\documentclass{amsart}

\usepackage{amssymb,amsmath,amscd,graphicx,amsthm,color,bm}
\usepackage[center]{caption}
\newtheorem{theorem}{Theorem}
\newtheorem{lemma}[theorem]{Lemma}
\newtheorem{proposition}[theorem]{Proposition}
\newtheorem{corollary}[theorem]{Corollary}

\theoremstyle{definition}

\theoremstyle{remark}
\newtheorem{remark}[theorem]{Remark}

\numberwithin{equation}{section}
\numberwithin{theorem}{section}

\def\A{{\mathcal A}}
\def\AA{{\mathbb A}}

\def\BD{G}
\def\C{{\mathbb C}}
\def\CC{{\mathcal C}}
\def\D{{\mathfrak d}}
\def\DD{{\mathcal D}}

\def\FF{{\mathcal F}}
\def\FFF{{\mathcal F}}
\def\ttF{{\tt F}}
\def\ttD{{\tt D}}
\def\G{{\mathcal G}}
\def\cGamma{\mathring{\Gamma}}
\def\H{\mathcal H}
\def\J{\mathcal J}

\def\L{{\mathcal L}}
\def\bL{{\mathbf L}} 
\def\M{{\mathcal M}}

\def\nuu{U}
\def\O{{\mathcal O}}
\def\P{{\mathcal P}}

\def\Q{{\mathbb Q}}

\def\cPi{\mathring{\Pi}}
\def\R{{\mathbb R}}

\def\S{{\mathbf S}}

\def\UU{\overline{\A}}

\def\X{{\mathcal X}}

\def\Z{{\mathbb Z}}
\def\ZZ{{\mathcal Z}}

\def\bfG{\mathbf \Gamma}
\def\bfGr{\bfG^{{\rm r}}}
\def\bfGc{\bfG^{{\rm c}}}
\def\tbfGr{\tilde\bfG^{{\rm r}}}
\def\tbfGc{\tilde\bfG^{{\rm c}}}
\def\Gammar{{\Gamma^\er}}
\def\Gammac{{\Gamma^\ec}}
\def\Lo{{(\L^1)}}
\def\Ld{{(\L^2)}}
\def\nalo{{\left(\nabla_{\L}^1\right)}}
\def\nald{{\left(\nabla_{\L}^2\right)}}
\def\lgrado{{\left(\L^1\nabla_{\L}^1\right)}}
\def\lgradd{{\left(\L^2\nabla_{\L}^2\right)}}
\def\gradlo{{\left(\nabla_{\L}^1\L^1\right)}}
\def\gradld{{\left(\nabla_{\L}^2\L^2\right)}}
\def\badB{B^{\romon}}
\def\badC{B^{\romtw}}
\def\badBp{B^{\romth}}
\def\badG{B^{\romfo}}
\def\badD{\bar B^{\romon}}
\def\badE{\bar B^{\romtw}}
\def\badDp{\bar B^{\romth}}
\def\badF{\bar B^{\romfo}}

\def\b{\mathfrak b}

\def\ceta{\mathring{\eta}}
\def\cgamma{{\mathring{\gamma}}}
\def\cgammar{{\cgamma^\er}}
\def\cgammac{{{\cgamma^\ec}}}

\def\cxi{\mathring{\xi}}

\def\ea{{\rm a}}

\def\ec{{\rm c}}
\def\el{{\rm l}}
\def\er{{\rm r}}
\def\fy{\varphi}
\def\tfy{\tilde\fy}

\def\ttd{{\tt d}}
\def\ttf{{\tt f}}
\def\ttg{{\tt g}}
\def\ttw{{\tt w}}
\def\g{\mathfrak g}
\def\gammar{{\gamma^\er}}
\def\gammac{{\gamma^\ec}}

\def\gl{\mathfrak g\mathfrak l}

\def\h{\mathfrak h}
\def\ii{{\hat\imath}}
\def\jj{\hat\jmath}

\def\ml{{\mathfrak l}}
\def\n{\mathfrak n}
\def\one{\mathbf 1}

\def\pp{{\mathfrak p}}

\def\q{{\bf q}}

\def\sl{\mathfrak {sl}}

\def\wB{{\widetilde{B}}}

\def\x{{\bf x}}

\def\Diag{\operatorname{Diag}}

\def\End{\operatorname{End}}

\def\Id{{\operatorname {Id}}}
\def\Ima{{\operatorname {Im}}}
\def\Kil{\langle \cdot,\cdot\rangle}

\def\Mat{\operatorname{Mat}}

\def\Poi{{\{\cdot,\cdot\}}}

\def\Tr{\operatorname{Tr}}

\def\deg{{\operatorname{deg}}}
\def\diag{\operatorname{diag}}
\def\dim{\operatorname{dim}}
\def\dnabla{{\raisebox{2pt}{$\bigtriangledown$}}\negthinspace}

\def\romon{\mbox{\tiny\rm I}}
\def\romtw{\mbox{\tiny\rm II}}
\def\romth{\mbox{\tiny\rm III}}
\def\romfo{\mbox{\tiny\rm IV}}

\def\:{{:\ }}

\begin{document}

\title
{A plethora of cluster structures on $GL_n$}

\author{Misha Gekhtman}

\address{Department of Mathematics, University of Notre Dame, Notre Dame,
IN 46556}
\email{mgekhtma@nd.edu}

\author{Michael Shapiro}
\address{Department of Mathematics, Michigan State University, East Lansing,
MI 48823, and National Research University Higher School of Economics, Russia}
\email{mshapiro@math.msu.edu}

\author{Alek Vainshtein}
\address{Department of Mathematics \& Department of Computer Science, University of Haifa, Haifa,
Mount Carmel 31905, Israel}
\email{alek@cs.haifa.ac.il}

\begin{abstract}
We continue the study of multiple cluster structures in the rings of regular functions on
$GL_n$, $SL_n$ and $\Mat_n$ that are compatible with Poisson—-Lie  and
Poisson-homogeneous structures. According to our initial conjecture, each class in the Belavin--Drinfeld classification of Poisson--Lie structures on a semisimple complex group $\G$ corresponds to a cluster structure in $\O(\G)$. Here we prove this conjecture for a large subset of Belavin--Drinfeld (BD) data of $A_n$ type, which includes all the previously known examples. 
Namely, we subdivide all possible $A_n$ type BD data into oriented and non-oriented kinds. 
We further single out BD data satisfying a certain combinatorial condition that we call aperiodicity and prove that for any 
oriented BD data of this kind there exists a regular cluster structure compatible with the corresponding Poisson--Lie bracket. In fact, we extend the aperiodicity condition to pairs of oriented BD data and prove a more general result that establishes an existence of a regular cluster structure on $SL_n$ compatible with a Poisson bracket homogeneous with respect to the right and left action of two copies of $SL_n$ equipped with two different Poisson-Lie brackets. Similar results hold for aperiodic non-oriented BD data, but the analysis of the corresponding regular cluster structure is more involved and not given here. If the aperiodicity condition is not satisfied, a compatible cluster structure has to be replaced with a generalized cluster structure. We will address these situations in future publications. 
\end{abstract}


\maketitle

\tableofcontents

\section{Introduction}

In this paper we continue the systematic study of multiple cluster structures in the rings of regular functions on
$GL_n$, $SL_n$ and $\Mat_n$  started in  \cite{GSVM, GSVPNAS, GSVMem}. It follows an approach developed and implemented 
in \cite{GSV1, GSV2, GSVb} for constructing cluster structures on algebraic varieties. 

Recall that given a smooth complex  algebraic Poisson variety $\left(\M, \Poi\right)$, a compatible cluster structure $\CC_\M$ on $\M$ 
is a collection of coordinate charts (called clusters) comprised of regular functions with simple birational 
transition maps between charts (called cluster transformations, see \cite{FZ1}) such that the logarithms of any two 
functions in the same chart have a constant Poisson bracket. Once found, any such chart can be used as a starting point, 
and our construction allows us to restore the whole $\CC_\M$, provided the arising birational maps preserve regularity.
Algebraic structures corresponding to $\CC_\M$ (the cluster algebra and the upper cluster algebra)
are closely related to the ring $\O(\M)$ of regular functions on $\mathcal{M}$. 
In fact, under certain rather mild conditions, $\O(\M)$ can be obtained by tensoring the upper cluster
algebra with $\C$, see \cite{GSVb}.

This construction was applied in \cite[Ch.~4.3]{GSVb} to double Bruhat cells in semisimple Lie groups
equipped with (the restriction of) the {\em standard\/} Poisson--Lie structure. It was shown that
the resulting cluster structure coincides with the one built in \cite{BFZ}. 
The standard Poisson--Lie structure is a particular case of Poisson--Lie structures corresponding to quasi-triangular
Lie bialgebras. Such structures are associated with solutions to the classical Yang--Baxter equation.
Their complete classification was obtained by Belavin and Drinfeld in \cite{BD}. Solutions are parametrized by the data that consists of a continuous and a discrete component. The former is governed by the latter, called the Belavin--Drinfeld triple.  It is defined as an isometry between two subsets of positive roots in the root system of the Lie algebra of the corresponding semisimple Lie group.
In \cite{GSVM} we conjectured that any such solution gives rise to a compatible cluster structure on 
this Lie group. This conjecture was verified in \cite{Eis} for $SL_5$ and proved in \cite{Eis1, Eis2} for all
simplest non-trivial Belavin--Drinfeld triples in $SL_n$ and in \cite{GSVMem} for the Cremmer--Gervais case.

In this paper we extend these results to a wide class of Belavin--Drinfeld triples in $SL_n$.  We define a subclass 
of {\em oriented\/} triples for which the corresponding isometry preserves orientation of intervals in the Dynkin diagram,
see Section \ref{sec:combdata}, and encode the corresponding information in a combinatorial
object called a Belavin--Drinfeld graph. Our main result claims that the conjecture of \cite{GSVM} holds true whenever
the corresponding Belavin--Drinfeld graph is acyclic. In this case the structure of the Belavin--Drinfeld graph is mirrored in
the explicit construction of the initial cluster,  see Section \ref{thebasis}. 
In fact, we have proved a stronger result: given two oriented 
Belavin--Drinfeld triples in $SL_n$ we define the graph of the pair, and if this graph possesses a certain 
acyclicity property then the Poisson bracket defined by the pair (note that it is not Poisson--Lie anymore) gives rise 
to a compatible cluster structure on $SL_n$. The construction we present in this paper can be adjusted to treat non-oriented triples, but the analysis is more involved and will be described elsewhere.

If the Belavin--Drinfeld graph has cycles then the conjecture of \cite{GSVM}
needs to be modified: one has to consider generalized cluster structures instead of the ordinary ones. 
Unlike in the acyclic case, our construction of the basis presented in Section~\ref{thebasis} then necessitates consideration of infinite periodic staircase matrices; consequently, defining  the initial seed and generalized cluster transformations require additional tools. These tools were developed in \cite{GSVpest}, where appropriate identities for minors of such matrices were described, along with the application to a construction of a generalized cluster structure in the smallest case when cycles are present.

In \cite{GY}, Goodearl and Yakimov developed a uniform approach for constructing cluster algebra structures in 
symmetric Poisson nilpotent algebras using sequences of Poisson-prime elements in chains of Poisson unique factorization
domains. These results apply to a large class of Poisson varieties, e.g., Schubert cells in Kac--Moody groups viewed as
Poisson subvarieties with respect to the standard Poisson-Lie bracket. It is worth pointing out, however, that the approach of  \cite{GY}, in its current form, does not seem to be applicable to the situation we consider here. This is evident from the fact that for cluster structures constructed in \cite{GY}, the cluster algebra and the corresponding upper cluster algebra always coincide. In contrast, as we have shown in \cite{GSVPNAS}, the simplest non-trivial Belavin--Drinfreld data 
in $SL_3$ results in a strict inclusion of the cluster algebra into  the upper cluster algebra. 

The paper is organized as follows. Section \ref{sec:prelim} contains a concise description of necessary definitions and results on cluster algebras and Poisson--Lie groups. Section~\ref{sec:mainres} presents main constructions and results. The Belavin--Drinfeld graph and related combinatorial data are defined in Section~\ref{sec:combdata}. 
The same section contains the formulations of the main Theorems~\ref{mainth} and~\ref{genmainth}. 
An explicit construction of the initial cluster is contained in Section~\ref{thebasis} and summarized in 
Theorem~\ref{logcanbasis}. Section~\ref{sec:basis} is dedicated to the proof of this theorem. 
The quiver that together with the initial cluster defines the compatible cluster 
structure is built in Section \ref{thequiver}, see Theorem~\ref{quiver} whose proof is contained in Section~\ref{sec:quiver}. 
Section \ref{outline} outlines the proof of the main Theorems~\ref{mainth} and~\ref{genmainth}.
It contains, inter alia, Theorem~\ref{prototype} 
that enables us to implement the induction step in the proof of an isomorphism between the constructed upper cluster algebra and the ring of regular functions on $\Mat_n$. A detailed constructive proof of this isomorphism is the subject of 
Section~\ref{sec:induction}. Section~\ref{sec:regtor} is devoted to showing that cluster structures we constructed are regular and admit a global toric action.

Our research was supported in part by the NSF research grants DMS \#1362801 and DMS \#1702054 (M.~G.), NSF research grants DMS \#1362352 and  DMS-1702115, and International Laboratory of Cluster Geometry NRU HSE, RF Government grant ag.~\# 2020-220-08-7077
(M.~S.), and ISF grants \#162/12 and \#1144/16 (A.~V.).  
While working on this project, we benefited from support of the following institutions and programs: 
Universit\'e Claude Bernard Lyon 1 (M.~S., Spring 2016), University of Notre Dame (A.~V., Spring 2016), 
Research in Pairs Program at the Mathematisches Forschungsinstitut Oberwolfach (M.~G., M.~S., A.~V., Summer 2016), 
Max Planck Institute for Mathematics, Bonn (M.~G.~and A.~V., Fall 2016), Bernoulli Brainstorm Program at EPFL, 
Lausanne (M.~G.~and A.~V., Summer 2017), Research in Paris Program at the Institut Henri Poincar\'e (M.~G., M.~S., A.~V., 
Fall 2017), Institute Des Hautes \'Etudes Scientifiques in  (M.~G.~and A.~V., Fall 2017), Mathematical Institute of the University of Heidelberg (M.~G., Spring 2017 and Summer 2018), Michigan State University (A.~V., Fall 2018). This paper was finished during the joint visit of the authors to the University of Notre Dame Jerusalem Global Gateway and the University of Haifa in December 2018. We are grateful to all these institutions for their hospitality and outstanding working conditions they provided. Special thanks are due to Salvatore Stella who pointed to a mistake in the original proof of
Theorem \ref{logcanbasis}, to Dmitry Voloshin for his crucial help with proof-reading the paper, and to Gus Schrader, Alexander Shapiro and Milen Yakimov for valuable discussions. We also thank the anonymous referee for constructive suggestions that improved the presentation.

\section{Preliminaries} \label{sec:prelim}

\subsection{ Skew-symmetric cluster structures of geometric type and compatible Poisson brackets}
\label{sec:cluster}
Let $\FFF$ be the field of rational functions in $N+M$ independent variables
with rational coefficients. There are $M$  distinguished variables; they are denoted $x_{N+1},\dots,x_{N+M}$ and 
called {\em frozen\/}, or {\em stable\/}.  The $(N+M)$-tuple  $\x=(x_1,\dots,\allowbreak x_{N+M})$ is called a {\em cluster\/}, 
and its elements $x_1,\dots,x_N$ are called {\em cluster variables\/}. The {\it quiver\/} $Q$ is a directed 
multigraph with no cycles of length~$!$ and~$2$ on the vertices $1,\dots,N+M$ corresponding to all variables; the vertices corresponding to frozen variables are called frozen. An arrow going from a vertex $i$ to a vertex $j$ is denoted $i\to j$. The pair $\Sigma=(\x,Q)$ is called a {\em seed}.

Given a seed as above, the {\em adjacent cluster\/} in direction $k$, $1\le k\le N$, is defined by
$\x'=(\x\setminus\{x_k\})\cup\{x'_k\}$, where the new cluster variable $x'_k$ is given by the {\em exchange relation}
\begin{equation*}
x_kx'_k=\prod_{k\to i}x_i+\prod_{i\to k}x_i.
\end{equation*}

The {\em quiver mutation\/} of $Q$ in direction $k$ is given by the following three steps: (i) for any two-arrow
path $i\to k\to j$ in $Q$, $e(i,j)$ arrow $i\to j$ are added, where $e(i,j)$ is the number of two-arrow paths $i\to k\to j$;
(ii) every arrow $j\to i$ (if it exists) annihilates with an arrow $i\to j$; (iii) all arrow $i\to k$ and all
arrows $k \to i$ are reversed.
The resulting quiver is denoted $Q'=\mu_k(Q)$. It is sometimes convenient to represent the quiver by an
$N\times(N+M)$ integer matrix $B=B(Q)$ called the {\it exchange matrix}, where $b_{ij}$ is the number of arrows 
$i\to j$ in $Q$. Note that the  
principal part of $B$ is skew-symmetric (recall that the principal part of a rectangular matrix  
is its maximal leading square submatrix). 

Given a seed $\Sigma=(\x,Q)$, we say that a seed $\Sigma'=(\x',Q')$ is {\em adjacent\/} to $\Sigma$ (in direction
$k$) if $\x'$ is adjacent to $\x$ in direction $k$ and $Q'=\mu_k(Q)$. Two seeds are {\em mutation equivalent\/} if 
they can be connected by a sequence of pairwise adjacent seeds. 
The set of all seeds mutation equivalent to $\Sigma$ is called the {\it cluster structure\/} 
(of geometric type) in $\FFF$ associated with $\Sigma$ and denoted by $\CC(\Sigma)$; in what follows, 
we usually write  just $\CC$ instead. 

Let $\AA$ be a {\em ground ring\/} satisfying the condition
\[
\Z[x_{N+1},\dots,x_{N+M}]\subseteq\AA\subseteq\Z[x_{N+1}^{\pm1},\dots,x_{N+M}^{\pm1}]
\]
(we write $x^{\pm1}$ instead of $x,x^{-1}$).
Following \cite{FZ1, BFZ}, we associate with $\CC$ two algebras of rank $N$ over $\AA$: 
the {\em cluster algebra\/} $\A=\A(\CC)$, which 
is the $\AA$-subalgebra of $\FF$ generated by all cluster
variables in all seeds in $\CC$, and the {\it upper cluster algebra\/}
$\UU=\UU(\CC)$, which is the intersection of the rings of Laurent polynomials over $\AA$ in cluster variables
taken over all seeds in $\CC$. The famous {\it Laurent phenomenon\/} \cite{FZ2}
claims the inclusion $\A(\CC)\subseteq\UU(\CC)$. Note that originally upper cluster algebras were defined over the ring of Laurent polynomials 
in frozen variables. In \cite{GSVD} we proved that upper cluster algebras over subrings of this ring retain all properties 
of usual upper cluster algebras.
In what follows we assume that the ground ring is the polynomial ring in frozen variables, unless
explicitly stated otherwise. 

Let $V$ be a quasi-affine variety over $\C$, $\C(V)$ be the field of rational functions on $V$, and
$\O(V)$ be the ring of regular functions on $V$. Let $\CC$ be a cluster structure in $\FF$ as above.
Assume that $\{f_1,\dots,f_{N+M}\}$ is a transcendence basis of $\C(V)$. Then the map $\varphi: x_i\mapsto f_i$,
$1\le i\le N+M]$, can be extended to a field isomorphism $\varphi: \FF_\C\to \C(V)$,  
where $\FF_\C=\FF\otimes\C$ is obtained from $\FF$ by extension of scalars.
The pair $(\CC,\varphi)$ is called a cluster structure {\it in\/}
$\C(V)$ (or just a cluster structure {\it on\/} $V$), $\{f_1,\dots,f_{N+M}\}$ is called a cluster in
 $(\CC,\varphi)$.
Occasionally, we omit direct indication of $\varphi$ and say that $\CC$ is a cluster structure on $V$. 
A cluster structure $(\CC,\varphi)$ is called {\it regular\/}
if $\varphi(x)$ is a regular function for any cluster variable $x$. 
The two algebras defined above have their counterparts in $\FF_\C$ obtained by extension of scalars; they are
denoted $\A_\C$ and $\UU_\C$.
If, moreover, the field isomorphism $\varphi$ can be restricted to an isomorphism of 
$\A_\C$ (or $\UU_\C$) and $\O(V)$, we say that 
$\A_\C$ (or $\UU_\C$) is {\it naturally isomorphic\/} to $\O(V)$.

Let $\Poi$ be a Poisson bracket on the ambient field $\FFF$, and $\CC$ be a cluster structure in $\FFF$. 
We say that the bracket and the cluster structure are {\em compatible\/} if, for any 
cluster $\x=(x_1,\dots,x_{N+M})$,  one has $\{x_i,x_j\}=\omega_{ij} x_ix_j$,
where $\omega_{ij}\in\Q$ are constants for all $1\le i,j\le N+M$. The matrix
$\Omega^{\x}=(\omega_{ij})$ is called the {\it coefficient matrix\/}
of $\Poi$ (in the basis $\x$); clearly, $\Omega^{\x}$ is
skew-symmetric. The notion of compatibility  extends to Poisson brackets on $\FF_\C$ without any changes.

 Fix an arbitrary cluster $\x=(x_1,\dots,x_{N+M})$ and define a {\it local toric action\/} of rank $s$ at $\x$ as a map  
\begin{equation}
\x\mapsto \left ( x_i \prod_{\alpha=1}^s q_\alpha^{w_{i\alpha}}\right )_{i=1}^{N+M},\qquad
\q=(q_1,\dots,q_s)\in (\C^*)^s,
\label{toricact}
\end{equation}
where $W=(w_{i\alpha})$ is an integer $(N+M)\times s$ {\it weight matrix\/} of full rank. 
Let $\x'$ be another cluster in $\CC$, then the corresponding local toric action defined by the weight matrix $W'$
is {\it compatible\/} with the local toric action \eqref{toricact} if it commutes with the sequence of cluster transformations that takes $\x$ to $\x'$.
 If local toric actions for all clusters are compatible, they define a {\it global toric action\/} on $\CC$ called 
the $\CC$-extension of the local toric action \eqref{toricact}. 

\subsection{Poisson--Lie groups}
A reductive complex Lie group $\G$ equipped with a Poisson bracket $\Poi$ is called a {\em Poisson--Lie group\/}
if the multiplication map $\G\times \G \ni (X,Y) \mapsto XY \in \G$
is Poisson. Perhaps, the most important class of Poisson--Lie groups
is the one associated with quasitriangular Lie bialgebras defined in terms of  {\em classical R-matrices\/} 
(see, e.~g., \cite[Ch.~1]{CP}, \cite{r-sts} and \cite{Ya} for a detailed exposition of these structures).

Let $\g$ be the Lie algebra corresponding to $\G$, $\Kil$ be an invariant nondegenerate form on $\g$,
 and let $\mathfrak{t}\in \g\otimes\g$ be the corresponding Casimir element.
For an arbitrary element $r=\sum_i a_i\otimes b_i\in\g\otimes\g$ denote
\[
[[r,r]]=\sum_{i,j} [a_i,a_j]\otimes b_i\otimes b_j+\sum_{i,j} a_i\otimes [b_i,a_j]\otimes b_j+
\sum_{i,j} a_i\otimes a_j\otimes [ b_i,b_j]
\]
and $r^{21}=\sum_i b_i\otimes a_i$.
A {\em classical R-matrix} is an element $r\in \g\otimes\g$ that satisfies
{\em the classical Yang-Baxter equation (CYBE)\/} $[[r, r]] =0$
together with the condition $r + r^{21} = \mathfrak{t}$.
The Poisson--Lie bracket on $\G$ that corresponds to $r$ can be written as
\begin{equation}\label{sklyabra}
\begin{aligned}
\{f^1,f^2\}_r &= \langle R_+(\nabla^L f^1), \nabla^L f^2 \rangle - \langle R_+(\nabla^R f^1), \nabla^R f^2 \rangle\\
&= \langle R_-(\nabla^L f^1), \nabla^L f^2 \rangle - \langle R_-(\nabla^R f^1), \nabla^R f^2 \rangle,
\end{aligned}
\end{equation} 
where $R_+,R_- \in \End \g$ are given by $\langle R_+ \eta, \zeta\rangle = \langle r, \eta\otimes\zeta \rangle$, 
$-\langle R_- \zeta, \eta\rangle = \langle r, \eta\otimes\zeta \rangle$ for any $\eta,\zeta\in \g$ and  
$\nabla^L$, $\nabla^R$ are the right and the left gradients of functions on $\G$ with respect to $\Kil$ 
defined by
\begin{equation*}
\left\langle \nabla^R f(X),\xi\right\rangle=\left.\frac d{dt}\right|_{t=0}f(e^{t\xi}X),  \quad
\left\langle \nabla^L f(X),\xi\right\rangle=\left.\frac d{dt}\right|_{t=0}f(Xe^{t\xi})
\end{equation*}
for any $\xi\in\g$, $X\in\G$.

Following \cite{r-sts}, let us recall the construction of the {\em  Drinfeld double}. First, note that CYBE implies that
\begin{equation}\label{g_pm}
\g_+=\Ima(R_+),\qquad \g_-=\Ima(R_-)
\end{equation}
are subalgebras in $\g$. The double of $\g$ is 
$D(\g)=\g  \oplus \g$ equipped with an invariant nondegenerate bilinear form
$$
\langle\langle (\xi,\eta), (\xi',\eta')\rangle\rangle = \langle \xi, \xi'\rangle - \langle \eta, \eta'\rangle. 
$$
Define subalgebras $\D_\pm$ of $D(\g)$ by
\begin{equation}\label{ddeco}
\D_+=\{( \xi,\xi)\: \xi \in\g\}, \quad \D_-=\{ (R_+(\xi),R_-(\xi))\: \xi \in\g\},
\end{equation}
then $\D_\pm$ are isotropic subalgebras of $D(\g)$ and $D(\g)= \D_+ \dot + \D_-$. In other words,
$(D(\g), \D_+, \D_-)$ is {\em a Manin triple}. Then the operator $R_D= \pi_{\D_+} - \pi_{\D_-}$ can be used to define 
a Poisson--Lie structure on $D(\G)=\G\times \G$, the double of the group $\G$, via
\begin{equation}
\{f^1,f^2\}^D_r = \frac{1}{2}\left (\langle\langle R_D(\dnabla^L f^1), \dnabla{^L} f^2 \rangle\rangle 
- \langle\langle R_D(\dnabla^R f^1), \dnabla^R f^2 \rangle\rangle \right),
\label{sklyadouble}
\end{equation}
where $\dnabla^R$ and $\dnabla^L$ are right and left gradients with respect to $\langle\langle \cdot ,\cdot \rangle\rangle$.
Restriction of this bracket to $\G$ identified with the diagonal subgroup of $D(\G)$ (whose Lie algebra is $\D_+$) 
coincides with the Poisson--Lie bracket $\Poi_r$ on $\G$. Let $\DD_-$ be the subgroup of $D(\G)$ that corresponds to $\D_-$
Double cosets of $\DD_-$ in $D(\G)$ play an important role in the description of symplectic leaves in Poisson--Lie
groups $\G$ and $D(\G)$, see \cite{Ya}.

The classification of classical R-matrices for simple complex Lie groups was given by Belavin and Drinfeld in \cite{BD}.
Let $\G$ be a simple complex Lie group, $\Phi$ be the root system associated with its Lie algebra $\g$, $\Phi^+$ be the set of positive roots, and $\Pi\subset \Phi^+$ be the set of positive simple roots. 
A {\em Belavin--Drinfeld triple} $\bfG=(\Gamma_1,\Gamma_2, \gamma)$ (in what follows, a {\em BD triple\/})
consists of two subsets $\Gamma_1,\Gamma_2$ of $\Pi$ and an isometry $\gamma\:\Gamma_1\to\Gamma_2$ nilpotent in the 
following sense: for every $\alpha \in \Gamma_1$ there exists $m\in\mathbb{N}$ such that $\gamma^j(\alpha)\in \Gamma_1$ 
for $j\in [0,m-1]$, but $\gamma^m(\alpha)\notin \Gamma_1$.

 The isometry $\gamma$ yields an isomorphism, also denoted by $\gamma$, between the Lie subalgebras $\g_{\Gamma_1}$ 
and $\g_{\Gamma_2}$ that correspond to $\Gamma_1$ and $\Gamma_2$. It is uniquely defined by the property 
$\gamma e_\alpha = e_{\gamma(\alpha)}$ for $\alpha\in \Gamma_1$, where $e_\alpha$ is the Chevalley generator corresponding to 
the root $\alpha$. The isomorphism $\gamma^*\: \g_{\Gamma_2} \to \g_{\Gamma_1}$ is defined as the adjoint to $\gamma$ with respect to the form $\Kil$. 
It is given by $\gamma^* e_{\gamma(\alpha)}=e_{\alpha}$ for $\gamma(\alpha)\in \Gamma_2$.
 Both $\gamma$ and $\gamma^*$ can be extended to maps of $\g$ to itself by applying first the orthogonal
projection on $\g_{\Gamma_1}$ (respectively, on $\g_{\Gamma_2}$) with respect to $\Kil$; clearly, the extended
maps remain adjoint to each other. Note that the restrictions of $\gamma$ and $\gamma^*$ to the positive and the negative nilpotent subalgebras $\n_+$ and $\n_-$ of $\g$ are Lie algebra homomorphisms 
of $\n_+$ and $\n_-$  to themselves, and $\gamma(e_{\pm\alpha})=0$
for all $\alpha\in\Pi\setminus\Gamma_1$.

 By the classification theorem, each classical R-matrix is equivalent to an R-matrix from the 
{\it Belavin--Drinfeld class\/} defined by a BD triple $\bfG$. 
Following \cite{ESS}, we write down an expression for the members of this class:
\begin{equation}
\label{r-matrix}
r = \frac 1 2 \Omega_\h + s + \sum_{\alpha} e_{-\alpha}\otimes e_\alpha + 
\sum_{\alpha} e_{-\alpha}\wedge \frac \gamma {1-\gamma} e_\alpha; 
\end{equation}
here the summation is over the set of all positive roots, $\Omega_\h \in \h \otimes \h$ is given by  
$\Omega_\h= \sum h_\alpha \otimes \hat{h}_\alpha$ where $\{h_\alpha\}$ is the standard basis of the Cartan subalgebra $\h$,
$\{\hat h_\alpha\}$ is the dual basis with respect to the restriction of $\Kil$ to $\h$, 
and $s \in \h \wedge \h$ satisfies
\begin{equation}
\label{s-eq}
\left ( (1-\gamma) \alpha \otimes \one \right ) (2 s) = \left ( (1+\gamma) \alpha \otimes \one \right ) \Omega_\h
\end{equation}
for any $\alpha \in \Gamma_1$. Solutions to \eqref{s-eq} form a linear space of dimension $\frac{k_{\bfG}(k_{\bfG}-1)}2$ 
with $k_{\bfG}=|\Pi\setminus\Gamma_1|$. 
More precisely, define
\begin{equation}\label{smalltorus}
\h_\bfG=\{ h\in\h \ : \ \alpha(h)=\beta(h)\ \mbox{if}\ \gamma^j(\alpha)=\beta \quad\text{for some $j$}\}, 
\end{equation}
then $\dim\h_\bfG=k_\bfG$, and if $s'$ is a fixed solution of \eqref{s-eq}, then
every other solution has a form $s=s' + s_0$, where $s_0$ is an arbitrary element of $\h_\bfG\wedge\h_\bfG$.
The subalgebra $\h_\bfG$ defines a torus $\H_\bfG=\exp \h_\bfG$ in $\G$.

Let $\pi_{>}$, $\pi_{<}$ be projections of  
$\g$ onto $\n_+$ and $\n_-$, $\pi_\h$ be the projection onto $\h$. 
It follows from \eqref{r-matrix} that $R_+$ in \eqref{sklyabra} is given by
\begin{equation}
\label{RplusSL}
R_+=\frac1{1-\gamma}\pi_{>}-\frac{\gamma^*}{1-\gamma^*}\pi_{<}+ \left (\frac 1 2 + S\right )\pi_\h,
\end{equation}
where 
$S \in \End \h$ is skew-symmetric with respect to the restriction of $\Kil$ to $\h$ and 
satisfies  $\langle S h, h' \rangle = \langle  s, h \otimes h'\rangle$ for any $h,h' \in \h$  
and conditions  
\begin{equation}
\label{S-eq}
 S (1-\gamma) h_\alpha = \frac 1 2(1+\gamma) h_\alpha
\end{equation}
for any $\alpha \in \Gamma_1$, translated from \eqref{s-eq}.
 
For an R-matrix given by \eqref{r-matrix}, subalgebras $\g_\pm$ from \eqref{g_pm} are contained in parabolic subalgebras
$\pp_\pm$ of $\g$ determined by the BD triple: $\pp_+$ contains $\b_+$ and all the negative root spaces in 
$\g_{\Gamma_1}$, while $\pp_-$ contains $\b_-$ and all the positive root spaces in $\g_{\Gamma_2}$. Then one has
\begin{equation}
\label{parabolics}
\pp_+ = \g_+ \oplus \h_+, \qquad \pp_- = \g_- \oplus \h_-
\end{equation}
with $\h_\pm \subset \h$. An explicit description of subalgebras $\h_\pm$ can be found, e.g., in \cite[Sect.~3.1]{Ya}. 
Let $\ml_\pm$ denote the Levi component of $\pp_\pm $. Then $\ml_+=\g_{\Gamma_1}$, $\ml_-=\g_{\Gamma_2}$, and the Lie algebra isomorphism $\gamma$ described above restricts to  $\ml_+ \cap \g_+ \to \ml_- \cap \g_-$. This allows to describe 
the subalgebra $\D_-$ as 
\begin{multline}\label{d_-}
\D_-=\{ (\xi_+,\xi_-))\: \xi_\pm \in\g_\pm,\ \gamma(\pi_{\ml_+ \cap \g_+} \xi_+)= \pi_{\ml_- \cap \g_-}\xi_- \}\\
\subset \{ (\xi_+,\xi_-))\: \xi_\pm \in\pp_\pm,\ \gamma(\pi_{\ml_+} \xi_+)= \pi_{\ml_-}\xi_- \},
\end{multline}
where $\pi_\cdot$ are the projections to the corresponding subalgebras.

In what follows we will use a Poisson bracket on $\G$ that is a generalization of the bracket \eqref{sklyabra}.
Let $r, r'$ be two classical R-matrices, and $R_+, R'_+$ be the corresponding operators, then we write
\begin{equation}
\{f^1,f^2\}_{r,r'} = \langle R'_+(\nabla^L f^1), \nabla^L f^2 \rangle - \langle R_+(\nabla^R f^1), \nabla^R f^2 \rangle.
\label{sklyabragen}
\end{equation} 
By \cite[Proposition 12.11]{r-sts}, the above expression defines a Poisson bracket, which is not Poisson--Lie unless $r=r'$,
in which case $\{f^1,f^2\}_{r,r}$ evidently coincides with $\{f^1,f^2\}_{r}$. The bracket \eqref{sklyabragen} defines a Poisson homogeneous structure on $\G$ with respect to the left and right multiplication by Poisson--Lie groups $(\G,\Poi_r)$ and
$(\G,\Poi_{r'})$, respectively.
The bracket on the Drinfeld double that corresponds to $\{f^1,f^2\}_{r,r'}$ 
is defined similarly to \eqref{sklyadouble} via
\begin{equation}
\{f^1,f^2\}^D_{r,r'} = \frac{1}{2}\left (\langle\langle R'_D(\dnabla^L f^1), \dnabla{^L} f^2 \rangle\rangle 
- \langle\langle R_D(\dnabla^R f^1), \dnabla^R f^2 \rangle\rangle \right).
\label{sklyadoublegen}
\end{equation}

\section{Main results and the outline of the proof}\label{sec:mainres}

\subsection{Combinatorial data and main results} \label{sec:combdata}
In this paper, we only deal with $\g = \sl_n$, and hence 
$\Gamma_1$ and $\Gamma_2$ can be identified with subsets of $[1,n-1]$. Note that the isometry condition on $\gamma$ implies that 
if $i,i+1\in\Gamma_1$ then $\gamma(i+1)=\gamma(i)\pm1$. We say that $\bfG$ is {\it oriented\/}  
if $i,i+1\in\Gamma_1$ yields $\gamma(i+1)=\gamma(i)+1$. In other words, the orientation of every subset of $\Gamma_1$ that consists of consecutive roots is preserved by $\gamma$.

For any $i\in [1,n]$ put
\[
i_+=\min\{j\in [1,n]\setminus\Gamma_1\: \ j\ge i\}, \qquad
i_-=\max\{j\in [0,n]\setminus\Gamma_1\: \ j<i\}.
\]
The interval $\Delta(i)=[i_-+1,i_+]$ is called the {\it $X$-run\/} of $i$. Clearly, all distinct $X$-runs form a 
partition of $[1,n]$. The $X$-runs are numbered consecutively from left to right. For example, let
$n=7$ and  $\Gamma_1=\{1,2,4\}$, then there are four $X$-runs: $\Delta_1=[1,3]$, $\Delta_2=[4,5]$, 
$\Delta_3=[6,6]$ and $\Delta_4=[7,7]$. Clearly, $\Delta(2)=\Delta_1$, $\Delta(4)=\Delta_2$, etc.

In a similar way, $\Gamma_2$ defines another partition of $[1,n]$ into $Y$-runs $\bar\Delta(i)$.
For example, let in the above example $\Gamma_2=\{1,3,4\}$, then $\bar\Delta_1=[1,2]$, 
$\bar\Delta_2=[3,5]$, $\bar\Delta_3=[6,6]$ and $\bar\Delta_4=[7,7]$.

Runs of length one are called trivial. The map $\gamma$ induces a bijection on the sets of nontrivial $X$-runs and 
$Y$-runs: we say that $\bar\Delta_i=\gamma(\Delta_j)$ if there exists  
$k\in\Delta_j$ such that $\bar\Delta(\gamma(k))=\bar\Delta_i$. The inverse of the bijection $\gamma$ is denoted 
$\gamma^*$ (the reasons for this notation will become clear later). Let in the previous example $\gamma(1)=3, 
\gamma(2)=4, \gamma(4)=1$, then $\bar\Delta_1=\gamma(\Delta_2)$ and $\bar\Delta_2=\gamma(\Delta_1)$.

The {\it BD graph\/} $\BD_\bfG$ is defined as follows. The vertices of $\BD_\bfG$ are two copies of the set of 
positive simple roots identified with $[1,n-1]$. One of the sets is called the {\it upper\/} part of the graph, 
and the other is called the
{\it lower\/} part. A vertex $i\in\Gamma_1$ is connected with an {\it inclined\/} edge to the vertex 
$\gamma(i)\in\Gamma_2$. Finally, vertices $i$ and $n-i$ in the same part are connected with a {\it horizontal\/} edge. 
If $n=2k$ and $i=n-i=k$, the corresponding horizontal edge is a loop. The BD graph for the above example is shown
in~Fig.~\ref{fig:BDgraph} on the left. In the same figure on the right one finds the BD graph for the case of
$SL_6$ with $\Gamma_1=\{1,3,4\}$, $\Gamma_2=\{2,4,5\}$ and $\gamma\: i\mapsto i+1$.

\begin{figure}[ht]
\begin{center}
\includegraphics[height=3.5cm]{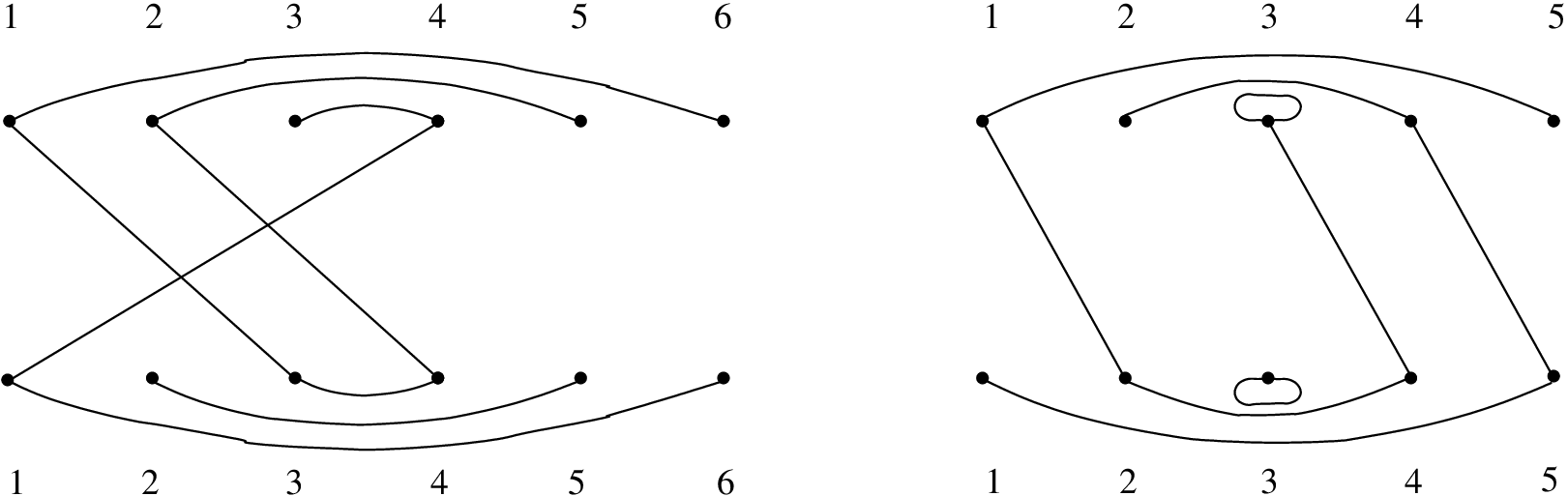}
\caption{BD graphs for aperiodic BD triples}
\label{fig:BDgraph}
\end{center}
\end{figure}

Clearly, there are four possible types of connected components in $\BD_\bfG$: a path, a path with a loop, a path with two loops,
and a cycle. We say that a BD triple $\bfG$ is {\em aperiodic\/} if each component in $\BD_\bfG$ is either a path or a path with a loop, and {\em periodic\/} otherwise. 
In what follows we assume that $\bfG$ is aperiodic. The case of periodic BD triples will be addressed in a separate paper.

\begin{remark}
\label{milen}
Let $w_0$ be the longest permutation in $S_n$.
Observe that horizontal edges in both rows of the BD graph can be seen as a depiction of the action of $\left ( -w_0\right)$ on the set of positive simple roots of $SL_n$.
Thus the BD graph can be used to analyze the properties of the map $w_0 \gamma w_0 \gamma^{-1}$. A map of this kind, with the pair $(w_0, w_0)$ replaced by a pair of elements of the Weyl group satisfying certain properties dictated by the BD triple in an arbitrary reductive Lie group, was defined in \cite[Sect.~5.1.1]{Ya} and utilized in the description of symplectic 
leaves of the corresponding Poisson--Lie structure.
\end{remark}

The main result of this paper states that the conjecture formulated in \cite{GSVM} holds for oriented aperiodic 
BD triples in $SL_n$. Namely, 

\begin{theorem}
\label{mainth}
For any oriented aperiodic Belavin--Drinfeld triple $\bfG=(\Gamma_1,\Gamma_2,\gamma)$ there exists a cluster structure
$\CC_\bfG$ on $SL_n$ such that

{\rm (i)}
the number of frozen variables is $2k_\bfG$, and the corresponding exchange matrix has a full rank;

{\rm (ii)} $\CC_\bfG$ is regular, and the corresponding upper cluster algebra $\UU_\C(\CC_\bfG)$ 
is naturally isomorphic to $\O(SL_n)$;

{\rm (iii)} the global toric action of $(\mathbb{C}^*)^{2k_\bfG}$ on $\CC_\bfG$ 
is generated by the action
of $\H_\bfG\times \H_\bfG$ on $SL_n$ given by $(H_1, H_2)(X) = H_1 X H_2$;

 {\rm (iv)} for any solution of CYBE that belongs to the Belavin--Drinfeld class specified  by $\bfG$, the corresponding Sklyanin bracket is compatible with $\CC_\bfG$;

{\rm (v)} a Poisson--Lie bracket on $SL_n$ is compatible with $\CC_\bfG$ only if it is a scalar multiple
of the Sklyanin bracket associated with a solution of CYBE that belongs to the Belavin--Drinfeld class specified  by $\bfG$.
\end{theorem}

This result was established previously for the Cremmer--Gervais case (given by $\gamma: i\mapsto i+1$ for $1\le i\le n-2$) 
in \cite{GSVMem} and for all cases when $k_\bfG=n-2$ in \cite{Eis1, Eis2}.

In fact, the construction above is a particular case of a more general construction. 
Let $r^{\rm r}$ and  $r^{\rm c}$ be two classical R-matrices that correspond to BD triples 
$\bfGr=(\Gamma_1^\er,\Gamma_2^{\rm r}, \gammar)$ and  $\bfGc=(\Gamma_1^\ec,\Gamma_2^\ec, \gammac)$,  which we call the {\em row} and the {\em column} BD triples, respectively.

Assume that both $\bfGr$ and  $\bfGc$ are oriented.
 Similarly to  the BD graph $\BD_{\bfG}$ for $\bfG$, one can define a  graph $\BD_{\bfGr, \bfGc}$ for the pair 
$(\bfGr, \bfGc)$ as follows. Take $\BD_{\bfGr}$ with all inclined edges directed downwards and 
$\BD_{\bfGc}$ in which all inclined edges are directed upwards. Superimpose these graphs by identifying the corresponding vertices.  In the resulting graph, for every pair of vertices $i, n -i$ in either top or bottom row there are two edges joining them. We give these edges opposite orientations. If $n$ is even, then we retain only one loop at each of the 
two vertices labeled $\frac{n}{2}$. The result is a directed graph $\BD_{\bfGr, \bfGc}$ on $2(n-1)$ vertices. 
For example, consider the case of $GL_5$ with $\bfGr=\left(\{1,2\}, \{2,3\}, 1\mapsto 2, 2\mapsto 3\right)$ and
$\bfGc=\left(\{1,2\}, \{3,4\}, 1\mapsto3, 2\mapsto4\right)$. The corresponding graph $\BD_{\bfGr, \bfGc}$ is shown 
on the left in Fig.~\ref{fig:altpaths}. For horizontal edges, no direction is indicated, which means that they can be traversed in both directions. The graph shown on in Fig.~\ref{fig:altpaths} on the right corresponds to 
the case of $GL_8$ with $\bfGr=\left(\{2,6\}, \{3,7\}, 2\mapsto 3, 6\mapsto7 \right)$ and
$\bfGc=\left(\{2,6\}, \{1,5\}, 6\mapsto1, 2\mapsto5\right)$. 

A directed path in $\BD_{\bfGr, \bfGc}$ is called {\em alternating\/} if horizontal and inclined edges in the path alternate. In particular, an edge is a (trivial) alternating path. 
An alternating path with coinciding endpoints and an even number of edges is called an {\em alternating cycle}. Similarly to
the decomposition of $\BD_{\bfG}$ into connected components, we can decompose the edge set of  $\BD_{\bfGr, \bfGc}$ into a disjoint union of maximal alternating paths and alternating cycles. 
If the resulting collection contains no alternating cycles, we call the pair $(\bfGr, \bfGc)$ {\em aperiodic\/}; clearly,
$(\bfG,\bfG)$ is aperiodic if and only if $\bfG$ is aperiodic. 
For the graph on the left in Fig.~\ref{fig:altpaths}, 
the corresponding maximal paths are $41\bar 2\bar 3 14$, $32\bar 3\bar 2$, $\bar1\bar4 23$, and $\bar4\bar1$ (here vertices in the lower part are marked with a dash for better visualization). None of them is an alternating cycle, so the corresponding pair is aperiodic.
 For the graph on the right in Fig.~\ref{fig:altpaths}, 
the path $62\bar3\bar5 26\bar7\bar1 6$ is an alternating cycle; the edges
$\bar1\bar7$ and $\bar5\bar3$ are trivial alternating paths.

\begin{figure}[ht]
\begin{center}
\includegraphics[height=3.5cm]{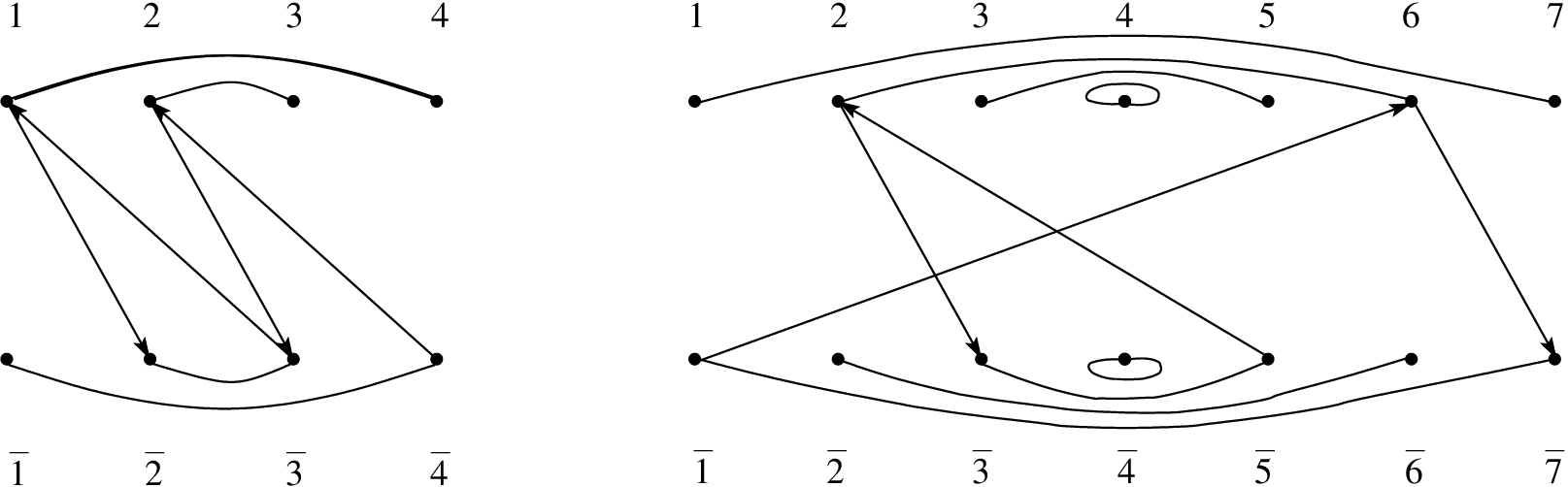}
\caption{Alternating paths and cycles in $\BD_{\bfGr, \bfGc}$}
\label{fig:altpaths}
\end{center}
\end{figure}

The following result generalizes the first four claims of Theorem \ref{mainth}

\begin{theorem}
\label{genmainth}
For any aperiodic pair of oriented Belavin--Drinfeld triples $(\bfGr, \bfGc)$ there exists a cluster structure
$\CC_{\bfGr, \bfGc}$ on $SL_n$ such that

{\rm (i)}
the number of frozen variables is $k_{\bfGr}+k_{\bfGc}$, and the corresponding exchange matrix has a full rank;

{\rm (ii)} $\CC_{\bfGr, \bfGc}$ is regular, and the corresponding upper cluster algebra $\UU_\C(\CC_{\bfGr, \bfGc})$ 
is naturally isomorphic to $\O(SL_n)$.

{\rm (iii)} the global toric action of $(\C^*)^{k_\bfGr+k_\bfGc}$ on $\CC_{\bfGr, \bfGc}$ is generated by the action
of $\H_{\bfGr}\times \H_{\bfGc}$ on $SL_n$ given by $(H_1, H_2)(X) = H_1 X H_2$.

 {\rm (iv)} for any pair of solutions of CYBE that belong to the Belavin--Drinfeld classes specified  by $\bfGr$ and 
$\bfGc$, the corresponding bracket~\eqref{sklyabragen} is compatible with $\CC_{\bfGr, \bfGc}$.
\end{theorem}

Following the approach suggested in \cite{GSVMem}, we will construct a cluster
structure on the space $\Mat_n$ of $n\times n$ matrices  and derive the required properties of $\CC_{\bfGr, \bfGc}$
from similar features of the latter cluster structure. Note that in the 
case of $GL_n$ we also obtain a regular cluster structure with the same properties, however, in this case the ring of regular
functions on $GL_n$ is isomorphic to the localization of the upper cluster algebra with
respect to $\det X$, which is equivalent to replacing the ground ring by the corresponding 
localization of the polynomial ring in frozen variables.
In what follows we use the same notation $\CC_{\bfGr, \bfGc}$ for all
three cluster structures and indicate explicitly which one is meant when needed.

\subsection{The basis}\label{thebasis}

Consider connected components of $\BD_\bfG$ for an aperiodic $\bfG$. 
The choice of the endpoint of a component  induces directions of its edges: the first edge is directed from the endpoint, the second one from the head of the first one, and so on. Note that for a path with a loop, each edge except for the loop gets two opposite directions. Consequently, the choice of an endpoint of a component defines  a matrix built of blocks curved out from 
two $n\times n$ matrices of indeterminates $X=(x_{ij})$ and $Y=(y_{ij})$. Each block is defined by a horizontal directed edge, that is, an edge whose head and tail belong to the same part of the graph. The block corresponding to
a horizontal edge $i\to (n-i)$ in the upper part, called an {\em $X$-block\/}, is the submatrix $X_{I}^{J}$ 
with $I=[\alpha,n]$ and $J=[1,\beta]$, where $\alpha=(n-i+1)_-+1$ is the leftmost point of the $X$-run 
containing $n-i+1$, and $\beta=i_+$ is the rightmost point of the $X$-run containing $i$. The entry $(n-i+1,1)$ is called the {\em exit point\/} of the $X$-block.
Similarly, the block corresponding to a horizontal edge $i\to (n-i)$ in the lower part, called a {\em $Y$-block\/}, 
is the submatrix $Y_{\bar I}^{\bar J}$ with $\bar I=[1,\bar\alpha]$ and $\bar J=[\bar\beta,n]$, where $\bar\alpha=i_+$ is the rightmost point of the $Y$-run containing $i$ and $\bar\beta=(n-i+1)_-+1$ is the leftmost point of the 
$Y$-run containing $n-i+1$. The entry $(1,n-i+1)$ is called the {\em exit point\/} of the $Y$-block. 
In the example shown in Fig.~\ref{fig:BDgraph} on the left, the edge $5\to2$ in the upper part defines the 
$X$-block $X_{[1,7]}^{[1,5]}$ with the exit 
point $(3,1)$, the edge $4\to3$ in the lower part defines the $Y$-block $Y_{[1,5]}^{[3,7]}$ with the exit point 
$(1,4)$, and the edge $1\to6$ in the upper part defines the $X$-block $X_{[7,7]}^{[1,3]}$ with the exit point 
$(7,1)$, see the left part of Fig.~\ref{fig:blockmatrix} where the exit points of the blocks are circled.

\begin{figure}[ht]
\begin{center}
\includegraphics[height=4.5cm]{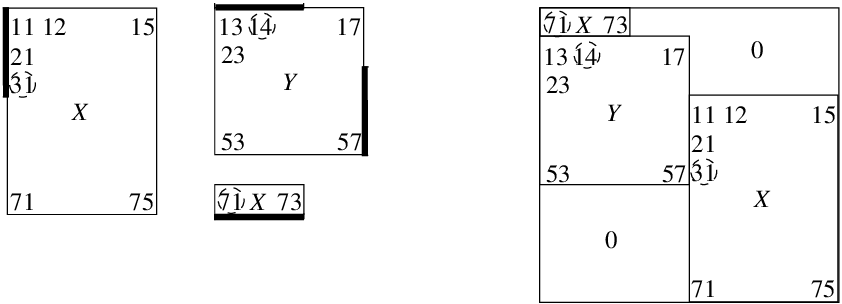}
\caption{Blocks and their gluing}
\label{fig:blockmatrix}
\end{center}
\end{figure}

The number of directed edges is odd  and the blocks of different types alternate; therefore, if this number equals 
$4b-1$, then there are $b$ blocks of each type. If there are $4b-3$ directed edges, there are $b$ blocks of one type and $b-1$ blocks of the other type. By adding at most two dummy blocks with empty sets of rows or columns at the beginning and at the end of the sequence, we may assume that the number of blocks of each type is equal, and that the first block is of $X$-type.
 
The blocks are glued together with the help of inclined edges whose head and tail belong to different parts of the graph. An inclined edge $i\to j$ directed downwards stipulates placing the entry $(j,n)$ of the $Y$-block defined by 
$j\to (n-j)$ immediately to the left of the entry $(i,1)$ of the $X$-block 
defined by $(n-i)\to i$. In other words, the two blocks are glued in such a way that $\Delta(\alpha)$ and 
$\bar\Delta(\bar\alpha)=\gamma(\Delta(\alpha))$ coincide. Similarly, an inclined edge $i\to j$ directed upwards stipulates placing the entry $(n,j)$ of the $X$-block defined by $j\to(n-j)$ immediately above the entry $(1,i)$  of the $Y$-block defined by $(n-i)\to i$. In other words, the two blocks are glued in such a way that $\bar\Delta(\bar\beta)$ and
$\Delta(\beta)=\gamma^*(\bar\Delta(\bar\beta))$ coincide. Clearly, the exit points of all blocks lie on the main diagonal of the resulting matrix. 
For example, the directed path $5\to2\to4\to3\to1\to6$ in the BD graph shown in Fig.~\ref{fig:BDgraph} on the left defines the gluing shown in Fig.~\ref{fig:blockmatrix} on the right. The runs along which the blocks are glued are shown in bold. The same path
traversed in the opposite direction defines a matrix glued from the blocks $X_{[1,7]}^{[1,6]}$, $Y_{[1,5]}^{[3,7]}$ and 
$X_{[6,7]}^{[1,3]}$.

Given an aperiodic pair $(\bfGr, \bfGc)$ and the decomposition of $\BD_{\bfGr, \bfGc}$ into maximal alternating paths, 
the blocks are defined in a similar way.  To each edge $i\to (n-i)$ in the upper part of 
$\BD_{\bfGr, \bfGc}$, assign the block $X_{I}^{J}$ with $I=[\alpha,n]$ and $J=[1,\beta]$, where $\alpha=(n-i+1)_-(\bfGr)+1$  and
$\beta=i_+(\bfGc)$ are defined by $X$-runs exactly as before except with respect to different BD triples  $\bfGr$ and $\bfGc$. 
Similarly, the block corresponding to a horizontal edge $i\to (n-i)$ in the lower part is the submatrix 
$Y_{\bar I}^{\bar J}$ with $\bar I=[1,\bar\alpha]$ and $\bar J=[\bar\beta,n]$, where $\bar\alpha=i_+(\bfGr)$  and $\bar\beta=(n-i+1)_-(\bfGc)+1$ are defined by $Y$-runs. These blocks are glued together in the same fashion as before, except that gluing of 
a $Y$-block to an $X$-block on the left (respectively,  at the bottom) is governed by the row triple $\bfGr$ 
(respectively, the column triple $\bfGc$).
In what follows, we will call $X-$ and $Y-$runs corresponding to $\bfGr$ (respectively, to $\bfGc$)  {\em row\/} (respectively, 
{\em column\/}) runs. 

Let $\L=\L(X,Y)$ denote the matrix glued from $X$- and $Y$-blocks as explained above.
It follows immediately from the construction that if $\L$ is defined by an alternating path 
$i_1\to i_2\to\dots\to i_{2k}$ then it is a square $N(\L)\times N(\L)$ matrix with 
\begin{equation*}
N(\L)=\sum_{j=1}^k i_{2j-1}.
\end{equation*}  
The matrices $\L$ defined by all maximal alternating paths in $\BD_{\bfGr, \bfGc}$ 
form a collection denoted $\bL=\bL_{\bfGr, \bfGc}$ (or $\bL_\bfG$ if $\bfGr=\bfGc=\bfG$). Thus, 
 
(i) each $\L\in\bL$ is a square $N(\L)\times N(\L)$ matrix, 

(ii) for any $1\leq i< j \leq n$, there is a unique pair $(\L \in \bL, s\in [1,N(\L)])$ such that $\L_{ss}=y_{i j}$, and 

(iii) for any $1\leq j< i \leq n$, there exists 
 and a unique pair $(\L \in \bL, s\in [1,N(\L)])$ such that $\L_{ss}=x_{ij}$. 

We thus have a bijection $\J=\J_{\bfGr, \bfGc}$ between 
$[1,n]\times [1,n]\setminus\cup_{i=1}^n(i,i)$ and the set of pairs $\left \{ (\L, s) : \L \in \bL, s\in [1, N(\L)]  \right \}$ that takes a pair $(i,j)$, $i\ne j$, to $(\L(i,j), s(i,j))$. 
We then define
\begin{equation}
\label{f_ij_gen}
{\tt f}_{ij}(X,Y)= \det \L(i,j)_{[s(i,j), N(\L(i,j))]}^{[s(i,j), N(\L(i,j))]}, \quad i\ne j.
\end{equation}
The block of $\L(i,j)$ that contains the entry $(s(i,j),s(i,j))$ is called the {\it leading block\/} of $\ttf_{ij}$. 

Additionally, we define
\begin{equation}
\label{twof_ii}
{\tt f}_{ii}^<(X,Y)=\det X_{[i,n]}^{[i,n]}, \qquad {\tt f}_{ii}^>(X,Y)=\det Y_{[i,n]}^{[i,n]}.
\end{equation}
The leading block of $\ttf_{ii}^<$ is $X$, and the leading block of $\ttf_{ii}^>$ is $Y$.
Note that \eqref{twof_ii} means that $s$ is extended to the diagonal via $s(i,i)=i$, while
$\L(i,i)$ is not defined uniquely: it might denote either $X$ or $Y$. 

Finally, we put $f_{ij}(X) ={\tt f}_{ij}(X,X)$ for $i\ne j$ and $f_{ii}(X) ={\tt f}_{ii}^<(X,X)={\tt f}_{ii}^>(X,X)$, 
and define 
$$
F=F_{\bfGr, \bfGc}=\{ f_{ij}(X) :  i,j\in[1,n]\}.
$$

\begin{theorem}\label{logcanbasis}
Let $(\bfGr,\bfGc)$ be an oriented aperiodic pair of BD triples, 
then the family $F_{\bfGr,\bfGc}$ forms
a log-canonical coordinate system with respect to the Poisson bracket \eqref{sklyabragen} on 
$\Mat_n$ with $r=r^{\er}$ and $r'=r^{\ec}$ given
by \eqref{r-matrix}.
\end{theorem}

\begin{remark}
A log-canonical coordinate system on $SL_n$ with respect to the same bracket is formed 
by $F_{\bfGr,\bfGc}\setminus\{\det X\}$.
\end{remark}

Although the construction of the family of functions $F_{\bfGr,\bfGc}$ is admittedly {\em ad hoc}, the intuition behind it is
given by the collection $\bL=\bL_{\bfGr, \bfGc}$ that does have an intrinsic meaning. Recall the observation we previously utilized in \cite{GSVMem}: a function serving as a frozen variable in a cluster structure on a Poisson variety has a property that it is log-canonical with every cluster variable in every cluster. The vanishing locus  of such a function foliates into a 
union of non-generic symplectic leaves. On the other hand, in many examples of Poisson varieties supporting a cluster structure, 
the union of generic symplectic leaves forms an open orbit of a certain natural group action. Thus, it makes sense to select 
semi-invariants of this group action as frozen variables. Furthermore, a global toric action on the cluster structure arising 
this way can be described in two equivalent ways: it is generated by an action of a commutative subgroup of the group acting on the underlying Poisson variety or, alternatively, by Hamiltonian flows generated by the frozen variables.

In our current situation, the group action is determined by the BD data $\bfGr$, $\bfGc$. Let $\D_-^\er$ and $\D_-^\ec$ 
be subalgebras defined in \eqref{ddeco} that correspond to $\bfGr$ and $\bfGc$, respectively, and  let 
$\DD_-^\er=\exp (\D_-^\er)$ and $\DD_-^\ec=\exp (\D_-^\ec)$ be the corresponding subgroups of the double. Consider the action 
of $\DD_-^\er\times\DD_-^\ec$ on the double $D(GL_n)$ with $\DD_-^\er$ acting on the left and $\DD_-^\ec$ acting on the right.

\begin{proposition}\label{frozen} 
Let  $\L(X,Y) \in \bL_{\bfGr, \bfGc}$. Then

{\rm (i)} $\det\L(X,Y)$ is a semi-invariant of the action of $\DD_-^\er\times\DD_-^\ec$ described above;

{\rm (ii)} $\det\L(X,X)$ is log-canonical with all matrix entries $x_{ij}$ with respect to the Poisson bracket \eqref{sklyabragen}.
\end{proposition}

Consequently, we select the subcollection $\{\det\L(X,X): \L\in \bL_{\bfGr, \bfGc}\}\cup\{\det X\}\subset F_{\bfGr, \bfGc}$
as the set of frozen variables.

\begin{remark}
\label{newhorizons}
The construction of the initial basis of functions outlined above can be extended to accommodate both non-oriented BD data and a presence of cycles in the BD graph. To deal with the former, block matrices that constitute the collection $\L$ have to include, in addition to blocks from $X$ and $Y$, appropriate blocks from $X^\dagger, Y^\dagger$ where $\dagger$ denotes the involution $X^\dagger = w_0 J\left (X^T\right)^{-1}J w_0$ with $J=\diag((-1)^i)_{i=1}^n$. On the other hand, cycles in the BD graph result in some of the matrices in $\L$ becoming infinite periodic, which makes impossible the use of trailing principal minors as described above. Necessary adjustments in this case are outlined in \cite{GSVpest} 
as an application of a more general construction of generalized cluster transformations associated with periodic staircase matrices. With this adjustments, an appropriately updated version of Theorem \ref{genmainth} remains valid for any BD data. However, the proofs presented below become considerably more involved. We will present them in separate publications.
\end{remark}

\subsection{The quiver}\label{thequiver}
Let us choose the family  $F_{\bfGr,\bfGc}$ as the initial cluster for our cluster structure. We now define the 
quiver $Q_{\bfGr,\bfGc}$ that corresponds to this cluster. 

The quiver has $n^2$ vertices labeled $(i,j)$. The function attached to a vertex $(i,j)$ is $f_{ij}$. Any vertex 
except for $(n,n)$ is frozen if and only if its degree is at most three. The vertex $(n,n)$ is never frozen. 
 We will show below that frozen vertices correspond bijectively to the determinants of the matrices 
$\L\in\bL\cup \{X\}$, as suggested by Proposition \ref{frozen}.

\begin{figure}[ht]
\begin{center}
\includegraphics[height=4cm]{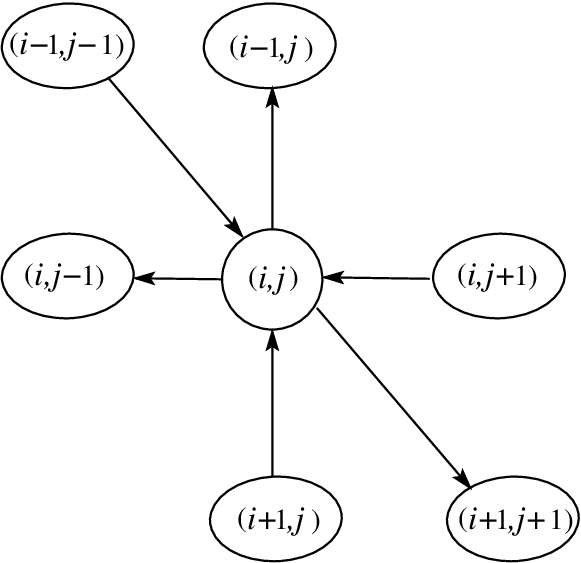}
\caption{The neighborhood of a vertex $(i,j)$, $1<i,j<n$}
\label{fig:ijnei}
\end{center}
\end{figure}

A vertex $(i,j)$ for $1<i<n$, $1<j<n$ has degree six, and its neighborhood looks as shown in Fig.~\ref{fig:ijnei}. 
Here and in what follows, mutable vertices are depicted by circles, frozen vertices by squares, and vertices of unspecified nature by ellipsa.

A vertex $(1,j)$ for $1<j<n$ can have degree two, three, five, or six.
If  $\bfGc$ stipulates both inclined edges $(j-1)\to (k-1)$ and $j\to k$ in the graph $G_{\bfGr,\bfGc}$ for some $k$, 
that is, if $\gammac(k-1)=j-1$ and $\gammac(k)=j$, then the degree of $(1,j)$ in  $Q_{\bfGr,\bfGc}$ equals six, and its
neighborhood looks as shown in Fig.~\ref{fig:1jnei}(a).

If  $\bfGc$ stipulates only the edge $(j-1)\to (k-1)$ as above but not the other one, that is, if $\gammac(k-1)=j-1$ and 
$j\notin \Gamma_2^{\rm c}$,  the degree of $(1,j)$ in  $Q_{\bfGr,\bfGc}$ equals five,  and its
neighborhood looks as shown in Fig.~\ref{fig:1jnei}(b).

If  $\bfGc$ stipulates only the edge $j\to k$ as above but not the other one, that is, if $j-1\notin \Gamma_2^{\rm c}$ and
$\gammac(k)=j$, the degree of $(1,j)$ in  $Q_{\bfGr,\bfGc}$ equals three, and its
neighborhood looks as shown in Fig.~\ref{fig:1jnei}(c).

Finally, if $\bfGc$ does not stipulate any one of the above two inclined edges  in $G_{\bfGr,\bfGc}$, that is, if 
$j-1,j\notin \Gamma_2^{\rm c}$, the degree  of $(1,j)$ in  $Q_{\bfGr,\bfGc}$ equals two, and its
neighborhood looks as shown in Fig.~\ref{fig:1jnei}(d).

\begin{figure}[ht]
\begin{center}
\includegraphics[height=8cm]{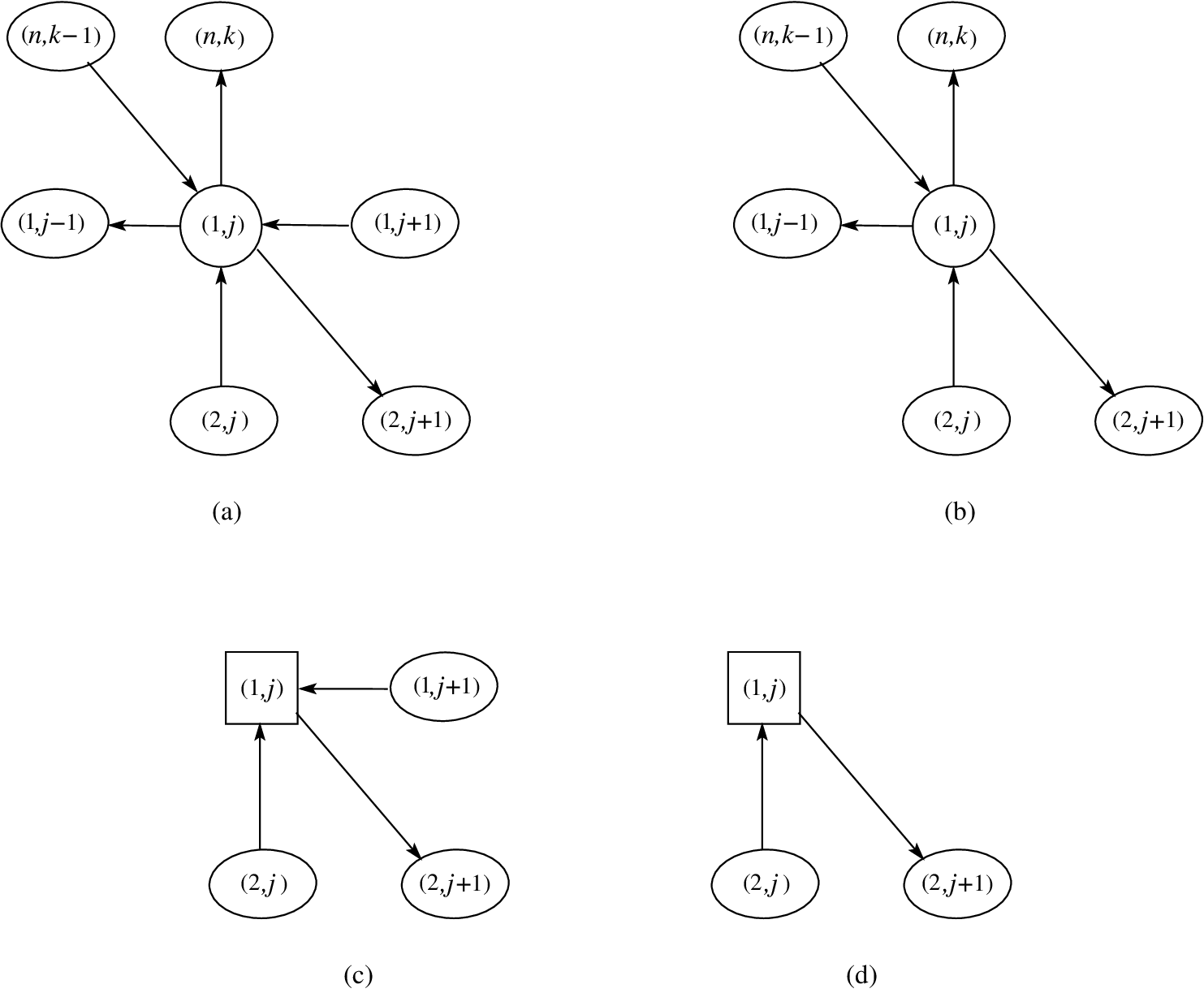}
\caption{Possible neighborhoods of a vertex $(1,j)$, $1<j<n$}
\label{fig:1jnei}
\end{center}
\end{figure}

Similarly, a vertex $(i,1)$ for $1<i<n$ can have degree two, three, five, or six.
If  $\bfGr$ stipulates both inclined edges $(i-1)\to (k-1)$ and $i\to k$ in the graph $G_{\bfGr,\bfGc}$ for some $k$, 
that is, if $\gammar(i-1)=k-1$ and $\gammar(i)=k$, then the degree of $(i,1)$ in  $Q_{\bfGr,\bfGc}$ equals six, and its
neighborhood looks as shown in Fig.~\ref{fig:i1nei}(a).

If  $\bfGr$ stipulates only the edge $(i-1)\to (k-1)$ as above but not the other one, that is, if $\gammar(i-1)=k-1$ and 
$i\notin \Gamma_1^{\rm r}$,  the degree of $(i,1)$ in  $Q_{\bfGr,\bfGc}$ equals five, and its
neighborhood looks as shown in Fig.~\ref{fig:i1nei}(b).

If  $\bfGr$ stipulates only the edge $i\to k$ as above but not the other one, that is, if $i-1\notin \Gamma_1^{\rm r}$ and
$\gammar(i)=k$, the degree of $(i,1)$ in  $Q_{\bfGr,\bfGc}$ equals three, and its
neighborhood looks as shown in Fig.~\ref{fig:i1nei}(c). 

Finally, if $\bfGr$ does not stipulate any one of the above two inclined edges  in $G_{\bfGr,\bfGc}$, that is, if 
$i-1,i\notin \Gamma_1^{\rm r}$, the degree  of $(i,1)$ in  $Q_{\bfGr,\bfGc}$ equals two, and its
neighborhood looks as shown in Fig.~\ref{fig:i1nei}(d).

\begin{figure}[ht]
\begin{center}
\includegraphics[height=8cm]{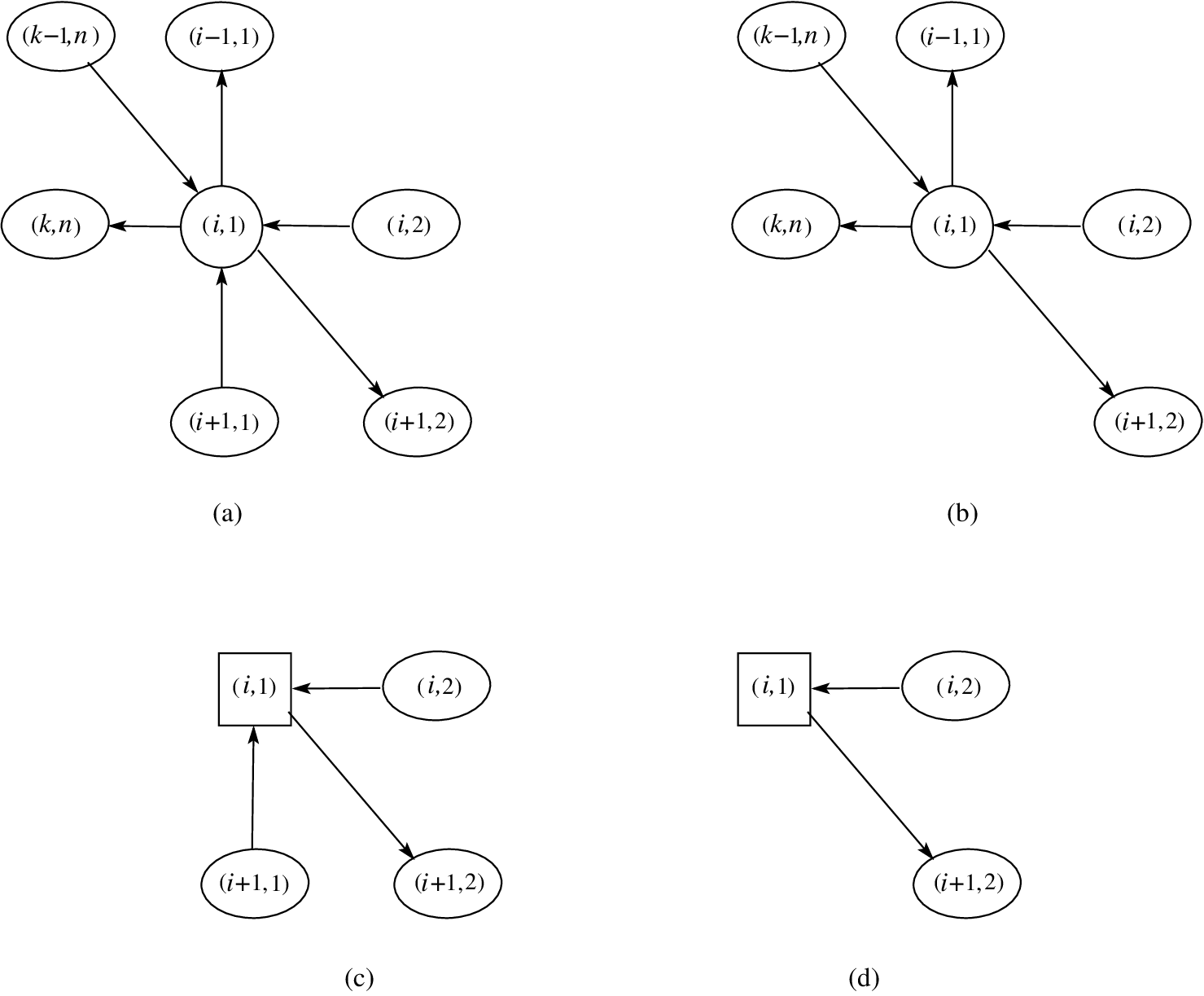}
\caption{Possible neighborhoods of a vertex $(i,1)$, $1<i<n$}
\label{fig:i1nei}
\end{center}
\end{figure}

A vertex $(n,j)$ for $1<j<n$ can have degree four, five, or six.
If  $\bfGc$ stipulates both inclined edges $(k-1)\to (j-1)$ and $k\to j$ in the graph $G_{\bfGr,\bfGc}$ for some $k$, 
that is, if $\gammac(j-1)=k-1$ and $\gammac(j)=k$, then the degree of $(n,j)$ in  $Q_{\bfGr,\bfGc}$ equals six,  and its
neighborhood looks as shown in Fig.~\ref{fig:njnei}(a).

If  $\bfGc$ stipulates only the edge $(k-1)\to (j-1)$ as above but not the other one, that is, if $\gammac(j-1)=k-1$ 
and $j\notin \Gamma_1^{\rm c}$, the degree of $(n,j)$ in  $Q_{\bfGr,\bfGc}$ equals five, and its
neighborhood looks as shown in Fig.~\ref{fig:njnei}(b).

If  $\bfGc$ stipulates only the edge $k\to j$ as above but not the other one, that is, if $j-1\notin \Gamma_1^{\rm c}$ and
$\gammac(j)=k$, the degree of $(n,j)$ in  $Q_{\bfGr,\bfGc}$ equals five as well, and its
neighborhood looks as shown in Fig.~\ref{fig:njnei}(c).

Finally, if $\bfGc$ does not stipulate any one of the above two inclined edges in $G_{\bfGr,\bfGc}$, that is, if 
$j-1,j\notin \Gamma_1^{\rm c}$, the degree  of $(n,j)$ in  $Q_{\bfGr,\bfGc}$ equals four, and its
neighborhood looks as shown in Fig.~\ref{fig:njnei}(d).

\begin{figure}[ht]
\begin{center}
\includegraphics[height=8cm]{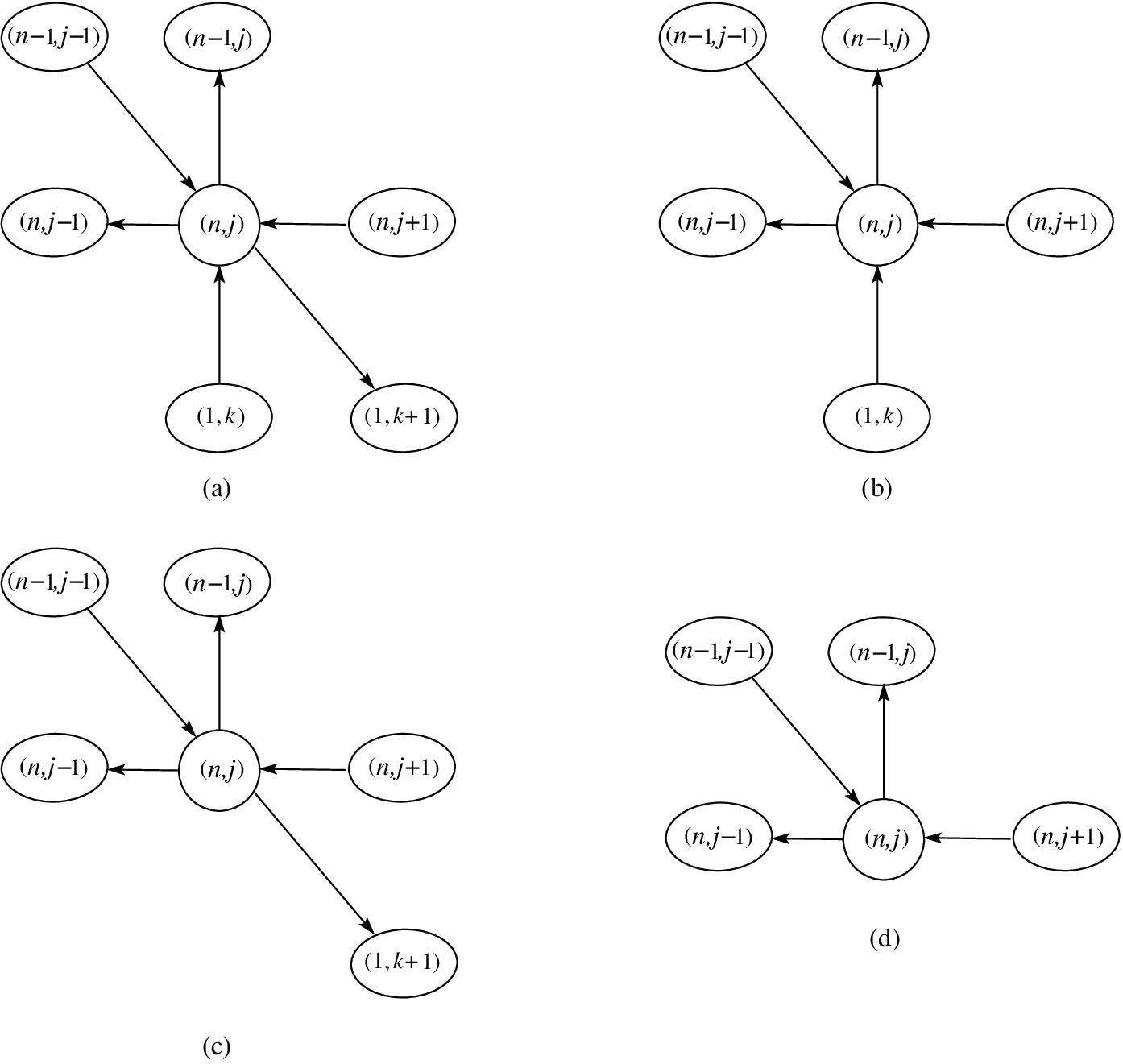}
\caption{Possible neighborhoods of a vertex $(n,j)$, $1<j<n$}
\label{fig:njnei}
\end{center}
\end{figure}

Similarly, a vertex $(i,n)$ for $1<i<n$ can have degree four, five, or six.
If  $\bfGr$ stipulates both inclined edges $(k-1)\to (i-1)$ and $k\to i$ in the graph $G_{\bfGr,\bfGc}$ for some $k$, 
that is, if $\gammar(k-1)=i-1$ and $\gammar(k)=i$, then the degree of $(i,n)$ in  $Q_{\bfGr,\bfGc}$ equals six, and its
neighborhood looks as shown in Fig.~\ref{fig:innei}(a). 

If  $\bfGr$ stipulates only the edge $(k-1)\to (i-1)$ as above but not the other one, that is, if $\gammar(k-1)=i-1$ 
and $i\notin \Gamma_2^{\rm r}$, the degree of $(i,n)$ in  $Q_{\bfGr,\bfGc}$ equals five, and its
neighborhood looks as shown in Fig.~\ref{fig:innei}(b). 

If  $\bfGr$ stipulates only the edge $k\to i$ as above but not the other one, that is, if $i-1\notin \Gamma_2^{\rm r}$ and
$\gammar(k)=i$, the degree of $(i,n)$ in  $Q_{\bfGr,\bfGc}$ equals five as well, and its
neighborhood looks as shown in Fig.~\ref{fig:innei}(c). 

Finally, if $\bfGr$ does not stipulate any one of the above two inclined edges in $G_{\bfGr,\bfGc}$, that is, if 
$i-1,i\notin \Gamma_2^{\rm r}$, the degree  of $(i,n)$ in  $Q_{\bfGr,\bfGc}$ equals four, and its
neighborhood looks as shown in Fig.~\ref{fig:innei}(d).  

\begin{figure}[ht]
\begin{center}
\includegraphics[height=8cm]{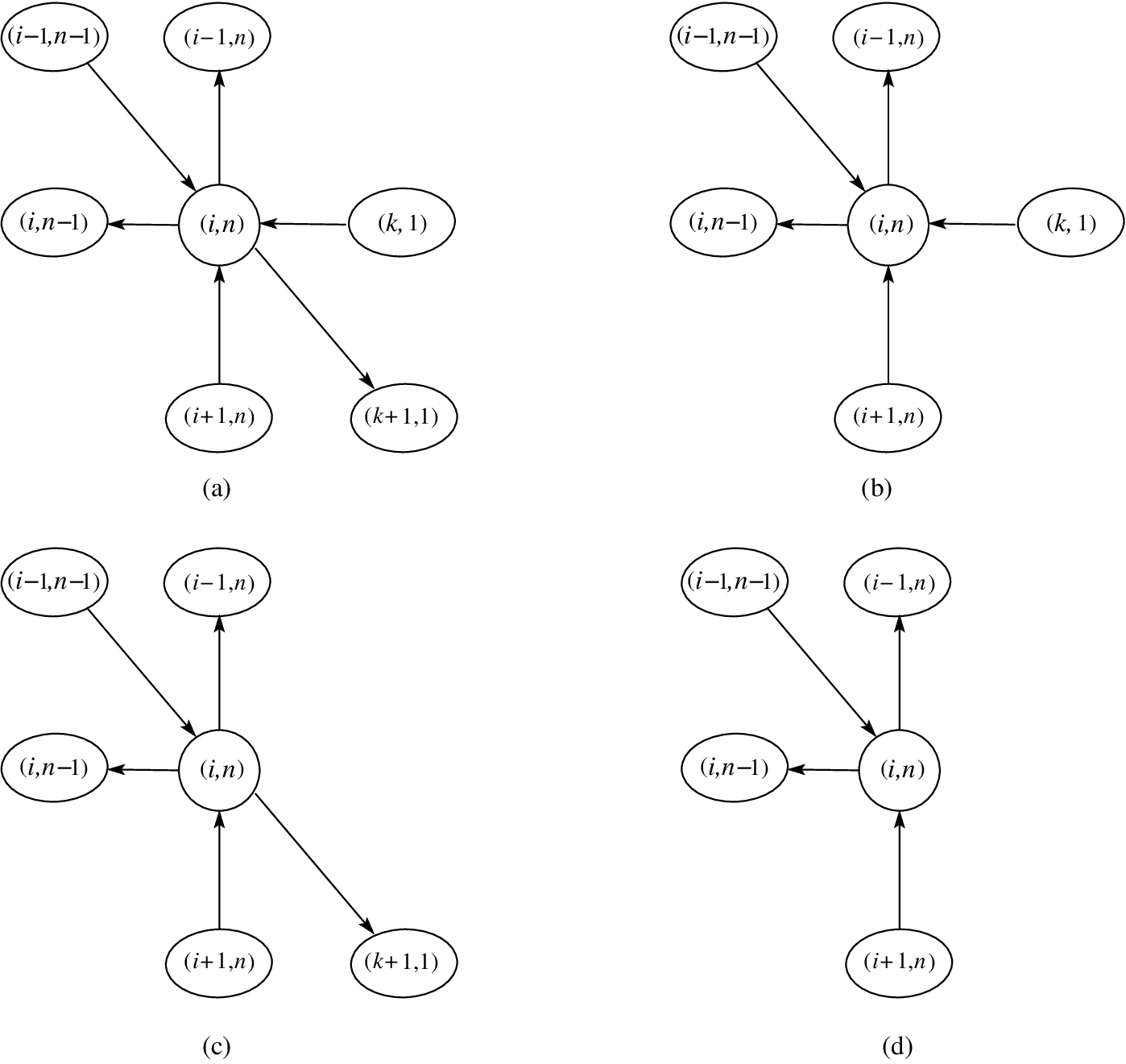}
\caption{Possible neighborhoods of a vertex $(i,n)$, $1<i<n$}
\label{fig:innei}
\end{center}
\end{figure}

The vertex $(1,n)$ can have degree one, two, four, or five. If $\bfGc$ stipulates an inclined edge $(n-1)\to j$ for some $j$,
and $\bfGr$ stipulates an inclined edge $i\to 1$ for some $i$, that is, if $\gammac(j)=n-1$ and $\gammar(i)=1$, then 
the degree of $(1,n)$ in $Q_{\bfGr,\bfGc}$ equals five, and its neighborhood looks as shown in Fig.~\ref{fig:1nnei}(a). 

If only the first of the above two edges is stipulated, that is, if $\gammac(j)=n-1$ and $1\notin \Gamma_2^{\rm r}$,
the degree of $(1,n)$ in $Q_{\bfGr,\bfGc}$ equals four, and its neighborhood looks as shown in Fig.~\ref{fig:1nnei}(b).

If only the second of the above two edges is stipulated, that is, if $\gammar(i)=1$ and $n-1\notin \Gamma_2^{\rm c}$, the degree of $(1,n)$ in $Q_{\bfGr,\bfGc}$ equals two, and its neighborhood looks as shown in Fig.~\ref{fig:1nnei}(c).

Finally, if none of the above two edges is stipulated, that is, if   $1\notin \Gamma_2^{\rm r}$ and 
$n-1\notin \Gamma_2^{\rm c}$, the degree of $(1,n)$ in $Q_{\bfGr,\bfGc}$ equals one, and its neighborhood looks as shown in Fig.~\ref{fig:1nnei}(d).

\begin{figure}[ht]
\begin{center}
\includegraphics[height=8cm]{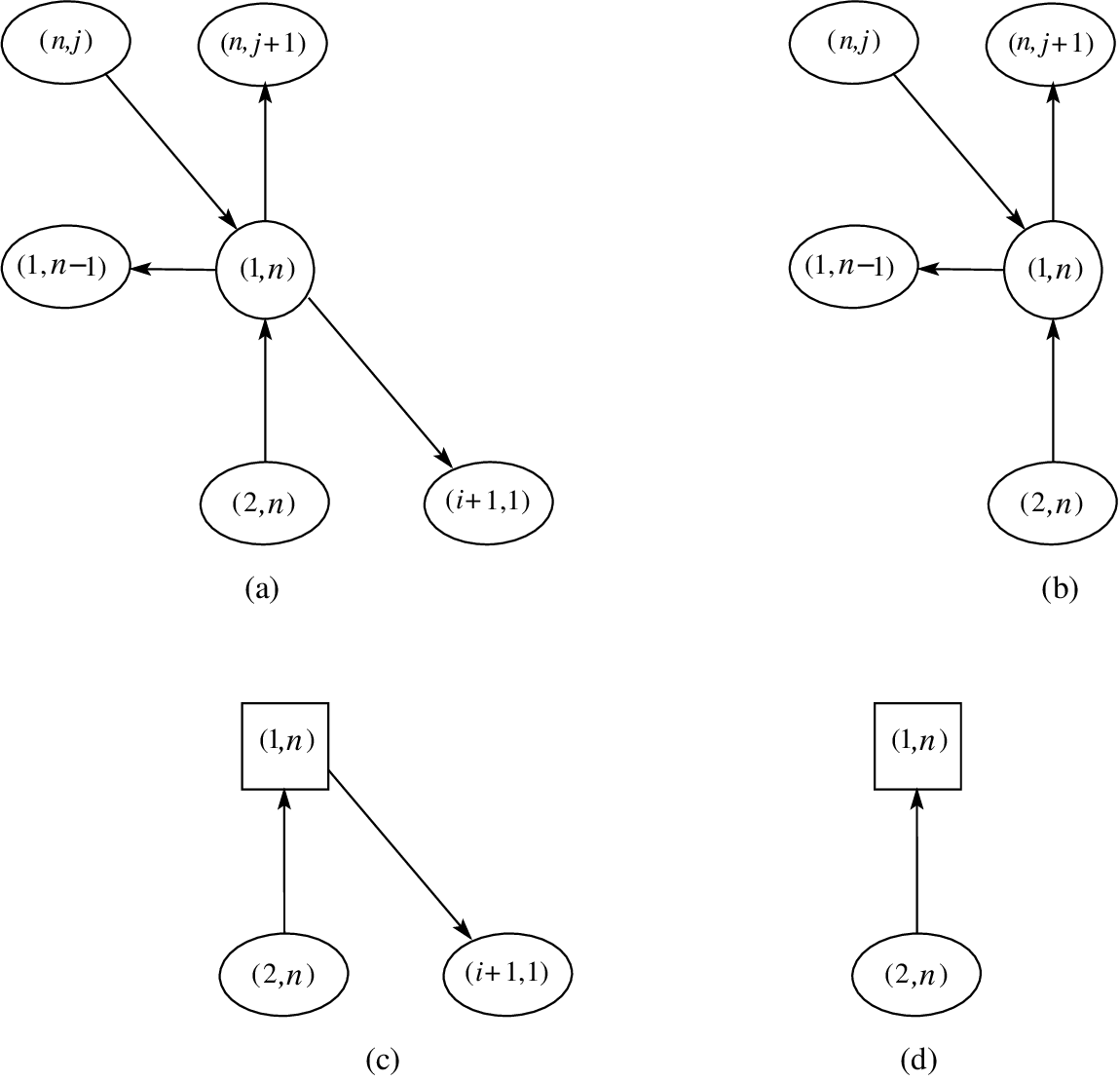}
\caption{Possible neighborhoods of the vertex $(1,n)$}
\label{fig:1nnei}
\end{center}
\end{figure}

Similarly, the vertex $(n,1)$ can have degree one, two, four, or five. If $\bfGr$ stipulates an inclined edge 
$(n-1)\to j$ for some $j$, and $\bfGc$ stipulates an inclined edge $i\to 1$ for some $i$, that is, if $\gammar(n-1)=j$ 
and $\gammac(1)=i$, then the degree of $(n,1)$ in $Q_{\bfGr,\bfGc}$ equals five, and its neighborhood looks as shown 
in Fig.~\ref{fig:n1nei}(a).

If only the first of the above two edges is stipulated, that is, if $\gammar(n-1)=j$ and $1\notin \Gamma_1^{\rm c}$,
the degree of $(n,1)$ in $Q_{\bfGr,\bfGc}$ equals four, and its neighborhood looks as shown 
in Fig.~\ref{fig:n1nei}(b).

If only the second of the above two edges is stipulated, that is, if
$\gammac(1)=i$ and $n-1\notin \Gamma_1^{\rm r}$, the degree of $(n,1)$ in $Q_{\bfGr,\bfGc}$ equals two, and its neighborhood looks as shown in Fig.~\ref{fig:n1nei}(c).

Finally, if none of the above
two edges is stipulated, that is, if   $1\notin \Gamma_1^{\rm c}$ and $n-1\notin \Gamma_1^{\rm r}$, the degree of $(n,1)$
in $Q_{\bfGr,\bfGc}$ equals one, and its neighborhood looks as shown 
in Fig.~\ref{fig:n1nei}(d).

\begin{figure}[ht]
\begin{center}
\includegraphics[height=8cm]{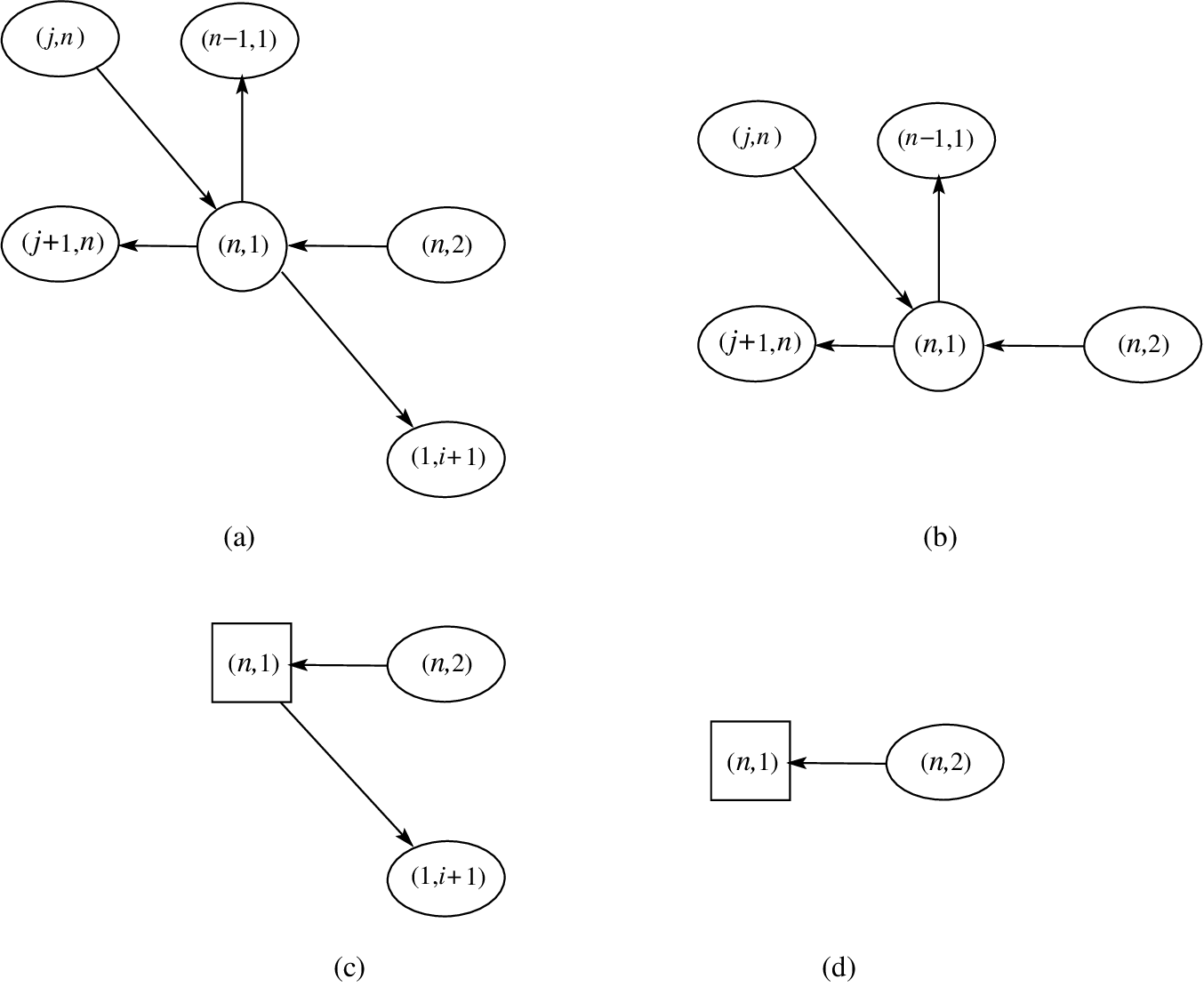}
\caption{Possible neighborhoods of the vertex $(n,1)$}
\label{fig:n1nei}
\end{center}
\end{figure}

The vertex $(n,n)$ can have degree three, four, or five. If $\bfGr$ stipulates an inclined edge 
$i\to(n-1)$ for some $i$, and $\bfGc$ stipulates an inclined edge $j\to(n-1)$ for some $j$, that is, if $\gammar(i)=n-1$ 
and $\gammac(n-1)=j$, then the degree of $(n,n)$ in $Q_{\bfGr,\bfGc}$ equals five, and its neighborhood looks as shown 
in Fig.~\ref{fig:nnnei}(a).

If only one of the above two edges is stipulated, that is, if either  $\gammar(i)=n-1$ and
$n-1\notin \Gamma_1^{\rm c}$, or $\gammac(n-1)=j$ and $n-1\notin \Gamma_2^{\rm r}$,
the degree of $(n,n)$ in $Q_{\bfGr,\bfGc}$ equals four, and its neighborhood looks as shown 
in Fig.~\ref{fig:nnnei}(b,c).

Finally, if none of the above two edges is stipulated, that is, if $n-1\notin \Gamma_1^{\rm c}$ and 
$n-1\notin \Gamma_2^{\rm r}$, 
the degree of $(n,n)$ in $Q_{\bfGr,\bfGc}$ equals three, and its neighborhood looks as shown 
in Fig.~\ref{fig:nnnei}(d).

\begin{figure}[ht]
\begin{center}
\includegraphics[height=8cm]{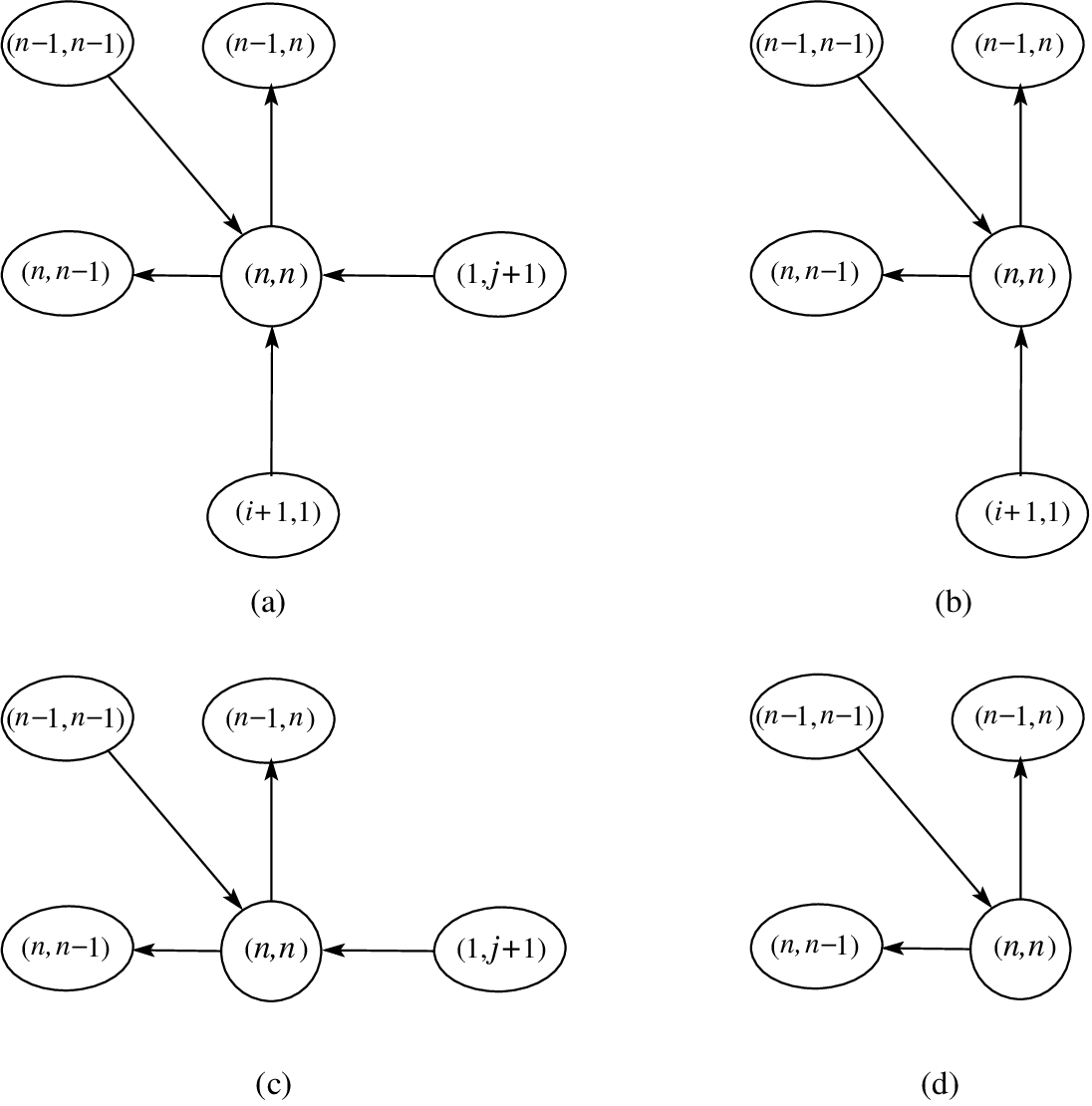}
\caption{Possible neighborhoods of the vertex $(n,n)$}
\label{fig:nnnei}
\end{center}
\end{figure}

Finally, the vertex $(1,1)$ can have degree one, two, or three. If $\bfGr$ stipulates an inclined edge 
$1\to i$ for some $i$, and $\bfGc$ stipulates an inclined edge $1\to j$ for some $j$, that is, if $\gammar(1)=i$ 
and $\gammac(j)=1$, then the degree of $(1,1)$ in $Q_{\bfGr,\bfGc}$ equals three, and its neighborhood looks as shown 
in Fig.~\ref{fig:11nei}(a).

If only one of the above two edges is stipulated, that is, if either  $\gammar(1)=i$ and $1\notin \Gamma_2^{\rm c}$, 
or $\gammac(j)=1$ and $1\notin \Gamma_1^{\rm r}$, the degree of $(1,1)$ in $Q_{\bfGr,\bfGc}$ equals two, and its neighborhood looks as shown in Fig.~\ref{fig:11nei}(b,c). 

If none of the above two edges is stipulated, that is, if 
$1\notin \Gamma_2^{\rm c}$ and $1\notin \Gamma_1^{\rm r}$, the degree of $(1,1)$ in $Q_{\bfGr,\bfGc}$ equals one, 
and its neighborhood looks as shown in Fig.~\ref{fig:11nei}(d).

\begin{figure}[ht]
\begin{center}
\includegraphics[height=6cm]{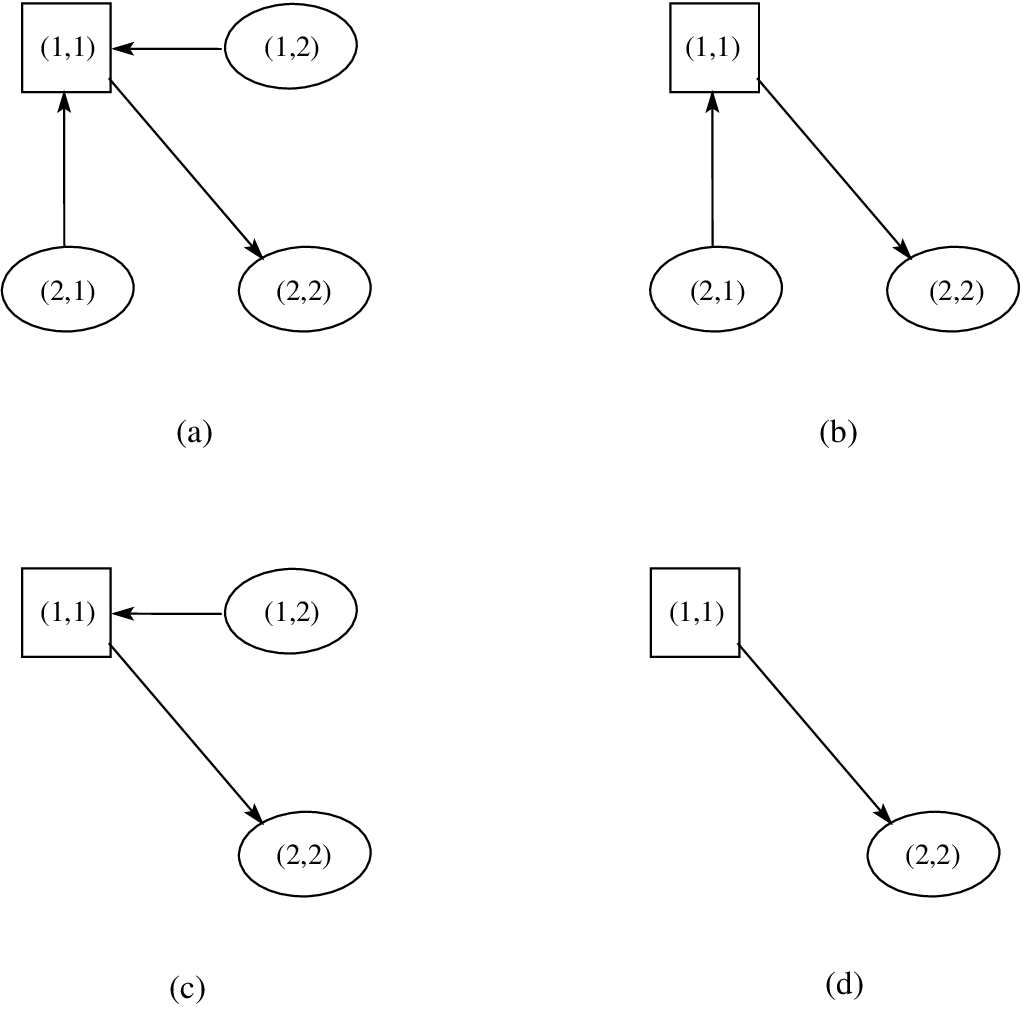}
\caption{Possible neighborhoods of the vertex $(1,1)$}
\label{fig:11nei}
\end{center}
\end{figure}

We can now prove the characterization of frozen vertices mentioned at the beginning of the section.

\begin{proposition}\label{frozenvert}
A vertex $(i,j)$ is frozen in $Q_{\bfGr,\bfGc}$ if and only if $i=j=1$ and $f_{11}=\det X$ or $f_{ij}$ is the restriction 
to the diagonal $X=Y$ of $\det\L$ for some $\L\in\bL_{\bfGr,\bfGc}$.
\end{proposition}

\begin{proof} It follows from the description of the quiver that there are two types of frozen vertices distinct from $(1,1)$:
vertices $(1,j)$ such that $j-1\notin\Gamma^\ec_2$, see Fig.~\ref{fig:1jnei}(c),(d) and Fig.~\ref{fig:1nnei}(c),(d), and
vertices $(i,1)$ such that $i-1\notin\Gamma^\er_1$, see Fig.~\ref{fig:i1nei}(c),(d) and Fig.~\ref{fig:n1nei}(c),(d).

In the first case, the horizontal edge $(n-j+1)\to(j-1)$ in the lower part of $G_{\bfGr,\bfGc}$ is the last edge of a maximal alternating path. Therefore, the $Y$-block defined by this edge is the uppermost block of the matrix $\L$ corresponding to
this path. Consequently, $\bar\beta=j_-(\Gamma_2^{\rm c})+1=j$, and hence  $(1,j)$ is indeed the upper left entry of $\L$.

The second case is handled in a similar manner.
\end{proof}

The quiver  $Q_{\bfGr,\bfGc}$ shown in Fig.~\ref {fig:quiver} corresponds to the BD data 
$\bfGr=\left(\{1,2\}, \right.$  $\left. \{2,3\}, 1\mapsto 2, 2\mapsto 3\right)$ and
$\bfGc=\left(\{1,2\}, \{3,4\}, 1\mapsto3, 2\mapsto4\right)$ in $GL_5$. The corresponding graph $\BD_{\bfGr, \bfGc}$ is shown 
on the left in Fig.~\ref{fig:altpaths}. For example, consider the vertex $(1,4)$ and note that $\BD_{\bfGr, \bfGc}$ contains both edges $\bar4\to2$ and $\bar3\to1$. Consequently, the first of the above conditions for the vertices of type
$(1,j)$ holds with $k=2$, and hence $(1,4)$ has outgoing edges $(1,4)\to(5,2)$, $(1,4)\to(2,5)$, and $(1,4)\to(1,3)$, and ingoing edges $(5,1)\to(1,4)$, $(1,5)\to(1,4)$, and $(2,4)\to(1,4)$. Alternatively, consider the vertex $(4,5)$ and note
that $\BD_{\bfGr, \bfGc}$ contains the edge $2\to\bar3$, while $4\notin \Gamma_2^{\rm r}$. Consequently, the second of
the above conditions for the vertices of type $(j,n)$ holds with $k=3$, and hence $(4,5)$ has outgoing edges $(4,5)\to(4,4)$ 
and  $(4,5)\to(3,5)$ and ingoing edges $(3,4)\to(4,5)$, $(3,1)\to (4,5)$, and $(5,5)\to(4,5)$.

\begin{figure}[ht]
\begin{center}
\includegraphics[height=8cm]{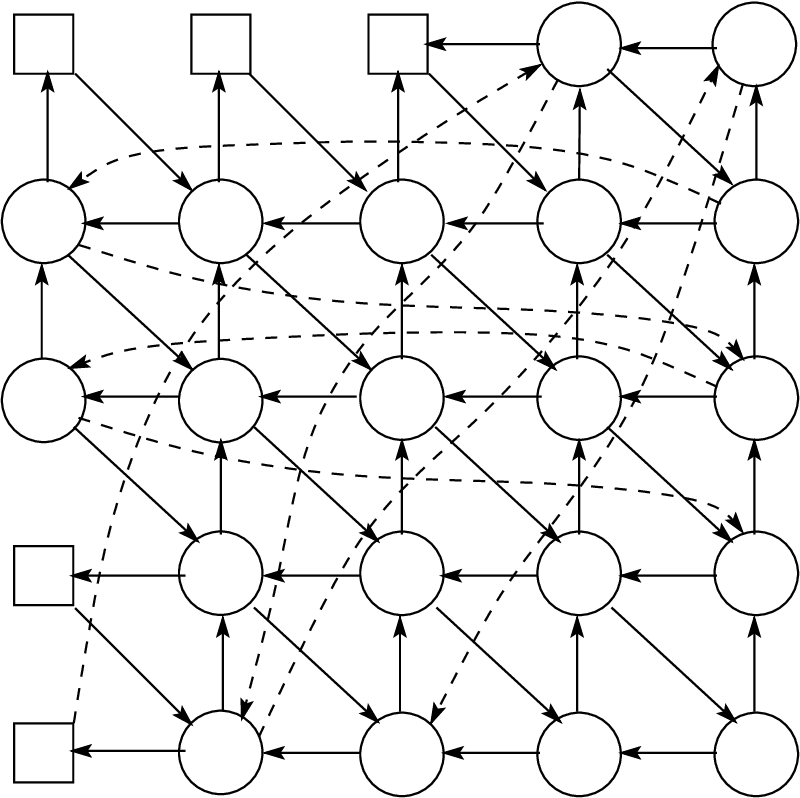}
\caption{An example of the quiver $Q_{\bfGr,\bfGc}$}
\label{fig:quiver}
\end{center}
\end{figure}

\begin{theorem}\label{quiver}
Let $(\bfGr,\bfGc)$ be an oriented aperiodic pair of BD triples, 
then the quiver $Q_{\bfGr,\bfGc}$ defines a cluster structure compatible with
the Poisson bracket \eqref{sklyabragen} on 
$\Mat_n$ with $r=r^{\er}$ and $r'=r^{\ec}$ given
by \eqref{r-matrix}.
\end{theorem}
  
\begin{remark} 
(i) The quiver that defines a cluster structure compatible with the same bracket on $SL_n$ is obtained from
$Q_{\bfGr,\bfGc}$ by deleting the vertex $(1,1)$.

(ii) It follows immediately from the description above that the quiver $Q_{\bfGr,\bfGc}$ defined by an oriented aperiodic pair of BD triples $(\bfGr,\bfGc)$ can be embedded into an orientable surface. This is not true for non-oriented BD triples. The simplest example is given by the BD data $\bfG=\left(\{1,2\}, \{3,4\}, 1\mapsto4, 2\mapsto3\right)$ in $GL_5$. 
In this case the corresponding quiver can be embedded into the projective plane, see \cite{Eis} for details.
\end{remark}

\subsection{Outline of the proof}\label{outline}
The proof of Theorem \ref{logcanbasis} is based on lengthy and rather involved calculations. Following the strategy introduced in
\cite{GSVMem}, we consider the bracket \eqref{sklyadoublegen} on the Drinfeld double of $SL_n$ and lift it to a bracket on 
$\Mat_n\times\Mat_n$. The family $F_{\bfGr,\bfGc}$ is obtained as the restriction onto the diagonal $X=Y$ of the family
$\ttF_{\bfGr,\bfGc}$ of functions defined on $\Mat_n\times\Mat_n$ via
$$
{\tt F}={\tt F}_{\bfGr, \bfGc}=\{{\tt f}_{ij}(X,Y) :  i,j\in[1,n], i\ne j\}
 \cup\{{\tt f}_{ii}^<(X,Y), {\tt f}_{ii}^>(X,Y):i\in[1,n]\},
$$
see \eqref{f_ij_gen}, \eqref{twof_ii}. 
The bracket of a pair of functions $f,g\in \ttF_{\bfGr,\bfGc}$ is decomposed into a large number of contributions that 
either vanish, or are proportional to the product $fg$. In the process we repeatedly use invariance properties
of functions in $\ttF_{\bfGr,\bfGc}$ with respect to the right and left action of certain subgroups of the double.  

The proof of Theorem \ref{quiver} is based on the standard characterization of Poisson structures compatible with a given cluster structure, see e.g. \cite[Ch.~4]{GSVb}. Note that the number of frozen variables in $Q_{\bfGr,\bfGc}$ equals 
$1+k_{\bfGr}+k_{\bfGc}$, and that $\det X$ is frozen. As an immediate consequence we get Theorem \ref{genmainth}(i), which for 
$\bfGr=\bfGc$ turns into Theorem \ref{mainth}(i).

The proof of Theorem \ref{genmainth}(iii) is based on the claim that right hand sides of all exchange relations in one cluster
are semi-invariants of the left-right action of $\H_{\bfGr}\times\H_{\bfGc}$, see Lemma~\ref{rlsemi}. It also  involves the regularity check for all clusters adjacent to the initial one, see Theorem~\ref{regneighbors}. 
Theorem \ref{mainth}(iii) follows when $\bfGr=\bfGc$. After this is done, Theorem \ref{mainth}(iv) and (v) 
follow from Theorem \ref{quiver} via \cite[Theorem~4.1]{GSVM}. To get Theorem~\ref{genmainth}(iv) 
we need a generalization of the latter result to the case of two different tori, which is straightforward.

The central part of the paper is the proof of Theorem \ref{genmainth}(ii) (Theorem \ref{mainth}(ii) then follows in the case 
$\bfGr=\bfGc$). It relies on Proposition~2.1 in \cite{GSVMem}, which is reproduced below for readers' convenience. 

\begin{proposition}\label{regfun}
Let $V$ be a Zariski open subset in $\C^{n+m}$ and $\CC$ be a cluster structure in $\C(V)$  
with $n$ cluster and $m$ frozen variables such that

{\rm(i)} there exists a cluster $(f_1,\dots,f_{n+m})$ in $\CC$ such that $f_i$ is
regular on $V$ for $i\in [1,n+m]$;

{\rm(ii)} any cluster variable $f_k'$ adjacent to $f_k$, $k\in [1,n]$, is regular on $V$;

{\rm(iii)} any frozen variable $f_{n+i}$, $i\in [1,m]$, vanishes at some point of $V$;

{\rm(iv)} each regular function on $V$ belongs to $\UU_\C(\CC)$.

\noindent Then $\CC$ is a regular cluster structure and $\UU_\C(\CC)$ is naturally isomorphic to $\O(V)$.
\end{proposition}

Conditions (i) and (iii) are established via direct observation, and condition (ii) was already discussed above. Therefore, 
the main task is to check condition (iv).  Note that Theorem \ref{genmainth}(i) and Theorem 3.11 in \cite{GSVD} 
imply that it is enough to check that every matrix entry can be written as a Laurent polynomial in the initial cluster and 
in any cluster adjacent to the initial one. In \cite{GSVMem} this goal was achieved by constructing two distinguished sequences of mutations.  Here we suggest a new approach: induction on the total size $|\Gamma^\er_1|+|\Gamma^\ec_1|$. 
Let $\tilde\bfG$ be the BD triple obtained from $\bfG$ by removing a certain root $\alpha$ from $\Gamma_1$ and the corresponding 
root $\gamma(\alpha)$ from $\Gamma_2$. Given an oriented aperiodic pair $(\bfGr, \bfGc)$ with $|\Gamma^\er_1|>0$, we choose 
$\alpha$ to be 
the rightmost root in an arbitrary nontrivial row $X$-run $\Delta^\er$ and define an oriented aperiodic pair $(\tbfGr,\bfGc)$. 
Since the total size of this pair is smaller, we assume that $\tilde\CC=\CC_{\tbfGr,\bfGc}$ possesses the above mentioned Laurent property. Recall that both $\CC$ and $\tilde \CC$ are cluster structures on the space of regular functions on
$\Mat_n$. To distinguish between them, the matrix entries in the latter are denoted $z_{ij}$; they form
an $n\times n$ matrix $Z=(z_{ij})$.

Let $F =\{ f_{ij}(X) \:  i,j\in[1,n]\}$ and $\tilde F=\{ \tilde f_{ij}(Z) \:  i,j\in[1,n]\}$ 
be initial clusters for $\CC$ and $\tilde \CC$, respectively, and 
$Q$ and $\tilde Q$ be the corresponding quivers. It is easy to see that all maximal alternating paths in 
$\BD_{\bfGr,\bfGc}$ are preserved in $\BD_{\tbfGr,\bfGc}$ except for the path that goes through the directed inclined edge
$\alpha\to \gammar(\alpha)$. The latter one is split into two: the initial segment up to the vertex $\alpha$ and the closing segment
starting with the vertex $\gammar(\alpha)$. Consequently, the only difference between
$Q$ and $\tilde Q$ is that the vertex $v=(\alpha+1,1)$ that corresponds to the endpoint of the initial segment
is mutable in $Q$ and frozen in $\tilde Q$, and that certain three edges incident to $v$ in $Q$ do not exist in $\tilde Q$. 

Let us consider four fields of rational functions in $n^2$ independent variables: 
$\X=\C(x_{11},\dots,x_{nn})$, $\ZZ=\C(z_{11},\dots,z_{nn})$,  $\FF=\C(\fy_{11},\dots,\fy_{nn})$, and 
$\tilde\FF=\C(\tfy_{11},\dots,\allowbreak\tfy_{nn})$.  Polynomial maps $f: \FF\to\X$ and $\tilde f: \tilde\FF\to\ZZ$
are given by $\fy_{ij}\mapsto f_{ij}(X)$ and $\tfy_{ij}\mapsto \tilde f_{ij}(Z)$. By the induction hypothesis, there exists
a map $\tilde P: \ZZ\to\tilde \FF$ that takes $z_{ij}$ to a Laurent polynomial in  variables $\tfy_{\alpha\beta}$ such
that $\tilde f\circ\tilde P=\Id$. Note that the polynomials $\tilde f_{ij}(Z)$ are algebraically independent, and hence
$\tilde f$ is an isomorphism. Consequently, $\tilde P\circ \tilde f=\Id$ as well. 
Our first goal is to build a map $P: \X\to\FF$ that takes $x_{ij}$ to a Laurent
polynomial in variables $\fy_{\alpha\beta}$ and satisfies condition $f\circ P=\Id$.

We start from the following result.

\begin{theorem}\label{prototype}
There exist a field isomorphism $U: \X\to \ZZ$ and an invertible polynomial map 
$T: \FF \to \tilde\FF$ satisfying the following conditions:

a) $\tilde f\circ T=U\circ f$;

b) the denominator of any $U(x_{ij})$ is a power of $\tilde f_v(Z)$;

c) the inverse of $T$ is a monomial transformation.
\end{theorem}

Put $P=T^{-1}\circ\tilde P\circ U$; it is a map $\X\to\FF$, and by a) and the induction hypothesis,
\[
P\circ f=T^{-1}\circ\tilde P\circ U\circ f=T^{-1}\circ\tilde P\circ\tilde f\circ T=T^{-1}\circ T=\Id.
\]
For the same reason as above this yields $f\circ P=\Id$. Let us check that $P$ takes $x_{ij}$ to a Laurent polynomial in 
variables $\fy_{\alpha\beta}$. Indeed,  by b), $U$ takes $x_{ij}$ into a rational expression whose denominator is a power 
of $\tilde f_v(Z)$. Consequently, by the induction hypothesis, $\tilde P$ takes the numerator of this expression to a 
Laurent polynomial in $\tfy_{\alpha\beta}$, and the denominator to a power of $\tfy_v$. As a result, $\tilde P\circ U$
takes $x_{ij}$ to a Laurent polynomial in $\tfy_{\alpha\beta}$. Finally, by c), $T^{-1}$ takes this
Laurent polynomial to a Laurent polynomial in $\fy_{\alpha\beta}$, and hence $P$ as above satisfies the required conditions.

 The next goal is to implement a similar construction at all adjacent clusters. 
Fix an arbitrary mutable vertex $u\ne v$ in $Q$; as it was explained above, $u$ remains mutable in $\tilde Q$ as well.
 Let $\mu_u(F)$ and $\mu_u(\tilde F)$ be the clusters obtained from $F$ and $\tilde F$, respectively, via
the mutation in direction $u$, and let $f'_u(X)$ and $\tilde f'_u(Z)$ be cluster variables that replace $f_u(X)$ and 
$\tilde f_u(Z)$ in $\mu_u(F)$ and $\mu_u(\tilde F)$. Replace variables $\fy_{u}$ and $\tfy_{u}$  by new variables 
$\fy'_{u}$ and $\tfy'_{u}$ and define two additional fields of rational functions in $n^2$ variables: 
$\FF'=\C(\fy_{11},\dots, \fy'_{u}, \dots, \fy_{nn})$ and $\tilde\FF'=\C(\tfy_{11},\dots, \tfy'_{u}, \dots, \tfy_{nn})$. 
Similarly to the situation discussed above, there are polynomial isomorphisms $f':\FF'\to\X$ and $\tilde f':\tilde \FF'\to\ZZ$ and a Laurent map $\tilde P':\ZZ\to \tilde \FF'$ such that $\tilde f'\circ\tilde P'=\Id$ 
(the latter exists by the induction hypothesis).

We define a map $T': \FF'\to\tilde \FF'$ via $T'(\fy_{ij})=T(\fy_{ij})$ for $(i,j)\ne u$ and 
$T'(\fy'_u)=\tfy'_u\tfy_v^{\lambda_u}$ for some integer $\lambda_u$ and prove that maps $U$ and $T'$ satisfy the analogs
of conditions a)--c) above. Consequently, the map $P'=(T')^{-1}\circ\tilde P'\circ U$ takes each $x_{ij}$ to a Laurent
polynomial in $\fy_{11},\dots,\fy'_u,\dots, \fy_{nn}$ and satisfies condition $P'\circ f'=\Id$.

Thus, we proved that every matrix entry can be written as a Laurent polynomial in the initial cluster $F$ of $\CC_{\bfGr,\bfGc}$ and in any cluster $\mu_u(F)$ adjacent to it, except for the cluster $\mu_v(F)$. 
To handle this remaining cluster, we pick a different $\alpha$: the rightmost root in another nontrivial row $X$-run (if there are other nontrivial row $X$-runs), 
or the leftmost root of the same row $X$-run (if it differs from the rightmost root), or the rightmost root of an arbitrary nontrivial column $X$-run and an aperiodic pair $(\bfGr,\tbfGc)$ (if $|\Gamma^\ec_1|>0$), and proceed in the same way as above.
Namely, we prove the existence of the analogs of the maps $U$ and $T$ satisfying conditions a)--c) above with a different
distinguished vertex $v$. Consequently, $\mu_v(F)$ is now covered by the above reasoning about adjacent clusters.

 Similarly, if the initial pair $(\bfGr, \bfGc)$ satisfies $|\Gamma^\ec_1|>0$, we apply the same strategy starting with column 
$X$-runs. It follows from the above description that the only case that cannot be treated in this way is $|\Gamma^\er_1|+|\Gamma^\ec_1|=1$. It is considered as the base of induction and treated via direct calculations. 

We thus obtain an analog of Theorem \ref{genmainth}(ii) for the cluster structure $\CC_{\bfGr,\bfGc}$ on $\Mat_n$. The sought-for 
statement for the cluster structure on $SL_n$ follows from the fact that both $\UU_\C(\CC_{\bfGr,\bfGc})$ and $\O(SL_n)$ are obtained from their $\Mat_n$ counterparts via the restriction to $\det X=1$.

\section{Initial basis}\label{sec:basis}
The goal of this Section is the proof of Theorem \ref{logcanbasis}. In \cite{GSVMem}, where the corresponding statement was proved for the case of the Cremmer--Gervais Poisson--Lie bracket, we were able to rely on the fact that all functions in the initial basis were realized as minors of {\em the same\/} block matrix. Thus, the proof of log-canonicity, while still rather cumbersome, was essentially reduced to a direct calculation. In the case of BD data of arbitrary complexity, such approach becomes untenable. Therefore, the proof below hinges on a number of observations exploiting the combinatorics of matrices in the family $\L$ and, most importantly, the invariance properties of functions of the initial basis (cf. \eqref{2.1} and \eqref{2.2} below).

\subsection{The bracket}\label{sec:bra} 
In this paper, we only deal with $\g = \sl_n$, and hence  
$\g_{\Gamma_{1}}$ and $\g_{\Gamma_{2}}$
are subalgebras of block-diagonal matrices with nontrivial traceless blocks determined by nontrivial runs of $\Gamma_{1}$ 
and $\Gamma_{2}$, respectively,  
and zeros everywhere else. Each diagonal component is isomorphic to 
$\sl_k$, where $k$ is the size of the corresponding run.  Formula \eqref{sklyabragen}, where $R_+=R_+^\ec$ and 
$R_+'=R_+^\er$ are given by \eqref{RplusSL} with $S$ skew-symmetric and subject to 
conditions \eqref{S-eq}, defines a Poisson bracket on $\G=SL_n$. It will be convenient to write down an extension of
the bracket \eqref{sklyadoublegen} to the double $D(GL_n)$ such that its restriction to the diagonal $X=Y$ is an extension of \eqref{sklyabragen} to $GL_n$ 
(for brevity, in what follows we write $\Poi^D$ instead of $\Poi^D_{r,r'}$). 

To provide an explicit expression for such an extension, we extend the maps $\gamma$ and $\gamma^*$ to the whole $\gl_n$. 
    Namely, $\gamma$ is re-defined as the projection from $\gl_n$ onto the union of diagonal blocks
 specified by $\Gamma_1$, which are then moved 
by the Lie algebra isomorphism between $\g_{\Gamma_{1}}$ and $\g_{\Gamma_{2}}$
to corresponding diagonal blocks specified by $\Gamma_2$. Similarly, the adjoint map $\gamma^*$  acts as the projection to  
$\g_{\Gamma_2}$ followed by the Lie algebra 
isomorphism that moves each  diagonal block of 
$\g_{\Gamma_{2}}$ back to the corresponding diagonal block of $\g_{\Gamma_{1}}$. 
Consequently, 
\begin{equation}
\label{gammaid}
\begin{aligned}
\gamma^*\gamma=\Pi_{\Gamma_1},\qquad \gamma\gamma^*=\Pi_{\Gamma_2},\\
\gamma\gamma^*\gamma=\gamma,\qquad \gamma^*\gamma\gamma^*=\gamma^*,
\end{aligned}
\end{equation}
where $\Pi_{\Gamma_1}$ is the projection to $\g_{\Gamma_{1}}$ and $\Pi_{\Gamma_2}$ is the projection to $\g_{\Gamma_{2}}$.
Note that the restriction of $\gamma$ to $\g_{\Gamma_1}$ is nilpotent, and hence $1-\gamma$ is invertible on the whole $\gl_n$.

We now view $\pi_>$, $\pi_<$ and $\pi_0$ as projections to the upper triangular, lower triangular and diagonal matrices, respectively.
Additionally, define $\pi_{\geq}=\pi_>+\pi_0$, $\pi_\leq=\pi_<+\pi_0$ and for any square matrix $A$ write $A_>$, $A_<$, $A_0$, $A_\geq$, $A_\leq$ instead of $\pi_>A$, $\pi_<A$, $\pi_0A$, $\pi_\geq A$, $\pi_\leq A$, respectively.
Finally, define operators $\nabla_X$ and $\nabla_Y$ via
\[
\nabla_X f=\left(\frac{\partial f}{\partial x_{ji}}\right)_{i,j=1}^n,\qquad
\nabla_Y f=\left(\frac{\partial f}{\partial y_{ji}}\right)_{i,j=1}^n,
\]
and operators
\begin{align*}
E_L=\nabla_X X+ \nabla_Y Y, \quad & \quad E_R=X \nabla_X + Y \nabla_Y,\\
\xi_L=\gammac(\nabla_X X)+\nabla_Y Y,\quad & \quad \xi_R=X\nabla_X+\gammar^*(Y\nabla_Y),\\
\eta_L=\nabla_X X+\gammac^*(\nabla_Y Y), \quad & \quad\eta_R=\gammar(X\nabla_X)+Y\nabla_Y
\end{align*}
via $E_L f=\nabla_X f \cdot X+ \nabla_Y f\cdot Y$, $E_R f=X \nabla_X f + Y \nabla_Y f$, and so on.
The following simple relations will be used repeatedly in what follows:
\begin{equation}
\label{gammarel}
\begin{aligned}
\frac{1} {1 - \gammac} E_L  = \nabla_X X + \frac{1} {1 - \gammac} \xi_L, \qquad&\frac{1} {1 - \gammar} E_R  = 
X \nabla_X  + \frac{1} {1 - \gammar} \eta_R,\\
\frac{1} {1 - \gammac^*} E_L = \nabla_Y Y + \frac{1} {1 - \gammac^*} \eta_L, \qquad&\frac{1} {1 - \gammar^*} E_R  = 
Y \nabla_Y + \frac{1} {1 - \gammar^*} \xi_R,\\
\eta_L=\gammac^*(\xi_L)+\Pi_{\hat\Gamma_1^\ec}(\nabla_X X), \qquad& \eta_R=\gammar(\xi_R)+\Pi_{\hat\Gamma_2^\er}(Y\nabla_Y),
\end{aligned}
\end{equation}
where $\Pi_{\hat\Gamma_{j}^\el}$ is the orthogonal projection complementary to $\Pi_{\Gamma_{j}^\el}$ for $j=1,2$, ${\rm l=r,c}$.

The statement below is a generalization of \cite[Lemma 4.1]{GSVMem}.

\begin{theorem}\label{doublebrack}
The bracket \eqref{sklyadoublegen} on the double $D(GL_n)$ is given by
\begin{multline}
\label{bracket}
\{f^1,f^2\}^D(X,Y)
=\left\langle R^{\ec}_+(E_L f^1),E_L f^2\right\rangle-\left\langle R^{\er}_+(E_R f^1),E_R f^2\right\rangle\\+
\left\langle X\nabla_{X} f^1, Y\nabla_{Y} f^2\right\rangle-\left\langle \nabla_X f^1\cdot X, 
\nabla_Y f^2\cdot Y\right\rangle,
\end{multline}
where 
\begin{multline}
\label{rplusfin}
R^{\el}_+(\zeta)=\frac1{1-\gamma^\el}\zeta_{\ge}-\frac{{\gamma^\el}^*}{1-{\gamma^\el}^*}\zeta_{<}\\-\frac12 \left 
(\frac{\gamma^\el}{1-\gamma^\el} + 
\frac{1}{1 - {\gamma^\el}^*}\right ) \zeta_0 - \frac 1 n \left(\Tr(\zeta)\S^\el - \Tr \left(\zeta\S^\el\right)\one\right)
\end{multline}
with 
\[
\S^\el = \frac 12 \left(\frac 1{1-\gamma^\el} - \frac 1{1 - {\gamma^\el}^*}\right )\one
\]
for $\rm l=r,c$.
\end{theorem}

\begin{proof}
We need to ``tweak'' $R_+$ to extend the bracket \eqref{sklyabragen} to $GL_n$ in such a way that the function $\det$ is a Casimir function. This is guaranteed by requiring that  $R_+$ is extended to an operator on $\gl_n$ which coincides with 
the one given by 
\eqref{RplusSL} on $\sl_n$ and for which  $\one \in \gl_n$ is an eigenvector. 
The latter goal can be achieved by replacing \eqref{RplusSL}  with 
\begin{equation}
\label{RplusGL}
R_+=\frac1{1-\gamma}\pi_{>}-\frac{\gamma^*}{1-\gamma^*}\pi_{<}+ \frac 1 2 \pi_0 +  \pi^* S\pi\pi_0,
\end{equation}
where  $\pi$ is the projection to the space of traceless diagonal matrices given by 
$\pi(\zeta)= \zeta - \frac 1 n \Tr(\zeta)\one$, $\pi^*$ is the adjoint to $\pi$ with respect to the restriction 
of the trace form
to the space of diagonal matrices in $\gl_n$, and $S$ is an operator on this space which is 
skew-symmetric  with respect to the restriction of the trace form  and satisfies \eqref{S-eq}.

The operator $S$ in \eqref{RplusGL} can be selected as follows. 

\begin{lemma} 
\label{tildeS}
The operator
\begin{equation}
\label{Sanswer}
S = \frac 1 2 \left  ( \frac 1 {1-\gamma} - \frac 1 {1 - \gamma^*}     \right )
\end{equation}
with $\gamma, \gamma^*$ understood as acting on the space of diagonal matrices in $\gl_n$ 
is skew-symmetric with respect to the restriction of the trace form to this space and satisfies \eqref{S-eq}.
\end{lemma}

\begin{proof} Rewrite \eqref{Sanswer} as
\[
S =  \frac 1 2 \frac {1+\gamma} {1-\gamma} - \frac 1 2 \left  ( \frac \gamma {1-\gamma} + \frac 1 {1 - \gamma^*}  \right ).
\]
The first term above clearly satisfies \eqref{S-eq}. The second term, multiplied by $(1-\gamma)$ on the right, becomes
\[
 - \frac 1 2 \left  (\gamma + \frac 1 {1 - \gamma^*} (1-\gamma) \right ) = - \frac 1 2  \frac 1 {1 - \gamma^*} \left  (1 - \gamma^*\gamma \right )
\]
and vanishes on  $\h_{\Gamma_1}\subset \h$ spanned by $\h_\alpha, \alpha \in \Gamma_1$.
\end{proof}

We can now compute
\begin{align*}
\pi^* S\pi(\zeta_0) &=  S(\zeta_0) - \frac 1 n \left (\Tr (\zeta) S(\one) + \Tr (S(\zeta_0)) \one\right)\\
&=S(\zeta_0)- \frac 1 n \left (\Tr (\zeta) S(\one) - \Tr (\zeta S(\one)) \one\right)
\end{align*}
and plug into \eqref{RplusGL} taking into account \eqref{Sanswer}, which gives \eqref{rplusfin}.
Expression \eqref{bracket} is obtained from \eqref{sklyadouble} in the same way as formula (4.2) in \cite{GSVMem}.
\end{proof}

\subsection{Handling functions in $\ttF$} \label{handling}
It will be convenient to carry out all computations in the double with functions in 
${\tt F}_{\bfGr,\bfGc}$, and to retrieve the statements for $F_{\bfGr,\bfGc}$ via the restriction to the diagonal. 

Recall that matrices $\L$ used for the definition of the collection ${\tt F}_{\bfGr,\bfGc}$ are built from 
$X$- and $Y$-blocks, see Section \ref{thebasis}. We will frequently use the following comparison statement,
which is an easy consequence of the definitions, see Fig.~\ref{fig:nesting}.

\begin{proposition}
\label{compar}
Let $X_I^J$, $X_{I'}^{J'}$ be two $X$-blocks and $Y_{\bar I}^{\bar J}$, $Y_{\bar I'}^{\bar J'}$ be two $Y$-blocks.
 
{\rm (i)} If $\beta'<\beta$ (respectively, $\alpha'>\alpha$) then $X_{I'}^{J'}$ fits completely inside $X_I^J$; in particular, $\alpha'\ge\alpha$ (respectively, $\beta'\le\beta$).

{\rm (ii)} If $\bar\beta'>\bar\beta$ (respectively, $\bar\alpha'<\bar\alpha$) then $Y_{\bar I'}^{\bar J'}$ fits completely inside $Y_{\bar I}^{\bar J}$; in particular, $\bar\alpha'\le\bar\alpha$ (respectively, $\bar\beta'\ge\bar\beta$).
\end{proposition}

\begin{figure}[ht]
\begin{center}
\includegraphics[height=4.5cm]{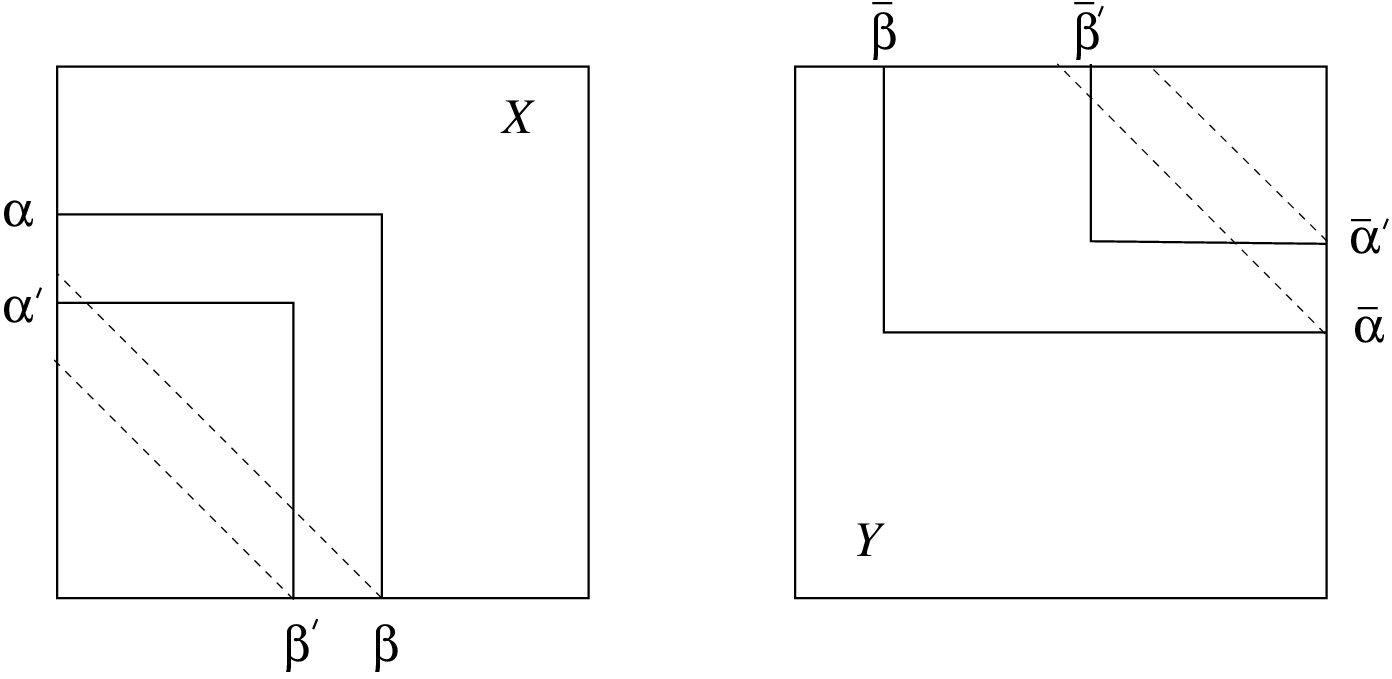}
\caption{Fitting of $X$- and $Y$-blocks}
\label{fig:nesting}
\end{center}
\end{figure}

Consider a matrix $\L$ defined by a maximal alternating path in $\BD_{\bfGr,\bfGc}$.
Let us number the $X$-blocks along the path consecutively, so that the $t$-th $X$-block is 
denoted $X_{I_t}^{J_t}$.
In a similar way we number the $Y$-blocks, so that the $t$-th $Y$-block is denoted 
$Y_{\bar I_t}^{\bar J_t}$.   The glued blocks form a matrix $\L$ so that $\L_{K_t}^{L_t}=X_{I_t}^{J_t}$ and 
$\L_{\bar K_t}^{\bar L_t}=Y_{\bar I_t}^{\bar J_t}$, which we write as
\begin{equation}\label{LLmat}
\L=\sum_{t=1}^s X_{I_t\to K_t}^{J_t\to L_t}+
\sum_{t=1}^s Y_{\bar I_t\to\bar K_t}^{\bar J_t\to\bar L_t}. 
\end{equation}

According to the agreement above, if the $t$-th $X$-block is non-dummy, then the $t$-th $Y$-block lies immediately to the left of it, and if the $t$-th $Y$-block is non-dummy, then the $(t+1)$-th $X$-block lies immediately above it. In more detail,  
all $K_t$'s are disjoint, and the same holds for all $\bar K_t$'s; moreover, $K_t\cap \bar K_{t-1}=\varnothing$. If both
$t$-th blocks are not dummy, put $\Phi_t=K_t\cap\bar K_t$. Then $\Phi_t\ne\varnothing$  corresponds to the nontrivial row runs 
$\Delta(\alpha_t)$ and $\bar\Delta(\bar\alpha_t)=\gamma^{\er}(\Delta(\alpha_t))$ along which the two blocks are glued. 
Consequently, $\Phi_t$ is the uppermost segment in $K_t$ and the lowermost segment in $\bar K_t$. 
If the first block is a dummy $X$-block and $\bar\Delta(\bar\alpha_1)$ is a nontrivial row $Y$-run,
define $\Phi_1$ as the set of rows corresponding to $\bar\Delta(\bar\alpha_1)$; if this $Y$-run is trivial,  put
$\Phi_1=\varnothing$. Similarly, if the last block is a dummy $Y$-block and $\Delta(\alpha_s)$ is a nontrivial row $X$-run,
define $\Phi_s$ as the set of rows corresponding to $\Delta(\alpha_s)$ and put $\bar I_s=\gammar(\Delta(\alpha_s))$; 
if this $X$-run is trivial, put
$\Phi_s=\varnothing$. We put $K_1=\Phi_1$ for a dummy first $X$-block and $\bar K_s=\Phi_s$ for a dummy last $Y$-block
to keep relation $\Phi_t=K_t\cap\bar K_t$ valid for dummy blocks as well.

\begin{figure}[ht]
\begin{center}
\includegraphics[height=6cm]{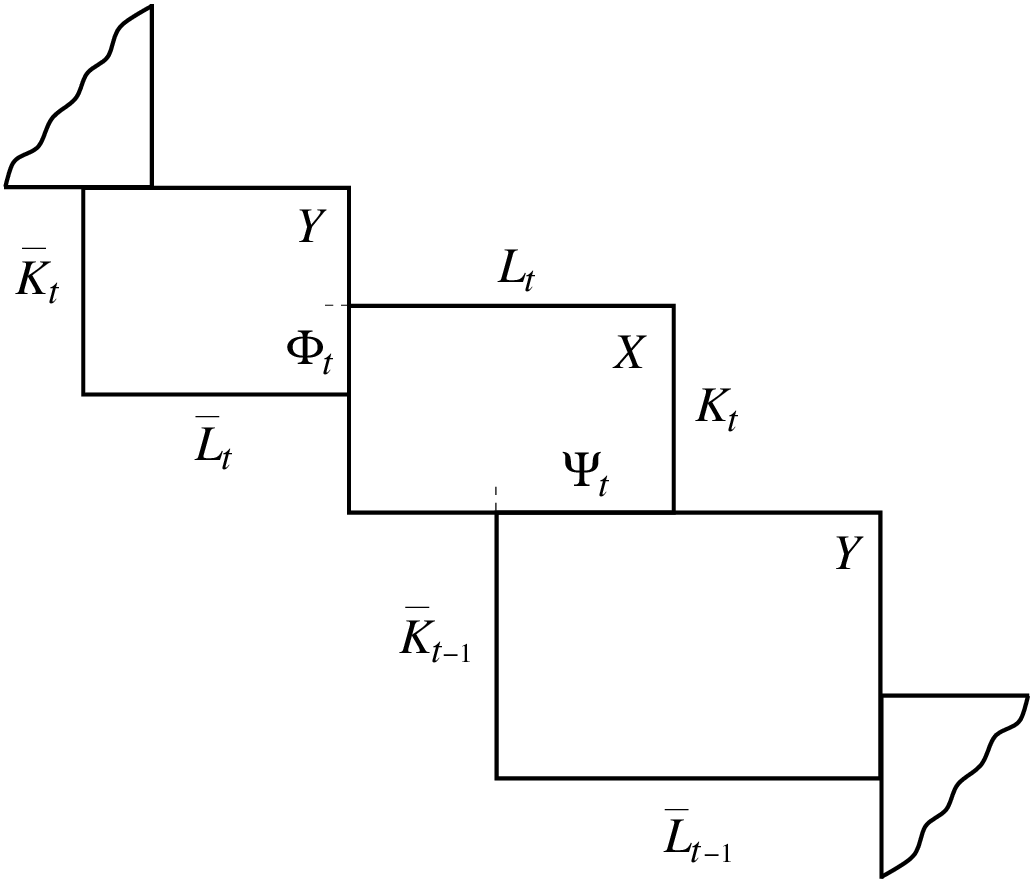}
\caption{The structure of $\L$}
\label{fig:ladder}
\end{center}
\end{figure}

Further, all $L_t$'s are disjoint, and the same holds for all $\bar L_t$'s; 
moreover, $L_t\cap\bar L_t=\varnothing$. For $2\le t\le s$, put $\Psi_t=L_{t}\cap\bar L_{t-1}$, then $\Psi_t\ne\varnothing$ corresponds to the nontrivial column runs $\bar\Delta(\bar\beta_{t-1})$ and 
$\Delta(\beta_t)=\gamma^{\ec*}(\bar\Delta(\bar\beta_{t-1}))$.  
Consequently, $\Psi_t$ is the rightmost segment in $L_t$ and the leftmost segment in $\bar L_{t-1}$. If the first block
is a non-dummy $X$-block and $\Delta(\beta_1)$ is a nontrivial column $X$-run, define $\Psi_1$ as the set of 
columns corresponding to $\Delta(\beta_1)$; if this $X$-run is trivial, or the block is dummy, define 
$\Psi_1=\varnothing$. Similarly, if the last
block is a non-dummy $Y$-block and $\bar\Delta(\bar\beta_s)$ is a nontrivial column $Y$-run, define $\Psi_{s+1}$ as the set of columns corresponding to $\bar\Delta(\bar\beta_s)$ and put $J_{s+1}=\gamma^{\ec*}(\bar\Delta(\bar\beta_{s}))$ (note that 
$J_{s+1}$ does not correspond to any $X$-block of $\L$); 
if this $Y$-run is trivial, or the block is dummy, define $\Psi_{s+1}=\varnothing$. We put $\bar L_0=\Psi_1$ and $L_{s+1}=\Psi_{s+1}$ to keep relation $\Psi_t=L_{t}\cap\bar L_{t-1}$ valid for $1\le t\le s+1$.  
The structure of the obtained matrix $\L$ is shown in Fig.~\ref{fig:ladder}.

It follows from \eqref{LLmat} that the gradients $\nabla_X g$ and $\nabla_Y g$ of a function $g=g(\L)$ can be written as
\begin{equation}\label{naxnay}
\nabla_X g=\sum_{t=1}^s (\nabla_{\L}g)^{K_t\to I_t}_{L_t\to J_t},\qquad
\nabla_Y g=\sum_{t=1}^s (\nabla_{\L}g)^{\bar K_t\to \bar I_t}_{\bar L_t\to \bar J_t}.
\end{equation}
Note that unlike \eqref{LLmat}, 
the blocks in \eqref{naxnay} may overlap.

 Direct computation shows that for $I=[\alpha,n]$, $J=[1,\beta]$, $\bar I=[1,\bar\alpha]$, $\bar J=[\bar\beta,n]$
one has
\begin{equation}\label{xynaxy}
X(\nabla_{\L}g)^{K\to I}_{L\to J}=\begin{bmatrix} 0 & \ast\\ 0 & X_I^J(\nabla_{\L}g)_L^K \end{bmatrix},
\qquad
Y(\nabla_{\L} g)^{\bar K\to \bar I}_{\bar L\to \bar J} =
\begin{bmatrix} Y_{\bar I}^{\bar J}(\nabla_{\L}g)_{\bar L}^{\bar K} & 0\\ \ast & 0 \end{bmatrix}.
\end{equation}
Here and in what follows we denote by an asterisk parts of matrices that are not relevant for further considerations. 
Note that the square block $X_I^J(\nabla_{\L}g)_L^K$ is the diagonal block defined by the index set $I$, whereas the square block $Y_{\bar I}^{\bar J}(\nabla_{\L}g)_{\bar L}^{\bar K}$ is the diagonal block defined by the index set $\bar I$.

Similarly, for $I$, $J$, $\bar I$, $\bar J$ as above,
\begin{equation}\label{naxyxy}
(\nabla_{\L}g)^{K\to I}_{L\to J}\cdot X=\begin{bmatrix}  (\nabla_{\L}g)_L^K \cdot X_I^J & \ast \\ 0 & 0 \end{bmatrix},
\qquad
(\nabla_{\L}g)^{\bar K\to \bar I}_{\bar L\to \bar J}\cdot Y=
\begin{bmatrix} 0 & 0 \\ \ast & (\nabla_{\L}g)_{\bar L}^{\bar K}\cdot Y_{\bar I}^{\bar J} \end{bmatrix},
\end{equation}
and the corresponding square blocks are diagonal blocks defined by the index sets $J$ and $\bar J$, respectively.

Let $N_+,N_-\in GL_n$ be arbitrary unipotent upper- and lower-triangular elements and $T_1,T_2\in H$ be arbitrary
diagonal elements. It is easy to see that the structure of $X$- and $Y$-blocks as defined in Section~\ref{thebasis} and
the way they are glued together, as shown in  
Fig.~\ref{fig:ladder},  imply that
for any $\ttf\in  {\ttF}_{\bfGr,\bfGc}$ one has
\begin{equation}\label{2.1}
\ttf\left(N_+X,\exp(\gammar)(N_+)Y\right)=\ttf\left(X\exp(\gammac^*)(N_-),YN_-\right)=\ttf(X,Y)
\end{equation}
and 
\begin{equation}\label{2.2}
 \ttf\left((T_1X\exp(\gammac^*)(T_2),\exp(\gammar)(T_1)YT_2\right)=a^\ec(T_1)a^\er(T_2) \ttf(X,Y),
\end{equation}
where $a^\ec(T_1)$ and $a^\er(T_2)$ are constants depending only on $T_1$ and $T_2$, respectively.

It will be more convenient to work with the logarithms of the functions $\ttf\in  {\ttF}_{\bfGr,\bfGc}$, instead of the functions $\ttf$ themselves. 
The corresponding infinitesimal form of the invariance properties \eqref{2.1} and~\eqref{2.2} reads: for any 
$\ttf\in  {\ttF}_{\bfGr,\bfGc}$,
\begin{equation}\label{infinv1} 
\left\langle \xi_R\ttg, n_+\right\rangle=
\left\langle \xi_L\ttg, n_-\right\rangle=0
\end{equation}
 and
\begin{equation}\label{infinv2}
(\xi_L\ttg )_0=\text{const},\quad (\xi_R\ttg)_0=\text{const}
\end{equation}
with $\ttg=\log\ttf$.
Additional invariance properties of the functions in ${\ttF}_{\bfGr,\bfGc}$ are given by the following statement.

\begin{lemma}\label{partrace}
For any $\ttf\in  {\ttF}_{\bfGr,\bfGc}$, any  $X$-run $\Delta$ and any $Y$-run $\bar\Delta$,
\begin{align*}
\Tr(\nabla_X\ttg\cdot X)_\Delta^\Delta&= \rm{const}, \qquad \Tr(X\nabla_X\ttg)_\Delta^\Delta =\rm{const},\\
\Tr(\nabla_Y\ttg\cdot Y)_{\bar\Delta}^{\bar\Delta}&= \rm{const}, \qquad
\Tr(Y\nabla_Y\ttg)_{\bar\Delta}^{\bar\Delta}= \rm{const}
\end{align*}
with $\ttg=\log\ttf$.
\end{lemma}

\begin{proof}
Consider for example the second equality above. Let $\one_\Delta$ denote the diagonal $n\times n$ matrix whose entry $(j,j)$
equals $1$ if $j\in\Delta$ and $0$ otherwise. Condition $\Tr(X\nabla_X\ttg)_\Delta^\Delta=a_\Delta$ for an integer constant
$a_\Delta$ is the infinitesimal version of the equality 
\begin{equation}\label{globalpt}
\ttf((\one_n+(z-1)\one_\Delta)X,Y)=z^{a_\Delta}\ttf(X,Y). 
\end{equation}

To establish
the latter, recall that $\ttf(X,Y)$ is a principal minor of a matrix $\L\in\bL$. Clearly, $\ttf((\one_n+(z-1)\one_\Delta)X,Y)$
represents the same principal minor in the matrix $\L(z)$ obtained from $\L$ via multiplying by $z$ every submatrix 
$\L^{L_t}_{R_t}$ such that the row set $R_t$ corresponds to the $X$-run $\Delta$. There are two types of such submatrices: 
those for which $R_t$ lies strictly below $\Phi_t$ and those for which $R_t$ coincides with $\Phi_t$ (the latter might
happen only when the run $X$ is nontrivial). To perform the above operation on each submatrix of the first type it suffices to multiply $\L$ on the left by the diagonal matrix having $z$ in all 
positions corresponding to $R_t$ and $1$ in all other positions. To handle a submatrix of the second type, we multiply by $z$ all rows of $\L$ starting from the first one and ending at the lowest row in $\bar K_t$, and divide by $z$ all columns starting from
the first one and ending at the rightmost column in $\bar L_t$, see Fig.~\ref{fig:ladder}. Clearly, this is equivalent to the left
multiplication of $\L$ by a diagonal matrix whose entries are either $z$ or $1$ and the right multiplication of $\L$ by a diagonal matrix whose entries are either $z^{-1}$ or $1$. Consequently, every principal minor of $\L(z)$ is an integer power of $z$ times the corresponding minor of $\L$, and \eqref{globalpt} follows.

A similar reasoning shows that the remaining three equalities in the statement of the lemma hold as well.  
 \end{proof}

Furthermore, the following statement holds true.

\begin{lemma}\label{twomoreinv}
For any $\ttf\in  {\ttF}_{\bfGr,\bfGc}$, 
\begin{equation}\label{pigammac}
\begin{aligned}
\Pi_{\hat\Gamma_1^\el} (\nabla_X\ttg\cdot X)_0 &= \rm{const}, \qquad \Pi_{\hat\Gamma_1^\el} (X\nabla_X\ttg)_0 =\rm{const},\\
\Pi_{\hat\Gamma_2^\el} (\nabla_Y\ttg\cdot Y)_0  &= \rm{const},\qquad \Pi_{\hat\Gamma_2^\el} (Y\nabla_Y\ttg)_0  =\rm{const}
\end{aligned}
\end{equation}
with $\ttg=\log\ttf$ and $\el=\ec,\er$.
\end{lemma}

\begin{proof}
Same as in the proof of Lemma \ref{partrace}, we will only focus on the second equality in \eqref{pigammac}, since
the other three can be treated in a similar way. 

For any diagonal matrix $\zeta$ we have
\begin{equation}\label{complproj}
\Pi_{\hat\Gamma_1^\el}(\zeta)=\sum_{\Delta}\frac1{|\Delta|}\Tr (\zeta_\Delta^\Delta)\one_\Delta,
\end{equation}
where the sum is taken over all $X$-runs. Let $\zeta=(X\nabla_X\ttg)_0$, then by Lemma~\ref{partrace} all terms in the sum above are constant.
\end{proof}

\begin{corollary}
{\rm(i)} For any $\ttf\in  {\ttF}_{\bfGr,\bfGc}$,
\begin{equation}\label{traces}
\begin{aligned}
\Tr(\nabla_X\ttg\cdot X)=\rm{const}, \qquad \Tr(X\nabla_X\ttg)=\rm{const},\\ 
\Tr(\nabla_Y\ttg\cdot Y)=\rm{const}, \qquad \Tr(Y\nabla_Y\ttg)=\rm{const}
\end{aligned}
\end{equation}
with $\ttg=\log\ttf$.

{\rm(ii)} For any $\ttf\in  {\ttF}_{\bfGr,\bfGc}$, 
\begin{equation}\label{infinv3}
(\eta_L\ttg)_0=\rm{const},\qquad (\eta_R\ttg)_0=\rm{const}
\end{equation}
with $\ttg=\log\ttf$.
\end{corollary} 

\begin{proof} (i) Follows immediately form Lemma \ref{twomoreinv} and equality 
$\Tr\zeta=\Tr\Pi_{\hat\Gamma_1^\el}(\zeta)=\Tr\Pi_{\hat\Gamma_2^\el}(\zeta)$ for any $\zeta$
and $\el=\ec,\er$.

(ii) Follows immediately form Lemma \ref{twomoreinv} and~\eqref{infinv2} via the last two relations in~\eqref{gammarel}.
\end{proof}

\subsection{Proof of Theorem \ref{logcanbasis}: first steps}

Theorem \ref{logcanbasis} is an immediate corollary of the following result.

\begin{theorem}
\label{logcandouble}
For any $\ttf^1, \ttf^2\in  {\ttF}_{\bfGr,\bfGc}$, the bracket $\{\log\ttf^1,\log\ttf^2\}^D$
is constant.
\end{theorem}

The proof of the theorem is given in this and the following sections. It comprises a number of explicit formulas for the objects involved.

\subsubsection{Explicit expression for the bracket} 
Let us derive an explicit expression for $\{\log\ttf^1,\log\ttf^2\}^D$. To indicate that an operator is applied to a function 
$\log\ttf^i$, $i=1,2$, we add $i$ as an upper index of the corresponding operator, so that 
\begin{equation*}
\nabla^1_X X=\nabla_X \log\ttf^1\cdot X, \qquad E_L^2=E_L \log\ttf^2,
\end{equation*}
etc.  

Let
\begin{equation}
\label{R0}
R_0(\zeta)=-\frac12 \left (\frac{\gamma}{1-\gamma} + 
\frac{1}{1 - \gamma^*}\right ) \zeta_0 - \frac 1 n \left(\Tr(\zeta)\S - \Tr \left(\zeta\S\right)\one\right),
\end{equation}
for  $\zeta\in\gl_n$, cf.~\eqref{rplusfin}; clearly, $R_0(\zeta)$ is a diagonal matrix. 

\begin{proposition} For any $\ttf^1, \ttf^2\in  {\ttF}_{\bfGr,\bfGc}$,
\begin{multline}\label{bra}
\{\log\ttf^1,\log\ttf^2\}^D\\=\left\langle R_0^\ec(E_L^1),E_L^2\right\rangle -\left\langle R_0^\er(E_R^1),E_R^2\right\rangle 
+ \left\langle ( \xi_L^1)_0 ,  \frac{1}{1-\gammac^*}  (\eta_L^2)_0 \right\rangle\\- \left\langle ( \eta_R^1)_0 , \frac{1}{1-\gammar^*} (\xi_R^2)_0 \right\rangle+ \left\langle \Pi_{\hat\Gamma_2^\ec}( \xi_L^1)_0 , \Pi_{\hat\Gamma_2^\ec}(\nabla_Y^2 Y)_0   \right\rangle 
\\
-\left\langle(\eta_L^1)_{<},(\eta_L^2)_{>}\right\rangle-
\left\langle(\eta_R^1)_{\ge},(\eta_R^2)_{\le}\right\rangle+
\left\langle\gammac^*(\xi_L^1)_{\le},\gammac^*(\nabla_Y^2 Y)\right\rangle+
\left\langle\gammar(\xi_R^1)_{\ge},\gammar(X\nabla_X^2)\right\rangle.
\end{multline}
\end{proposition}

\begin{proof} 
 First, it follows from Theorem \ref{doublebrack} that 
\begin{equation}
\label{bracket1}
\{\log\ttf^1,\log\ttf^2\}^D
=\left\langle R^\ec_+(E_L^1)-\nabla_X^1 X ,E_L^2\right\rangle-\left\langle R^\er_+(E_R^1) - X\nabla_{X}^1,E_R^2\right\rangle.
\end{equation}
By \eqref{gammarel} and \eqref{R0}, 
\begin{align*}
R^\ec_+(E_L^1) - \nabla_X^1 X  
&= R^\ec_0(E_L^1) + \frac{1}{1-\gammac} ( \xi_L^1)_{\ge} - \frac{1}{1-\gammac^*} (\eta_L^1)_{<}\\ 
&= R^\ec_0(E_L^1) + \frac{1}{1-\gammac} ( \xi_L^1)_0 - \frac{1}{1-\gammac^*}(\eta_L^1)_{<}; 
\end{align*}
the second equality holds since $\xi_L^1 \in \b_-$ by \eqref{infinv1}. Similarly,
\begin{equation}\label{term1left}
\begin{aligned}
R^\er_+(E_R^1) -  X \nabla_X^1  
&= R^\er_0(E_R^1) + \frac{1}{1-\gammar} ( \eta_R^1)_{\ge} - \frac{1}{1-\gammar^*} (\xi_R^1)_{<}\\
 &= R^\er_0(E_R^1) + \frac{1}{1-\gammar} ( \eta_R^1)_{\ge}; 
\end{aligned}
\end{equation}
the second equality holds since $\xi_R^1 \in \b_+$ by \eqref{infinv1}.

Consequently, the first term in \eqref{bracket1} is equal to
\begin{equation}
\label{term1}
\left\langle R_0^\ec(E_L^1), E_L^2 \right\rangle +  \left\langle \frac{1}{1-\gammac} ( \xi_L^1)_0 ,  E_L^2 \right\rangle\\ 
- \left\langle\frac{1}{1-\gammac^*} (\eta_L^1)_{<},  E_L^2\right\rangle.
\end{equation}
The second term in \eqref{term1} can be re-written via \eqref{gammarel} as
\begin{multline*}
 \left\langle \frac{1}{1-\gammac} ( \xi_L^1)_0 ,  E_L^2 \right\rangle 
 = \left\langle ( \xi_L^1)_0 , \nabla_Y^2  Y  + \frac{1}{1-\gammac^*}  \eta_L^2 \right\rangle\\
 = \left\langle ( \xi_L^1)_0 ,  \frac{1}{1-\gammac^*}  \eta_L^2\right\rangle + 
\left\langle \Pi_{\hat\Gamma_2^\ec}(\xi_L^1)_0, \Pi_{\hat\Gamma_2^\ec}(\nabla_Y^2  Y)   \right\rangle+
\left\langle \Pi_{\Gamma_2^\ec}(\xi_L^1)_0, \Pi_{\Gamma_2^\ec}(\nabla_Y^2  Y)   \right\rangle\\
 = \left\langle ( \xi_L^1)_0,  \frac{1}{1-\gammac^*}  (\eta_L^2)_0 \right\rangle + 
\left\langle \Pi_{\hat\Gamma_2^\ec}(\xi_L^1)_0, \Pi_{\hat\Gamma_2^\ec}(\nabla_Y^2  Y)_0   \right\rangle
 + \left\langle \gammac^*( \xi_L^1)_0 , \gammac^* (\nabla_Y^2  Y)  \right\rangle, 
\end{multline*}
where the last equality follows from \eqref{gammaid}.

We re-write the third term in \eqref{term1} as
\begin{multline*}
\left\langle(\eta_L^1)_{<}, \frac{1}{1-\gammac} E_L^2\right\rangle =
\left\langle (\eta_L^1)_{<},  \nabla_X^2  X + \frac{1}{1-\gammac} \xi_L^2\right\rangle = \left\langle(\eta_L^1)_{<},  
\nabla_X^2  X )\right\rangle\\
= \left\langle (\eta_L^1)_{<},  \eta_L^2 \right\rangle - 
\left\langle (\eta_L^1)_{<},  \gammac^* (\nabla_Y^2  Y )\right\rangle
= \left\langle (\eta_L^1)_{<},  \eta_L^2 \right\rangle - \left\langle \gammac^* (\xi_L^1)_{<},  \gammac^* (\nabla_Y^2  Y )\right\rangle,
\end{multline*}
where the second equality follows from \eqref{infinv1}, and the last equality, from \eqref{gammarel} and
$\left\langle \Pi_{\hat\Gamma_1^\ec}(A),\gammac^*(B)\right\rangle=0$ for any $A, B$.

Similarly, the second term in in \eqref{bracket1} is equal to
\begin{multline}\label{term2}
\left\langle R^\er_0(E_R^1), E_R^2\right\rangle +  \left\langle\frac{1}{1-\gammar}(\eta_R^1)_{\ge}, E_R^2\right\rangle \\
=\left\langle R_0^\er(E_R^1), E_R^2 \right\rangle +  \left\langle ( \eta_R^1)_{\ge} , Y \nabla_Y^2 \right\rangle
+ \left\langle ( \eta_R^1)_0 , \frac{1}{1-\gammar^*} (\xi_R^2)_0 \right\rangle\\
=\left\langle R_0^\er(E_R^1), E_R^2 \right\rangle 
+ \left\langle ( \eta_R^1)_0 , \frac{1}{1-\gammar^*} (\xi_R^2)_0 \right\rangle +  \left\langle ( \eta_R^1)_{\ge} , \eta_R^2 \right\rangle -
\left\langle \gammar (\xi_R^1)_{\ge},  \gammar(X\nabla_X^2   )\right\rangle.
\end{multline}

Combining \eqref{term1}, \eqref{term2} and plugging the result into \eqref{bracket1}, we obtain  \eqref{bra} 
as required.
\end{proof}

\subsubsection{Diagonal contributions}
Note that the third, the fourth and the fifth terms in \eqref{bra}
are constant due to  \eqref{infinv2} and  \eqref{pigammac}. The first two terms are handled by the
following statement.

\begin{lemma}
\label{R0const}
The quantities $\left\langle R_0(E_L^1), E_L^2 \right\rangle$ and  $\left\langle R_0(E_R^1), E_R^2 \right\rangle$ are constant for any $\ttf^1, \ttf^2\in  {\ttF}_{\bfGr,\bfGc}$.
\end{lemma} 

\begin{proof} 
Let us start with 
\begin{multline}
\label{R0L12}
\left\langle R_0(E_L^1), E_L^2 \right\rangle = - \frac 1 2  \left\langle   \left (   \frac {\gamma}{1-\gamma} + \frac{1} {1 - \gamma^*}  \right )  (E_L^1)_0  ,  E_L^2 \right\rangle\\
 - \frac 1 n \left (   \Tr ( E_L^1 ) \Tr(E_L^2\S) - \Tr(E_L^1\S) \Tr ( E_L^2 )   \right ),
\end{multline}
where $\gamma=\gammac$.
First, note that
\begin{multline}\label{traceELS}
\Tr ( E_L^i\S) = \left \langle E_L^i,\left (\frac {1}{1-\gamma} - \frac{1} {1 - \gamma^*}\right)\one\right \rangle =
\Tr\left ( \left (   \frac {1}{1-\gamma^*} - \frac{1} {1 - \gamma}  \right )  E_L^i \right ) \\
= \Tr\left(\frac {1}{1-\gamma^*}\eta_L^i - \frac{1} {1 - \gamma} \xi_L^i + \nabla_Y^i Y  - \nabla_X^i X \right ) =\text{const} 
\end{multline}
for $i=1,2$ by \eqref{gammarel},  \eqref{infinv2}, \eqref{traces} and \eqref{infinv3}.  
Thus, the terms in the second line in \eqref{R0L12} are constant.

Next, by \eqref{gammarel},
\begin{equation}\label{for_frozen}
\begin{aligned}
 \left (   \frac {\gamma}{1-\gamma} + \frac{1} {1 - \gamma^*}  \right ) E_L &= \frac {1}{1-\gamma} \xi_L 
+ \frac{1} {1 - \gamma^*}\eta_L,\\
\left\langle\frac1{1-\gamma}\xi_L^1,E_L^2\right\rangle&= 
\left\langle \xi_L^1,\nabla_Y^2 Y+\frac1{1-\gamma^*}\eta_L^2\right\rangle,\\
\left\langle\frac1{1-\gamma^*}\eta_L^1,E_L^2\right\rangle&= 
\left\langle \eta_L^1,\nabla_X^2 X+\frac1{1-\gamma}\xi_L^2\right\rangle,\\
\end{aligned}
\end{equation}
and hence
\begin{multline}
\label{R0L}
\left\langle  \left (   \frac {\gamma}{1-\gamma} + \frac{1} {1 - \gamma^*}  \right )(E_L^1)_0 , E_L^2 \right\rangle \\
= \left\langle  (\xi_L^1)_0, \nabla_Y^2 Y+ \frac{1} {1 - \gamma^*}\eta_L^2 \right\rangle + \left\langle (\eta_L^1)_0, \nabla_X^2 X + \frac{1} {1 - \gamma}\xi_L^2 \right\rangle \\
= \left\langle  (\xi_L^1)_0,  \frac{1} {1 - \gamma^*}(\eta_L^2)_0 \right\rangle 
+  \left\langle  (\eta_L^1)_0,  \frac{1} {1 - \gamma}(\xi_L^2)_0  \right\rangle 
+ \left\langle  (\xi_L^1)_0, (\xi_L^2 )_0  \right\rangle\\
 + \left\langle  (\eta_L^1)_0, \nabla_X^2 X \right\rangle -
\left\langle  (\xi_L^1)_0, \gamma (\nabla_X^2 X) \right\rangle.
\end{multline}
Each of the three first terms in \eqref{R0L} is constant by \eqref{infinv2} and \eqref{infinv3}. Note that by \eqref{gammaid}, 
\[
\left\langle  (\xi_L^1)_0, \gamma (\nabla_X^2 X)\right\rangle 
= \left\langle  \gamma^*\gamma (\nabla_X^1 X)_0  +  \gamma^*(\nabla_Y^1 Y)_0, \nabla_X^2 X \right\rangle 
=\left\langle  \Pi_{\Gamma_1}(\eta_L^1)_0 , \nabla_X^2 X \right\rangle
\]
with $\Gamma_1=\Gamma_1^\ec$, and so the last two terms in \eqref{R0L} combine into
\[
\left\langle  \Pi_{\hat\Gamma_1}(\eta_L^1)_0 , \Pi_{\hat\Gamma_1}(\nabla_X^2 X)_0\right\rangle,
\]
which is constant by  \eqref{pigammac}.
 
 Similarly,
 \begin{multline}
\label{R0R12}
\left\langle R_0(E_R^1), E_R^2 \right\rangle = - \frac 1 2  \left\langle   \left (   \frac {\gamma}{1-\gamma} + \frac{1} {1 - \gamma^*}  \right )  (E_R^1)_0  ,  E_R^2 \right\rangle\\
 - \frac 1 n \left (   \Tr ( E_R^1 ) \Tr (E_R^2\S) - \Tr(E_R^1 \S) \Tr ( E_R^2 )   \right ) 
\end{multline}
with $\gamma=\gammar$. As before, 
\begin{multline*}
\Tr (E_R^i\S) = \left \langle E_R^i ,  \left (\frac {1}{1-\gamma} - \frac{1} {1 - \gamma^*}\right )\one  \right \rangle \\
= \Tr\left ( \frac {1}{1-\gamma^*}\xi_R^i - \frac{1} {1 - \gamma}    \eta_R^i + Y \nabla_Y^i   - X \nabla_X^i  \right ) = \text{const} 
\end{multline*}
for $i=1,2$, 
and
\begin{multline*}
\left\langle  \left(\frac {\gamma}{1-\gamma} + \frac{1} {1 - \gamma^*}\right)(E_R^1)_0 , E_R^2 \right\rangle \\
= \left\langle  (\eta_R^1)_0, Y\nabla_Y^2 + \frac{1} {1 - \gamma^*}\xi_R^2 \right\rangle + \left\langle  (\xi_R^1)_0, 
X\nabla_X^2 + \frac{1} {1 - \gamma}\eta_R^2  \right\rangle \\
= \left\langle  (\eta_R^1)_0,  \frac{1} {1 - \gamma^*}(\xi_R^2)_0 \right\rangle 
+ \left\langle  (\xi_R^1)_0,  \frac{1} {1 - \gamma}(\eta_R^2)_0  \right\rangle 
+ \left\langle  (\xi_R^1)_0, (\xi_R^2 )_0  \right\rangle\\
+ \left\langle  (\eta_R^1)_0, Y\nabla_Y^2  \right\rangle 
- \left\langle  (\xi_R^1)_0, \gamma^* (Y\nabla_Y^2) \right\rangle.
\end{multline*}
Each of the three first terms above is constant by \eqref{infinv2} and \eqref{infinv3}, while
\[
\left\langle  (\eta_R^1)_0, Y\nabla_Y^2  \right\rangle -
\left\langle  (\xi_R^1)_0, \gamma^* (Y\nabla_Y^2) \right\rangle=
\left\langle  \Pi_{\hat\Gamma_2}(\eta_R^1)_0 , \Pi_{\hat\Gamma_2}(Y\nabla_Y^2)_0 \right\rangle = \text{const}
\]
with $\Gamma_2=\Gamma_2^\er$. Thus, the right hand side of \eqref{R0R12} is constant as well, and we are done.
\end{proof}

\subsubsection{Simplified version of the maps $\gamma$ and $\gamma^*$}\label{simplega}
To proceed further, we define more ``accessible'' versions of the maps $\gamma$ and $\gamma^*$. Recall that
$\g_{\Gamma_{1}}$ and $\g_{\Gamma_{2}}$ defined above are subalgebras of block-diagonal matrices with nontrivial 
traceless blocks determined by nontrivial runs of $\Gamma_{1}$ and $\Gamma_{2}$, respectively, and zeros everywhere else. 
Each diagonal component is isomorphic to $\sl_k$, where $k$ is the size of the corresponding run.   
To modify the definition of $\gamma$, we first modify each nontrivial diagonal block in $\g_{\Gamma_{1}}$ and 
$\g_{\Gamma_{2}}$ from $\sl_k$ to $\Mat_k$ by dropping the tracelessness condition. Next, $\cgamma$ is defined as the projection from $\Mat_n$ onto the union of diagonal blocks specified by $\Gamma_1$, which are then moved 
to corresponding diagonal blocks specified by $\Gamma_2$. Similarly, the adjoint map $\cgamma^*$  acts as the projection to  
$\Mat_{\Gamma_2}$ followed by a map that moves each  diagonal block of $\Mat_{\Gamma_{2}}$ back to the corresponding diagonal block of $\Mat_{\Gamma_{1}}$. Consequently, ringed analogs of relations \eqref{gammaid} remain valid  with $\cPi_{\Gamma_1}$ 
understood as the orthogonal projection to $\Mat_{\Gamma_{1}}$ and $\cPi_{\Gamma_2}$ as the orthogonal projection to 
$\Mat_{\Gamma_{2}}$. Further, we define $\cxi_L$, $\cxi_R$, $\ceta_L$ and $\ceta_R$ with $\cgammar$ and $\cgammac$ replacing
$\gammar$ and $\gammac$ and note that the ringed versions of the last two relations in \eqref{gammarel} remain valid 
with $\cPi_{\hat\Gamma_1}$ and $\cPi_{\hat\Gamma_2}$ being  orthogonal projections complementary to $\cPi_{\Gamma_1}$ and
$\cPi_{\Gamma_2}$, respectively. Observe that the ringed versions of the other four relations in \eqref{gammarel} are no
longer true, since $1-\cgamma$ and $1-\cgamma^*$ might be non-invertible.

It is easy to see that $\cgamma$ and $\cgamma^*$ differ from $\gamma$ and $\gamma^*$, respectively, only on the diagonal. 
Consequently, invariance properties \eqref{2.1} and \eqref{infinv1} remain valid in ringed versions.
Further, the ringed version of the invariance property \eqref{2.2} remains valid as well, albeit with different 
constants $a^\ec(T_1)$ and $a^\er(T_2)$, which yields the ringed version of \eqref{infinv2}. Ringed relations \eqref{pigammac}
also hold true: indeed, the sum in \eqref{complproj} is now taken only over trivial $X$-runs.
As a corollary, we restore ringed versions of relations \eqref{infinv3}.

Recall that to complete the proof of Theorem \ref{logcandouble}, it remains to consider the four last terms in \eqref{bra}.
The following observation plays a crucial role in handling these terms. 

\begin{lemma}\label{difference}
For each one of the last four terms in \eqref{bra}, the difference between the initial and the ringed version is constant.
\end{lemma} 

\begin{proof} The equality 
$\left\langle (\eta^1_L)_{<},(\eta_L^2)_{>}\right\rangle=\left\langle (\ceta^1_L)_{<},(\ceta_L^2)_{>}\right\rangle$ 
is trivial, since $\gamma^*$ and $\cgamma^*$ coincide on $\n_+$ and $\n_-$. 

For the second of the four terms, we have to consider the difference 
\begin{multline*}
\left\langle (\ceta^1_R)_{0},(\ceta_R^2)_{0}\right\rangle-
\left\langle (\eta^1_R)_{0},(\eta_R^2)_{0}\right\rangle= 
\left\langle \cgammar(X\nabla_X^1)_{0}-\gammar(X\nabla_X^1)_{0},(Y\nabla_Y^2)_{0}\right\rangle\\+
\left\langle (Y\nabla_Y^1)_{0},\cgammar(X\nabla_X^2)_{0}-\gammar(X\nabla_X^2)_{0}\right\rangle\\+
\left\langle (\cgammar-\gammar)(X\nabla_X^1)_{0},\cgammar(X\nabla_X^2)_{0}\right\rangle+
\left\langle \gammar(X\nabla_X^1)_{0},(\cgammar-\gammar)(X\nabla_X^2)_{0}\right\rangle.
\end{multline*}
The first summand in the right hand side above  equals
\[
\sum_\Delta\frac1{|\Delta|}\Tr(X\nabla_X^1)_\Delta^\Delta\Tr(Y\nabla_Y^2)_{\gammar(\Delta)}^{\gammar(\Delta)},
\]
where the sum is taken over all nontrivial row $X$-runs. By Lemma \ref{partrace}, each factor in this expression is constant, 
and hence the same holds true for the whole sum. The remaining three summands can be treated in a similar way. 

The remaining two terms in \eqref{bra} are treated in the same way as the second term.
\end{proof}

Based on Lemma \ref{difference}, from now on we proceed with the ringed versions of the last four terms in \eqref{bra}.

\subsubsection{Explicit expression for $\left\langle (\ceta^1_L)_{<},(\ceta_L^2)_{>}\right\rangle$}\label{etaletasec}
Let $\ttf^i$ be the $l^i\times l^i$ trailing minor of $\L^i$, then
\begin{equation}\label{lnal}
\L^i\nabla_{\L}^i=\begin{bmatrix} 0 & \ast \\ 0 & \one_{l^i} \end{bmatrix}, \qquad
\nabla_{\L}^i\L^i =\begin{bmatrix} 0 & 0 \\ \ast & \one_{l^i}\end{bmatrix}.
\end{equation}

Denote $\hat l^i=N(\L^i)-l^i+1$.
From now on we assume without loss of generality that 
\begin{equation}
\label{blockp}
\hat l^1\in L^1_p\cup\bar L^1_{p-1}.
\end{equation}

Consider the fixed block $X_{I^1_p}^{J^1_p}$ in $\L^1$ and an arbitrary block $X_{I^2_t}^{J^2_t}$ in $\L^2$. 
If $\beta_p^1>\beta_t^2$ then, by Proposition \ref{compar}(i) the second block fits completely inside the first one.  
This defines an injection $\rho$ of the subsets $K^2_t$ and $L^2_t$ of rows and columns of the matrix 
$\L^2$ into the subsets $K^1_p$ and $L^1_p$ of rows and columns of the matrix $\L^1$. Put
\begin{align}
\label{bad1}
\badB_t&=-\left\langle \lgrado_{\rho(\Phi_t^2)}^{\rho(\Phi_t^2)}
\Ld_{\Phi_t^2}^{{\bar L}_t^2}\nald_{\bar L_t^2}^{\Phi_t^2}\right\rangle,\\
\label{bad2}
\badC_t&=\left\langle \gradlo_{\rho(\Psi_t^2)}^{\rho(\Psi_t^2)}
\nald_{\Psi_t^2}^{{\bar K}_{t-1}^2}\Ld_{{\bar K}^2_{t-1}}^{\Psi_t^2}
\right\rangle,\\
\label{bad1eq}
\badBp_t&=
\left\langle\gradlo_{\Psi_p^1}^{L_p^1\setminus\Psi_p^1}
\nald_{L_t^2\setminus\Psi_t^2}^{K_t^2}\Ld_{K_t^2}^{\Psi_t^2}\right\rangle.
\end{align}

\begin{lemma}\label{etaletalemma} 
{\rm (i)} The expression for $\left\langle (\ceta^1_L)_{<},(\ceta_L^2)_{>}\right\rangle$ 
is given by
 \begin{multline}\label{etaleta}
\left\langle (\ceta^1_L)_{<},(\ceta_L^2)_{>}\right\rangle=\sum_{\beta_t^2< \beta_p^1}\left(\badB_t+\badC_t\right)+
\sum_{\beta_t^2= \beta_p^1}\badBp_t\\
+\sum_{\beta_t^2< \beta_p^1}\left(
\left\langle \lgrado_{\rho(K_t^2)}^{\rho(K_t^2)}\lgradd_{K_t^2}^{K_t^2}\right\rangle-
\left\langle \gradlo_{\rho(L_t^2)}^{\rho(L_t^2)}
\gradld_{L_t^2}^{L_t^2}\right\rangle\right)
\end{multline}
if $\hat l^1\in L_p^1$, and vanishes otherwise.

{\rm (ii)} Both summands in the last sum in \eqref{etaleta} are constant.
\end{lemma}

\begin{remark} Since $\left\langle A_1A_2\dots,A^1A^2\dots\right\rangle=\Tr(A_1A_2\dots A^1A^2\dots)$, 
here and in what follows we omit the comma and write just $\left\langle A_1A_2\dots A^1A^2\dots\right\rangle$
whenever $A_1, A_2, \dots$ and $A^1, A^2, \dots$ are matrices given by explicit expressions.
\end{remark}

\begin{proof}
 First of all, write
\begin{equation}\label{firstterm}
\left\langle (\ceta^1_L)_{<},(\ceta_L^2)_{>}\right\rangle=
\left\langle \cPi_{\Gamma_1}\left((\ceta^1_L)_{<}\right),\cPi_{\Gamma_1}\left((\ceta_L^2)_{>}\right)\right\rangle+ 
\left\langle \cPi_{\hat\Gamma_1}\left((\ceta^1_L)_{<}\right),\cPi_{\hat\Gamma_1}\left((\ceta_L^2)_{>}\right)\right\rangle
\end{equation}
with $\Gamma_1=\Gamma_1^\ec$.

It follows from the ringed version of \eqref{gammaid} that for $i=1,2$,
\begin{equation}\label{pigamma}
\cPi_{\Gamma_1}(\ceta_L^i)=\cgamma^*(\cxi_L^i)
\end{equation}
with $\cgamma=\cgammac$. Consequently, 
\[
\left\langle \cPi_{\Gamma_1}\left((\ceta^1_L)_{<}\right),\cPi_{\Gamma_1}\left((\ceta_L^2)_{>}\right)\right\rangle=
\left\langle \cPi_{\Gamma_1}\left((\ceta^1_L)_{<}\right),\cgamma^*\left((\cxi_L^2)_{>}\right)\right\rangle=0
\]
via the ringed version of \eqref{infinv1}.

Note that $\cPi_{\hat\Gamma_1}\left(\cgamma^*(\nabla^i_Y Y)\right)=0$ by the definition of $\cgamma^*$, 
therefore $\cPi_{\hat\Gamma_1}(\ceta_L^i)= \cPi_{\hat\Gamma_1}(\nabla_X^i X)$.

Let us compute $\nabla_X^i X$. Taking into account \eqref{naxnay} and \eqref{naxyxy}, we get
\begin{multline*}
\nabla_X^i X=\sum_{t=1}^{s^i}\begin{bmatrix}
(\nabla_{\L}^i)_{L_t^i}^{K_t^i}X_{I_t^i}^{J_t^i} & (\nabla_{\L}^i)_{L_t^i}^{K_t^i}X_{I_t^i}^{\hat J_t^i}\\0 & 0\end{bmatrix}\\=
\sum_{t=1}^{s^i}\begin{bmatrix}
(\nabla_{\L}^i\L^i)_{L_t^i}^{L_t^i\setminus\Psi_t^i} & (\nabla_{\L}^i)_{L_t^i}^{K_t^i}\L_{K_t^i}^{\Psi_t^i} & 
(\nabla_{\L}^i)_{L_t^i}^{K_t^i}X_{I_t^i}^{\hat J_t^i}\\
0 & 0 & 0\end{bmatrix},
\end{multline*}
where $\hat J_t^i=[1,n]\setminus J_t^i$. 
The latter equality follows from the fact that in columns $L_t^i\setminus\Psi_t^i$ all nonzero entries of $\L^i$  belong to the block $(\L^i)_{K_t^i}^{L_t^i}=X_{I_t^i}^{J_t^i}$, whereas in columns $\Psi_t^i$ nonzero entries of $\L^i$ belong also to the block $(\L^i)_{\bar K_{t-1}^i}^{\bar L_{t-1}^i}=Y_{\bar I_{t-1}^i}^{\bar J_{t-1}^i}$, see Fig.~\ref{fig:ladder}.
In more detail,
\begin{equation}\label{naxx}
\nabla_X^i X=\sum_{t=1}^{s^i}\begin{bmatrix}
(\nabla_{\L}^i\L^i)_{L_t^i\setminus\Psi_t^i}^{L_t^i\setminus\Psi_t^i} & 
(\nabla_{\L}^i)_{L_t^i\setminus\Psi_t^i}^{K_t^i}(\L^i)_{K_t^i}^{\Psi_t^i} & 
(\nabla_{\L}^i)_{L_t^i\setminus \Psi_t^i}^{K_t^i}X_{I_t^i}^{\hat J_t^i}\\
(\nabla_{\L}^i\L^i)_{\Psi_t^i}^{L_t^i\setminus\Psi_t^i} & (\nabla_{\L}^i)_{\Psi_t^i}^{K_t^i}(\L^i)_{K_t^i}^{\Psi_t^i} & 
(\nabla_{\L}^i)_{\Psi_t^i}^{K_t^i}X_{I_t^i}^{\hat J_t^i}\\
0 & 0 & 0 \end{bmatrix}.
\end{equation}

Note that the upper left block in \eqref{naxx} is lower triangular by \eqref{lnal}. Besides, the projection of the middle block onto $\hat\Gamma_1$ vanishes, since it corresponds to the diagonal block defined by the nontrivial $X$-run 
$\Delta(\beta_t^i)$ (or is void if $t=1$ and $\Psi_1^i=\varnothing$). 

It follows from the explanations above and
\eqref{lnal} that the contribution of the $t$-th summand in \eqref{naxx} to $\cPi_{\hat\Gamma_1}\left((\ceta^1_L)_{<}\right)$ vanishes, unless $t=p$. Moreover, if $\hat l^1\in \bar L_{p-1}^1\setminus\Psi_p^1$, it vanishes for $t=p$
as well. So, in what follows we assume that $\hat l^1\in L_p^1$. In this case \eqref{naxx} yields 
\begin{equation}\label{pigalo}
\cPi_{\hat\Gamma_1}\left((\ceta_L^1)_{<}\right)=\cPi_{\hat\Gamma_1}\begin{bmatrix}
\left((\nabla_{\L}^1\L^1)_{L_p^1}^{L_p^1}\right)_{<} & 0 \\ 0 & 0\end{bmatrix}.
\end{equation}

On the other hand,
\begin{equation}\label{pigaup}
\cPi_{\hat\Gamma_1}\left((\ceta_L^2)_{>}\right)=\sum_{t=1}^{s^2}\begin{bmatrix}
0 & (\nabla_{\L}^2)_{L_t^2\setminus\Psi_t^2}^{K_t^2}(\L^2)_{K_t^2}^{\Psi_t^2} & 
(\nabla_{\L}^2)_{L_t^2\setminus \Psi_t^2}^{K_t^2}X_{I_t^2}^{\hat J^2_t}\\
0 & 0 & 
(\nabla_{\L}^2)_{\Psi_t^2}^{K_t^2}X_{I_t^2}^{\hat J^2_t}\\
0 & 0 & 0 \end{bmatrix},
\end{equation}
where the $t$-th summand corresponds to the $t$-th $X$-block of $\L^2$.

If $\beta_p^1<\beta_t^2$, then the contribution of the $t$-th summand in \eqref{pigaup}
to the second term in \eqref{firstterm} vanishes by \eqref{pigalo}, since in this case 
$J^1_p\subseteq J^2_t\setminus\Delta(\beta_t^2)$, which means that the upper left block in
\eqref{pigalo} fits completely within the zero upper left block in \eqref{pigaup}.

 Assume that $\beta_p^1>\beta_t^2$. Then, to the contrary, $J^2_t\subseteq J^1_p\setminus
\Delta(\beta_p^1)$, and hence  
$\rho(L_t^2)\subseteq L_p^1\setminus\Psi_p^1$. Note that by \eqref{pigalo}, to compute the second term in \eqref{firstterm}  one can replace $\hat J_t^2$ in \eqref{pigaup} by $J_p^1\setminus J_t^2$. So, using the above injection $\rho$, one can rewrite the two upper blocks at the $t$-th summand of $\cPi_{\hat\Gamma_1}\left((\ceta^2_L)_{>}\right)$ in \eqref{pigaup} as one block
\[
\nald_{L_t^2\setminus \Psi_t^2}^{K_t^2}\Lo_{\rho(K_t^2)}^{L_p^1\setminus \rho(L_t^2\setminus \Psi_t^2)},
\]
and the remaining nonzero block in the same summand as
\[
\nald_{\Psi_t^2}^{K_t^2}(\L^1)_{\rho(K_t^2)}^{L_p^1\setminus \rho(L_t^2)}.
\]
The corresponding blocks of $\cPi_{\hat\Gamma_1}\left((\ceta^1_L)_{<}\right)$ in \eqref{pigalo} are
\[
\gradlo_{L_p^1\setminus \rho(L_t^2\setminus \Psi_t^2)}^{\rho(L_t^2\setminus\Psi_t^2)}=
\nalo_{L_p^1\setminus \rho(L_t^2\setminus \Psi_t^2)}^{K_p^1}\Lo_{K_p^1}^{\rho(L_t^2\setminus\Psi_t^2)}
\]
and
\[
\gradlo_{L_p^1\setminus \rho(L_t^2)}^{\rho(\Psi_t^2)}=
\nalo_{L_p^1\setminus \rho(L_t^2)}^{K_p^1}\Lo_{K_p^1}^{\rho(\Psi_t^2)}.
\]
The equalities follow from the fact that all nonzero entries in the columns $\rho(L_t^2)$ of $\L^1$ belong to the $X$-block,
see Fig.~\ref{fig:ladder}.

 The contribution of the first blocks in each pair can be rewritten as
\begin{equation}\label{cont1}
\left\langle \Lo_{\rho(K_t^2)}^{L_p^1\setminus \rho(L_t^2\setminus \Psi_t^2)}
\nalo_{L_p^1\setminus \rho(L_t^2\setminus \Psi_t^2)}^{K_p^1}
\Lo_{K_p^1}^{\rho(L_t^2\setminus\Psi_t^2)}
\nald_{L_t^2\setminus \Psi_t^2}^{K_t^2}\right\rangle.
\end{equation}
Recall that $\rho(K_t^2)\subseteq K_p^1$. If the inclusion is strict, then immediately
\begin{multline}\label{trick1}
\Lo_{\rho(K_t^2)}^{L_p^1\setminus \rho(L_t^2\setminus \Psi_t^2)}
\nalo_{L_p^1\setminus \rho(L_t^2\setminus \Psi_t^2)}^{K_p^1}\\
=\lgrado_{\rho(K_t^2)}^{K_p^1}-
\Lo_{\rho(K_t^2)}^{\rho\left(L_t^2\setminus \Psi_t^2\right)}
\nalo_{\rho\left(L_t^2\setminus \Psi_t^2\right)}^{K_p^1}\\
=\lgrado_{\rho(K_t^2)}^{K_p^1}-
\Ld_{K_t^2}^{L_t^2\setminus \Psi_t^2}
\nalo_{\rho\left(L_t^2\setminus \Psi_t^2\right)}^{K_p^1}.
\end{multline}
Otherwise there is an additional term
\[
-\Lo_{K_p^1}^{{\bar L}_p^1}\nalo_{{\bar L}_p^1}^{K_p^1}
\]
in the right hand side of \eqref{trick1}. However, for the same reason as above, 
\[
\nalo_{{\bar L}_p^1}^{K_p^1}\Lo_{K_p^1}^{\rho(L_t^2\setminus\Psi_t^2)}=
\gradlo_{{\bar L}_p^1}^{\rho(L_t^2\setminus\Psi_t^2)}.
\]
Note that $\rho(L_t^2\setminus\Psi_t^2)\subset L_p^1$, and ${\bar L}_p^1$ lies strictly to the left of $L_p^1$, see
Fig.~\ref{fig:ladder}. Consequently, by \eqref{lnal}, the latter submatrix vanishes. Therefore, the additional term does not contribute to \eqref{cont1}.

To find the contribution of the second term in \eqref{trick1} to \eqref{cont1}, note that
\begin{equation}
\label{tfor1}
\nalo_{\rho\left(L_t^2\setminus \Psi_t^2\right)}^{K_p^1}\Lo_{K_p^1}^{\rho(L_t^2\setminus\Psi_t^2)}=
\gradlo_{\rho(L_t^2\setminus\Psi_t^2)}^{\rho(L_t^2\setminus\Psi_t^2)}
\end{equation}
and
\[
\nald_{L_t^2\setminus \Psi_t^2}^{K_t^2}\Ld_{K_t^2}^{L_t^2\setminus \Psi_t^2}=
\gradld_{L_t^2\setminus \Psi_t^2}^{L_t^2\setminus \Psi_t^2}
\]
for the same reason as above, and hence the contribution in question equals 
\[
-\left\langle \gradld_{L_t^2\setminus \Psi_t^2}^{L_t^2\setminus \Psi_t^2}
\gradlo_{\rho(L_t^2\setminus\Psi_t^2)}^{\rho(L_t^2\setminus\Psi_t^2)}\right\rangle=
\text{const}
\]
by \eqref{lnal}.

 Similarly to \eqref{cont1}, \eqref{trick1}, the contribution of the second blocks in each pair can be rewritten as 
\begin{equation}\label{cont2}
\left\langle \lgrado_{\rho(K_t^2)}^{K_p^1}-
\Lo_{\rho(K_t^2)}^{\rho(L_t^2)}\nalo_{\rho(L_t^2)}^{K_p^1},
\Lo_{K_p^1}^{\rho(\Psi_t^2)}\nald_{\Psi_t^2}^{K_t^2}\right\rangle.
\end{equation}
As in the previous case, an additional term arises if $\rho(K_t^2)= K_p^1$, and its contribution to
\eqref{cont2} vanishes.

Note that by \eqref{lnal}, one has
\[
\lgrado_{\rho(K_t^2)}^{K_p^1}\Lo_{K_p^1}^{\rho(L_t^2\setminus\Psi_t^2)}=
\lgrado_{\rho(K_t^2)}^{\rho(K_t^2)}\Lo_{\rho(K_t^2)}^{\rho(L_t^2\setminus\Psi_t^2)}
\]
and
\[
\lgrado_{\rho(K_t^2)}^{K_p^1}\Lo_{K_p^1}^{\rho(\Psi_t^2)}=
\lgrado_{\rho(K_t^2)}^{\rho(K_t^2)}\Lo_{\rho(K_t^2)}^{\rho(\Psi_t^2)},
\]
hence the total contribution of the first terms in \eqref{trick1} and \eqref{cont2} equals
\begin{multline}\label{tfor2}
\left\langle \lgrado_{\rho(K_t^2)}^{\rho(K_t^2)},
\Lo_{\rho(K_t^2)}^{\rho(L_t^2\setminus\Psi_t^2)}\nald_{L_t^2\setminus \Psi_t^2}^{K_t^2}+
\Lo_{\rho(K_t^2)}^{\rho(\Psi_t^2)}\nald_{\Psi_t^2}^{K_t^2}\right\rangle\\
=\left\langle \lgrado_{\rho(K_t^2)}^{\rho(K_t^2)},
\Ld_{K_t^2}^{L_t^2\setminus\Psi_t^2}\nald_{L_t^2\setminus \Psi_t^2}^{K_t^2}+
\Ld_{K_t^2}^{\Psi_t^2}\nald_{\Psi_t^2}^{K_t^2}\right\rangle\\
=\left\langle \lgrado_{\rho(K_t^2)}^{\rho(K_t^2)},
\lgradd_{K_t^2}^{K_t^2}-U_t\nald_{\bar L_t^2}^{K_t^2}
\right\rangle,
\end{multline}
where
\[
U_t=\begin{bmatrix} \Ld_{\Phi_t^2}^{{\bar L}_t^2} \\ 0\end{bmatrix}.
\]
Note that 
\[
\left\langle \lgrado_{\rho(K_t^2)}^{\rho(K_t^2)}
\lgradd_{K_t^2}^{K_t^2}\right\rangle=\text{const}
\]
by \eqref{lnal}, which gives the first summand 
in the last sum in \eqref{etaleta}. The remaining term equals
\begin{multline*}
-\left\langle \lgrado_{\rho(K_t^2)}^{\rho(K_t^2)}U_t\nald_{\bar L_t^2}^{K_t^2}\right\rangle=
-\left\langle \lgrado_{\rho(K_t^2)}^{\rho(\Phi_t^2)}
\Ld_{\Phi_t^2}^{{\bar L}_t^2}\nald_{\bar L_t^2}^{K_t^2}\right\rangle\\
=-\left\langle \lgrado_{\rho(\Phi_t^2)}^{\rho(\Phi_t^2)}
\Ld_{\Phi_t^2}^{{\bar L}_t^2}\nald_{\bar L_t^2}^{\Phi_t^2}\right\rangle,
\end{multline*}
which coincides with the expression for $\badB_t$ in \eqref{bad1};
the last equality above follows from \eqref{lnal}.

It remains to compute the contribution of the second term in \eqref{cont2}. Similarly to \eqref{tfor1}, we have
\[
\nalo_{\rho\left(L_t^2\right)}^{K_p^1}\Lo_{K_p^1}^{\rho(\Psi_t^2)}=
\gradlo_{\rho(L_t^2)}^{\rho(\Psi_t^2)}.
\]
On the other hand, similarly to \eqref{tfor2}, we have
\[
\nald_{\Psi_t^2}^{K_t^2}\Ld_{K_t^2}^{L_t^2}=
\gradld_{\Psi_t^2}^{L_t^2}-\nald_{\Psi_t^2}^{{\bar K}_{t-1}^2}V_t,
\]
where
\[
V_t=\begin{bmatrix} 0 & \Ld_{{\bar K}^2_{t-1}}^{\Psi_t^2}\end{bmatrix}.
\]
As before, we use \eqref{lnal} to get 
\[
-\left\langle
\gradlo_{\rho(L_t^2)}^{\rho(\Psi_t^2)}
\gradld_{\Psi_t^2}^{L_t^2}\right\rangle=
-\left\langle\gradlo_{\rho(\Psi_t^2)}^{\rho(\Psi_t^2)}
\gradld_{\Psi_t^2}^{\Psi_t^2}\right\rangle=\text{const},
\]
which together with the contribution of the second term in \eqref{trick1} computed above yields the second summand 
in the last sum in \eqref{etaleta}.
The remaining term is given by 
\begin{equation*}
\left\langle \gradlo_{\rho(L_t^2)}^{\rho(\Psi_t^2)}
\nald_{\Psi_t^2}^{{\bar K}_{t-1}^2}V_t
\right\rangle=
\left\langle \gradlo_{\rho(\Psi_t^2)}^{\rho(\Psi_t^2)} 
\nald_{\Psi_t^2}^{{\bar K}_{t-1}^2}\Ld_{{\bar K}^2_{t-1}}^{\Psi_t^2}
\right\rangle,
\end{equation*}
which coincides with the expression for $\badC_t$ in \eqref{bad2}.

Assume now that $\beta_p^1=\beta_t^2$ and hence $J_p^1=J_t^2$. In this case 
the blocks $X_{I^2_t}^{J^2_t}$ and $X_{I^1_p}^{J^1_p}$ have the same width, and one of them lies inside the other, but the direction of the inclusion may vary, and hence $\rho$ is not defined. 

Note that by \eqref{pigalo}, to compute the second term in \eqref{firstterm} in this case, one can omit the columns 
$\hat J_t^2$ in \eqref{pigaup}, 
and hence the contribution in question equals
\begin{equation*}
\left\langle\gradlo_{\Psi_p^1}^{L_p^1\setminus\Psi_p^1}
\nald_{L_t^2\setminus\Psi_t^2}^{K_t^2}\Ld_{K_t^2}^{\Psi_t^2}\right\rangle,
\end{equation*} 
which coincides with the expression for $\badBp_t$ in \eqref{bad1eq}.
\end{proof}

\subsubsection{Explicit expression for $\left\langle (\ceta_R^1)_{\ge},(\ceta_R^2)_{\le}\right\rangle$}\label{etaretasec}
Recall that $\hat l^1\in L_p^1\cup \bar L_{p-1}^1$ by \eqref{blockp}. Consequently, $\hat l^1\in K_p^1\cup\bar K_{p-1}^1$;
more exactly, either $\hat l^1\in K_p^1\setminus \Phi_p^1$, or
\begin{equation}\label{blockq}
\hat l^1\in \bar K_q^1\quad\text{with $q=p$ or $q=p-1$},
\end{equation}
see Fig.~\ref{fig:ladder}.
Consider a fixed block $Y_{\bar I^1_q}^{\bar J^1_q}$ in $\L^1$ and an arbitrary block $Y_{\bar I^2_t}^{\bar J^2_t}$ in $\L^2$. 
If $\bar\alpha_q^1> \bar\alpha_t^2$ then, by Proposition \ref{compar}(ii) the second block fits completely inside 
the first one.  This defines an injection $\sigma$ of the subsets $\bar K^2_t$ and $\bar L^2_t$ of rows and columns of the matrix $\L^2$ into the subsets $\bar K^1_{q}$ and $\bar L^1_{q}$ of rows and columns of the matrix $\L^1$. Put
\begin{align}\label{bad3}
\badD_t&=-\left\langle \gradlo_{\sigma(\Psi_{t+1}^2)}^{\sigma(\Psi_{t+1}^2)}
\nald^{ K_{t+1}^2}_{\Psi_{t+1}^2}\Ld_{ K_{t+1}^2}^{\Psi_{t+1}^2}\right\rangle,\\
\label{bad4}
\badE_t&=\left\langle \lgrado^{\sigma(\Phi_t^2)}_{\sigma(\Phi_t^2)}
\Ld_{\Phi_t^2}^{L_t^2}\nald^{\Phi_t^2}_{L_{t}^2}
\right\rangle,\\
\label{bad3eq}
\badDp_t&=\left\langle\lgrado^{\Phi_q^1}_{\bar K_q^1\setminus\Phi_q^1}
\Ld^{\bar L_t^2}_{\Phi_t^2}\nald^{\bar K_t^2\setminus\Phi_t^2}_{\bar L_t^2}\right\rangle.
\end{align}

\begin{lemma}\label{etaretalemma} 
{\rm (i)} The expression for $\left\langle (\ceta_R^1)_{\ge},(\ceta_R^2)_{\le}\right\rangle$  
is given by
 \begin{multline}\label{etareta}
\left\langle (\ceta_R^1)_{\ge},(\ceta_R^2)_{\le}\right\rangle=\left\langle (\ceta_R^1)_{0},(\ceta_R^2)_{0}\right\rangle
+\sum_{\bar\alpha_t^2 < \bar\alpha_q^1}\left(\badD_t+\badE_t\right)+
\sum_{ \bar\alpha_t^2=\bar\alpha_q^1}\badDp_t\\
+\sum_{\bar\alpha_t^2 < \bar\alpha_q^1}\left(
\left\langle \gradlo^{\sigma(\bar L_t^2)}_{\sigma(\bar L_t^2)}
\gradld^{\bar L_t^2}_{\bar L_t^2}\right\rangle
-\left\langle \lgrado^{\sigma(\bar K_t^2)}_{\sigma(\bar K_t^2)}
\lgradd^{\bar K_t^2}_{\bar K_t^2}\right\rangle\right)
\end{multline}
if $\hat l^1\in\bar K_q^1$, and equals 
$\left\langle (\ceta_R^1)_{0},(\ceta_R^2)_{0}\right\rangle$ otherwise. 

{\rm (ii)} The first term  and both summands in the last sum in the right hand side of \eqref{etareta} are constant.
\end{lemma}

\begin{proof}
 Clearly, 
$\left\langle (\ceta_R^1)_{\ge},(\ceta_R^2)_{\le}\right\rangle=
\left\langle (\ceta_R^1)_{0},(\ceta_R^2)_{0}\right\rangle+
\left\langle (\ceta_R^1)_{>},(\ceta_R^2)_{<}\right\rangle$.  
The first term on the right is constant by the ringed version of \eqref{infinv3}, so in what follows we only look at the second term.
Similarly to \eqref{firstterm}, we have
\begin{equation}\label{secondterm}
\left\langle (\ceta^1_R)_{>},(\ceta_R^2)_{<}\right\rangle=
\left\langle \cPi_{\Gamma_2}\left((\ceta^1_R)_{>}\right),\cPi_{\Gamma_2}\left((\ceta_R^2)_{<}\right)\right\rangle+ 
\left\langle \cPi_{\hat\Gamma_2}\left((\ceta^1_R)_{>}\right),\cPi_{\hat\Gamma_2}\left((\ceta_R^2)_{<}\right)\right\rangle
\end{equation}
with $\Gamma_2=\Gamma_2^\er$.

It follows from the ringed version of \eqref{gammaid} that for $i=1,2$,
\begin{equation}\label{pigamma2}
\cPi_{\Gamma_2}(\ceta_R^i)=\cgamma(\cxi_R^i)
\end{equation}
with $\cgamma=\cgammar$. Consequently, 
\[
\left\langle \cPi_{\Gamma_2}\left((\ceta^1_R)_{>}\right),\cPi_{\Gamma_2}\left((\ceta_R^2)_{<}\right)\right\rangle=
\left\langle \cPi_{\Gamma_2}\left((\ceta^1_R)_{>}\right),\cgamma\left((\cxi_R^2)_{<}\right)\right\rangle=0
\]
via the ringed version of \eqref{infinv1}.

 Note that $\cPi_{\hat\Gamma_2}\left(\cgamma(X\nabla^i_X)\right)=0$ by the definition of $\cgamma$, therefore
$\cPi_{\hat\Gamma_2}(\ceta_R^i)= \cPi_{\hat\Gamma_2}(Y\nabla_Y^i)$.

Let us compute $Y\nabla_Y^i$. Taking into account \eqref{naxnay} and \eqref{xynaxy}, we get
\begin{equation*}
Y\nabla_Y^i =\sum_{t=1}^{s^i}\begin{bmatrix}
Y_{\bar I_t^i}^{\bar J_t^i}(\nabla_{\L}^i)_{\bar L_t^i}^{\bar K_t^i} & 0 \\
Y_{\hat{\bar I}_t^i}^{\bar J_t^i}(\nabla_{\L}^i)_{\bar L_t^i}^{\bar K_t^i} & 0\end{bmatrix}=
\sum_{t=1}^{s^i}\begin{bmatrix}
(\L\nabla_{\L}^i)^{\bar K_t^i}_{\bar K_t^i\setminus\Phi_t^i} & 0\\
(\L^i)^{\bar L_t^i}_{\Phi_t^i}(\nabla_{\L}^i)^{\bar K_t^i}_{\bar L_t^i} & 0\\
Y_{\hat{\bar I}_t^i}^{\bar J_t^i}(\nabla_{\L}^i)_{\bar L_t^i}^{\bar K_t^i} & 0\end{bmatrix},
\end{equation*}
where $\hat{\bar I}_t^i=[1,n]\setminus \bar I_t^i$;
the latter equality follows from the fact that in rows $\bar K_t^i\setminus\Phi_t^i$ all nonzero entries of $\L^i$  belong to the block $(\L^i)_{\bar K_t^i}^{\bar L_t^i}=Y_{\bar I_t^i}^{\bar J_t^i}$, whereas in rows 
$\Phi_t^i$ nonzero entries of $\L^i$ belong also to the block $(\L^i)_{K_{t}^i}^{L_{t}^i}=X_{I_{t}^i}^{J_{t}^i}$,
see Fig.~\ref{fig:ladder}.
In more detail,
\begin{equation}\label{ynay}
Y\nabla_Y^i=\sum_{t=1}^{s^i}\begin{bmatrix}
(\L^i\nabla_{\L}^i)^{\bar K_t^i\setminus\Phi_t^i}_{\bar K_t^i\setminus\Phi_t^i} & 
(\L^i\nabla_{\L}^i)^{\Phi_t^i}_{\bar K_t^i\setminus\Phi_t^i} & 0\\
(\L^i)^{\bar L_t^i}_{\Phi_t^i}(\nabla_{\L}^i)^{\bar K_t^i\setminus\Phi_t^i}_{\bar L_t^i} & 
(\L^i)^{\bar L_t^i}_{\Phi_t^i}(\nabla_{\L}^i)^{\Phi_t^i}_{\bar L_t^i} & 0 \\
Y^{\bar J_t^i}_{\hat{\bar I}^i_t} (\nabla_{\L}^i)^{\bar K_t^i\setminus \Phi_t^i}_{\bar L_t^i} & 
Y^{\bar J_t^i}_{\hat{\bar I}^i_t}(\nabla_{\L}^i)^{\Phi_t^i}_{\bar L_t^i} & 0\end{bmatrix}.
\end{equation}

Note that the upper left block in \eqref{ynay} is upper triangular by \eqref{lnal}. Besides, the projection of the middle block onto $\hat\Gamma_2$ vanishes, since for $\Phi_t^i\ne\varnothing$, the middle block corresponds to the diagonal block defined by the nontrivial $Y$-run $\bar\Delta(\bar\alpha_t^i)$.

Recall that  $\hat l^1\in K_p^1\cup\bar K_{p-1}^1$, therefore
by \eqref{lnal}, the contribution of the $t$-th summand in \eqref{ynay} to $\cPi_{\hat\Gamma_2}\left((\ceta_R^1)_{>}\right)$ vanishes, unless $t\ne q$, where $q$ is either $p$ or $p-1$. Moreover, if 
$\hat l^1\in K_p^1\setminus\Phi_p^1$, this contribution vanishes for $t=q$ as well, see Fig.~\ref{fig:ladder}. So, in what follows $\hat l^1\in \bar K_q^1$, in which case 
\begin{equation}\label{pigalo2}
\cPi_{\hat\Gamma_2}\left((\ceta_R^1)_{>}\right)=\cPi_{\hat\Gamma_2}
\begin{bmatrix} \left((\L^1\nabla_{\L}^1)^{\bar K_q^1}_{\bar K_q^1}\right)_{>} 
& 0\\0 & 0 \end{bmatrix}.
\end{equation}
On the other hand, 
\begin{equation}\label{pigaup2}
\cPi_{\hat\Gamma_2}\left((\ceta_R^2)_{<}\right)=\sum_{t=1}^{s^2}
\begin{bmatrix} 0 & 0 & 0 \\
(\L^2)^{\bar L_t^2}_{\Phi_t^2}(\nabla_{\L}^2)^{\bar K_t^2\setminus\Phi_t^2}_{\bar L_t^2} & 0 & 0 \\
Y^{\bar J_t^2}_{\hat{\bar I}^2_t} (\nabla_{\L}^2)^{\bar K_t^2\setminus \Phi_t^2}_{\bar L_t^2} & 
Y^{\bar J_t^2}_{\hat{\bar I}^2_t}(\nabla_{\L}^2)^{\Phi_t^2}_{\bar L_t^2} & 0\end{bmatrix},
\end{equation}
where the $t$-th summand corresponds to the $t$-th $Y$-block in $\L^2$.

If $\bar\alpha_q^1< \bar\alpha_t^2$, then the contribution of the $t$-th summand in 
\eqref{pigaup2} to the second term in \eqref{secondterm} vanishes by \eqref{pigalo2}, since in this case $\bar I^1_{q}\subseteq \bar I^2_t\setminus\bar\Delta(\bar\alpha^2_t)$.

Assume that $\bar\alpha_q^1> \bar\alpha_t^2$. Then, to the contrary, $\bar I^2_t\subseteq 
\bar I^1_{q}\setminus\bar\Delta(\bar\alpha_q^1)$, and hence $\sigma(\bar K_t^2)\subseteq \bar K_q^1\setminus \Phi_q^1$. 
Note that by \eqref{pigalo2}, to compute the second term in \eqref{secondterm}, one can replace $\hat{\bar I}_t^2$ 
in \eqref{pigaup2} by $\bar I_{q}^1\setminus \bar I_t^2$. So, using the above injection $\sigma$, one can rewrite the two upper blocks at the $t$-th summand of $\cPi_{\hat\Gamma_2}\left((\ceta^2_R)_{<}\right)$ in \eqref{pigaup2} as one block
\[
\Lo^{\sigma(\bar L_t^2)}_{\bar K_{q}^1\setminus \sigma(\bar K_t^2\setminus\Phi_t^2)}
\nald^{\bar K_t^2\setminus \Phi_t^2}_{\bar L_t^2},
\]
and the remaining nonzero block in the same summand as
\[
\Lo^{\sigma(\bar L_t^2)}_{\bar K_{q}^1\setminus \sigma(\bar K_t^2)}
\nald^{\Phi_t^2}_{\bar L_t^2}.
\]
The corresponding blocks of $\cPi_{\hat\Gamma_2}\left((\ceta^1_R)_{>}\right)$ in \eqref{pigalo2} are
\[
\lgrado^{\bar K_{q}^1\setminus \sigma(\bar K_t^2\setminus\Phi_t^2)}
_{\sigma(\bar K_t^2\setminus\Phi_t^2)}=
\Lo^{\bar L_{q}^1}_{\sigma(\bar K_t^2\setminus\Phi_t^2)}
\nalo^{\bar K_{q}^1\setminus \sigma(\bar K_t^2\setminus\Phi_t^2)}_{\bar L_{q}^1}
\]
and
\[
\lgrado^{\bar K_{q}^1\setminus \sigma(\bar K_t^2)}_{\sigma(\Phi_t^2)}=
\Lo^{\bar L_{q}^1}_{\sigma(\Phi_t^2)}
\nalo^{\bar K_{q}^1\setminus \sigma(\bar K_t^2)}_{\bar L_{q}^1}.
\]
The equalities follow from the fact that all nonzero entries in the rows $\sigma(\bar K_t^2)$ of $\L^1$ belong to the $Y$-block, see Fig.~\ref{fig:ladder}.

 The contribution of the first blocks in each pair can be rewritten as
\begin{equation}\label{cont12}
\left\langle 
\nalo^{\bar K_{q}^1\setminus \sigma(\bar K_t^2\setminus\Phi_t^2)}_{\bar L_{q}^1}
\Lo^{\sigma(\bar L_t^2)}_{\bar K_{q}^1\setminus \sigma(\bar K_t^2\setminus\Phi_t^2)}
\nald^{\bar K_t^2\setminus \Phi_t^2}_{\bar L_t^2}
\Lo^{\bar L_{q}^1}_{\sigma(\bar K_t^2\setminus\Phi_t^2)}
\right\rangle.
\end{equation}
Recall that $\sigma(\bar L_t^2)\subseteq \bar L_{q}^1$. If the inclusion is strict, then immediately
\begin{multline}\label{trick12}
\nalo^{\bar K_{q}^1\setminus \sigma(\bar K_t^2\setminus\Phi_t^2)}_{\bar L_{q}^1}
\Lo^{\sigma(\bar L_t^2)}_{\bar K_{q}^1\setminus \sigma(\bar K_t^2\setminus\Phi_t^2)}\\
=\gradlo^{\sigma(\bar L_t^2)}_{\bar L_{q}^1}-
\nalo^{\sigma\left(\bar K_t^2\setminus \Phi_t^2\right)}_{\bar L_{q}^1}
\Lo^{\sigma(\bar L_t^2)}_{\sigma\left(\bar K_t^2\setminus \Phi_t^2\right)}\\
=\gradlo^{\sigma(\bar L_t^2)}_{\bar L_{q}^1}-
\nalo^{\sigma\left(\bar K_t^2\setminus \Phi_t^2\right)}_{\bar L_{q}^1}
\Ld^{\bar L_t^2}_{\bar K_t^2\setminus \Phi_t^2}.
\end{multline}
Otherwise there is an additional term
\[
-\nalo^{K_{q}^1}_{\bar L_{q}^1}\Lo^{\bar L_{q}^1}_{K_{q}^1}
\]
in the right hand of \eqref{trick12}. However, for the same reason as those discussed during the treatment of \eqref{cont1}, 
\[
\Lo^{\bar L_{q}^1}_{\sigma(\bar K_t^2\setminus\Phi_t^2)}
\nalo^{K_{q}^1}_{\bar L_{q}^1}=
\lgrado^{K_{q}^1}_{\sigma(\bar K_t^2\setminus\Phi_t^2)}.
\]
Note that $\sigma(\bar K_t^2\setminus\Phi_t^2)\subseteq \bar K_{q}^1\setminus\Phi_q^1$ and $K_{q}^1$ lies strictly below  $\bar K_{q}^1\setminus\Phi_q^1$, see Fig.~\ref{fig:ladder}. Hence by \eqref{lnal} the above submatrix vanishes, and 
the additional term does not contribute to \eqref{cont12}.

To find the contribution of the second term in \eqref{trick12} to \eqref{cont12}, note that
\begin{equation}\label{tfor3}
\Lo^{\bar L_{q}^1}_{\sigma(\bar K_t^2\setminus\Phi_t^2)}
\nalo^{\sigma\left(\bar K_t^2\setminus \Phi_t^2\right)}_{\bar L_{q}^1}=
\lgrado^{\sigma(\bar K_t^2\setminus\Phi_t^2)}_{\sigma(\bar K_t^2\setminus\Phi_t^2)}
\end{equation}
and
\[
\Ld^{\bar L_t^2}_{\bar K_t^2\setminus \Phi_t^2}
\nald^{\bar K_t^2\setminus \Phi_t^2}_{\bar L_t^2}=
\lgradd^{\bar K_t^2\setminus \Phi_t^2}_{\bar K_t^2\setminus \Phi_t^2},
\]
and hence the contribution in question equals 
\[
-\left\langle \lgradd^{\bar K_t^2\setminus \Phi_t^2}_{\bar K_t^2\setminus \Phi_t^2}
\lgrado^{\sigma(\bar K_t^2\setminus\Phi_t^2)}_{\sigma(\bar K_t^2\setminus\Phi_t^2)}\right\rangle=
\text{const}
\]
by \eqref{lnal}.

 Similarly to \eqref{cont2}, the contribution of the second blocks in each pair above can be rewritten as 
\begin{equation}\label{cont22}
\left\langle \gradlo^{\sigma(\bar L_t^2)}_{\bar L_{q}^1}-
\nalo^{\sigma(\bar K_t^2)}_{\bar L_{q}^1}
\Lo^{\sigma(\bar L_t^2)}_{\sigma(\bar K_t^2)},
\nald^{\Phi_t^2}_{\bar L_t^2}\Lo^{\bar L_{q}^1}_{\sigma(\Phi_t^2)}\right\rangle.
\end{equation}
As in the previous case, an additional term arises if $\sigma(\bar L_t^2)= \bar L_{q}^1$, and its contribution to
\eqref{cont22} vanishes.

 To find the total contribution of the first terms in \eqref{trick12} and \eqref{cont22}, note that by \eqref{lnal}, in this computation one can replace the row set $\bar L_q^1$
of $\L^1\nabla_{\L^1}$ with $\sigma(\bar L_t^2)$. Therefore,  the contribution in question equals
\begin{multline}\label{tfor4}
\left\langle \gradlo^{\sigma(\bar L_t^2)}_{\sigma(\bar L_t^2)},
\nald^{\bar K_t^2\setminus \Phi_t^2}_{\bar L_t^2}
\Lo^{\sigma(\bar L_t^2)}_{\sigma(\bar K_t^2\setminus\Phi_t^2)}
+\nald^{\Phi_t^2}_{\bar L_t^2}\Lo^{\sigma(\bar L_t^2)}_{\sigma(\Phi_t^2)}
\right\rangle\\
=\left\langle \gradlo^{\sigma(\bar L_t^2)}_{\sigma(\bar L_t^2)},
\nald^{\bar K_t^2\setminus \Phi_t^2}_{\bar L_t^2}
\Ld^{\bar L_t^2}_{\bar K_t^2\setminus\Phi_t^2}+
\nald^{\Phi_t^2}_{\bar L_t^2}\Ld^{\bar L_t^2}_{\Phi_t^2}
\right\rangle\\
=\left\langle \gradlo^{\sigma(\bar L_t^2)}_{\sigma(\bar L_t^2)},
\gradld^{\bar L_t^2}_{\bar L_t^2}-
\nald^{K_{t+1}^2}_{\bar L_t^2}W_t
\right\rangle,
\end{multline}
where
\[
W_t=\begin{bmatrix} \Ld_{K_{t+1}^2}^{\Psi_{t+1}^2} & 0\end{bmatrix}.
\]
Note that 
\[
\left\langle \gradlo^{\sigma(\bar L_t^2)}_{\sigma(\bar L_t^2)}
\gradld^{\bar L_t^2}_{\bar L_t^2}\right\rangle=\text{const}
\]
by \eqref{lnal}, which gives the first summand in the last sum in \eqref{etareta}. 
The remaining term is given by
\begin{equation*}
-\left\langle \gradlo_{\sigma(\bar L_t^2)}^{\sigma(\bar L_t^2)}
\nald^{ K_{t+1}^2}_{\bar L_t^2}W_t\right\rangle=
-\left\langle \gradlo_{\sigma(\Psi_{t+1}^2)}^{\sigma(\Psi_{t+1}^2)}
\nald^{ K_{t+1}^2}_{\Psi_{t+1}^2}\Ld_{ K_{t+1}^2}^{\Psi_{t+1}^2}\right\rangle,
\end{equation*}
which coincides with the expression for $\badD_t$  in \eqref{bad3}.

It remains to compute the contribution of the second term in \eqref{cont22}. Similarly to
\eqref{tfor3}, we have
\[
\Lo^{\bar L_{q}^1}_{\sigma(\Phi_t^2)}
\nalo^{\sigma\left(\bar K_t^2\right)}_{\bar L_{q}^1}=
\lgrado^{\sigma(\bar K_t^2)}_{\sigma(\Phi_t^2)}.
\]
On the other hand, similarly to \eqref{tfor4}, we have
\[
\Ld^{\bar L_t^2}_{\bar K_t^2}\nald^{\Phi_t^2}_{\bar L_t^2}=
\lgradd^{\Phi_t^2}_{\bar K_t^2}-
Z_t\nald^{\Phi_t^2}_{{ L}_{t}^2},
\]
where
\[
Z_t=\begin{bmatrix} 0 \\ \Ld^{{L}^2_{t}}_{\Phi_t^2}\end{bmatrix}.
\]
Using \eqref{lnal} once again, we get
\[
-\left\langle
\lgrado^{\sigma(\bar K_t^2)}_{\sigma(\Phi_t^2)}
\lgradd^{\Phi_t^2}_{\bar K_t^2}\right\rangle=
-\left\langle\lgrado^{\sigma(\Phi_t^2)}_{\sigma(\Phi_t^2)}
\lgradd^{\Phi_t^2}_{\Phi_t^2}\right\rangle=\text{const},
\]
which together with the contribution of the second term in \eqref{trick12} computed above yields the second summand in
the last sum in \eqref{etareta}. 
The remaining term is given by
\begin{equation*}
\left\langle \lgrado^{\sigma(\bar K_t^2)}_{\sigma(\Phi_t^2)} 
Z_t\nald^{\Phi_t^2}_{{L}_{t}^2}\right\rangle=
\left\langle \lgrado^{\sigma(\Phi_t^2)}_{\sigma(\Phi_t^2)}
\Ld_{\Phi_t^2}^{L_t^2}\nald^{\Phi_t^2}_{L_{t}^2}
\right\rangle,
\end{equation*}
which coincides with the expression for $\badE_t$ in \eqref{bad4}.

 Assume now that $\bar\alpha_t^2=\bar\alpha_q^1$ and hence $\bar I^2_t= \bar I^1_{q}$. 
In this case the blocks $Y^{\bar J^2_t}_{\bar I^2_t}$ and  
$Y^{\bar J^1_{q}}_{\bar I^1_{q}}$ have the same height, and one of them lies inside the other, but the direction of the inclusion may vary, and hence $\sigma$ is not defined.

 Note that by \eqref{pigalo2}, to compute the second term in \eqref{secondterm} in this case, one can omit the rows 
$\hat{\bar I}_t^2$ in \eqref{pigaup2}, and hence the contribution in question equals
\begin{equation*}
\left\langle\lgrado^{\Phi_q^1}_{\bar K_q^1\setminus\Phi_q^1}
\Ld^{\bar L_t^2}_{\Phi_t^2}\nald^{\bar K_t^2\setminus\Phi_t^2}_{\bar L_t^2}\right\rangle,
\end{equation*}
which coincides with the expression for  $\badDp_t$ in \eqref{bad3eq}.
\end{proof}

\subsubsection{Explicit expression for $\left\langle\cgamma^{\ec*}(\cxi_L^1)_{\le},\cgamma^{\ec*}(\nabla_Y^2 Y)\right\rangle$}
\label{xinaysection} Assume that $p$ is defined by \eqref{blockp}. For any fixed $q'$ with $\bar\beta_t^2<\bar\beta_{q'}^1$
we define the injection $\sigma$ of $\bar K_t^2$ and $\bar L_t^2$ into $\bar K_{q'}^1$ and $\bar L_{q'}^1$, respectively, 
similarly to the one defined at the beginning of Section \ref{etaretasec}, and put
\begin{equation}\label{bad5}
\badF_{t}=\left\langle \gradlo_{\sigma(\Psi_{t+1}^2)}^{\sigma(\Psi_{t+1}^2)}
\nald^{\bar K_{t}^2}_{\Psi_{t+1}^2}\Ld_{\bar K_{t}^2}^{\Psi_{t+1}^2}\right\rangle.
\end{equation}
In what follows specific values of $q'$ are indicated in the summation index. Note that $q'$ may not coincide with $q$
defined by \eqref{blockq}.
 
\begin{lemma}\label{xinaylemma} 
{\rm (i)} The expression for $\left\langle\cgamma^{\ec*}(\cxi_L^1)_{\le},\cgamma^{\ec*}(\nabla_Y^2 Y)\right\rangle$ 
is given by
 \begin{multline}\label{xinay}
\left\langle\cgamma^{\ec*}(\cxi_L^1)_{\le},\cgamma^{\ec*}(\nabla_Y^2 Y)\right\rangle=
\sum_{\beta_t^2\le \beta_p^1 }\badC_t+\sum_{ \bar\beta_{t}^2>\bar\beta_{p-1}^1 }\badF_{t}\\
+\sum_{u=1}^{p} \sum_{t=1}^{s^2}
\left\langle(\nabla_{\L}^1\L^1)_{L_u^1\to J_u^1}^{L_u^1\to J_u^1},
\cgamma^{\ec*}(\nabla_{\L}^2\L^2)_{\bar L_{t}^2\setminus\Psi_{t+1}^2\to \bar J_{t}^2\setminus\bar\Delta(\bar\beta_{t}^2)}
^{\bar L_{t}^2\setminus\Psi_{t+1}^2\to \bar J_{t}^2\setminus\bar\Delta(\bar\beta_{t}^2)}\right\rangle\\
+\sum_{u=1}^{p-1} \sum_{t=1}^{s^2}
\left\langle
(\nabla_{\L}^1\L^1)_{\bar L_{u}^1\setminus\Psi_{u+1}^1\to \bar J_{u}^1\setminus\bar\Delta(\bar\beta_{u}^1)}
^{\bar L_{u}^1\setminus\Psi_{u+1}^1\to \bar J_{u}^1\setminus\bar\Delta(\bar\beta_{u}^1)},
\cPi_{\Gamma_2^\ec}(\nabla_{\L}^2\L^2)_{\bar L_{t}^2\setminus\Psi_{t+1}^2\to \bar J_{t}^2\setminus\bar\Delta(\bar\beta_{t}^2)}
^{\bar L_{t}^2\setminus\Psi_{t+1}^2\to \bar J_{t}^2\setminus\bar\Delta(\bar\beta_{t}^2)}\right\rangle\\
+\sum_{t=1}^{s^2}\left(|\{u<p: \beta_u^1\ge \beta_{t+1}^2\}|+ |\{u< p: \bar\beta_{u-1}^1< \bar\beta_t^2\}|  \right)
\left\langle\nald_{\Psi_{t+1}^2}^{\bar K_{t}^2}\Ld_{\bar K_{t}^2}^{\Psi_{t+1}^2}\right\rangle,
\end{multline}
where 
$\badC_t$ is given by \eqref{bad2} with $\rho(\Phi_t^2)$ replaced by $\Phi_p^1$ for $\beta_p^1= \beta_t^2$,  and
$\badF_{t}$ is given by \eqref{bad5}.

{\rm (ii)} Each summand in the last three sums in \eqref{xinay} is constant.
\end{lemma}

\begin{proof}
 Recall that by \eqref{pigamma}, this term can be rewritten as 
$\left\langle\cPi_{\Gamma_1}(\ceta^1_L)_\le, \cgamma^*(\nabla_Y^2 Y)\right\rangle$ with $\Gamma_1=\Gamma_1^\ec$ and 
$\cgamma=\cgammac$. 

Note that $\nabla_X^i X$ has been already computed in \eqref{naxx}. Let us compute $\cgamma^*(\nabla_Y^i Y)$. Taking into account \eqref{naxnay} and \eqref{naxyxy}, we get
\begin{equation*}
\cgamma^*(\nabla_Y^i Y)=\sum_{t=2}^{s^i+1}\cgamma^*\begin{bmatrix} 0 & 0\\ \ast &
(\nabla_{\L}^i)_{\bar L_{t-1}^i}^{\bar K_{t-1}^i}Y_{\bar I_{t-1}^i}^{\bar J_{t-1}^i} \end{bmatrix}=
\sum_{t=2}^{s^i+1}\cgamma^*\begin{bmatrix} 0 & 0\\ \ast & 
(\nabla_{\L}^i)_{\Psi_t^i}^{\bar K_{t-1}^i}(\L^i)_{\bar K_{t-1}^i}^{\bar L_{t-1}^i}\\
\ast & (\nabla_{\L}^i\L^i)_{\bar L_{t-1}^i\setminus\Psi_t^i}^{\bar L_{t-1}^i}\end{bmatrix};
\end{equation*}
the latter equality is similar to the one used in the derivation of the expression for $\nabla_X^i X$ in the proof of Lemma
\ref{etaletalemma}. In more detail,
\begin{multline}\label{ganayy}
\cgamma^*(\nabla_Y^i Y)=\\
\sum_{t=2}^{s^i+1}\cgamma^*\begin{bmatrix} 0 & 0 & 0 \\
0 & (\nabla_{\L}^i)_{\Psi_t^i}^{\bar K_{t-1}^i}(\L^i)_{\bar K_{t-1}^i}^{\Psi_t^i} & 0 \\
0 & 0 & 0\end{bmatrix}+
\sum_{t=2}^{s^i+1}\cgamma^*\begin{bmatrix} 0 & 0 \\
0  &(\nabla_{\L}^i\L^i)_{\bar L_{t-1}^i\setminus\Psi_t^i}^{\bar L_{t-1}^i\setminus\Psi_t^i}\end{bmatrix}.
\end{multline}

Note that the diagonal block in the first term in \eqref{ganayy} corresponds to the nontrivial column $Y$-run 
$\bar\Delta(\bar\beta_{t-1}^i)$, unless $t=s^i+1$ and $\Psi_{s^i+1}^i=\varnothing$. Therefore,
$\cgamma^*$ moves it to the diagonal block corresponding to the nontrivial column $X$-run $\Delta(\beta_t^i)$ occupied by  
$(\nabla_{\L}^i)_{\Psi_t^i}^{K_t^i}(\L^i)_{K_t^i}^{\Psi_t^i}$ in \eqref{naxx}. 
 Consequently, the resulting diagonal block in $\ceta_L^i$ is equal to
\begin{equation}\label{compli}
(\nabla_{\L}^i)_{\Psi_t^i}^{K_t^i}(\L^i)_{K_t^i}^{\Psi_t^i}+(\nabla_{\L}^i)_{\Psi_t^i}^{\bar K_{t-1}^i}
(\L^i)_{\bar K_{t-1}^i}^{\Psi_t^i}=(\nabla_{\L}^i\L^i)_{\Psi_t^i}^{\Psi_t^i}
\end{equation}
for $1\le t\le s^i+1$; note that the first term in the left hand side of \eqref{compli} vanishes
for $t=s^i+1$, and the second term vanishes for $t=1$.

Further, the projection $\cPi_{\Gamma_1}$ of the second block in the first row of \eqref{naxx} vanishes.
Summing up and applying \eqref{lnal}, we get
\begin{multline}\label{leftt}
\cPi_{\Gamma_1}(\ceta_L^1)_\le=
\sum_{u=1}^{s^1} \cPi_{\Gamma_1}
\begin{bmatrix} (\nabla_{\L}^1\L^1)_{L_u^1}^{L_u^1} & 0\\ 0 & 0\end{bmatrix}+
\cPi_{\Gamma_1}\begin{bmatrix} 0 & 0 & 0 \\
0 & (\nabla_{\L}^1\L^1)_{L_{s^1+1}^1}^{L_{s^1+1}^1} & 0 \\
0 & 0 & 0\end{bmatrix}\\
+\sum_{u=2}^{s^1+1} \cgamma^*\begin{bmatrix} 0 & 0\\ 0 &
(\nabla_{\L}^1\L^1)_{\bar L_{u-1}^1\setminus\Psi_u^1}^{\bar L_{u-1}^1\setminus\Psi_u^1}\end{bmatrix},
\end{multline}
where the nonzero block in the second term occupies rows and columns $J_{s^1+1}$, see Section~\ref{handling} after~\eqref{LLmat}.

Recall that $\hat l^1\in L_p^1\cup \bar L_{p-1}^1$ by \eqref{blockp}. 
Therefore, for any $u>p$ all three terms in \eqref{leftt} vanish. Consequently, 
by the ringed version of \eqref{gammaid}, the contribution of the second term in expression \eqref{ganayy} for the second function to the final result equals
\begin{multline*}
\sum_{u=1}^{p} \sum_{t=1}^{s^2}
\left\langle
(\nabla_{\L}^1\L^1)_{L_u^1\to J_u^1}^{L_u^1\to J_u^1},
\cgamma^*(\nabla_{\L}^2\L^2)_{\bar L_{t}^2\setminus\Psi_{t+1}^2\to \bar J_{t}^2\setminus\bar\Delta(\bar\beta_{t}^2)}
^{\bar L_{t}^2\setminus\Psi_{t+1}^2\to \bar J_{t}^2\setminus\bar\Delta(\bar\beta_{t}^2)}\right\rangle\\
+\sum_{u=1}^{p-1} \sum_{t=1}^{s^2}
\left\langle
(\nabla_{\L}^1\L^1)_{\bar L_{u}^1\setminus\Psi_{u+1}^1\to \bar J_{u}^1\setminus\bar\Delta(\bar\beta_{u}^1)}
^{\bar L_{u}^1\setminus\Psi_{u+1}^1\to \bar J_{u}^1\setminus\bar\Delta(\bar\beta_{u}^1)},
\cPi_{\Gamma_2}
(\nabla_{\L}^2\L^2)_{\bar L_{t}^2\setminus\Psi_{t+1}^2\to \bar J_{t}^2\setminus\bar\Delta(\bar\beta_{t}^2)}
^{\bar L_{t}^2\setminus\Psi_{t+1}^2\to \bar J_{t}^2\setminus\bar\Delta(\bar\beta_{t}^2)}\right\rangle,
\end{multline*}
which yields the third and the fourth sums in \eqref{xinay}. Note that each summand in both sums is constant by \eqref{lnal}.

Further, for any $u<p$, the nonzero blocks in all terms in \eqref{leftt} are just identity matrices by \eqref{lnal}. Hence, the corresponding contribution of the first term in expression 
\eqref{ganayy} for the second function to the final result equals 
\begin{equation}\label{trace}
\sum_{t=1}^{s^2}\left(|\{u<p: \beta_u^1\ge \beta_{t+1}^2\}|+ |\{u<p: \bar\beta_{u-1}^1< \bar\beta_t^2\}|  \right)
\left\langle\nald_{\Psi_{t+1}^2}^{\bar K_{t}^2}\Ld_{\bar K_{t}^2}^{\Psi_{t+1}^2}\right\rangle,
\end{equation}
which yields the fifth sum in \eqref{xinay}. It follows immediately from the proof of Lemma \ref{partrace}
that the trace $\left\langle(\nabla_{\L})_{\Psi_{t+1}}^{\bar K_{t}}\L_{\bar K_{t}}^{\Psi_{t+1}}\right\rangle$ is a constant. 

Finally, let $u=p$. 
Let us find the contribution of the first two terms in \eqref{leftt}. From now on we are looking at the $t$-th summand in the first term of \eqref{ganayy} for the second function. 
If $\beta_p^1<\beta_t^2$ then the contribution of this summand vanishes for the same size considerations
as in the proof of Lemma \ref{etaletalemma}.

If $\beta_p^1>\beta_t^2$ then the contribution in question equals
\begin{equation*}
\left\langle\gradlo_{\rho(\Psi_t^2)}^{\rho(\Psi_t^2)} 
\nald_{\Psi_t^2}^{\bar K_{t-1}^2}\Ld^{\Psi_t^2}_{\bar K_{t-1}^2}\right\rangle,
\end{equation*}
which coincides with $\badC_t$ given by \eqref{bad2} and yields the first sum in \eqref{xinay}.

If $\beta_p^1=\beta_t^2$ then the contribution in question remains the same as in the previous case with $\rho(\Psi_t^2)$ replaced by $\Psi_p^1$.

Let us find the contribution of the third term in \eqref{leftt}. Note that $\cgamma^*$ enters both the third term in 
\eqref{leftt} and the first term in \eqref{ganayy}, consequently, we can drop it in the former and replace by $\cPi_{\Gamma_2}$ in the latter, which effectively means that
$\cgamma^*$ is simultaneously dropped in both terms.

From now on we are looking at the  $t$-th summand in the first term of \eqref{ganayy}. However, since we have dropped 
$\cgamma^*$, this means that we are comparing the $(t-1)$-st $Y$-block in $\L^2$ 
with the $(p-1)$-st $Y$-block in  $\L^1$.  If $\bar\beta_{p-1}^1\geq \bar\beta_{t-1}^2$ 
then the contribution of this summand vanishes for the same size considerations as before.

If $\bar\beta_{p-1}^1< \bar\beta_{t-1}^2$, then the contribution in question equals
\begin{equation*}
\left\langle\gradlo_{\sigma(\Psi_{t}^2)}^{\sigma(\Psi_{t}^2)}
\nald_{\Psi_{t}^2}^{\bar K_{t-1}^2}\Ld^{\Psi_{t}^2}_{\bar K_{t-1}^2}\right\rangle,
\end{equation*}
which coincides with $\badF_{t-1}$ given by \eqref{bad5}, and hence yields the second sum in \eqref{xinay}.
 \end{proof}

\subsubsection{Explicit expression for $\left\langle\cgamma^{\er}(\cxi_R^1)_{\ge},\cgamma^{\er}(X\nabla_X^2 )\right\rangle$}
\label{xinaxsection}
Assume that $p$ and $q$ are defined by \eqref{blockp} and \eqref{blockq}, respectively, $q'$ and $\sigma$ are the same as in Section \ref{xinaysection}, and let $\rho$ be the injection of $K_t^2$ and 
$L_t^2$ into $K_p^1$ and $L_p^1$, respectively, defined at the beginning of Section \ref{etaletasec}. Put
\begin{equation}\label{bad6}
\badG_t=\left\langle\lgrado_{\rho(\Phi_t^2)}^{\rho(\Phi_t^2)} 
\Ld_{\Phi_t^2}^{L_{t}^2}\nald^{\Phi_t^2}_{L_{t}^2}\right\rangle.
\end{equation}

\begin{lemma}\label{xinaxlemma} 
{\rm (i)} The expression for $\left\langle\cgamma^{\er}(\cxi_R^1)_{\ge},\cgamma^{\er}(X\nabla_X^2 )\right\rangle$  
is given by
 \begin{multline}\label{xinax}
\left\langle\cgamma^{\er}(\cxi_R^1)_{\ge},\cgamma^{\er}(X\nabla_X^2 )\right\rangle=
\sum_{\bar\alpha_t^2\le \bar\alpha_{p-1}^1}\badE_t
+\sum_{\bar\alpha_t^2\le \bar\alpha_p^1}\badE_t
+\sum_{\alpha_t^2>\alpha_p^1}\badG_{t}\\
+\sum_{u=1}^{p} \sum_{t=1}^{s^2}
\left\langle(\L^1\nabla_{\L}^1)_{\bar K_u^1\to \bar I_u^1}^{\bar K_u^1\to \bar I_u^1},
\cgammar(\L^2\nabla_{\L}^2)_{K_{t}^2\setminus\Phi_{t}^2\to I_{t}^2\setminus\Delta(\alpha_{t}^2)}
^{K_{t}^2\setminus\Phi_{t}^2\to I_{t}^2\setminus\Delta(\alpha_{t}^2)}
\right\rangle\\
+\sum_{u=1}^{p} \sum_{t=1}^{s^2}
\left\langle
(\L^1\nabla_{\L}^1)_{K_{u}^1\setminus\Phi_{u}^1\to I_{u}^1\setminus\Delta(\alpha_{u}^1)}
^{K_{u}^1\setminus\Phi_{u}^1\to I_{u}^1\setminus\Delta(\alpha_{u}^1)},
\cPi_{\Gamma_1^\er}(\L^2\nabla_{\L}^2)_{K_{t}^2\setminus\Phi_{t}^2\to I_{t}^2\setminus\Delta(\alpha_{t}^2)}
^{K_{t}^2\setminus\Phi_{t}^2\to I_{t}^2\setminus\Delta(\alpha_{t}^2)}
\right\rangle\\
+\sum_{t=1}^{s^2}\left(|\{u<p-1: \bar\alpha_u^1\ge \bar\alpha_t^2\}|+ |\{u< p: \alpha_u^1< \alpha_t^2\}|  \right)
\left\langle\Ld^{L_{t}^2}_{\Phi_t^2}\nald^{\Phi_t^2}_{L_{t}^2}\right\rangle,
\end{multline}
where 
$\badE_t$ is given by \eqref{bad4} with 
$\sigma(\Phi_t^2)$ replaced by $\Phi_q^1$ for $\bar\alpha_q^1=\bar\alpha_t^2$,  and $\badG_t$ is given by
\eqref{bad6}.

{\rm (ii)} Each summand in the last three sums in \eqref{xinay} is constant.
\end{lemma}

\begin{proof}
Recall that by \eqref{pigamma2}, this term can be rewritten as 
$\left\langle\cPi_{\Gamma_2}(\ceta^1_R)_\ge, \cgamma(X\nabla_X^2)\right\rangle$ with $\Gamma_2=\Gamma_2^\er$ and
$\cgamma=\cgammar$. 

Note that $Y\nabla_Y^i$ has been already computed in \eqref{ynay}. Let us compute $\cgamma(X\nabla_X^i)$. Taking into account \eqref{naxnay} and \eqref{xynaxy}, we get
\begin{equation}\label{gaxnax}
\cgamma(X\nabla_X^i)=\\
\sum_{t=1}^{s^i}\cgamma\begin{bmatrix} 0 & 0 & 0 \\
0 & (\L^i)^{ L_{t}^i}_{\Phi_t^i}(\nabla_{\L}^i)^{\Phi_t^i}_{L_{t}^i} & 0 \\
0 & 0 & 0\end{bmatrix}+
\sum_{t=1}^{s^i}\cgamma\begin{bmatrix} 0 & 0 \\
0  &(\L^i\nabla_{\L}^i)_{K_{t}^i\setminus\Phi_t^i}^{K_{t}^i\setminus\Phi_t^i}\end{bmatrix},
\end{equation}
similarly to \eqref{ganayy}.

Note first that the diagonal block in the first term in \eqref{gaxnax} corresponds to 
the nontrivial row $X$-run $\Delta(\beta_t^i)$, unless $t=1$ and the first $X$-block is dummy, or $t=s^i$ and 
$\Phi_{s^i}=\varnothing$. Hence, $\cgamma$ moves it to the diagonal block 
corresponding to the nontrivial row $Y$-run $\bar\Delta(\bar\beta_t^i)$ occupied by 
$(\L^i)^{\bar L_t^i}_{\Phi_t^i}(\nabla_{\L}^i)^{\Phi_t^i}_{\bar L_t^i}$ in \eqref{ynay}. Consequently, the resulting diagonal block in $\ceta_R^i$ is equal to
\begin{equation}\label{compli2}
(\L^i)^{\bar L_t^i}_{\Phi_t^i}(\nabla_{\L}^i)^{\Phi_t^i}_{\bar L_t^i}+
(\L^i)^{L_{t}^i}_{\Phi_t^i}(\nabla_{\L}^i)^{\Phi_t^i}_{L_{t}^i}=
(\L^i\nabla_{\L}^i)_{\Phi_t^i}^{\Phi_t^i}
\end{equation}
(if the first $X$-block is dummy and $\Phi_1^i\ne\varnothing$, the second term in the left hand side vanishes;
for $\Phi_{t}^i=\varnothing$ relation \eqref{compli2}  holds trivially with all three terms void).

Moreover, the projection $\cPi_{\Gamma_2}$ of the second block in the first column of \eqref{ynay} vanishes. Summing up and applying \eqref{lnal}, we get
\begin{equation}\label{leftt2}
\cPi_{\Gamma_2}(\ceta_R^1)_\ge=
\sum_{u=1}^{s^1} \cPi_{\Gamma_2}\begin{bmatrix} (\L^1\nabla_{\L}^1)^{\bar K_u^1}_{\bar K_u^1} & 0\\ 0 & 0
\end{bmatrix}+
\sum_{u=1}^{s^1} \cgamma\begin{bmatrix} 0 & 0\\ 0 &
(\L^1\nabla_{\L}^1)_{K_{u}^1\setminus\Phi_u^1}^{K_{u}^1\setminus\Phi_u^1}\end{bmatrix};
\end{equation}
 note that if the last $Y$-block is dummy then the block $(\L^1\nabla_{\L}^1)^{\bar K_{s^1}^1}_{\bar K_{s^1}^1}$ is not leading any more: it occupies rows and columns $\bar I_{s^1}$, see Section~\ref{handling} after~\eqref{LLmat}.

Recall that $\hat l^1\in K_p\cup \bar K_{p-1}$, see Section \ref{etaretasec}. 
Therefore, for any $u>p$ both terms in \eqref{leftt2} vanish. Therefore,
the contribution of the second term in \eqref{gaxnax} to the final result equals
\begin{multline*}
\sum_{u=1}^{p} \sum_{t=1}^{s^2}
\left\langle(\L^1\nabla_{\L}^1)_{\bar K_u^1\to \bar I_u^1}^{\bar K_u^1\to \bar I_u^1},
\cgamma(\L^2\nabla_{\L}^2)_{K_{t}^2\setminus\Phi_{t}^2\to I_{t}^2\setminus\Delta(\alpha_{t}^2)}
^{K_{t}^2\setminus\Phi_{t}^2\to I_{t}^2\setminus\Delta(\alpha_{t}^2)}
\right\rangle\\
+\sum_{u=1}^{p} \sum_{t=1}^{s^2}
\left\langle
(\L^1\nabla_{\L}^1)_{K_{u}^1\setminus\Phi_{u}^1\to I_{u}^1\setminus\Delta(\alpha_{u}^1)}
^{K_{u}^1\setminus\Phi_{u}^1\to I_{u}^1\setminus\Delta(\alpha_{u}^1)},
\cPi_{\Gamma_1}(\L^2\nabla_{\L}^2)_{K_{t}^2\setminus\Phi_{t}^2\to I_{t}^2\setminus\Delta(\alpha_{t}^2)}
^{K_{t}^2\setminus\Phi_{t}^2\to I_{t}^2\setminus\Delta(\alpha_{t}^2)}
\right\rangle,
\end{multline*}
which yields the fourth and the fifth sums in \eqref{xinax}. Note that each summand in both sums is constant by \eqref{lnal}.

For any $u<p-1$, the nonzero blocks in both terms in \eqref{leftt2} are just identity matrices by \eqref{lnal}. Therefore, the corresponding contribution of the first term of
\eqref{gaxnax} for the second function to the final result equals 
\begin{equation*}
\sum_{t=1}^{s^2}\left(|\{u<p-1: \bar\alpha_u^1\ge \bar\alpha_t^2\}|+ |\{u< p-1: \alpha_u^1< \alpha_t^2\}|  \right)
\left\langle\Ld^{L_{t}^2}_{\Phi_t^2}\nald^{\Phi_t^2}_{L_{t}^2}\right\rangle,
\end{equation*}
which is similar to \eqref{trace} and is constant for the same reason.

Further, let $u=p-1$. Then the nonzero block in the second term in\eqref{leftt2} is again an identity matrix, and hence
the inequality $u<p-1$ in the second term above is replaced by $u<p$, which yields the last sum in \eqref{xinax}.

Let us find the contribution of the first term in \eqref{leftt2}. From now on we are looking at the summation index $t$ 
in \eqref{gaxnax} for the second function; recall that it corresponds to the $t$-th $Y$-block. If 
$\bar\alpha_{p-1}^1<\bar\alpha_t^2$ then the contribution of this summand vanishes for the size considerations, similarly to the proof of Lemma \ref{xinaylemma}.
 If $\bar\alpha_{p-1}^1>\bar\alpha_t^2$, then the contribution in question equals
\begin{equation*}\label{qcontr}
\left\langle\lgrado_{\sigma(\Phi_t^2)}^{\sigma(\Phi_t^2)}
\Ld_{\Phi_t^2}^{L_{t}^2}\nald^{\Phi_t^2}_{L_{t}^2}\right\rangle,
\end{equation*}
which coincides with $\badE_t$ given by \eqref{bad4}.
If $\bar\alpha_{p-1}^1=\bar\alpha_t^2$ then the contribution in question remains the same as in the previous case with 
$\sigma(\Phi_t^2)$ replaced by $\Phi_{p-1}^1$. 
Consequently, we get the first sum in \eqref{xinax}.

Finally, let $u=p$. Then the first term in \eqref{leftt2} is treated exactly as in the case $u=p-1$, which gives the 
second sum in \eqref{xinax}. 

Let us find the contribution of the second term in 
\eqref{leftt2}. Note that $\cgamma$ enters both the second term in \eqref{leftt2} and the first term in \eqref{gaxnax}, consequently, we can drop it in the former and replace by $\cPi_{\Gamma_1}$ in the latter, which effectively means that
$\cgamma$ is simultaneously dropped in both terms.

From now on we are looking at the summation index $t$ in \eqref{gaxnax} for the second function. However, since we have 
dropped $\cgamma$, this means that we are comparing the $t$-th $X$-block in $\L^2$ 
with the $p$-th $X$-block in  $\L^1$. If 
$\alpha_p^1\geq \alpha_t^2$ then the contribution of the $t$-th term in \eqref{gaxnax} vanishes for the size considerations.

If $\alpha_p^1<\alpha_t^2$ then the contribution in question equals
\begin{equation*}
\left\langle\lgrado_{\rho(\Phi_t^2)}^{\rho(\Phi_t^2)} 
\Ld_{\Phi_t^2}^{L_{t}^2}\nald^{\Phi_t^2}_{L_{t}^2}\right\rangle,
\end{equation*}
which coincides with the expression \eqref{bad6} for $\badG_t$ and yields the third sum in \eqref{xinax}.
\end{proof}

\subsection{Proof of Theorem \ref{logcanbasis}: final steps}

Let us find the total contribution of all $B$-terms in the right hand side of \eqref{etaleta}, \eqref{etareta},
\eqref{xinay} and \eqref{xinax}. Recall that $\hat l^1$ lies in rows $K^1_p\cup \bar K^1_{p-1}$ and columns 
$L^1_p\cup \bar L^1_{p-1}$. We consider the following two cases.

\subsubsection{Case 1: $\hat l^1$ lies in rows $K^1_p$ and columns $L^1_p$}\label{case1} 
Note that under these conditions, the matrix $\gradlo_{\sigma(\Psi_{t+1}^2)}^{\sigma(\Psi_{t+1}^2)}$
in the expression \eqref{bad3} for $\badD_t$ in \eqref{etareta} vanishes, since rows and columns 
$\sigma(\Psi_{t+1}^2)$ lie strictly above and to the left of $\hat l^1$. Besides, the matrix 
$\lgrado_{\bar K_p^1\setminus \Phi_p^1}^{\Phi_p^1}$ in the expression \eqref{bad3eq} for $\badDp_t$ in \eqref{etareta}
vanishes as well. Indeed, the column  $\Lo_{\bar K_p^1\setminus \Phi_p^1}^j$ vanishes 
if $j$ lies to the right of $\bar L_p$. On the other hand, the $i$-th row of $\nabla_\L^1$ vanishes if $i$ lies above the intersection of the main diagonal with the vertical line corresponding to the right endpoint of $\bar L_p$. 

Finally, for any $t$ such that $\beta_p^1>\beta_t^2$,
the contributions of the term $\badC_t$ given by \eqref{bad2} in \eqref{etaleta} and \eqref{xinay} cancel each other.
Similarly, for any $t$ such that  $\bar\alpha_p^1>\bar\alpha_t^2$,
the contributions of the term $\badE_t$ given by \eqref{bad4} in \eqref{etareta} and \eqref{xinax} cancel each other as 
well. Taking into account that $\bar\alpha_p^1=\bar\alpha_t^2$ is equivalent to $\alpha_p^1=\alpha_t^2$, we can
rewrite the remaining terms as
\begin{multline}\label{zone1}
\sum \{\badG_t-\badB_t :\ \beta_p^1>\beta_t^2, \alpha_p^1<\alpha_t^2\}
+\sum \{\badE_t-\badB_t :\  \beta_p^1>\beta_t^2, \alpha_p^1=\alpha_t^2\}\\
+\sum \{\badE_t :\  \beta_p^1<\beta_t^2, \alpha_p^1=\alpha_t^2\}
+\sum \{\badC_t-\badBp_t :\ \beta_p^1=\beta_t^2, \alpha_p^1>\alpha_t^2\}\\
+\sum \{\badC_t-\badBp_t+\badG_t :\ \beta_p^1=\beta_t^2, \alpha_p^1<\alpha_t^2\}
+\sum \{\badC_t-\badBp_t+\badE_t :\  \beta_p^1=\beta_t^2, \alpha_p^1=\alpha_t^2\}\\
+\sum \{\badF_t :\  \bar\beta_{p-1}^1<\bar\beta_t^2\}
+\sum \{\badE_t :\  \bar\alpha_{p-1}^1\ge\bar\alpha_t^2\},
\end{multline}
where $\badB_t$,  $\badBp_t$, $\badG_t$, and $\badF_t$ are given by \eqref{bad1}, \eqref{bad1eq}, \eqref{bad6}, 
and \eqref{bad5}, respectively.

\begin{lemma}\label{zone1lemma}
{\rm (i)} Expression \eqref{zone1} is given by
\begin{multline*}
\sum_{{{\beta_t^2<\beta_p^1}\atop{\alpha_t^2>\alpha_p^1}}}
\left\langle\lgrado_{\rho(\Phi_t^2)}^{\rho(\Phi_t^2)}\lgradd_{\Phi_t^2}^{\Phi_t^2}\right\rangle
+\sum_{{{\beta_t^2\ne\beta_p^1}\atop{\alpha_t^2=\alpha_p^1}}}
\left\langle\lgrado_{\Phi_p^1}^{\Phi_p^1}\lgradd_{\Phi_t^2}^{\Phi_t^2}\right\rangle\\
+\sum_{{\beta_t^2=\beta_p^1}\atop{\alpha_t^2<\alpha_p^1}}
\left\langle \Ld_{\bar K_{t-1}^2}^{L_t^2}\nald_{L_t^2}^{\bar K_{t-1}^2}\right\rangle
+\sum_{{\beta_t^2=\beta_p^1}\atop{\alpha_t^2\ge\alpha_p^1}}
\left\langle\gradlo_{L_p^1}^{L_p^1}\gradld_{L_t^2}^{L_t^2}\right\rangle\\
-\sum_{{\beta_t^2=\beta_p^1}\atop{\alpha_t^2\ge \alpha_p^1}}
\left\langle\lgrado_{\rho(K_t^2\setminus\Phi_t^2)}^{\rho(K_t^2\setminus\Phi_t^2)}
\lgradd_{K_t^2\setminus\Phi_t^2}^{K_t^2\setminus\Phi_t^2}\right\rangle
+{\sum_{{\beta_t^2=\beta_p^1}\atop{\alpha_t^2=\alpha_p^1}}}^\ea
\left\langle \Ld_{\bar K_{t-1}^2}^{\Psi_t^2}\nald_{\Psi_t^2}^{\bar K_{t-1}^2}\right\rangle\\
+{\sum_{{\beta_t^2=\beta_p^1}\atop{\alpha_t^2=\alpha_p^1}}}^\ea
\left\langle\lgrado_{K_p^1}^{K_p^1}\lgradd_{K_t^2}^{K_t^2}\right\rangle
-{\sum_{{{\beta_t^2=\beta_p^1}\atop{\alpha_t^2=\alpha_p^1}}}}^\ea
\left\langle\gradlo_{L_p^1}^{L_p^1}\gradld_{L_t^2}^{L_t^2}\right\rangle\\
+\sum_{\bar\beta_t^2>\bar\beta_{p-1}^1}\left\langle\Ld_{\bar K_t^2}^{\Psi_{t+1}^2}\nald_{\Psi_{t+1}^2}^{\bar K_t^2}\right\rangle
+\sum_{\bar\alpha_t^2\le \bar\alpha_{p-1}^1}\left\langle\Ld_{\Phi_t^2}^{L_t^2}\nald_{L_t^2}^{\Phi_t^2}\right\rangle, 
\end{multline*}
where $\sum^\ea$ is taken over the cases when the exit point of $X_{I_t^2}^{J_t^2}$ lies above  
the exit point of $X_{I_p^1}^{J_p^1}$. 

{\rm (ii)} Each summand in the expression above is a constant.
\end{lemma}

\begin{proof}
To find the first term in \eqref{zone1} note that for any fixed $t$ satisfying the corresponding conditions one has 
\begin{multline}\label{motive}
\badG_t-\badB_t=
\left\langle\lgrado_{\rho(\Phi_t^2)}^{\rho(\Phi_t^2)} 
\Ld_{\Phi_t^2}^{L_{t}^2}\nald_{L_t^2}^{\Phi_t^2}\right\rangle+
\left\langle\lgrado_{\rho(\Phi_t^2)}^{\rho(\Phi_t^2)} 
\Ld_{\Phi_t^2}^{\bar L_{t}^2}\nald_{\bar L_{t}^2}^{\Phi_t^2}\right\rangle\\
=\left\langle\lgrado_{\rho(\Phi_t^2)}^{\rho(\Phi_t^2)} 
\lgradd_{\Phi_t^2}^{\Phi_t^2}\right\rangle=\text{const}
\end{multline}
via \eqref{compli2} and \eqref{lnal}, which yields the first term in the statement of the lemma.

Similarly, to treat the second term in \eqref{zone1} we note that under the corresponding conditions
\begin{multline}\label{termtwo}
\badE_t-\badB_t=
\left\langle\lgrado_{\Phi_p^1}^{\Phi_p^1} 
\Ld_{\Phi_t^2}^{L_{t}^2}\nald_{L_{t}^2}^{\Phi_t^2}\right\rangle+
\left\langle\lgrado_{\Phi_p^1}^{\Phi_p^1} 
\Ld_{\Phi_t^2}^{\bar L_{t}^2}\nald_{\bar L_t^2}^{\Phi_t^2}\right\rangle\\
=\left\langle\lgrado_{\Phi_p^1}^{\Phi_p^1} 
\lgradd_{\Phi_t^2}^{\Phi_t^2}\right\rangle=\text{const}
\end{multline}
via \eqref{compli2} and \eqref{lnal}. 

To find the contribution of the third term in \eqref{zone1}, rewrite it as
\begin{equation*}
\left\langle
\lgrado_{\Phi_p^1}^{\Phi_p^1} \lgradd_{\Phi_t^2}^{\Phi_t^2}
\right\rangle-
\left\langle
\lgrado_{\Phi_p^1}^{\Phi_p^1}\Ld_{\Phi_t^2}^{\bar L_t^2}
\nald_{\bar L_t^2}^{\Phi_t^2}
\right\rangle
\end{equation*}
and note that  the second term equals
\begin{equation}\label{3inone}
-\left\langle
\Lo_{\Phi_p^1}^{L_p^1}
\nalo_{L_p^1}^{\Phi_p^1}
\Ld_{\Phi_t^2}^{\bar L_t^2}
\nald_{\bar L_t^2}^{\Phi_t^2}
\right\rangle,
\end{equation}
since $\nalo_{\bar L_p^1}^{\Phi_p^1}$ vanishes. 
Further, the block 
$X_{I_p^1}^{J_p^1}$ is contained completely inside the block $X_{I_t^2}^{J_t^2}$. We denote by $\rho$ the corresponding injection, so $\Lo_{\Phi_p^1}^{L_p^1}=\Ld_{\Phi_t^2}^{\rho(L_p^1)}$. Therefore,
\eqref{3inone} can be written as
\[
\left\langle
\nalo_{L_p^1}^{\Phi_p^1}
\Ld_{\Phi_t^2}^{\bar L_t^2}
\nald_{\bar L_t^2}^{K_t^2\setminus\Phi_t^2}
\Ld_{K_t^2\setminus\Phi_t^2}^{\rho(L_p^1)}
\right\rangle,
\]
where we used the fact that
\[
\nald_{\bar L_t^2}^{\Phi_t^2}
\Ld_{\Phi_t^2}^{\rho(L_p^1)}+
\nald_{\bar L_t^2}^{K_t^2\setminus\Phi_t^2}
\Ld_{K_t^2\setminus\Phi_t^2}^{\rho(L_p^1)}=
\gradld_{\bar L_t^2}^{\rho(L_p^1)}=0.
\]
Finally, $\Ld_{K_t^2\setminus\Phi_t^2}^{\rho(L_p^1)}=\Lo_{K_p^1\setminus\Phi_p^1}^{L_p^1}$,
and
\[
\Lo_{K_p^1\setminus\Phi_p^1}^{L_p^1}
\nalo_{L_p^1}^{\Phi_p^1}=
\lgrado_{K_p^1\setminus\Phi_p^1}^{\Phi_p^1}=0,
\]
hence \eqref{3inone} vanishes, and the contribution in question is given by the same expression as in \eqref{termtwo}, and
thus yields the second term in the statement of the lemma. 

To find the fourth term in \eqref{zone1} note that for any fixed $t$ satisfying the corresponding conditions we get
\begin{multline}\label{bmicinit}
\badC_t-\badBp_t\\=\left\langle\gradlo_{\Psi_p^1}^{\Psi_p^1}
\nald^{\bar K_{t-1}^2}_{\Psi_t^2}\Ld^{\Psi_t^2}_{\bar K_{t-1}^2}\right\rangle-
\left\langle\gradlo_{\Psi_p^1}^{L_p^1\setminus\Psi_p^1}
\nald^{K_t^2}_{L_{t}^2\setminus\Psi_t^2}\Ld^{\Psi_t^2}_{K_{t}^2}\right\rangle.
\end{multline}
Applying \eqref{compli} to the first expression and using the equality
\[
\gradlo_{\Psi_p^1}^{L_p^1\setminus\Psi_p^1} 
\nald^{K_t^2}_{L_{t}^2\setminus\Psi_t^2}+
\gradlo_{\Psi_p^1}^{\Psi_p^1}
\nald^{K_{t}^2}_{\Psi_t^2}=
\gradlo_{\Psi_p^1}^{L_p^1}\nald^{K_{t}^2}_{L_t^2}
\]
we get
\begin{equation}\label{bmic}
\badC_t-\badBp_t=\left\langle\gradlo_{\Psi_p^1}^{\Psi_p^1}
\gradld_{\Psi_t^2}^{\Psi_t^2}\right\rangle-
\left\langle\gradlo_{\Psi_p^1}^{L_p^1} 
\nald^{K_t^2}_{L_{t}^2}\Ld^{\Psi_t^2}_{K_{t}^2}\right\rangle.
\end{equation}
Clearly, the first term above is a constant. 

Note that $\alpha_p^1>\alpha_t^2$, and hence the block 
$X_{I_p^1}^{J_p^1}$ is contained completely inside the block $X_{I_t^2}^{J_t^2}$, which means, in particular, that $p>1$.
Consider two sequences of blocks
\begin{equation}\label{blocks}
\{Y_{\bar I_{p-1}^1}^{\bar J_{p-1}^1}, X_{I_{p-1}^1}^{J_{p-1}^1}, Y_{\bar I_{p-2}^1}^{\bar J_{p-2}^1}, \dots\}\quad\text{and}
\quad\{Y_{\bar I_{t-1}^2}^{\bar J_{t-1}^2}, X_{I_{t-1}^2}^{J_{t-1}^2}, Y_{\bar I_{t-2}^2}^{\bar J_{t-2}^2}, \dots\}.
\end{equation}

There are four possibilities:

(i) there exists a pair of blocks $Y_{\bar I_{p-m}^1}^{\bar J_{p-m}^1}$ and 
$Y_{\bar I_{t-m}^2}^{\bar J_{t-m}^2}$ such that $\bar J_{p-m}^1=\bar J_{t-m}^2$, 
$\bar I_{p-m}^1\ne \bar I_{t-m}^2$, 
and the subsequences of blocks to the left of $Y_{\bar I_{p-m}^1}^{\bar J_{p-m}^1}$ and $Y_{\bar I_{t-m}^2}^{\bar J_{t-m}^2}$ coincide;

(ii) there exists a pair of blocks $X_{I_{p-m}^1}^{J_{p-m}^1}$ and $X_{I_{t-m}^2}^{J_{t-m}^2}$ such that 
$I_{p-m}^1=I_{t-m}^2$, $J_{p-m}^1\ne J_{t-m}^2$,
and the subsequences of blocks to the left of 
$X_{I_{p-m}^1}^{J_{p-m}^1}$ and $X_{I_{t-m}^2}^{J_{t-m}^2}$ coincide;

(iii) the first sequence is a proper subsequence of the second one;

(iv) the second sequence is a proper subsequence of the first one, or is empty.

{\it Case\/} (i): Clearly, this can be possible only if $\bar I_{t-m}^2\subset \bar I_{p-m}^1$, see Fig.~\ref{fig:chainy}
where  blocks $X_{I_{k}^i}^{J_{k}^i}$ and $Y_{\bar I_{k}^i}^{\bar J_{k}^i}$ are for brevity denoted $X_k^i$ and 
$Y_k^i$, respectively.

\begin{figure}[ht]
\begin{center}
\includegraphics[height=7cm]{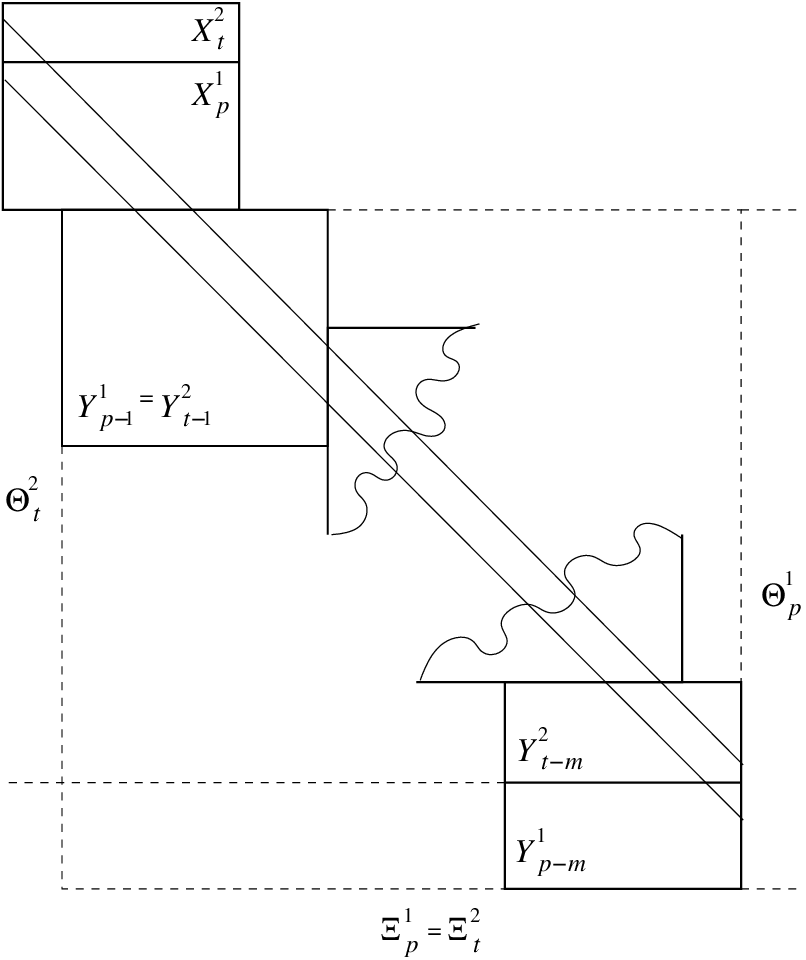}
\caption{Case (i)}
\label{fig:chainy}
\end{center}
\end{figure}

Denote 
\begin{equation}\label{thetaxi}
\Theta_{r}^i=\bigcup_{j=1}^{m-1}(\bar K_{r-j}^i\cup K_{r-j}^i)\cup \bar K_{r-m}^i, \qquad
\Xi_{r}^i=\bigcup_{j=1}^{m-1} (\bar L_{r-j}^i\cup L_{r-j}^i)\cup \bar L_{r-m}^i. 
\end{equation}
Note that the matrix $\Ld_{\Theta_{t}^2}^{\Xi_{t}^2}$ coincides with a proper submatrix of 
$\Lo_{\Theta_{p}^1}^{\Xi_{p}^1}$; we denote the corresponding injection $\sigma$ (it can be considered as an analog of the injection $\sigma$ defined in Section \ref{etaretasec}). Clearly,
\begin{equation}\label{interim}
\nald_{L_t^2}^{K_t^2}\Ld_{K_t^2}^{\Psi_t^2}=
\gradld_{L_t^2}^{\Psi_t^2}- \nald_{L_t^2}^{\Theta_{t}^2}\Ld_{\Theta_{t}^2}^{\Psi_t^2}.
\end{equation}
The contribution of the first term in \eqref{interim} to the second term in \eqref{bmic} equals
\[
-\left\langle\gradlo_{\Psi_p^1}^{L_p^1}\gradld_{L_t^2}^{\Psi_t^2}\right\rangle=
-\left\langle\gradlo_{\Psi_p^1}^{\Psi_p^1}\gradld_{\Psi_t^2}^{\Psi_t^2}\right\rangle
\]
and cancels the contribution of the first term in \eqref{bmic} computed above. 

To find the contribution of the second term in \eqref{interim} to the second term in \eqref{bmic}
note that
\begin{equation}\label{superdec}
\gradlo_{\Psi_p^1}^{L_p^1}=
\nalo_{\Psi_p^1}^{K_p^1\cup\Theta_{p}^1}
\Lo_{K_p^1\cup\Theta_{p}^1}^{L_p^1},
\end{equation}
so the contribution in question equals
\begin{equation}\label{crosscases}
\left\langle\nald_{L_t^2}^{\Theta_t^2}\Ld_{\Theta_t^2}^{\Psi_t^2}
\nalo_{\Psi_p^1}^{K_p^1\cup\Theta_{p}^1}\Lo_{K_p^1\cup\Theta_{p}^1}^{L_p^1}\right\rangle.
\end{equation}
Taking into account that $\Ld_{\Theta_t^2}^{\Psi_t^2}=\Lo_{\sigma(\Theta_t^2)}^{\Psi_p^1}$, 
$\Ld_{\Theta_t^2}^{\Xi_t^2\setminus\Psi_t^2}=\Lo_{\sigma(\Theta_t^2)}^{\Xi_p^1\setminus\Psi_p^1}$
and that
\begin{equation}\label{twoterm}
\Lo_{\sigma(\Theta_t^2)}^{\Psi_p^1}\nalo_{\Psi_p^1}^{K_p^1\cup\Theta_{p}^1}=
 \lgrado_{\sigma(\Theta_t^2)}^{K_p^1\cup\Theta_{p}^1}-
\Lo_{\sigma(\Theta_t^2)}^{\Xi_p^1\setminus\Psi_p^1}
\nalo_{\Xi_p^1\setminus\Psi_p^1}^{K_p^1\cup\Theta_p^1},
\end{equation}
this contribution can be rewritten as
\begin{multline*}
\left\langle \nald_{L_t^2}^{\Theta_t^2}
\lgrado_{\sigma(\Theta_t^2)}^{K_p^1\cup\Theta_{p}^1}
\Lo_{K_p^1\cup\Theta_{p}^1}^{L_p^1}\right\rangle\\
-\left\langle \nald_{L_t^2}^{\Theta_t^2}
\Ld_{\Theta_t^2}^{\Xi_t^2\setminus\Psi_t^2}
\nalo_{\Xi_p^1\setminus\Psi_p^1}^{K_p^1\cup\Theta_p^1}
\Lo_{K_p^1\cup\Theta_{p}^1}^{L_p^1}\right\rangle.
\end{multline*}
Next, by\eqref{lnal}, 
\[
\nald_{L_t^2}^{\Theta_t^2}
\Ld_{\Theta_t^2}^{\Xi_t^2\setminus\Psi_t^2}=
\gradld_{L_t^2}^{\Xi_t^2\setminus\Psi_t^2}=0,
\]
since the columns $L_t^2$ lie to the left of $\Xi_t^2\setminus\Psi_t^2$.

Finally, by \eqref{lnal},
\[
\lgrado_{\sigma(\Theta_t^2)}^{K_p^1\cup\Theta_{p}^1}=
\begin{bmatrix} 0& \one & 0\end{bmatrix},
\]
where the unit block occupies the rows and the columns $\sigma(\Theta_t^2)$. Therefore, the remaining contribution equals
\[
\left\langle \nald_{L_t^2}^{\Theta_t^2}
\Lo_{\sigma(\Theta_t^2)}^{L_p^1}\right\rangle=
\left\langle \Ld_{\Theta_t^2}^{L_t^2}\nald_{L_t^2}^{\Theta_t^2}\right\rangle=
\left\langle \Ld_{\bar K_{t-1}^2}^{L_t^2}\nald_{L_t^2}^{\bar K_{t-1}^2}
\right\rangle,
\]
which is a constant via Lemma \ref{partrace} and yields the third term in the statement of the lemma.

{\it Case\/} (ii): Clearly, this can be possible only if $J_{p-m}^1\subset J_{t-m}^2$, see Fig.~\ref{fig:chainx} where we use 
the same convention as in Fig.~\ref{fig:chainy}.

\begin{figure}[ht]
\begin{center}
\includegraphics[height=7cm]{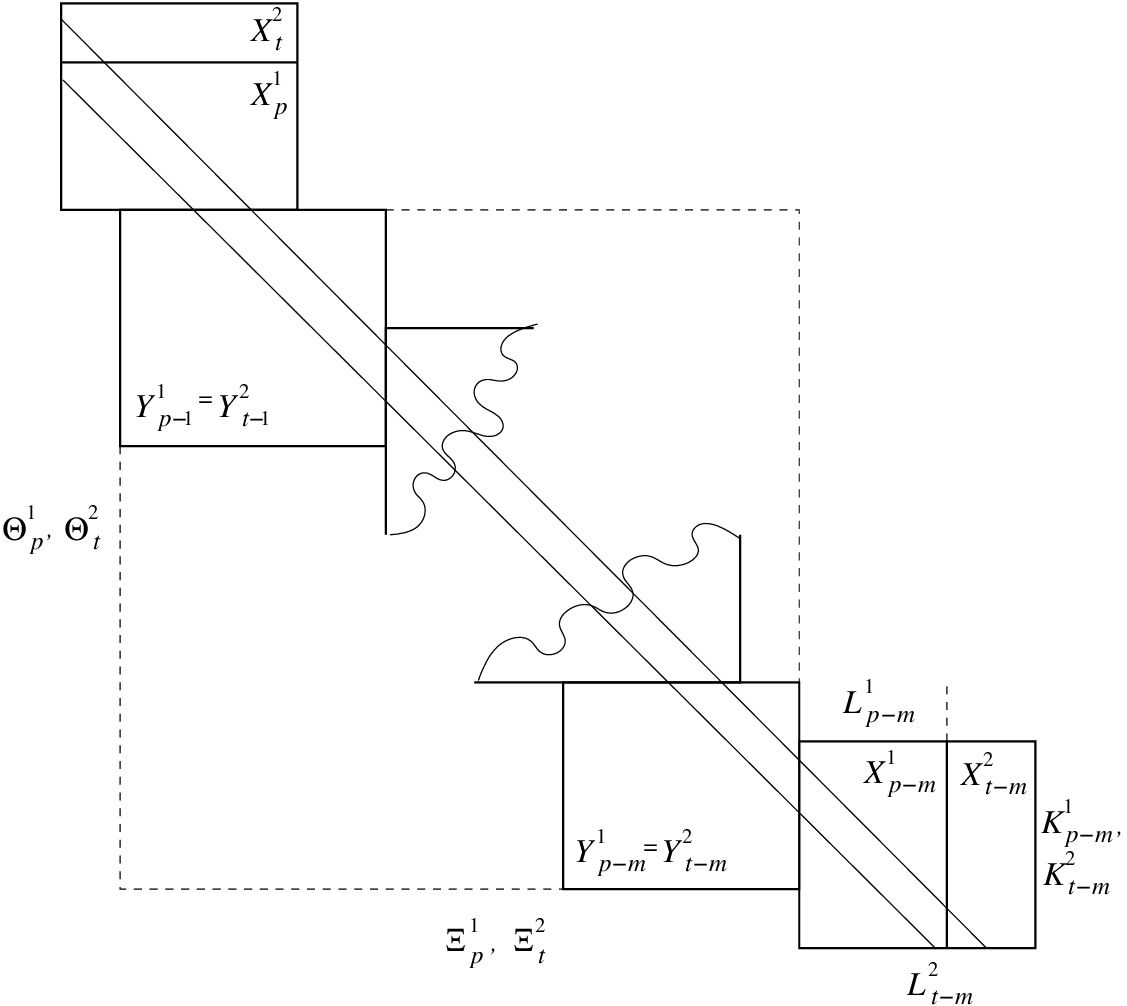}
\caption{Case (ii)}
\label{fig:chainx}
\end{center}
\end{figure}

Let $\Theta_r^i$ and $\Xi_r^i$ be defined by \eqref{thetaxi}.
Note that the matrix
$\Lo_{\Theta_{p}^1\cup K_{p-m}^1}^{L_{p-m}^1}$ coincides with a proper submatrix of 
$\Ld_{\Theta_{t}^2\cup K_{t-m}^2}^{L_{t-m}^2}$; we denote the corresponding 
injection $\rho$ (in a sense, it can be considered as an analog of the injection $\rho$ defined in Section \ref{etaletasec}; however, it acts in the opposite direction).
Clearly, $\rho(\Theta_p^1\cup  K_{p-m}^1)=\Theta_t^2\cup K_{t-m}^2$.
Similarly to \eqref{twoterm}, we have
\begin{multline*}
\Lo_{\Theta_p^1}^{\Psi_p^1}\nalo_{\Psi_p^1}^{K_p^1\cup\Theta_{p}^1}\\=
 \lgrado_{\Theta_p^1}^{K_p^1\cup\Theta_{p}^1}-
\Lo_{\Theta_p^1}^{\Xi_p^1\setminus\Psi_p^1}
\nalo_{\Xi_p^1\setminus\Psi_p^1}^{K_p^1\cup\Theta_p^1}
-\Lo_{\Theta_p^1}^{L_{p-m}^1}
\nalo_{L_{p-m}^1}^{K_p^1\cup\Theta_p^1}.
\end{multline*}
The first two terms in the right hand side of this equation are treated exactly as in Case (i) and yield the same contribution. 
The third term yields 
\begin{equation*}
-\left\langle \nald_{L_t^2}^{\Theta_t^2}
\Ld_{\Theta_t^2}^{\rho(L_{p-m}^1)}
\nalo_{L_{p-m}^1}^{K_p^1\cup\Theta_p^1}
\Lo_{K_p^1\cup\Theta_{p}^1}^{L_p^1}\right\rangle
\end{equation*}
since $\Lo_{\Theta_p^1}^{L_{p-m}^1}=\Ld_{\Theta_t^2}^{\rho(L_{p-m}^1)}$. To proceed further, note
that
\[
\nald_{L_t^2}^{\Theta_t^2}
\Ld_{\Theta_t^2}^{\rho(L_{p-m}^1)}=\gradld_{L_t^2}^{\rho(L_{p-m}^1)}-
\nald_{L_t^2}^{K_{t-m}^2\setminus\Phi_{t-m}^2}
\Ld_{K_{t-m}^2\setminus\Phi_{t-m}^2}^{\rho(L_{p-m}^1)}.
\]
The first term on the right hand side vanishes, since $\nabla_{\L}\L$ is lower triangular, and columns
$L_t^2$ lie to the left of $\rho(L_{p-m}^1)$. The second yields
\begin{multline*}
\left\langle \nald_{L_t^2}^{K_{t-m}^2\setminus\Phi_{t-m}^2}
\Lo_{K_{p-m}^1\setminus\Phi_{p-m}^1}^{L_{p-m}^1}
\nalo_{L_{p-m}^1}^{K_p^1\cup\Theta_p^1}
\Lo_{K_p^1\cup\Theta_{p}^1}^{L_p^1}\right\rangle\\
=\left\langle \nald_{L_t^2}^{K_{t-m}^2\setminus\Phi_{t-m}^2}
\lgrado_{K_{p-m}^1\setminus\Phi_{p-m}^1}^{K_p^1\cup\Theta_p^1}
\Lo_{K_p^1\cup\Theta_{p}^1}^{L_p^1}\right\rangle
\end{multline*}
via $\Ld_{K_{t-m}^2\setminus\Phi_{t-m}^2}^{\rho(L_{p-m}^1)}=\Lo_{K_p^1\cup\Theta_{p}^1}^{L_{p-m}^1}$.
Finally, $\lgrado_{K_{p-m}^1\setminus\Phi_{p-m}^1}^{K_p^1\cup\Theta_p^1}$ vanishes, since $\L\nabla_{\L}$ is upper triangular, and rows $K_{p-m}^1\setminus\Phi_{p-m}^1$ lie below $K_p^1\cup\Theta_p^1$.

{\it Case\/} (iii): This case is only possible if the last block in the first sequence is of type $Y$, see 
Fig.~\ref{fig:chains} on the left. Assuming that this block is $Y_{\bar I_{p-m}^1}^{\bar J_{p-m}^1}$, we proceed exactly as in Case (ii) with $L_{p-m}^1=\varnothing$ and get the same contribution.

\begin{figure}[ht]
\begin{center}
\includegraphics[height=6.5cm]{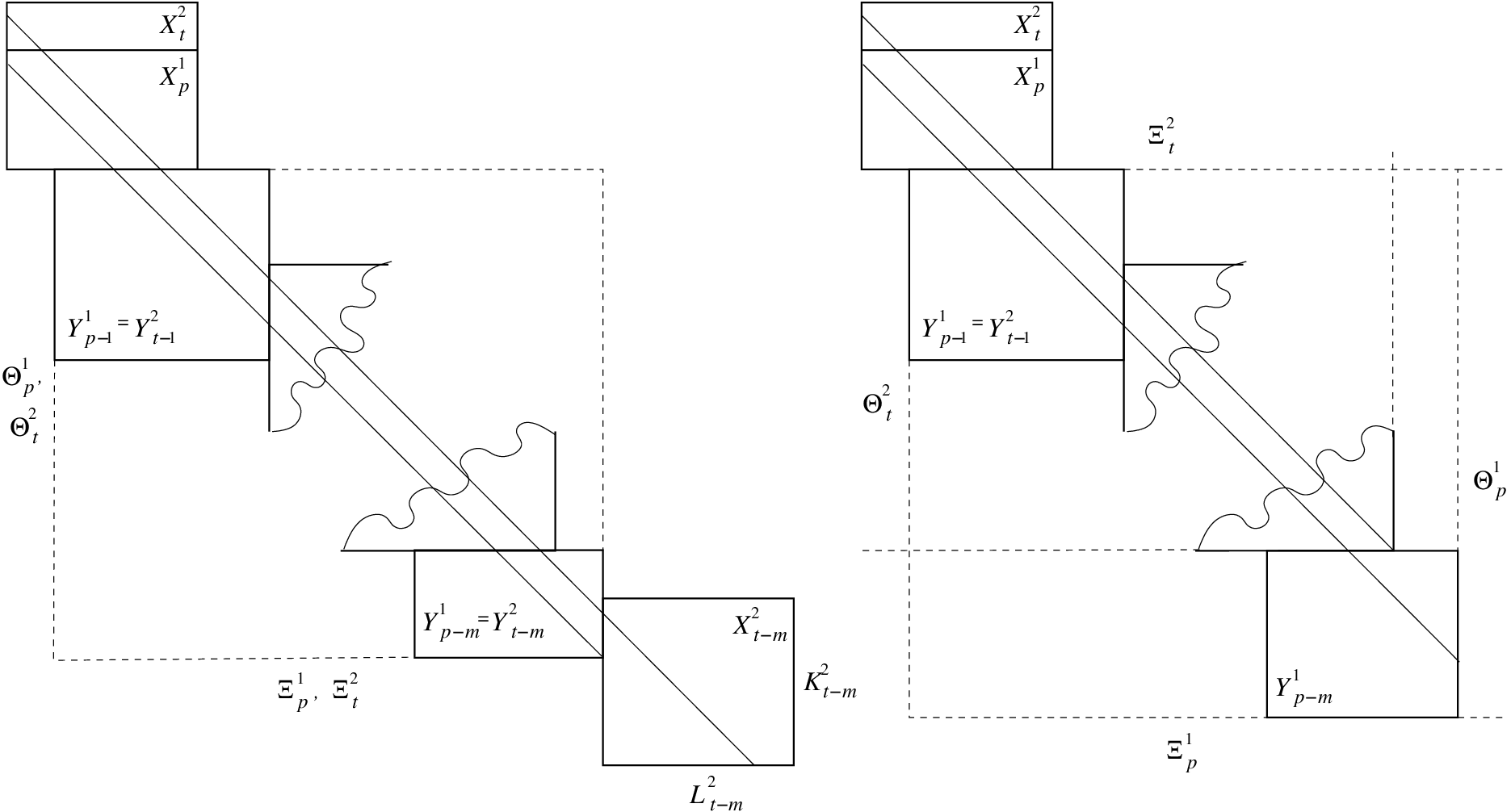}
\caption{Cases (iii) and (iv)}
\label{fig:chains}
\end{center}
\end{figure}

{\it Case\/} (iv): This case is only possible if the last block in the second sequence is of type $X$, see 
Fig.~\ref{fig:chains} on the right. 
Assuming that this block is $X_{I_{t-m+1}^2}^{J_{t-m+1}^2}$, we proceed exactly as in Case (i) with
$\bar K_{t-m}^2=\varnothing$ and get the same contribution.

To treat the fifth sum in \eqref{zone1}, note that $\alpha_p^1<\alpha_t^2$ implies that the block 
$X_{I_t^2}^{J_t^2}$ is contained completely inside the block $X_{I_p^1}^{J_p^1}$. 
Therefore, injection $\rho$ can be defined as in Section \ref{etaletasec}; moreover, $\rho(\Psi_t^2)=\Psi_p^1$ and 
$\rho(L_t^2)=L_p^1$, since $\beta_p^1=\beta_t^2$. Consequently, the block $Y_{\bar I_{p-1}^1}^{\bar J_{p-1}^1}$ is
contained completely inside the block $Y_{\bar I_{t-1}^2}^{\bar J_{t-1}^2}$, and injection $\sigma$  can be defined 
as in Section \ref{etaretasec}.

We proceed similarly to the previous case and arrive at
\begin{multline}\label{bmic0}
\badC_t-\badBp_t+\badG_t=\left\langle\gradlo_{\Psi_p^1}^{\Psi_p^1}
\gradld_{\Psi_t^2}^{\Psi_t^2}\right\rangle\\
-\left\langle\gradlo_{\Psi_p^1}^{L_p^1} 
\nald^{K_t^2}_{L_{t}^2}\Ld^{\Psi_t^2}_{K_{t}^2}\right\rangle+
\left\langle\lgrado_{\rho(\Phi_t^2)}^{\rho(\Phi_t^2)} 
\Ld_{\Phi_t^2}^{L_{t}^2}\nald_{L_t^2}^{\Phi_t^2}\right\rangle.
\end{multline}

Clearly,
$\gradlo_{\Psi_p^1}^{L_p^1}=\nalo_{\Psi_p^1}^{K_p^1\cup\bar K_{p-1}^1}
\Lo_{K_p^1\cup\bar K_{p-1}^1}^{L_p^1}$,
so the second term in \eqref{bmic0} equals
\begin{multline}\label{bmic1}
-\left\langle\Lo_{K_p^1\cup\bar K_{p-1}^1}^{L_p^1} 
\nald^{K_t^2}_{L_t^2}\Lo^{\Psi_p^1}_{\rho(K_{t}^2)}\nalo_{\Psi_p^1}^{K_p^1\cup\bar K_{p-1}^1} \right\rangle\\
=\left\langle\Lo_{K_p^1\cup\bar K_{p-1}^1}^{L_p^1} 
\nald^{K_t^2}_{L_{t}^2}\Lo_{\rho(K_t^2)}^{L_p^1\setminus\Psi_p^1}
\nalo_{L_p^1\setminus\Psi_p^1}^{K_p^1\cup\bar K_{p-1}^1} \right\rangle\\
-\left\langle\Lo_{K_p^1\cup\bar K_{p-1}^1}^{L_p^1} 
\nald^{K_t^2}_{L_{t}^2}\lgrado_{\rho(K_t^2)}^{K_p^1\cup\bar K_{p-1}^1}
\right\rangle.
\end{multline}

The first term in \eqref{bmic1} equals
\begin{multline*}
\left\langle\Lo_{K_p^1\cup\bar K_{p-1}^1}^{L_p^1} 
\nald^{K_t^2}_{L_{t}^2}\Ld_{K_t^2}^{L_t^2\setminus\Psi_t^2}
\nalo_{L_p^1\setminus\Psi_p^1}^{K_p^1\cup\bar K_{p-1}^1} \right\rangle\\
=\left\langle\gradld^{L_t^2\setminus\Psi_t^2}_{L_{t}^2}
\gradlo_{L_p^1\setminus\Psi_p^1}^{L_p^1}\right\rangle
=\left\langle\gradld^{L_t^2\setminus\Psi_t^2}_{L_{t}^2\setminus\Psi_t^2}
\gradlo_{L_p^1\setminus\Psi_p^1}^{L_p^1\setminus\Psi_p^1}\right\rangle
=\text{const},
\end{multline*}
which together with the contribution of the first term in \eqref{bmic0}  yields the fourth term in the statement 
of the lemma for $\alpha_t^2>\alpha_p^1$.

By~\eqref{lnal}, the matrix $\lgrado_{\rho(K_t^2)}^{K_p^1\setminus\rho(K_t^2)}$ vanishes. Next, we use injection 
$\sigma$ mentioned above to write $\Lo_{\rho(K_t^2)\cup\bar K_{p-1}^1}^{L_p^1}= \Ld_{K_t^2\cup \sigma(\bar K_{p-1}^1)}^{L_t^2}$,
and hence  the second term in \eqref{bmic1} can be written as
\begin{multline}\label{bmic2}
-\left\langle\lgrado_{\rho(K_t^2)}^{\rho(K_t^2)\cup\bar K_{p-1}^1}
\Ld_{K_t^2\cup \sigma(\bar K_{p-1}^1)}^{L_t^2} 
\nald^{K_t^2}_{L_{t}^2}\right\rangle\\
=-\left\langle\lgrado_{\rho(K_t^2)}^{\rho(K_t^2)\cup\bar K_{p-1}^1}
\lgradd_{K_t^2\cup\sigma(\bar K_{p-1}^1)}^{K_t^2}\right\rangle\\
+\left\langle\lgrado_{\rho(K_t^2)}^{\rho(K_t^2)\cup\bar K_{p-1}^1}
\Ld_{K_t^2\cup \sigma(\bar K_{p-1}^1)}^{\bar L_t^2} 
\nald^{K_t^2}_{\bar L_{t}^2}\right\rangle\\
+\left\langle\lgrado_{\rho(K_t^2)}^{\rho(K_t^2)\cup\bar K_{p-1}^1}
\Ld_{K_t^2\cup \sigma(\bar K_{p-1}^1)}^{\bar L_{t-1}^2\setminus\Psi_t^2} 
\nald^{K_t^2}_{\bar L_{t-1}^2\setminus\Psi_t^2}\right\rangle.
\end{multline}

By \eqref{lnal}, the first term in \eqref{bmic2} equals
\begin{equation}\label{bmic21}
-\left\langle\lgrado_{\rho(K_t^2)}^{\rho(K_t^2)}
\lgradd_{K_t^2}^{K_t^2}\right\rangle=\text{const}.
\end{equation}
 
Recall that the matrix $\Ld_{(K_t^2\setminus\Phi_t^2)\cup \sigma(\bar K_{p-1}^1)}^{\bar L_t^2}$ vanishes, and so 
the second term in \eqref{bmic2} can be rewritten as 
\[
\left\langle\lgrado_{\rho(K_t^2)}^{\rho(\Phi_t^2)}\Ld_{\Phi_t^2}^{\bar L_t^2} 
\nald^{K_t^2}_{\bar L_{t}^2}\right\rangle=
\left\langle\lgrado_{\rho(\Phi_t^2)}^{\rho(\Phi_t^2)}\Ld_{\Phi_t^2}^{\bar L_t^2} 
\nald^{\Phi_t^2}_{\bar L_{t}^2}\right\rangle
\]
by \eqref{lnal}. Taking into account the third term in \eqref{bmic0}, we get exactly the same contribution 
as in \eqref{motive}, which together with \eqref{bmic21} yields the fifth term in the statement of the lemma
for $\alpha_t^2>\alpha_p^1$.

To treat the third term in \eqref{bmic2} note that 
\[
\lgrado_{\rho(K_t^2)}^{\rho(K_t^2)\cup\bar K_{p-1}^1}=
\Lo_{\rho(K_t^2)}^{L_p^1}\nalo_{L_p^1}^{\rho(K_t^2)\cup\bar K_{p-1}^1}
\]
 and that the matrix
$\Ld_{K_t^2}^{\bar L_{t-1}^2\setminus\Psi_t^2}$ vanishes. Consequently, the term in question equals
\begin{multline*}
\left\langle\Lo_{\rho(K_t^2)}^{L_p^1}\nalo_{L_p^1}^{\rho(K_t^2)\cup\bar K_{p-1}^1}
\Ld_{K_t^2\cup \sigma(\bar K_{p-1}^1)}^{\bar L_{t-1}^2\setminus\Psi_t^2}
\nald^{K_t^2}_{\bar L_{t-1}^2\setminus\Psi_t^2}\right\rangle\\
=\left\langle\Lo_{\rho(K_t^2)}^{L_p^1}\nalo_{L_p^1}^{\bar K_{p-1}^1}
\Lo_{\bar K_{p-1}^1}^{\bar L_{p-1}^2\setminus\Psi_p^1}
\nald^{K_t^2}_{\bar L_{t-1}^2\setminus\Psi_t^2}\right\rangle,
\end{multline*}
since $\Ld_{\sigma(\bar K_{p-1}^1)}^{\bar L_{t-1}^2\setminus\Psi_t^2} =\Lo_{\bar K_{p-1}^1}^{\bar L_{p-1}^2\setminus\Psi_p^1}$.
The obtained expression vanishes since 
\[
\nalo_{L_p^1}^{\bar K_{p-1}^1}
\Lo_{\bar K_{p-1}^1}^{\bar L_{p-1}^2\setminus\Psi_p^1}=\gradlo_{L_p^1}^{\bar L_{p-1}^2\setminus\Psi_p^1}
\]
vanishes by \eqref{lnal}.

Further, consider the sixth term in \eqref{zone1}. Using \eqref{bmic} we arrive at
\begin{multline}\label{bmicmie}
\badC_t-\badBp_t+\badE_t=\left\langle\gradlo_{\Psi_p^1}^{\Psi_p^1}
\gradld_{\Psi_t^2}^{\Psi_t^2}\right\rangle\\
-\left\langle\gradlo_{\Psi_p^1}^{L_p^1} 
\nald^{K_t^2}_{L_{t}^2}\Ld^{\Psi_t^2}_{K_{t}^2}\right\rangle
+\left\langle\lgrado_{\Phi_p^1}^{\Phi_p^1} 
\Ld_{\Phi_t^2}^{L_{t}^2}\nald_{L_{t}^2}^{\Phi_t^2}\right\rangle.
\end{multline}
Clearly, the first term in \eqref{bmicmie} is a constant.

Note that the blocks $X_{I_p^1}^{J_p^1}$ and $X_{I_t^2}^{J_t^2}$ coincide. Similarly to the analysis above, 
we consider two nonempty sequences of blocks \eqref{blocks} 
(the cases $p=1$ or $t=1$ are trivial). We have the same four possibilities as before, and, additionally,

(v) the sequences coincide. 

Each one of the possibilities (i)--(iv) is further split into two cases: 

a) the exit point of $X_{I_t^2}^{J_t^2}$ lies below the exit point of $X_{I_p^1}^{J_p^1}$;

b) the exit point of $X_{I_t^2}^{J_t^2}$ lies above the exit point of $X_{I_p^1}^{J_p^1}$. 

{\it Case\/} (ia): Clearly, this can be possible only if $\bar I_{p-m}^1\subset \bar I_{t-m}^2$, see Fig.~\ref{fig:chainya}.

\begin{figure}[ht]
\begin{center}
\includegraphics[height=7cm]{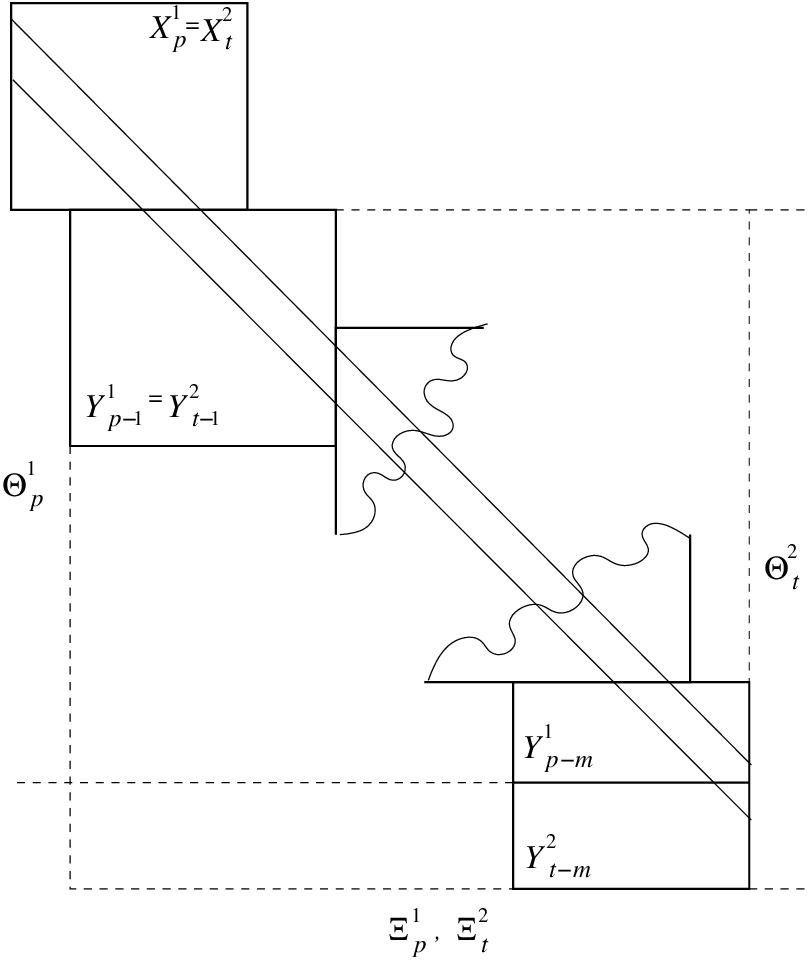}
\caption{Case (ia)}
\label{fig:chainya}
\end{center}
\end{figure}

Define $\Theta_r^i$ and $\Xi_r^i$ in the same way as in \eqref{thetaxi}. Using equalities \eqref{superdec} and 
$\Ld_{K_t^2}^{\Psi_t^2}=\Lo_{K_p^1}^{\Psi_p^1}$, 
we rewrite the second term in \eqref{bmicmie} as
\begin{multline*}
-\left\langle\nald^{K_t^2}_{L_{t}^2}
\lgrado_{K_p^1}^{K_p^1\cup\Theta_p^1}
\Lo_{K_p^1\cup\Theta_p^1}^{L_p^1}\right\rangle\\
+\left\langle\nald^{K_t^2}_{L_{t}^2}
\Lo_{K_p^1}^{L_p^1\setminus\Psi_p^1}
\nalo_{L_p^1\setminus\Psi_p^1}^{K_p^1\cup\Theta_p^1}
\Lo_{K_p^1\cup\Theta_p^1}^{L_p^1}\right\rangle\\
+\left\langle\nald^{K_t^2}_{L_{t}^2}
\Lo_{K_p^1}^{\bar L_p^1}
\nalo_{\bar L_p^1}^{K_p^1\cup\Theta_p^1}
\Lo_{K_p^1\cup\Theta_p^1}^{L_p^1}\right\rangle.
\end{multline*}
Note that $\Lo_{K_p^1}^{L_p^1\setminus\Psi_p^1}=\Ld_{K_t^2}^{L_t^2\setminus\Psi_t^2}$ and
\begin{equation*}
\begin{aligned}
\nalo_{L_p^1\setminus\Psi_p^1}^{K_p^1\cup\Theta_p^1}
\Lo_{K_p^1\cup\Theta_p^1}^{L_p^1}&=
\gradlo_{L_p^1\setminus\Psi_p^1}^{L_p^1},\\
\nald^{K_t^2}_{L_{t}^2}
\Ld_{K_t^2}^{L_t^2\setminus\Psi_t^2}&=
\gradld^{L_t^2\setminus\Psi_t^2}_{L_t^2},
\end{aligned}
\end{equation*}
hence the second term in the expression above equals 
\[
\left\langle\gradlo_{L_p^1\setminus\Psi_p^1}^{L_p^1}\gradld^{L_t^2\setminus\Psi_t^2}_{L_t^2}\right\rangle=
\left\langle\gradlo_{L_p^1\setminus\Psi_p^1}^{L_p^1\setminus\Psi_p^1}
\gradld^{L_t^2\setminus\Psi_t^2}_{L_t^2\setminus\Psi_t^2}\right\rangle=\text{const},
\]
which together with the first term in \eqref{bmicmie} yields the eighth term in the statement of the lemma, as well as the fourth term for $\alpha_t^2=\alpha_p^1$.

Finally, 
$\nalo_{\bar L_p^1}^{K_p^1\cup\Theta_p^1}$ vanishes since the columns $\bar L_p^1$ are strictly to the left of 
$K_p^1\cup\Theta_p^1$, so the third term in the expression above vanishes.

Note that
\begin{multline*}
\lgrado_{K_p^1}^{K_p^1\cup\Theta_p^1}
\Lo_{K_p^1\cup\Theta_p^1}^{L_p^1}\\=
\lgrado_{K_p^1}^{\Phi_p^1}
\Lo_{\Phi_p^1}^{L_p^1}+
\lgrado_{K_p^1}^{K_p^1\setminus\Phi_p^1}
\Lo_{K_p^1\setminus\Phi_p^1}^{L_p^1}+
\lgrado_{K_p^1}^{\Theta_p^1}
\Lo_{\Theta_p^1}^{L_p^1}.
\end{multline*}
By \eqref{lnal}, $\lgrado_{K_p^1\setminus\Phi_p^1}^{\Phi_p^1}$ vanishes; 
besides, $\Ld_{\Phi_t^2}^{L_{t}^2}=\Lo_{\Phi_p^1}^{L_{p}^1}$. Hence
\begin{equation*}
-\left\langle\lgrado_{K_p^1}^{\Phi_p^1}
\Lo_{\Phi_p^1}^{L_{p}^1}\nald_{L_{t}^2}^{K_t^2}\right\rangle
=-\left\langle\lgrado_{\Phi_p^1}^{\Phi_p^1} 
\Ld_{\Phi_t^2}^{L_{t}^2}\nald_{L_{t}^2}^{\Phi_t^2}\right\rangle,
\end{equation*}
that is, the first term in the equation above cancels the third term in \eqref{bmicmie}. 
Further, $\Lo_{K_p^1\setminus\Phi_p^1}^{L_{p}^1}=\Ld_{K_t^2\setminus\Phi_t^2}^{L_{t}^2}$ and
\[
\Ld_{K_t^2\setminus\Phi_t^2}^{L_{t}^2}\nald_{L_{t}^2}^{K_t^2}=
\lgradd_{K_t^2\setminus\Phi_t^2}^{K_t^2},
\]
and hence
\begin{multline}\label{foroneb}
-\left\langle\lgrado_{K_p^1}^{K_p^1\setminus\Phi_p^1} 
\Lo_{K_p^1\setminus\Phi_p^1}^{L_{p}^1}\nald_{L_{t}^2}^{K_t^2}\right\rangle
=-\left\langle\lgrado_{K_p^1}^{K_p^1\setminus\Phi_p^1} 
\lgradd_{K_{t}^2\setminus\Phi_t^2}^{K_t^2}\right\rangle\\
=-\left\langle\lgrado_{K_p^1\setminus\Phi_p^1}^{K_p^1\setminus\Phi_p^1} 
\lgradd_{K_{t}^2\setminus\Phi_t^2}^{K_t^2\setminus\Phi_t^2}\right\rangle
=\text{const}.
\end{multline}

The remaining contribution of \eqref{bmicmie} equals
\begin{equation}\label{fortwoa}
-\left\langle\lgrado_{K_p^1}^{\Theta_p^1} 
\Lo_{\Theta_p^1}^{L_{p}^1}\nald_{L_{t}^2}^{K_t^2}\right\rangle=
-\left\langle\lgrado_{\Phi_p^1}^{\Theta_p^1} 
\Lo_{\Theta_p^1}^{\Psi_{p}^1}\nald_{\Psi_{t}^2}^{\Phi_t^2}\right\rangle,
\end{equation}
since the deleted columns and rows of $\L^1\nabla_{\L}^1$ and $\L^1$ vanish. 

Next we use the injection $\sigma$ (similar to the one used in Case (i) above but acting in the opposite direction) to rewrite 
$\Lo_{\Theta_p^1}^{\Psi_{p}^1}=\Ld_{\sigma(\Theta_p^1)}^{\Psi_t^2}$, and to write
\[
\Ld_{\sigma(\Theta_p^1)}^{\Psi_t^2}\nald_{\Psi_{t}^2}^{\Phi_t^2}=
\lgradd_{\sigma(\Theta_p^1)}^{\Phi_t^2}-
\Ld_{\sigma(\Theta_p^1)}^{\Xi_t^2\setminus\Psi_t^2}
\nald_{\Xi_t^2\setminus\Psi_{t}^2}^{\Phi_t^2},
\]
which transforms the above contribution into
\[
-\left\langle\lgrado_{\Phi_p^1}^{\Theta_p^1}
\lgradd_{\sigma(\Theta_p^1)}^{\Phi_t^2}\right\rangle+
\left\langle\lgrado_{\Phi_p^1}^{\Theta_p^1}
\Ld_{\sigma(\Theta_p^1)}^{\Xi_t^2\setminus\Psi_t^2}
\nald_{\Xi_t^2\setminus\Psi_{t}^2}^{\Phi_t^2}\right\rangle.
\]
Clearly, the first term above vanishes since $\lgradd_{\sigma(\Theta_p^1)}^{\Phi_t^2}=0$. 
The second one vanishes since 
\begin{equation}\label{fortwoaa}
\lgrado_{\Phi_p^1}^{\Theta_p^1}=
\Lo_{\Phi_p^1}^{L_p^1\cup\bar L_p^1}
\nalo_{L_p^1\cup\bar L_p^1}^{\Theta_p^1},
\end{equation}
$\Ld_{\sigma(\Theta_p^1)}^{\Xi_t^2\setminus\Psi_t^2}=\Lo_{\Theta_p^1}^{\Xi_p^1\setminus\Psi_p^1}$ and
\[
\nalo_{L_p^1\cup\bar L_p^1}^{\Theta_p^1}
\Lo_{\Theta_p^1}^{\Xi_p^1\setminus\Psi_p^1}=
\gradlo_{L_p^1\cup\bar L_p^1}^{\Xi_p^1\setminus\Psi_p^1}=0.
\]

{\it Case\/} (ib): Clearly, this can be possible only if $\bar I_{t-m}^2\subset \bar I_{p-m}^1$, 
cf.~Fig.~\ref{fig:chainy}.
We proceed exactly as in Case (ia), retaining the definitions of 
$\Theta_r$ and $\Xi_r$, and arrive at \eqref{fortwoa}. 
As a result, we obtain two contributions similar to those obtained in Case (ia): one is similar to the 
eighth term in the statement of the lemma and is given by
\begin{equation}\label{tbdes}
{\sum_{{\beta_p^1=\beta_t^2}\atop{\alpha_p^1=\alpha_t^2}}}^\ea
\left\langle\gradlo_{L_p^1}^{L_p^1}\gradld_{L_t^2}^{L_t^2}\right\rangle,
\end{equation} 
while the other together with \eqref{foroneb} yields the fifth term in the statement of the lemma
for $\alpha_t^2=\alpha_p^1$.

Next, we note that 
$\lgrado_{\Phi_p^1}^{\Theta_p^1}=\Lo_{\Phi_p^1}^{L_p^1}\nalo_{L_p^1}^{\Theta_p^1}$, since
$\nalo_{\bar L_p^1}^{\Theta_p^1}=0$. Applying $\Lo_{\Phi_p^1}^{L_p^1}=\Ld_{\Phi_t^2}^{L_t^2}$, 
we arrive at
\[
-\left\langle\nalo_{L_p^1}^{\Theta_p^1}\Lo_{\Theta_p^1}^{\Psi_p^1}\nald_{\Psi_t^2}^{\Phi_t^2}
\Ld_{\Phi_t^2}^{L_t^2}\right\rangle.
\]
Note that
\begin{equation}\label{kern}
\nald_{\Psi_t^2}^{\Phi_t^2}\Ld_{\Phi_t^2}^{L_t^2}=\gradld_{\Psi_t^2}^{L_t^2}-
\nald_{\Psi_t^2}^{K_t^2\setminus\Phi_t^2}\Ld_{K_t^2\setminus\Phi_t^2}^{L_t^2}
-\nald_{\Psi_t^2}^{\Theta_t^2}\Ld_{\Theta_t^2}^{L_t^2}.
\end{equation}

To treat the first term in \eqref{kern}, we use an analog of \eqref{compli} and get
\[
-\left\langle \gradlo_{L_p^1}^{\Psi_p^1} \gradld_{\Psi_t^2}^{L_t^2} \right\rangle+
\left\langle \nalo_{L_p^1}^{K_p^1}\Lo_{K_p^1}^{\Psi_p^1} \gradld_{\Psi_t^2}^{L_t^2}\right\rangle.
\]
Clearly, the first term above equals 
\begin{equation}\label{tbdes1}
-\left\langle \gradlo_{\Psi_p^1}^{\Psi_p^1} \gradld_{\Psi_t^2}^{\Psi_t^2} \right\rangle=\text{const}.
\end{equation}

The second term above can be rewritten as
\[
\left\langle \nalo_{L_p^1}^{K_p^1}\Ld_{K_t^2}^{\Psi_t^2}\nald_{\Psi_t^2}^{K_t^2\cup\Theta_t^2}
\Ld_{K_t^2\cup\Theta_t^2}^{L_t^2}\right\rangle.
\]
Next, we write
\begin{equation}\label{kern1}
\Ld_{K_t^2}^{\Psi_t^2}\nald_{\Psi_t^2}^{K_t^2\cup\Theta_t^2}=
\lgradd_{K_t^2}^{K_t^2\cup\Theta_t^2}-\Ld_{K_t^2}^{L_t^2\setminus\Psi_t^2}
\nald_{L_t^2\setminus\Psi_t^2}^{K_t^2\cup\Theta_t^2}-
\Ld_{K_t^2}^{\bar L_t^2}\nald_{\bar L_t^2}^{K_t^2\cup\Theta_t^2}.
\end{equation}
The contribution of the first term in \eqref{kern1} can be written as
\begin{multline*}
\left\langle \lgradd_{K_t^2}^{K_t^2\cup\Theta_t^2}\Lo_{K_p^1\cup\sigma(\Theta_t^2)}^{L_p^1}
\nalo_{L_p^1}^{K_p^1}\right\rangle\\
=-\left\langle \lgradd_{K_t^2}^{K_t^2\cup\Theta_t^2}
\Lo_{K_p^1\cup\sigma(\Theta_t^2)}^{\Xi_p^1\setminus\Psi_p^1}
\nalo_{\Xi_p^1\setminus\Psi_p^1}^{K_p^1}\right\rangle\\
+\left\langle \lgradd_{K_t^2}^{K_t^2\cup\Theta_t^2}
\lgrado_{K_p^1\cup\sigma(\Theta_t^2)}^{K_p^1}\right\rangle,
\end{multline*}
where injection $\sigma$ is defined as in Case (i) above. 
The second term above equals
\[
 \left\langle \lgradd_{K_t^2}^{K_t^2}\lgrado_{K_p^1}^{K_p^1}\right\rangle=\text{const},
\]
and yields the seventh term in the statement of the lemma,
while the first term equals
\[
 -\left\langle \Ld_{K_t^2}^{L_t^2\cup \bar L_t^2}\nald_{L_t^2\cup \bar L_t^2}^{K_t^2\cup\Theta_t^2}
\Ld_{K_t^2\cup\Theta_t^2}^{\Xi_t^2\setminus\Psi_t^2}
\nalo_{\Xi_p^1\setminus\Psi_p^1}^{K_p^1}\right\rangle
\]
and vanishes, since
\[
\nald_{L_t^2\cup \bar L_t^2}^{K_t^2\cup\Theta_t^2}
\Ld_{K_t^2\cup\Theta_t^2}^{\Xi_t^2\setminus\Psi_t^2}=
\gradld_{L_t^2\cup \bar L_t^2}^{\Xi_t^2\setminus\Psi_t^2}=0
\]
by \eqref{lnal}.

The contribution of the second term in \eqref{kern1} equals
\begin{multline*}
-\left\langle \nalo_{L_p^1}^{K_p^1}\Lo_{K_p^1}^{L_p^1\setminus\Psi_p^1}
\nald_{L_t^2\setminus\Psi_t^2}^{K_t^2\cup\Theta_t^2}
\Ld_{K_t^2\cup\Theta_t^2}^{L_t^2}\right\rangle\\
=-\left\langle \gradlo_{L_p^1}^{L_p^1\setminus\Psi_p^1}
\gradld_{L_t^2\setminus\Psi_t^2}^{L_t^2}\right\rangle
=-\left\langle \gradlo_{L_p^1\setminus\Psi_p^1}^{L_p^1\setminus\Psi_p^1}
\gradld_{L_t^2\setminus\Psi_t^2}^{L_t^2\setminus\Psi_t^2}\right\rangle
=\text{const}
\end{multline*}
and together with \eqref{tbdes1} cancels the contribution of \eqref{tbdes}.

The contribution of the third term in \eqref{kern1} equals
\[
-\left\langle \nalo_{L_p^1}^{K_p^1}\Ld_{K_t^2}^{\bar L_t^2}
\nald_{\bar L_t^2}^{K_t^2\cup\Theta_t^2}\Ld_{K_t^2\cup\Theta_t^2}^{L_t^2}\right\rangle
\]
and vanishes, since
\[
\nald_{\bar L_t^2}^{K_t^2\cup\Theta_t^2}\Ld_{K_t^2\cup\Theta_t^2}^{L_t^2}=
\gradld_{\bar L_t^2}^{L_t^2}=0
\]
by \eqref{lnal}.

The contribution of the second term in \eqref{kern} equals
\[
\left\langle \Lo_{K_p^1\setminus\Phi_p^1}^{L_p^1}
\nalo_{L_p^1}^{\Theta_p^1}\Lo_{\Theta_p^1}^{\Psi_p^1}
\nald_{\Psi_t^2}^{K_t^2\setminus\Phi_t^2}\right\rangle
\]
and vanishes, since
\[
\Lo_{K_p^1\setminus\Phi_p^1}^{L_p^1}
\nalo_{L_p^1}^{\Theta_p^1}=\lgrado_{K_p^1\setminus\Phi_p^1}^{\Theta_p^1}=0;
\]
the latter equality follows from the fact 
$\lgrado_{(K_p^1\setminus\Phi_p^1)\cup\Theta_p^1}^{(K_p^1\setminus\Phi_p^1)\cup\Theta_p^1}=\one$.

The contribution of the third term in \eqref{kern} equals
\[
\left\langle \Lo_{\sigma(\Theta_t^2)}^{L_p^1}\nalo_{L_p^1}^{\Theta_p^1}
\Lo_{\Theta_p^1}^{\Psi_p^1}\nald_{\Psi_t^2}^{\Theta_t^2}\right\rangle
\]
via $\Ld_{\Theta_t^2}^{L_t^2}=\Lo_{\sigma(\Theta_t^2)}^{L_p^1}$. Note that
\[
\Lo_{\sigma(\Theta_t^2)}^{L_p^1}\nalo_{L_p^1}^{\Theta_p^1}=
\lgrado_{\sigma(\Theta_t^2)}^{\Theta_p^1}=
\begin{bmatrix} \one & 0 \end{bmatrix},
\]
and hence $\Lo_{\sigma(\Theta_t^2)}^{L_p^1}\nalo_{L_p^1}^{\Theta_p^1}
\Lo_{\Theta_p^1}^{\Psi_p^1}=\Ld_{\Theta_t^2}^{\Psi_t^2}$. Consequently, the contribution in question
equals 
\[
-\left\langle \Ld_{\Theta_t^2}^{\Psi_t^2}\nald_{\Psi_t^2}^{\Theta_t^2}\right\rangle=
-\left\langle \Ld_{\bar K_{t-1}^2}^{\Psi_t^2}\nald_{\Psi_t^2}^{\bar K_{t-1}^2}\right\rangle,
\]
which is a constant by Lemma \ref{partrace} yielding the sixth term in the statement of the lemma.

{\it Case\/} (iia): Clearly, this can be possible only if $J_{t-m}^2\subset J_{p-m}^1$, see Fig.~\ref{fig:chainxa}. 

\begin{figure}[ht]
\begin{center}
\includegraphics[height=7cm]{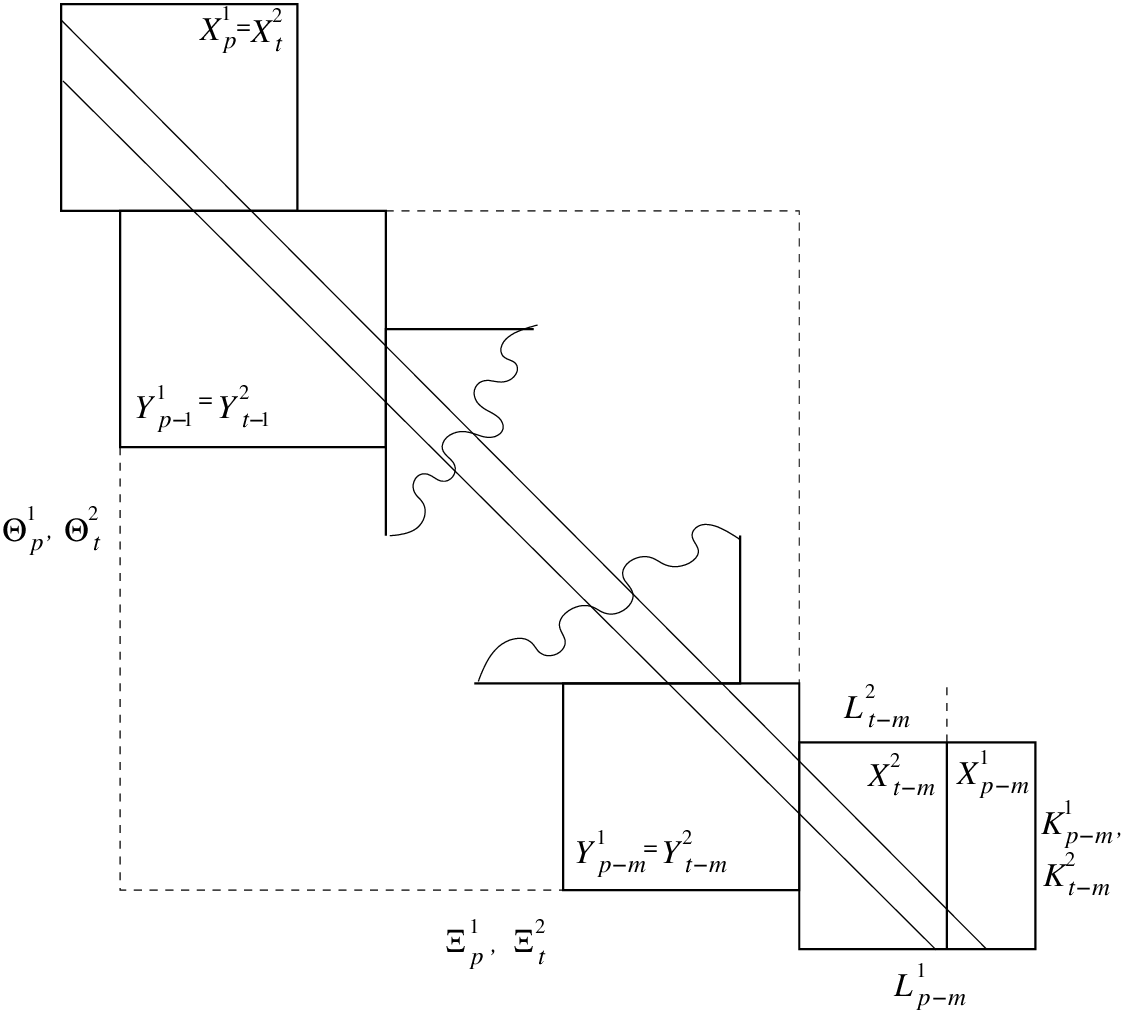}
\caption{Case (iia)}
\label{fig:chainxa}
\end{center}
\end{figure}

We proceed exactly as in Case (ia), retaining the definitions of 
$\Theta_r^i$ and $\Xi_r^i$, and arrive at \eqref{fortwoa}. 
Next, we apply  $\Lo_{\Theta_p^1}^{\Psi_{p}^1}=\Ld_{\Theta_t^2}^{\Psi_t^2}$, and note that
\begin{equation*}
\Ld_{\Theta_t^2}^{\Psi_t^2}\nald_{\Psi_{t}^2}^{\Phi_t^2}=
\lgradd_{\Theta_t^2}^{\Phi_t^2}-
\Ld_{\Theta_t^2}^{\Xi_t^2\setminus\Psi_t^2}
\nald_{\Xi_t^2\setminus\Psi_{t}^2}^{\Phi_t^2}-
\Ld_{\Theta_t^2}^{L_{t-m}^2}
\nald_{L_{t-m}^2}^{\Phi_t^2}.
\end{equation*}
Consequently, \eqref{fortwoa} can be written as a sum of three terms. The first two are treated exactly as in Case (ia)
and yield the same contribution.  With the help of \eqref{fortwoaa}, the third term can be rewritten as
\[
\left\langle\Lo_{\Phi_p^1}^{L_p^1\cup \bar L_p^1} \nalo_{L_p^1\cup \bar L_p^1}^{\Theta_p^1}
\Ld_{\Theta_t^2}^{L_{t-m}^2}\nald_{L_{t-m}^2}^{\Phi_t^2}\right\rangle.
\]
Next, we use the injection $\rho$ (similar to the one defined in Section~\ref{etaletasec}) to write 
$\Ld_{\Theta_t^2}^{L_{t-m}^2}=
\Lo_{\Theta_p^1}^{\rho(L_{t-m}^2)}$, which together with
\[
\nalo_{L_p^1\cup \bar L_p^1}^{\Theta_p^1}\Lo_{\Theta_p^1}^{\rho(L_{t-m}^2)}+
\nalo_{L_p^1\cup \bar L_p^1}^{K_{p-m}^1\setminus\Phi_{p-m}^1}
\Lo_{K_{p-m}^1\setminus\Phi_{p-m}^1}^{\rho(L_{t-m}^2)}=
\gradlo_{L_p^1\cup \bar L_p^1}^{\rho(L_{t-m}^2)}=0
\]
transforms the third term into
\[
-\left\langle\Lo_{\Phi_p^1}^{L_p^1\cup \bar L_p^1}
\nalo_{L_p^1\cup \bar L_p^1}^{K_{p-m}^1\setminus\Phi_{p-m}^1}
\Lo_{K_{p-m}^1\setminus\Phi_{p-m}^1}^{\rho(L_{t-m}^2)}\nald_{L_{t-m}^2}^{\Phi_t^2}\right\rangle.
\]
Finally, we use $\Lo_{K_{p-m}^1\setminus\Phi_{p-m}^1}^{\rho(L_{t-m}^2)}=
\Ld_{K_{t-m}^2\setminus\Phi_{t-m}^2}^{L_{t-m}^2}$ and
\[
\Ld_{K_{t-m}^2\setminus\Phi_{t-m}^2}^{L_{t-m}^2}\nald_{L_{t-m}^2}^{\Phi_t^2}=
\lgradd_{K_{t-m}^2\setminus\Phi_{t-m}^2}^{\Phi_t^2}=0
\]
to make sure that the contribution of this term vanishes.

{\it Case\/} (iib): Clearly, this can be possible only if $J_{p-m}^1\subset J_{t-m}^2$, cf.~Fig.~\ref{fig:chainx}. 
We proceed exactly as in Case (ib), with the only difference: the contribution of the first term in \eqref{kern1} contains an additional term
\[
 \left\langle
\Ld_{K_t^2}^{L_t^2\cup\bar L_t^2}\nald_{L_t^2\cup\bar L_t^2}^{K_t^2\cup\Theta_t^2}
\Lo_{K_p^1\cup\Theta_p^1}^{L_{p-m}^1}\nalo_{L_{p-m}^1}^{K_p^1}\right\rangle,
\]
which vanishes since $\Lo_{K_p^1\cup\Theta_p^1}^{L_{p-m}^1}=\Ld_{K_t^2\cup\Theta_t^2}^{\rho(L_{p-m}^1)}$ 
and
\[
\nald_{L_t^2\cup\bar L_t^2}^{K_t^2\cup\Theta_t^2}\Ld_{K_t^2\cup\Theta_t^2}^{\rho(L_{p-m}^1)}=
\gradld_{L_t^2\cup\bar L_t^2}^{\rho(L_{p-m}^1)}=0.
\]

{\it Case\/} (iiia):  This case is only possible if the last block in the first sequence is of type $X$, 
see Fig.~\ref{fig:chainsa} on the right. Assuming that this block is $X_{I_{p-m+1}^1}^{J_{p-m+1}^1}$, we proceed exactly as in Case (ia) with
$\bar K_{p-m}^1=\varnothing$ and get the same contribution.

\begin{figure}[ht]
\begin{center}
\includegraphics[height=6.5cm]{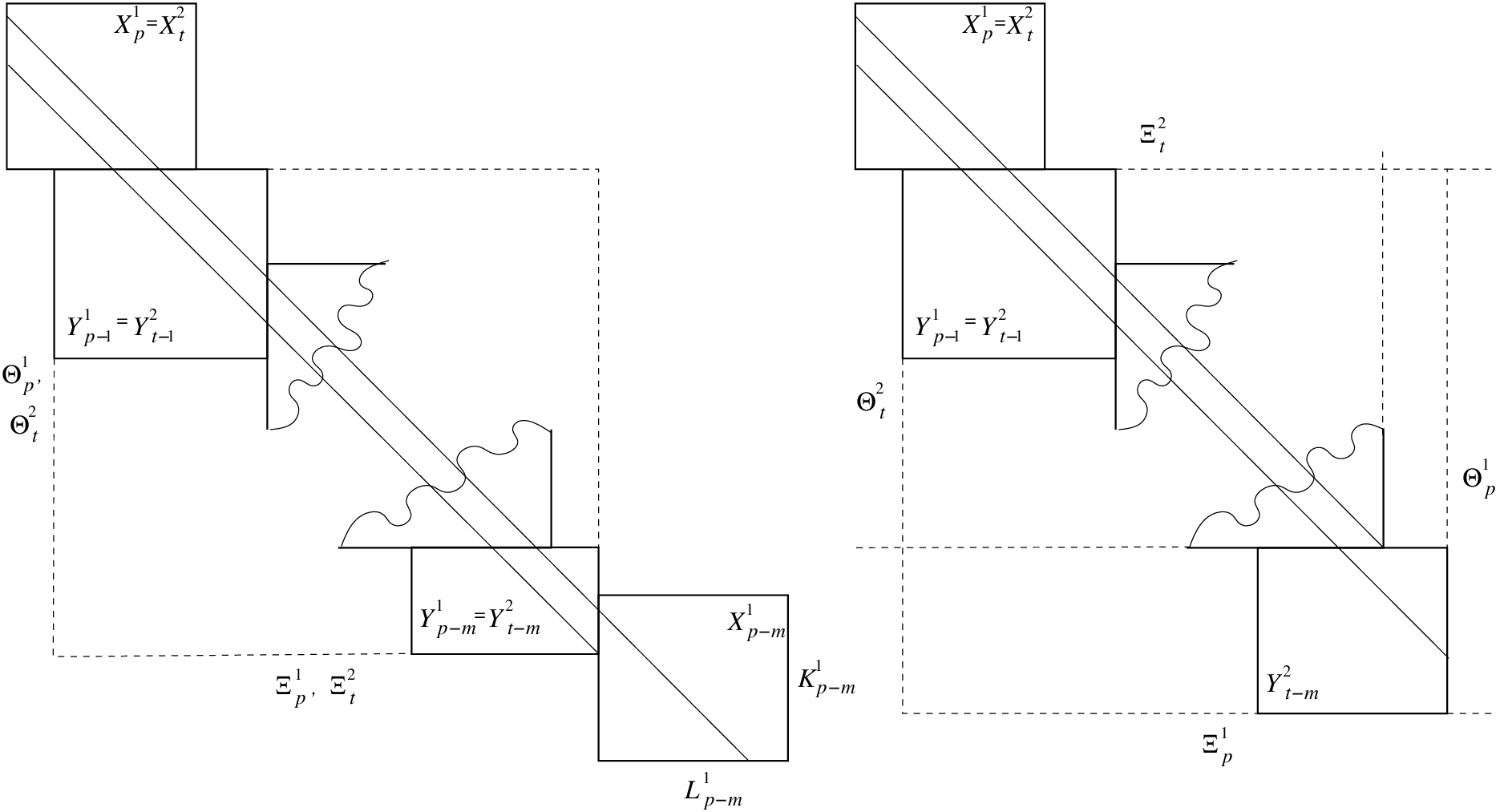}
\caption{Cases (iiia) and (iva)}
\label{fig:chainsa}
\end{center}
\end{figure}

{\it Case\/} (iiib): This case is only possible if the last block in the first sequence is of type $Y$, 
cf.~Fig.~\ref{fig:chains}. 
Assuming that this block is $Y_{\bar I_{p-m}^1}^{\bar J_{p-m}^1}$, we proceed exactly as in Case (iib) with 
$L_{p-m}^1=\varnothing$ and get the same contribution.

{\it Case\/} (iva): This case is only possible if the last block in the second sequence is of type $Y$, 
see Fig.~\ref{fig:chainsa} on the left. 
Assuming that this block is $Y_{\bar I_{t-m}^2}^{\bar J_{t-m}^2}$, we proceed exactly as in Case (iia) with 
$L_{t-m}^2=\varnothing$ and get the same contribution.
 
{\it Case\/} (ivb): This case is only possible if the last block in the second sequence is of type $X$, 
cf.~Fig.~\ref{fig:chains}. Assuming that this block is $X_{I_{t-m+1}^2}^{J_{t-m+1}^2}$, we proceed exactly as in Case (ib) with
$\bar K_{t-m}^2=\varnothing$ and get the same contribution.

{\it Case\/} (v): This case is only possible if the exit points of $X_{I_t^2}^{J_t^2}$ and $X_{I_p^1}^{J_p^1}$ coincide. 
The last block in both sequences is either of type $Y$ or of type $X$. In the former case we proceed as in Case (iva), and in the latter case, as in Case (iiia).

The last two terms in the statement of the lemma are obtained from the last two terms in \eqref{zone1} by taking into
account that $\lgrado_{\sigma(\Phi_t^2)}^{\sigma(\Phi_t^2)}$ in the expression \eqref{bad4} for $\badE_t$ and
$\gradlo_{\sigma(\Psi_{t+1}^2)}^{\sigma(\Psi_{t+1}^2)}$ in the expression \eqref{bad5} for $\badF_t$ are unit matrices, 
since in both cases $\sigma$ is an injection into the block $Y_{I_{p-1}^1}^{J_{p-1}^1}$. The remaining traces are treated in
the same way as in \eqref{trace}.
\end{proof}

\subsubsection{Case 2: $\hat l^1$ lies in rows $\bar K^1_{p-1}$ and columns $\bar L^1_{p-1}$}\label{case2} 
Similarly to the previous case,
$\lgrado_{\rho(\Phi_t^2)}^{\rho(\Phi_t^2)}$ in the expression \eqref{bad1} for $\badB_t$ in \eqref{etaleta}
and in the expression \eqref{bad6} for $\badG_t$ in \eqref{xinax}, $\lgrado_{\sigma(\Phi_t^2)}^{\sigma(\Phi_t^2)}$
in the expression \eqref{bad4} for $\badE_t$ in the fifth term in \eqref{xinax},  
as well as 
$\gradlo_{\Psi_p^1}^{L_p^1\setminus\Psi_p^1}$ in the expression \eqref{bad1eq} for $\badBp_t$ in \eqref{etaleta} vanish. 
Further, the contributions of $\badC_t$ to \eqref{etaleta} and to \eqref{xinay} cancel each other  for any $t$ such 
that $\beta_p^1>\beta_t^2$, while the contributions of $\badE_t$ to \eqref{etareta} and to \eqref{xinax} cancel each other  for any $t$ such that $\bar\alpha_{p-1}^1>\bar\alpha_t^2$.
Consequently, we arrive at
\begin{multline}\label{zone2}
\sum \{\badF_t-\badD_t :\ \bar\alpha_{p-1}^1>\bar\alpha_t^2, \bar\beta_{p-1}^1<\bar\beta_t^2 \}
+\sum \{\badC_{t+1}-\badD_{t} :\ \bar\alpha_{p-1}^1>\bar\alpha_t^2, \bar\beta_{p-1}^1=\bar\beta_t^2 \}\\
+\sum \{\badC_{t+1} :\  \bar\alpha_{p-1}^1<\bar\alpha_{t}^2, \bar\beta_{p-1}^1=\bar\beta_t^2  \}
+\sum \{\badE_{t}-\badDp_{t} :\ \bar\alpha_{p-1}^1=\bar\alpha_{t}^2, \bar\beta_{p-1}^1>\bar\beta_t^2  \}\\
+\sum \{\badE_{t}+\badF_t-\badDp_t :\ \bar\alpha_{p-1}^1=\bar\alpha_t^2, \bar\beta_{p-1}^1<\bar\beta_t^2 \}\\
+\sum \{\badE_{t}+\badC_{t+1}-\badDp_{t} :\  \bar\alpha_{p-1}^1=\bar\alpha_{t}^2, \bar\beta_{p-1}^1=\bar\beta_t^2 \}.
\end{multline}
A direct comparison shows that \eqref{zone2} can be obtained directly from the first six terms of \eqref{zone1} via 
switching the roles of $B^*_t$ and $\bar B^*_t$, replacing $\beta^*_t$ with $\bar\alpha^*_t$ and $\alpha^*_t$ with
$\bar\beta^*_t$, and shifting indices when necessary.

\begin{lemma}\label{zone2lemma}
{\rm (i)} Expression \eqref{zone2} is given by
\begin{multline*}
\sum_{{{\bar\alpha_{t-1}^2<\bar\alpha_{p-1}^1}\atop {\bar\beta_{t-1}^2>\bar\beta_{p-1}^1}}}
\left\langle\gradlo^{\sigma(\Psi_t^2)}_{\sigma(\Psi_t^2)}\gradld^{\Psi_t^2}_{\Psi_t^2}\right\rangle
+\sum_{{{\bar\alpha_{t-1}^2\ne\bar\alpha_{p-1}^1}\atop {\bar\beta_{t-1}^2=\bar\beta_{p-1}^1}}}
\left\langle\gradlo^{\Psi_p^1}_{\Psi_p^1}\gradld^{\Psi_t^2}_{\Psi_t^2}\right\rangle\\
+\sum_{{{\bar\alpha_{t-1}^2=\bar\alpha_{p-1}^1}\atop {\bar\beta_{t-1}^2<\bar\beta_{p-1}^1}}}
\left\langle\Ld^{L_{t-1}^2}_{\bar K_{t-1}^2}\nald^{\bar K_{t-1}^2}_{L_{t-1}^2}\right\rangle
+\sum_{{{\bar\alpha_{t-1}^2=\bar\alpha_{p-1}^1}\atop {\bar\beta_{t-1}^2\ge\bar\beta_{p-1}^1}}}
\left\langle\lgrado_{\bar K_{p-1}^1}^{\bar K_{p-1}^1}\lgradd_{\bar K_{t-1}^2}^{\bar K_{t-1}^2}\right\rangle
\\
-\sum_{{{\bar\alpha_{t-1}^2=\bar\alpha_{p-1}^1}\atop {\bar\beta_{t-1}^2\ge\bar\beta_{p-1}^1}}}
\left\langle\gradlo_{\sigma(\bar L_{t-1}^2\setminus\Psi_{t}^2)}^{\sigma(\bar L_{t-1}^2\setminus\Psi_{t}^2)}
\gradld_{\bar L_{t-1}^2\setminus\Psi_{t}^2}^{\bar L_{t-1}^2\setminus\Psi_{t}^2}\right\rangle
+{\sum_{{{\bar\alpha_{t-1}^2=\bar\alpha_{p-1}^1}\atop {\bar\beta_{t-1}^2=\bar\beta_{p-1}^1}}}}^{\hskip -9pt\el}
\left\langle\Ld^{L_{t-1}^2}_{\Phi_{t-1}^2}\nald^{\Phi_{t-1}^2}_{L_{t-1}^2}\right\rangle\\
+{\sum_{{{\bar\alpha_{t-1}^2=\bar\alpha_{p-1}^1}\atop {\bar\beta_{t-1}^2=\bar\beta_{p-1}^1}}}}^{\hskip -9pt\el}
\left\langle\gradlo_{\bar L_{p-1}^1}^{\bar L_{p-1}^1}\gradld_{\bar L_{t-1}^2}^{\bar L_{t-1}^2}\right\rangle
-{\sum_{{{\bar\alpha_{t-1}^2=\bar\alpha_{p-1}^1}\atop {\bar\beta_{t-1}^2=\bar\beta_{p-1}^1}}}}^{\hskip -9pt \el}
\left\langle\lgrado_{\bar K_{p-1}^1}^{\bar K_{p-1}^1}\lgradd_{\bar K_{t-1}^2}^{\bar K_{t-1}^2}\right\rangle,
\end{multline*}
where $\sum^\el$  is taken over the cases when the exit point of 
$Y_{I_{t-1}^2}^{J_{t-1}^2}$ lies to the left of the exit point
of $Y_{I_{p-1}^1}^{J_{p-1}^1}$. 

{\rm (ii)} Each summand in the expression above is a constant.
\end{lemma}

\begin{proof}
The contributions of the terms in \eqref{zone2} can be obtained from the computation of the contributions 
of the corresponding terms in \eqref{zone1} via a formal process, which 
replaces $K_*$, $L_*$, $\bar K_{*}$, $\bar L_{*}$, $\Phi_*$, 
$\Psi_*$, $\alpha_*$, $\beta_*$, $\bar\alpha_*$, $\bar\beta_*$ and $\sum^\ea$  
by $\bar L_{*-1}$, $\bar K_{*-1}$, $L_{*}$, $K_{*}$, $\Psi_*$, $\Phi_{*-1}$, $\bar\beta_{*-1}$, 
$\bar\alpha_{*-1}$, $\beta_*$, $\alpha_*$ and $\sum^\el$,
respectively, and interchanges $\rho$ and $\sigma$. Besides, matrix multiplication from the right should be replaced by the multiplication from the left, and the upper and lower indices should be interchanged. 

As an example of this formal process, let us consider the computation of the contribution of the fourth term in \eqref{zone2}. First observe, that the expression for  $\badC_t-\badBp_t$ in \eqref{bmicinit} is
transformed to
\begin{equation*}
\left\langle \Ld_{\Phi_{t-1}^2}^{L_{t-1}^2}
\nald_{L_{t-1}^2}^{\Phi_{t-1}^2}
\lgrado_{\Phi_{p-1}^1}^{\Phi_{p-1}^1} \right\rangle
-\left\langle \Ld_{\Phi_{t-1}^2}^{\bar L_{t-1}^2}
\nald_{\bar L_{t-1}^2}^{\bar K_{t-1}^2\setminus\Phi_{t-1}^2}
\lgrado^{\Phi_{p-1}^1}_{\bar K_{p-1}^1\setminus\Phi_{p-1}^1} 
\right\rangle,
\end{equation*}
which is exactly the expression for $\badE_{t-1}-\badDp_{t-1}$ (note that the summation index in the statement of the
lemma is shifted by one with respect to the summation index in \eqref{zone2}).

Next, we apply the transformed version of \eqref{compli} (which is identical to \eqref{compli2} with shifted indices) to the first expression above and use the transformed equality
\[
\nald_{\bar L_{t-1}^2}^{\bar K_{t-1}^2\setminus\Phi_{t-1}^2}
\lgrado^{\Phi_{p-1}^1}_{\bar K_{p-1}^1\setminus\Phi_{p-1}^1} 
+\nald_{\bar L_{t-1}^2}^{\Phi_{t-1}^2}
\lgrado^{\Phi_{p-1}^1}_{\Phi_{p-1}^1}
=\nald_{\bar L_{t-1}^2}^{\bar K_{t-1}^2}
\lgrado^{\Phi_{p-1}^1}_{\bar K_{p-1}^1}
\]
to get
\begin{equation}\label{bmic3}
\badE_{t-1}-\badDp_{t-1}=\left\langle\lgradd_{\Phi_{t-1}^2}^{\Phi_{t-1}^2}
\lgrado_{\Phi_{p-1}^1}^{\Phi_{p-1}^1}
\right\rangle-\left\langle \Ld_{\Phi_{t-1}^2}^{\bar L_{t-1}^2}
\nald_{\bar L_{t-1}^2}^{\bar K_{t-1}^2}
\gradlo^{\Phi_{p-1}^1}_{\bar K_{p-1}^1}
\right\rangle,
\end{equation}
which is the transformed version of \eqref{bmic}. Clearly, the first term above is a constant. 

Note that  $\bar\beta_{p-1}^1>\bar\beta_{t-1}^2$, which is the transformed version of 
$\alpha_p^1>\alpha_t^2$ and means that the block 
$Y_{I_{p-1}^2}^{J_{p-1}^2}$ is contained completely inside the block $Y_{I_{t-1}^1}^{J_{t-1}^1}$. Similarly to
Section \ref{case1}, we consider two sequences of blocks
\begin{equation*}
\{ X_{I_{p-1}^1}^{J_{p-1}^1}, Y_{\bar I_{p-2}^1}^{\bar J_{p-2}^1},  X_{I_{p-2}^1}^{J_{p-2}^1},\dots\}\quad\text{and}
\quad\{ X_{I_{t-1}^2}^{J_{t-1}^2}, Y_{\bar I_{t-2}^2}^{\bar J_{t-2}^2}, X_{I_{t-2}^2}^{J_{t-2}^2}, \dots\}
\end{equation*}
and study the same four cases.  
Let us consider Case (i) in detail. The analogs of $\Theta_r$ and $\Xi_r$ are  
\begin{equation*}
\bar\Theta_{r-1}=K_{r-1}\cup\bigcup_{i=2}^{m}(\bar K_{r-i}\cup K_{r-i}), \qquad
\bar\Xi_{r-1}=L_{r-1}\cup\bigcup_{i=2}^{m} (\bar L_{r-i}\cup L_{r-i}). 
\end{equation*}
We add the correspondence $\Theta_* \mapsto \bar\Xi_{*-1}$ and $\Xi_*\mapsto \bar\Theta_{*-1}$, which turns
the above relations into the transformed version of \eqref{thetaxi}.

Note that the matrix $\Ld_{\bar\Theta_{t-1}^2}^{\bar\Xi_{t-1}^2}$ coincides with a proper submatrix of 
$\Lo_{\bar\Theta_{p-1}^1}^{\bar\Xi_{p-1}^1}$; we denote the corresponding injection $\rho$. Clearly,
\begin{equation}\label{interim3}
\Ld^{\bar L_{t-1}^2}_{\Phi_{t-1}^2}\nald^{\bar K_{t-1}^2}_{\bar L_{t-1}^2}=
\lgradd^{\bar K_{t-1}^2}_{\Phi_{t-1}^2}- \Ld^{\bar\Xi_{t-1}^2}_{\Phi_{t-1}^2}\nald^{\bar K_{t-1}^2}_{\bar\Xi_{t-1}^2},
\end{equation}
which is the transformed version of \eqref{interim}.

The contribution of the first term in \eqref{interim3} to the second term in \eqref{bmic3} equals
\[
-\left\langle\lgrado^{\Phi_{p-1}^1}_{\bar K_{p-1}^1}\lgradd^{\bar K_{t-1}^2}_{\Phi_{t-1}^2}\right\rangle=
-\left\langle\lgrado^{\Phi_{p-1}^1}_{\Phi_{p-1}^1}\lgradd_{\Phi_{t-1}^2}^{\Phi_{t-1}^2}\right\rangle
\]
and cancels the contribution of the first term in \eqref{bmic3} computed above. 

To find the contribution of the second term in \eqref{interim3} to the second term in \eqref{bmic3}
note that
\begin{equation*}
\lgrado^{\Phi_{p-1}^1}_{\bar K_{p-1}^1}=\Lo^{\bar L_{p-1}^1\cup\bar\Xi_{p-1}^1}_{\bar K_{p-1}^1}
\nalo^{\Phi_{p-1}^1}_{\bar L_{p-1}^1\cup\bar\Xi_{p-1}^1},
\end{equation*}
which is the transformed version of \eqref{superdec}, 
so the contribution in question equals
\begin{equation*}
\left\langle\Ld^{\bar\Xi_{t-1}^2}_{\Phi_{t-1}^2}\nald^{\bar K_{t-1}^2}_{\bar\Xi_{t-1}^2}
\Lo^{\bar L_{p-1}^1\cup\bar\Xi_{p-1}^1}_{\bar K_{p-1}^1}
\nalo^{\Phi_{p-1}^1}_{\bar L_{p-1}^1\cup\bar\Xi_{p-1}^1}\right\rangle;
\end{equation*}
the latter expression is the transformed version of \eqref{crosscases}.
Taking into account that $\Ld^{\bar\Xi_{t-1}^2}_{\Phi_{t-1}^2}=\Lo^{\rho(\bar\Xi_{t-1}^2)}_{\Phi_{p-1}^1}$, 
$\Ld^{\bar\Xi_{t-1}^2}_{\bar\Theta_{t-1}^2\setminus\Phi_{t-1}^2}=
\Lo^{\rho(\bar\Xi_{t-1}^2)}_{\bar\Theta_{p-1}^1\setminus\Phi_{p-1}^1}$
and that
\begin{equation*}
\nalo^{\Phi_{p-1}^1}_{\bar L_{p-1}^1\cup\bar\Xi_{p-1}^1}\Lo^{\rho(\bar\Xi_{t-1}^2)}_{\Phi_{p-1}^1}=
 \gradlo^{\rho(\bar\Xi_{t-1}^2)}_{\bar L_{p-1}^1\cup\bar\Xi_{p-1}^1}-
\nalo_{\bar L_{p-1}^1\cup\bar\Xi_{p-1}^1}^{\bar\Theta_{p-1}^1\setminus\Phi_{p-1}^1}
\Lo^{\rho(\bar\Xi_{t-1}^2)}_{\bar\Theta_{p-1}^1\setminus\Phi_{p-1}^1},
\end{equation*}
which is the transformed version of \eqref{twoterm}, 
this contribution can be rewritten as
\begin{multline*}
\left\langle \Lo^{\bar L_{p-1}^1\cup\bar\Xi_{p-1}^1}_{\bar K_{p-1}^1}
\gradlo^{\rho(\bar\Xi_{t-1}^2)}_{\bar L_{p-1}^1\cup\bar\Xi_{p-1}^1}
\nald^{\bar K_{t-1}^2}_{\bar\Xi_{t-1}^2}\right\rangle\\
-\left\langle \Ld^{\bar\Xi_{t-1}^2}_{\bar\Theta_{t-1}^2\setminus\Phi_{t-1}^2}
\nald^{\bar K_{t-1}^2}_{\bar\Xi_{t-1}^2}
\Lo^{\bar L_{p-1}^1\cup\bar\Xi_{p-1}^1}_{\bar K_{p-1}^1}
\nalo_{\bar L_{p-1}^1\cup\bar\Xi_{p-1}^1}^{\bar\Theta_{p-1}^1\setminus\Phi_{p-1}^1}
\right\rangle.
\end{multline*}
Next, by\eqref{lnal}, 
\[
\Ld^{\bar\Xi_{t-1}^2}_{\bar\Theta_{t-1}^2\setminus\Phi_{t-1}^2}
\nald^{\bar K_{t-1}^2}_{\bar\Xi_{t-1}^2}=
\lgradd^{\bar K_{t-1}^2}_{\bar\Theta_{t-1}^2\setminus\Phi_{t-1}^2}=0,
\]
since the rows $\bar K_{t-1}^2$ lie above $\bar\Theta_{t-1}^2\setminus\Phi_{t-1}^2$.

Finally, by \eqref{lnal},
\[
\gradlo^{\rho(\bar\Xi_{t-1}^2)}_{\bar L_{p-1}^1\cup\bar\Xi_{p-1}^1}=
\begin{bmatrix} 0\\ \one \\ 0\end{bmatrix},
\]
where the unit block occupies the rows and the columns $\rho(\bar\Xi_{t-1}^2)$. Therefore, the remaining contribution equals
\[
\left\langle \Lo^{\rho(\bar\Xi_{t-1}^2)}_{\bar K_{p-1}^1}
\nald^{\bar K_{t-1}^2}_{\bar\Xi_{t-1}^2}
\right\rangle=
\left\langle \Ld_{\bar\Xi_{t-1}^2}^{\bar K_{t-1}^2}\nald_{\bar K_{t-1}^2}^{\bar\Xi_{t-1}^2}\right\rangle=
\left\langle \Ld_{L_{t-1}^2}^{\bar K_{t-1}^2}\nald_{\bar K_{t-1}^2}^{L_{t-1}^2}\right\rangle,
\]
which is a constant via Lemma \ref{partrace} an yields the third term in the statement of the lemma.

\end{proof}

\section{The quiver}\label{sec:quiver}

The goal of this Section is the proof of Theorem~\ref{quiver}. 

\subsection{Preliminary considerations} 
Consider an arbitrary ordering on the set of vertices of the quiver $Q_{\bfGr,\bfGc}$ in which all mutable vertices 
precede all frozen vertices. Let $B_{\bfGr,\bfGc}$ be the exchange matrix that encodes $Q_{\bfGr,\bfGc}$ under this
ordering, and let $\Omega_{\bfGr,\bfGc}$ be the (skew-symmetric) matrix of the constants $\{\log f^1, \log f^2\}$, 
$f^1, f^2\in F_{\bfGr,\bfGc}$, provided $F_{\bfGr,\bfGc}$ has the same ordering. Then by \cite[Theorem 4.5]{GSVb}, 
to prove Theorem~\ref{quiver} it suffices to check that
\begin{equation*}
B_{\bfGr,\bfGc}\Omega_{\bfGr,\bfGc}=\begin{bmatrix}\lambda \one & 0\end{bmatrix}
\end{equation*}
for some $\lambda\ne0$. In more detail, denote $\omega^{\ii\jj}_{rs}=\{\log f_{rs}, \log f_{\ii\jj}\}$,  
then the above equation can be rewritten as
\begin{equation}\label{bomega}
\sum_{(i,j)\to(r,s)}\omega^{\ii\jj}_{rs}-
\sum_{(r,s)\to(i,j)}\omega^{\ii\jj}_{rs}=
\begin{cases} 
\lambda\ \quad&\text{for $(\ii,\jj)=(i,j)$,}\\
0\ \quad & \text{otherwise}
\end{cases}							           
\end{equation}
for all pairs $(i,j), (\ii,\jj)$ such that $f_{ij}$ is not frozen. By the definition of the quiver $Q_{\bfGr,\bfGc}$
(see Section \ref{thequiver}), a non-frozen vertex can have degree six, five, four, or three. 
Consider first the case of degree six. All possible neighborhoods of a vertex in this case are shown in 
Fig.~\ref{fig:ijnei}, Fig.~\ref{fig:1jnei}(a), Fig.~\ref{fig:i1nei}(a), Fig.~\ref{fig:njnei}(a), and 
Fig.~\ref{fig:innei}(a).    

Consequently, the left hand side of \eqref{bomega} for $1<i,j<n$ can be rewritten as
\begin{multline}\label{neisum}
(\omega^{\ii\jj}_{i-1,j}-\omega^{\ii\jj}_{i,j+1})-(\omega^{\ii\jj}_{i-1,j-1}-\omega^{\ii\jj}_{i,j})-
(\omega^{\ii\jj}_{i,j}-\omega^{\ii\jj}_{i+1,j+1})+(\omega^{\ii\jj}_{i,j-1}-\omega^{\ii\jj}_{i+1,j})\\
=\delta_{ij}^1-\delta_{ij}^2-\delta_{ij}^3+\delta_{ij}^4,
\end{multline}
see Fig.~\ref{fig:ijnei}.
In other words, the neighborhood of $(i,j)$ is covered by the union of four pairs of vertices, and the contribution 
$\delta_{ij}^k$ of each pair is the difference of the corresponding values of $\omega$. More exactly, 
the first pair consists of the vertices to the north and to the east of $(i,j)$, the second pair 
consists of the vertex to the north-west of $(i,j)$ and of $(i,j)$ itself, the third pair consists of 
$(i,j)$ itself and of the vertex to the south-east of $(i,j)$, and the fourth pair consists of the vertices
to the west and to the south of $(i,j)$.  

It is easy to see that in all other cases of degree six, 
the left hand side of \eqref{bomega} can be rewritten in a similar way.
For example, for $i=1$, 
an analog of \eqref{neisum} holds with 
$\delta_{1j}^1=\omega^{\ii\jj}_{n,\gammac^*(j-1)+1}-\omega^{\ii\jj}_{1,j+1}$ and
$\delta_{1j}^2=\omega^{\ii\jj}_{n,\gammac^*(j-1)}-\omega^{\ii\jj}_{1j}$, see Fig.~\ref{fig:1jnei}(a).

Further, consider the case of degree five. All possible neighborhoods of a vertex in this case are shown in  
Fig.~\ref{fig:1jnei}(b), Fig.~\ref{fig:i1nei}(b), Fig.~\ref{fig:njnei}(b,c), 
Fig.~\ref{fig:innei}(b,c), Fig.~\ref{fig:1nnei}(a), Fig.~\ref{fig:n1nei}(a), and Fig.~\ref{fig:nnnei}(a). 
Direct inspection of all these cases shows that the lower vertex is missing either in the first pair 
(Fig.~\ref{fig:1jnei}(b),  Fig.~\ref{fig:innei}(c), and Fig.~\ref{fig:1nnei}(a)), or in the third pair
(Fig.~\ref{fig:njnei}(b),  Fig.~\ref{fig:innei}(b), and Fig.~\ref{fig:nnnei}(a)), or in the fourth pair
Fig.~\ref{fig:i1nei}(b),  Fig.~\ref{fig:njnei}(c), and Fig.~\ref{fig:n1nei}(a)). In all these cases the remaining function
in a deficient pair is a minor of size one, and hence all the above relations will remain valid if the missing function in 
the deficient pair is replaced by $f=1$ (understood as a minor of size zero).
 
Similarly, in the case of degree four there are two deficient pairs (any two of the pairs 1, 3, and 4), and in the case of degree three, all three pairs are deficient. However, adding at most three dummy functions $f=1$ as explained above, 
we can always rewrite \eqref{bomega} as
\begin{equation}\label{normneisum}
\Delta_{ij}=\delta_{ij}^1-\delta_{ij}^2-\delta_{ij}^3+\delta_{ij}^4=
\begin{cases} 
\lambda\ \quad&\text{for $(\ii,\jj)=(i,j)$}\\
0\ \ \quad& \text{otherwise.}
\end{cases}
\end{equation} 

Equation \eqref{normneisum} can be obtained as the restriction to the diagonal $X=Y$ of a similar equation in the double. Namely, assume that $\ii\ne\jj$, $r\ne s$, and put 
$\ttw_{rs}^{\ii\jj}= \{\log \ttf_{rs}, \log \ttf_{\ii\jj}\}^D$. If additionally $1<i,j<n$ and $i\ne j, j\pm 1$, we define
\begin{align*}
\ttd_{ij}^1&=\ttw^{\ii\jj}_{i-1,j}-\ttw^{\ii\jj}_{i,j+1}, \qquad 
\ttd_{ij}^2=\ttw^{\ii\jj}_{i-1,j-1}-\ttw^{\ii\jj}_{i,j},\\
\ttd_{ij}^3&=\ttw^{\ii\jj}_{i,j}-\ttw^{\ii\jj}_{i+1,j+1}, \qquad
\ttd_{ij}^4=\ttw^{\ii\jj}_{i,j-1}-\ttw^{\ii\jj}_{i+1,j}.
\end{align*}
If $i$ or $j$ equals $1$ or $n$, the above definition of $\ttd_{ij}^k$ should be modified similarly to the modification of
$\delta_{ij}^k$ explained above. 
It follows immediately from \eqref{f_ij_gen}, \eqref{twof_ii} that each $\ttd_{ij}^k$ is a difference 
$\{\log \ttf_{i^kj^k},\log \ttf_{\ii\jj}\}^D-\{\log \tilde \ttf_{i^kj^k},\log \ttf_{\ii\jj}\}^D$, where 
$\ttf_{i^kj^k}$ and $\tilde \ttf_{i^kj^k}$ are two trailing minors of the same matrix that differ 
in size by one. For example, for $i=1$ we get $\ttf_{i^1j^1}=\ttf_{n,\gammac^*(j-1)+1}$, $\ttf_{i^2j^2}=
\ttf_{n,\gammac^*(j-1)}$, $\ttf_{i^3j^3}=\ttf_{1j}$, and $\ttf_{i^4j^4}=\ttf_{1,j-1}$.  
We say that $\ttd_{ij}^k$ is of
$X$-{\it type\/} if the leading block of $\ttf_{i^kj^k}$ is an $X$-block, 
and of $Y$-{\it type\/} otherwise. 

If $i=j+1$ then we set $\ttf_{i^1j^1}=\ttf_{i-1,j}^<$.  Consequently, in this case all four 
$\ttd_{ij}^k$ are of $X$-type. Similarly, if $i=j-1$ then we set $\ttf_{i^4j^4}=\ttf_{i,j-1}^>$. 
Consequently, in this case all four $\ttd_{ij}^k$ are of $Y$-type. In what follows we will use the above conventions 
without indicating that explicitly. 

For $i\ne j$
equation \eqref{normneisum} is the restriction to the diagonal  $X=Y$ of the equation 
\begin{equation}\label{dubneisum}
\ttD_{ij}=\ttd_{ij}^1-\ttd_{ij}^2-\ttd_{ij}^3+\ttd_{ij}^4=
\begin{cases} 
\lambda\ \quad&\text{for $(\ii,\jj)=(i,j)$,}\\
0\ \ \quad&\text{otherwise}
\end{cases}
\end{equation} 
in the Drinfeld double. Note that all the quantities involved in the above equation are defined unambiguously.

The case $i=j$ requires a more delicate treatment. It is impossible to fix a choice of $\ttf_{i^2j^2}$ and $\ttf_{i^3j^3}$ in such a way
that \eqref{dubneisum} is satisfied. Consequently, to get \eqref{normneisum}, we treat each contribution to $\ttD_{ij}$ computed 
in Section \ref{sec:basis} separately, and restrict it to the diagonal $X=Y$. The obtained restrictions are combined in a proper way
to get $\Delta_{ij}$ and to prove \eqref{normneisum} directly. In more detail, we either set
$\ttf_{i^2j^2}=\ttf_{i-1,j-1}^<$ and $\ttf_{i^3j^3}=\ttf_{ij}^>$, or $\ttf_{i^2j^2}=\ttf_{i-1,j-1}^>$ and $\ttf_{i^3j^3}=\ttf_{ij}^<$. 
In the former case $\ttd_{ij}^2$ and $\ttd_{ij}^4$ 
are of $X$-type and $\ttd_{ij}^1$ and $\ttd_{ij}^3$ are of $Y$-type, while in the latter case 
$\ttd_{ij}^3$ and $\ttd_{ij}^4$ are of $X$-type and $\ttd_{ij}^1$ and $\ttd_{ij}^2$ are of $Y$-type. Note that in both cases the restriction
to the diagonal yields the same pair of functions.

Similarly, in the case  $\ii=\jj$ we set either $\ttf^2=\ttf_{\ii\jj}^<$ or $\ttf^2=\ttf_{\ii\jj}^>$, depending on the choice of the 
corresponding $\ttf^1$, so that $\ttf^1$ and $\ttf^2$ have the same type.

\subsection{Diagonal contributions} \label{trivicon}
Recall that the bracket in the double is computed via equation \eqref{bra}. 
In this section we find the contribution of the first five terms in \eqref{bra} to $\ttD_{ij}$.

\begin{proposition}\label{ftcon}
The contribution of the first term in \eqref{bra} to $\ttD_{ij}$ vanishes.
\end{proposition}

\begin{proof} Similarly to operators $E_L$ and $E_R$ defined in section \ref{sec:bra}, define operators $\bar E_L$ and 
$\bar E_R$ via $\bar E_L=\nabla_X X-\nabla_Y Y$ and $\bar E_R=X\nabla_X-Y\nabla_Y$.

Note that by \eqref{R0L12}, \eqref{R0L}, the first term in \eqref{bra} can be rewritten as
\begin{multline}\label{ft}
\left\langle R_0^{\ec}(E_L^1),E_L^2\right\rangle=\left\langle \left(\xi_L^1\right)_0, A_L^2\right\rangle+
\left\langle \left(\eta_L^1\right)_0, B_L^2\right\rangle+\Tr (E_L^1)\cdot p_L^2\\
+\Tr\left(\frac1{1-\gammac^*}\eta_L^1\right)\cdot q_L^2-\Tr\left(\frac1{1-\gammac}\xi_L^1\right)\cdot q_L^2-
\Tr(\bar E_L^1)\cdot q_L^2,
\end{multline}
where $A_L^2$ and $B_L^2$ are matrices depending only on $\ttf^2$ and $p_L^2$ and $q_L^2$ are functions depending only on 
$\ttf^2$.

\begin{lemma}\label{econ}
The contribution of the third term in \eqref{ft} to any one of $\ttd_{ij}^k$, $1\le k\le 4$, equals~$p_L^2$.
\end{lemma}

\begin{proof} For any $\ttf$,
\[
\Tr(E_L \log\ttf)=\frac1{\ttf}\sum_{i,j=1}^n\left(\frac{\partial\ttf}{\partial x_{ij}} x_{ij}
+\frac{\partial\ttf}{\partial y_{ij}}y_{ij}\right)=\left.\frac{d}{dt}\right|_{t=1}\log\ttf(tX,tY).
\]
If $\ttf$ is a homogeneous polynomial, then the above expression equals its total degree. 
Recall that $\ttf_{i^kj^k}$ satisfies this condition, and that 
$\deg \ttf_{i^kj^k}-\deg\tilde\ttf_{i^kj^k}=1$.
\end{proof}

\begin{lemma}\label{barecon}
The contribution of the sixth term in \eqref{ft} to any one of $\ttd_{ij}^k$, $1\le k\le 4$, equals~$q_L^2$ if $\ttd_{ij}^k$ is 
of $X$-type and $-q_L^2$ otherwise.
\end{lemma}

\begin{proof} For any $\ttf$,
\[
\Tr(\bar E_L \log\ttf)=\frac1{\ttf}\sum_{i,j=1}^n\left(\frac{\partial\ttf}{\partial x_{ij}} x_{ij}
-\frac{\partial\ttf}{\partial y_{ij}}y_{ij}\right)=\left.\frac{d}{dt}\right|_{t=1}\log\ttf(tX,t^{-1}Y).
\]
If $\ttf$ is a homogeneous polynomial both in $x$-variables and in $y$-variables, then the above expression equals 
$\deg_x\ttf-\deg_y\ttf$. Recall that $\ttf_{i^kj^k}$ satisfies this condition and that 
$\deg_x\ttf_{i^kj^k}-\deg_x\tilde\ttf_{i^kj^k}$ equals~$1$ if $\ttf_{i^kj^k}$ is of $X$-type and~$0$
if it is of $Y$-type, while $\deg_y\ttf_{i^kj^k}-\deg_y\tilde\ttf_{i^kj^k}$ equals~$0$ if 
$\ttf_{i^kj^k}$ is of $X$-type and~$1$ if it is of $Y$-type.
\end{proof}

Recall that every point of a nontrivial $X$-run except for the last point belongs to $\Gamma_1$. We denote by 
$\cGamma_1$ the union of all nontrivial $X$-runs, and by $\cgamma$ the extension of $\gamma$ that 
takes the last point of a nontrivial $X$-run $\Delta$ to the last point of $\gamma(\Delta)$. In a similar way we define
$\cGamma_2$ and $\cgamma^*$.

\begin{lemma}\label{xietacon}
{\rm (i)} The contribution of the first term in \eqref{ft} to $\ttd_{ij}^3$ equals $(A_L^2)_{jj}$ if $\ttd_{ij}^3$
is of $Y$-type, $(A_L^2)_{\cgammac(j)\cgammac(j)}-|\Delta(j)|^{-1}\sum_{k\in\Delta(j)}(A_L^2)_{kk}$ 
if $\ttd_{ij}^3$ is of $X$-type and $j\in\cGamma_1^\ec$, and $0$ otherwise.

{\rm (ii)} The contribution of the second term in \eqref{ft} to $\ttd_{ij}^3$ equals $(B_L^2)_{jj}$ if $\ttd_{ij}^3$
is of $X$-type, $(B_L^2)_{\cgamma^{\ec*}(j)\cgamma^{\ec*}(j)}-|\bar\Delta(j)|^{-1}\sum_{k\in\bar\Delta(j)}(B_L^2)_{kk}$ 
if  $\ttd_{ij}^3$ is of $Y$-type and $j\in\cGamma_2^\ec$, and $0$ otherwise.
\end{lemma}

\begin{proof}
(i) Define an $n\times n$ matrix $J_m(t)$ as the identity matrix with the entry $(m,m)$ replaced by $t$, and set 
$X_m(t)=X J_m(t)$, 
$Y_m(t)=Y J_m(t)$. By the definition of $\cxi_L$, for any $\ttf$ one has
\begin{align*}
(\cxi_L \log\ttf)_{ll}=&\frac1\ttf\sum_{i=1}^n\frac{\partial\ttf}{\partial x_{i\cgamma^{\ec*}(l)}}x_{i\cgamma^{\ec*}(l)}+
\frac1\ttf\sum_{i=1}^n\frac{\partial\ttf}{\partial y_{il}}y_{il}\\
=&\left.\frac{d}{dt}\right|_{t=1}\log\ttf(X_{{\cgamma}^{\ec*}(l)}(t),Y_l(t)).
\end{align*}
If $\ttf$ is a minor of a matrix $\L\in\bL\cup\{X,Y\}$, then the above expression equals the total number of columns $l$ 
in all column $Y$-blocks 
involved in this minor plus the total number of columns $\cgamma{^\ec*}(l)$ in all column $X$-blocks involved in this minor 
(note that $l\ne \cgamma^{\ec*}(l)$, and hence all such columns are different). Recall
that the minors $\ttf_{i^3j^3}=\ttf_{ij}$ and $\tilde\ttf_{i^3j^3}$ differ in size by one, and that the column missing 
in the latter minor is $j$. Consequently, if $\ttd_{ij}^3$ is of $Y$-type,
$(\cxi_L \log\ttf_{i^3j^3})_{ll}- (\cxi_L \log\tilde\ttf_{i^3j^3})_{ll}$ equals~$1$ if $l=j$, which yields $(A_L^2)_{jj}$,
and vanishes otherwise. Similarly, if $\ttd_{ij}^3$ is of $X$-type, this difference equals~$1$ if $j\in\cGamma_1^\ec$ 
and $l=\cgammac(j)$, which yields $(A_L^2)_{\cgammac(j)\cgammac(j)}$, and vanishes otherwise. Finally, the additional term
$-|\Delta(j)|^{-1}\sum_{k\in\Delta(j)}(A_L^2)_{kk}$ stems from the difference between $(\xi_L\log\ttf)_0$ and 
$(\cxi_L\log\ttf)_0$, see Section \ref{simplega}.

(ii) The proof is similar to the proof of (i).
\end{proof}

To prove Proposition~\ref{ftcon}, consider the contributions of the terms in the right hand side of \eqref{ft} to $\ttD_{ij}$.

Let us prove that the contributions of the first term to $\ttd_{ij}^1$ and $\ttd_{ij}^3$ cancel each other, as well as the
contributions to $\ttd_{ij}^2$ and $\ttd_{ij}^4$. 
Assume first that $1<i<j\le n$. Clearly, in this case all $\ttd_{ij}^k$ are of $Y$-type, and 
\begin{equation}\label{ttdiden}
\ttd_{ij}^1=\ttd_{i-1,j}^3,\qquad  
\ttd_{ij}^2=\ttd_{i-1,j-1}^3, \qquad\ttd_{ij}^4=\ttd_{i,j-1}^3. 
\end{equation}
Hence by Lemma \ref{xietacon}(i),  the sought for cancellations hold true, consequently, the contribution of the 
first term in \eqref{ft} to $\ttD_{ij}$ vanishes.

Assume next that $1<j<i\le n$. In this case all $\ttd_{ij}^k$ are of $X$-type, and \eqref{ttdiden} holds. Hence by Lemma \ref{xietacon}(i),  
the contribution of the first term in \eqref{ft} to $\ttD_{ij}$ vanishes, similarly to the previous case.

The next case is $1<i=j\le n$. In this case we choose $\ttf_{i^2j^2}$ and $\ttf_{i^3j^3}$ in such a way that 
$\ttd_{ij}^1$ and $\ttd_{ij}^3$ are of $Y$-type and  $\ttd_{ij}^2$ and $\ttd_{ij}^4$ are of $X$-type, and \eqref{ttdiden} holds, so the contribution of the first term in \eqref{ft} to $\ttD_{ij}$ vanishes once again. 

Assume now that $1=i<j\le n$. In this case $\ttd_{1j}^1$ and $\ttd_{1j}^2$ are of $X$-type and  $\ttd_{1j}^3$ and $\ttd_{1j}^4$ are of 
$Y$-type. Relations \eqref{ttdiden} are replaced by
\[
 \ttd_{1j}^1=\ttd_{nl}^3,\qquad  
\ttd_{1j}^2=\ttd_{n,l-1}^3, \qquad\ttd_{1j}^4=\ttd_{1,j-1}^3,
\]
where $\gammac(l-1)=j-1$, see Section \ref{thequiver}, and in particular, Fig.~\ref{fig:1jnei}. Consequently, 
$\cgammac(l-1)=j-1$ and $\cgammac(l)=j$, and hence by Lemma \ref{xietacon}(i), the sought for cancellations hold true.

Finally, assume that $1=j<i\le n$. In this case $\ttd_{i1}^1$ and $\ttd_{i1}^3$ are of $X$-type and  $\ttd_{i1}^2$ and $\ttd_{i1}^4$ are of 
$Y$-type.  Relations \eqref{ttdiden} are replaced by
\[
 \ttd_{i1}^1=\ttd_{i-1,1}^3,\qquad  
\ttd_{i1}^2=\ttd_{l-1,n}^3, \qquad\ttd_{i1}^4=\ttd_{ln}^3,
\]
where $\gammar(i-1)=l-1$, see Section \ref{thequiver}, and in particular, Fig.~\ref{fig:i1nei}. Consequently,  by Lemma \ref{xietacon}(i), 
the sought for cancellations hold true.

To treat the second term in \eqref{ft} we reason exactly in the same way and use Lemma \ref{xietacon}(ii) instead.

The third term in \eqref{ft} is treated trivially with the help of Lemma \ref{econ}. 

Cancellations for the fourth term follow from the cancellations for the second term established above and the fact that 
$\frac1{1-\gammac^*}$ 
is a linear operator. Similarly, cancellations for the fifths term follow from the cancellations for the first term established above and the fact that $\frac1{1-\gammac}$ is a linear operator.

Finally, the sixth term is treated similarly to the first one based on Lemma \ref{barecon}.
\end{proof}

\begin{proposition}\label{stcon}
The contribution of the second term in \eqref{bra} to $\ttD_{ij}$ vanishes.
\end{proposition}

\begin{proof} The proof of this proposition is similar to the proof of Proposition \ref{ftcon} and is based on analogs of
Lemmas \ref{econ}--\ref{xietacon}. Note that the analog of Lemma \ref{xietacon} claims that contributions of $(\xi_R^1)_0$ and
$(\eta_R^1)_0$ to $\ttD_{ij}$ depend on $i$, $\cgammar(i)$, and $\cgamma^{\er*}(i)$. In the treatment of the case
$1<i=j\le n$ we choose $\ttf_{i^2j^2}$ and $\ttf_{i^3j^3}$ in such a way that 
$\ttd_{ij}^1$ and $\ttd_{ij}^2$ are of $Y$-type and  $\ttd_{ij}^3$ and $\ttd_{ij}^4$ are of $X$-type.
\end{proof}

\begin{proposition}\label{tffcon}
The contributions of the third, fourth, and fifth term in \eqref{bra} to $\ttD_{ij}$ vanish.
\end{proposition}

\begin{proof} The claim for the third term essentially coincides with the similar claim for the first term in \eqref{ft}, the claim for the fourth term essentially coincides with the similar claim for the second term in \eqref{ft}, and the claim for the fifth term uses additionally 
the fact that $\Pi_{\hat\Gamma_1^\ec}$ is a linear operator.
\end{proof}

\subsection{Non-diagonal contributions} \label{nontrivicon}
In this section we find the contributions of the four remaining terms 
in \eqref{bra} to $\ttD_{ij}$. More exactly, we will be dealing with the contributions of the corresponding ringed
versions. The contribution of the difference between the ordinary and the ringed version to $\ttD_{ij}$ vanishes 
similarly to the contributions treated in the previous section.

\subsubsection{Case $1<j<i<n$}\label{typicase} 
In this case all seven functions $\ttf_{i^kj^k}$, $\tilde\ttf_{i^kj^k}$ satisfy 
the conditions of Case 1 in Section \ref{case1}. 
Consequently, the leading block of $\ttf_{i^1j^1}=\ttf_{i-1,j}$ and $\tilde\ttf_{i^1j^1}=\ttf_{i,j+1}$ is $X_I^J$, the leading block 
of  $\ttf_{i^2j^2}=\ttf_{i-1,j-1}$, $\tilde\ttf_{i^2j^2}=\ttf_{i^3j^3}=\ttf_{ij}$, and $\tilde\ttf_{i^3j^3}=\ttf_{i+1,j+1}$ is 
$X_{I'}^{J'}$, and the leading block of $\ttf_{i^4j^4}=\ttf_{i,j-1}$ and $\tilde\ttf_{i^4j^4}=\ttf_{i+1,j}$ is $X_{I''}^{J''}$.

We have to compute the contributions of \eqref{etaleta}, \eqref{etareta}, \eqref{xinay}, and \eqref{xinax}.
Note that the first term in \eqref{etareta} looks exactly the same as terms already treated in Section \ref{trivicon}, and hence its contribution to $\ttD_{ij}$ vanishes. The fourth term in \eqref{etareta} vanishes under the conditions of Case 1, 
since both $\gradlo_{\sigma(\bar L_t^2)}^{\sigma(\bar L_t^2)}$ and $\lgrado_{\sigma(\bar K_t^2)}^{\sigma(\bar K_t^2)}$ vanish.
Next, the contribution of the last term in \eqref{xinay} to any one of $\ttd_{ij}^k$ vanishes, 
since the leading blocks of $\ttf_{i^kj^k}$ and $\tilde\ttf_{i^kj^k}$ coincide. The same holds true for the last term 
in \eqref{xinax}. Further, the contributions of the third term in \eqref{xinay} to $\ttd_{ij}^1$ and to $\ttd_{ij}^3$
coincide, as well as the contributions of this term to $\ttd_{ij}^2$ and to $\ttd_{ij}^4$, since they depend only on $j^k$,
and $j^1=j^3=j$, $j^2=j^4=j-1$. The same holds true for the foutrh term in \eqref{xinay}. Similarly, the contributions of the
fourth term in \eqref{xinax} to $\ttd_{ij}^1$ and to $\ttd_{ij}^2$ coincide, as well as the contributions of this term to 
$\ttd_{ij}^3$ and to $\ttd_{ij}^4$, since they depend only on $i^k$, and $i^1=i^2=i-1$, $i^3=i^4=i$. The same holds true 
for the fifth term in \eqref{xinax}.

The total contribution of all $B$-terms involved in the above formulas is given in  Lemma \ref{zone1lemma}. Note that the contributions of the third, sixth, ninth and tenth terms in Lemma \ref{zone1lemma} to any one of $\ttd_{ij}^k$ vanish, 
since the dependence of all these terms on $\ttf^1$ is only over which blocks the summation goes. The latter fact, in turn, is 
completely defined by the leading block of $\ttf^1$, and the leading blocks of $\ttf_{i^kj^k}$ and $\tilde\ttf_{i^kj^k}$ coincide.

To proceed further assume first that $X_I^J=X_{I'}^{J'}=X_{I''}^{J''}$. Consider the first sum in the third term in 
\eqref{etaleta}. Each block involved in this sum contributes an equal amount to $\ttd_{ij}^1$ and $\ttd_{ij}^2$, as well as to
$\ttd_{ij}^3$ and $\ttd_{ij}^4$, so the total contribution of the block vanishes. Similarly, for the second sum in the third term in 
\eqref{etaleta}, each block involved contributes an equal amount to $\ttd_{ij}^1$ and $\ttd_{ij}^3$, as well as to
$\ttd_{ij}^2$ and $\ttd_{ij}^4$, so the total contribution of the block vanishes as well.

The first, the second, and the fifth term in Lemma \ref{zone1lemma} are treated exactly as the first sum in the third term in 
\eqref{etaleta}, and the fourth term, exactly as the the second sum in the third term in \eqref{etaleta}. Consequently, all these
contributions vanish. We thus see that $\ttD_{ij}=\ttD_{ij}[7]-\ttD_{ij}[8]$, where $\ttD_{ij}[7]$ and $\ttD_{ij}[8]$ are the contributions
of the seventh and the eights terms in Lemma \ref{zone1lemma} to $\ttD_{ij}$.

To treat $\ttD_{ij}[7]$, recall that the sum in the seventh term is taken over the cases when the exit point of 
$X_{I_t^2}^{J_t^2}$ lies above the exit point of $X_{I_p^1}^{J_p^1}$. Consequently, the treatment in the cases when the exit 
point of $\ttf^2$ lies above the exit point of $\ttf_{i^1j^1}$ is again exactly the same as for the first sum in the third term in 
\eqref{etaleta}, and the corresponding contribution vanishes. If the exit point of $\ttf^2$ coincides with the exit point of $\ttf_{i^1j^1}$, that is, if $\hat\imath-\hat\jmath=i-j-1$, one has
\begin{equation}\label{cont7u}
\ttD_{ij}[7]=-\ttd_{ij}^2[7]-\ttd_{ij}^3[7]+\ttd_{ij}^4[7]=
\begin{cases}
-\#^1-1 \quad &\text{for $\hat\imath< i$},\\
-\#^1 \quad &\text{for $\hat\imath\ge i$},
\end{cases}
\end{equation}
where $\#^1$ is the number of non-leading blocks of $\ttf^2$ satisfying the corresponding conditions. If the exit point 
of $\ttf^2$ coincides with the exit point of $\ttf_{i^2j^2}$, that is, if $\hat\imath-\hat\jmath=i-j$,
one has
\begin{equation*}\label{cont7m}
\ttD_{ij}[7]=\ttd_{ij}^4[7]=
\begin{cases}
\#^2+1 \quad &\text{for $\hat\imath\le i$},\\
\#^2 \quad &\text{for $\hat\imath>i$},
\end{cases}
\end{equation*}
where $\#^2$ is the number of non-leading blocks of $\ttf^2$ satisfying the corresponding conditions. The cases when the exit point of
$\ttf^2$ lies below the exit point of $\ttf_{i^2j^2}$ do not contribute to $\ttD_{ij}[7]$.

Similarly, the treatment of $\ttD_{ij}[8]$ in the cases when the exit 
point of $\ttf^2$ lies above the exit point of $\ttf_{i^1j^1}$ is exactly the same as for the second sum in the third term in
\eqref{etaleta}, and the corresponding contribution vanishes. If the exit point of $\ttf^2$ coincides with the exit point of 
$\ttf_{i^1j^1}$, one has
\begin{equation}\label{cont8u}
\ttD_{ij}[8]=-\ttd_{ij}^2[8]-\ttd_{ij}^3[8]+\ttd_{ij}^4[8]=
\begin{cases}
-\#^1-1 \quad &\text{for $\hat\jmath\le j$},\\
-\#^1 \quad &\text{for $\hat\jmath> j$},
\end{cases}
\end{equation}
where $\#^1$ is the same as above. If the exit point of $\ttf^2$ coincides with the exit point of $\ttf_{i^2j^2}$, one has
\begin{equation*}\label{cont8m}
\ttD_{ij}[8]=\ttd_{ij}^4[8]=
\begin{cases}
\#^2+1 \quad &\text{for $\hat\jmath< j$},\\
\#^2 \quad &\text{for $\hat\jmath\ge j$},
\end{cases}
\end{equation*}
where $\#^2$ is the same as above. The cases when the exit point of
$\ttf^2$ lies below the exit point of $\ttf_{i^2j^2}$ do not contribute to $\ttD_{ij}[8]$.

It follows from the above discussion that for $\hat\imath-\hat\jmath=i-j-1$ 
\begin{equation*}
\ttD_{ij}[7]-\ttD_{ij}[8]=\begin{cases}
1\quad &\text{for $\hat\imath\ge i$, $\hat\jmath\le j$},\\ 
-1\quad &\text{for $\hat\imath<i$, $\hat\jmath>j$},\\ 
0\quad &\text{otherwise}. 
\end{cases}
\end{equation*}
Consequently, $\ttD_{ij}$ vanishes everywhere on the line $\hat\imath-\hat\jmath=i-j-1$. 
Further, for $\hat\imath-\hat\jmath=i-j$ one has
\begin{equation*}
\ttD_{ij}[7]-\ttD_{ij}[8]=\begin{cases}
1\quad &\text{for $\hat\imath\le i$, $\hat\jmath\ge j$},\\ 
-1\quad &\text{for $\hat\imath>i$, $\hat\jmath<j$},\\ 
0\quad &\text{otherwise}. 
\end{cases}
\end{equation*}
 Consequently, $\ttD_{ij}$ vanishes everywhere on the line 
$\hat\imath-\hat\jmath=i-j$ except for the point $(\hat\imath,\hat\jmath)=(i,j)$, where it equals one. Therefore, 
for $X_I^J=X_{I'}^{J'}=X_{I''}^{J''}$ relation \eqref{dubneisum} holds with $\lambda=1$.

There are three more possibilities for relations between the blocks $X_I^J$, $X_{I'}^{J'}$, $X_{I''}^{J''}$: 

a) $X_I^J\ne X_{I'}^{J'}=X_{I''}^{J''}$; 

b) $X_I^J=X_{I'}^{J'}\ne X_{I''}^{J''}$; 

c) $X_I^J\ne X_{I'}^{J'}\ne X_{I''}^{J''}$.

To treat each of these three one has to consider correction terms with respect to the basic case $X_I^J=X_{I'}^{J'}=X_{I''}^{J''}$.
We illustrate this treatment for the first of the above possibilities.

By Lemma \ref{compar}, case a) can be further subdivided into three subcases:

a1) $I'=I$, $J'\subsetneq J$;

a2) $I'\subsetneq I$, $J'=J$;

a3) $I'\subsetneq I$, $J'\subsetneq J$.

In case a1) we have the following correction terms. For the third term in \eqref{etaleta}, there are blocks $X_{\tilde I}^{J'}$ that satisfy the summation condition $\beta_t^2<\beta_p^1$ for the pair $\ttf_{i^1j^1}$, 
$\tilde\ttf_{i^1j^1}$ but violate it for the other three pairs. By Lemma \ref{compar}, such blocks are characterized
by conditions $\tilde I \subseteq I$, $\tilde J=J'$.
 Consequently, these blocks produce the correction term
\begin{equation*}
-\sum_{\tilde J=J'} \left\langle\lgrado_{\rho(K_t^2)}^{\rho(K_t^2)}\lgradd_{K_t^2}^{K_t^2}\right\rangle+
\sum_{\tilde J=J'} \left\langle\gradlo_{\rho(L_t^2)}^{\rho(L_t^2)}\gradld_{L_t^2}^{L_t^2}\right\rangle
\end{equation*}
to $\ttd_{ij}^1$.

For the first term in Lemma \ref{zone1lemma}, the correction terms are defined by the same blocks as above except for the block $X_{I'}^{J'}$
itself (because of the additional summation condition $\alpha_t^2>\alpha_p^1$). Consequently, these blocks produce the correction term 
\begin{equation*}
\sum_{\tilde J=J'} \left\langle\lgrado_{\rho(\Phi_t^2)}^{\rho(\Phi_t^2)}\lgradd_{\Phi_t^2}^{\Phi_t^2}\right\rangle-
\sum_{\tilde J=J'\atop \tilde I=I'}\left\langle\lgrado_{\rho(\Phi')}^{\rho(\Phi')}\lgradd_{\Phi'}^{\Phi'}\right\rangle
\end{equation*}
to $\ttd_{ij}^1$, where $\Phi'$ corresponds to the block $X_{I'}^{J'}$.

For the second term in Lemma \ref{zone1lemma}, the block $X_I^J$ violates the summation condition $\beta_t^2\ne\beta_p^1$, 
$\alpha_t^2=\alpha_p^1$ for the pair $\ttf_{i^1j^1}$, $\tilde\ttf_{i^1j^1}$ but satisfies it for the other three pairs. Besides, the block
$X_{I'}^{J'}$ satisfies this condition for the pair $\ttf_{i^1j^1}$, $\tilde\ttf_{i^1j^1}$ but violates it for the other three pairs Consequently, these two blocks produce correction terms
\begin{equation*}
\sum_{\tilde J=J'\atop \tilde I=I'}\left\langle\lgrado_{\rho(\Phi')}^{\rho(\Phi')}\lgradd_{\Phi'}^{\Phi'}\right\rangle
-\sum_{\tilde J=J\atop \tilde I=I}\left\langle\lgrado_{\Phi}^{\Phi}\lgradd_{\Phi}^{\Phi}\right\rangle
\end{equation*}
to $\ttd_{ij}^1$, where $\Phi$ corresponds to the block $X_I^J$.

For the fourth term in Lemma \ref{zone1lemma}, the blocks $X_{\tilde I}^{J'}$ violate the summation condition
$\beta_t^2=\beta_p^1$, $\alpha_t^2\ge \alpha_p^1$  for the pair $\ttf_{i^1j^1}$, $\tilde\ttf_{i^1j^1}$ but satisfy it for 
the other three pairs. Besides, the block $X_I^J$ satisfies this condition  for the pair $\ttf_{i^1j^1}$, $\tilde\ttf_{i^1j^1}$ but violates it for the other three pairs. Consequently, these blocks produce correction terms
\begin{equation*}
-\sum_{\tilde J=J'} \left\langle\gradlo_{\rho(L_t^2)}^{\rho(L_t^2)}\gradld_{L_t^2}^{L_t^2}\right\rangle
+\sum_{\tilde J=J\atop \tilde I=I}\left\langle\gradlo_{L}^{L}\gradld_{L}^{L}\right\rangle
\end{equation*}
to $d_{ij}^1$, where $L$ corresponds to the block $X_I^J$.

Summation conditions in the fifth term in Lemma \ref{zone1lemma} are exactly the same as in the fourth term. Consequently, one
gets correction terms
\begin{equation*}
\sum_{\tilde J=J'} \left\langle\lgrado_{\rho(K_t^2\setminus\Phi_t^2)}^{\rho(K_t^2\setminus\Phi_t^2)}
\lgradd_{K_t^2\setminus\Phi_t^2}^{K_t^2\setminus\Phi_t^2}\right\rangle-
\sum_{\tilde J=J\atop \tilde I=I}\left\langle\lgrado_{K\setminus\Phi}^{K\setminus\Phi}\lgradd_{K\setminus\Phi}^{K\setminus\Phi}\right\rangle
\end{equation*}
to $\ttd_{ij}^1$, where $K$ corresponds to the block $X_I^J$.

For the seventh term in Lemma \ref{zone1lemma}, the block $X_I^J$ satisfies the summation condition 
$\beta_t^2=\beta_p^1$, $\alpha_t^2=\alpha_p^1$ for the pair $\ttf_{i^1j^1}$, $\tilde\ttf_{i^1j^1}$ but violates it for the other three pairs.
Besides, the additional condition on the exit points excludes the diagonal $\hat\imath-\hat\jmath=i-j-1$. Consequently, this block
produces correction terms
\begin{equation*}
\sum_{\tilde J=J\atop \tilde I=I}\left\langle\lgrado_{\Phi}^{\Phi}\lgradd_{\Phi}^{\Phi}\right\rangle
+\ttD_{ij}[7]
\end{equation*}
to $\ttd_{ij}^1$, where $\ttD_{ij}[7]$ is given by \eqref{cont7u}.

For the eights term in Lemma \ref{zone1lemma}, the situation is exactly the same as for the seventh term. Consequently, one gets correction terms 
\begin{equation*}
-\sum_{\tilde J=J\atop \tilde I=I}\left\langle\gradlo_{L}^{L}\gradld_{L}^{L}\right\rangle-\ttD_{ij}[8]
\end{equation*}
to $\ttd_{ij}^1$, where $\ttD_{ij}[8]$ is given by \eqref{cont8u}.

It is easy to note that the correction terms listed above cancel one another (recall that vanishing of $D_{ij}[7]-D_{ij}[8]$ for
$\hat\imath-\hat\jmath=i-j-1$ was already proved above), and hence relation \eqref{dubneisum} is established in the case a1).
Cases a2), a3), b), and c) are treated in a similar manner.

\subsubsection{Other cases} The case $1<i<j<n$ is treated in a similar way with \eqref{etaleta} replaced by \eqref{etareta} and
Lemma \ref{zone1lemma} replaced by Lemma \ref{zone2lemma}. 

Consider the case $1<i=j<n$. The treatment of the first term in \eqref{etareta}, the last terms in \eqref{xinay} 
and \eqref{xinax}, the third, sixth, ninth and tenth terms in Lemma \ref{zone1lemma}, and  the third and the sixth terms 
in Lemma \ref{zone2lemma} is exactly the same as in the previous section. The third and the fourth terms in \eqref{xinay}, 
as well as the fourth and the fifth terms in \eqref{xinax}, are treated almost in the same way as in the previous section; 
the only difference is an appropriate choice 
of the functions on the diagonal, which ensures required cancellations. To treat all the other contributions, recall that by the definition, the leading block of $\ttf_{ii}^<$ is $X$, and the leading block of $\ttf_{ii}^>$ is $Y$.
Denote by $X_I^J$ the leading block of $\ttf_{i,i-1}$, and by $Y_{\bar I}^{\bar J}$ the leading block of $\ttf_{i-1,i}$. Similarly to
Section \ref{typicase}, there are four possible cases: $X_I^J=X$, $Y_{\bar I}^{\bar J}=Y$; $X_I^J\ne X$, 
$Y_{\bar I}^{\bar J}=Y$; $X_I^J=X$, $Y_{\bar I}^{\bar J}\ne Y$; $X_I^J\ne X$, $Y_{\bar I}^{\bar J}\ne Y$.

Let us consider the first of the above four cases. Contributions of all terms except for the seventh and the eights terms in 
Lemmas \ref{zone1lemma} and \ref{zone2lemma} are treated in the same way as the third and the fourth terms in \eqref{xinay} 
above. For example, to treat the first sum in the third 
term in \eqref{etaleta} we choose $\ttf_{i^2j^2}=\ttf_{i-1,j-1}^<$ and $\ttf_{i^3j^3}=\ttf_{ij}^>$, so that this sum contributes only to 
$\delta_{ij}^2$ and $\delta_{ij}^4$, and the contributions cancel each other. For the remaining four terms, there is a subtlety in the case
$\ii=\jj$. We write $f_{\ii\ii}=\left.\frac12\ttf_{\ii\ii}^<\right|_{X=Y}+\left.\frac12\ttf_{\ii\ii}^>\right|_{X=Y}$ and note that
$X$ is the only block for $\ttf_{\ii\ii}^<$ and $Y$ is the only block for $\ttf_{\ii\ii}^>$. Consequently,  
for $\ttf^2= \frac12\ttf_{\ii\ii}^<$, the terms involved in Lemma \ref{zone1lemma} contribute zero for $\ii \ne i$ and $1/2$ for $\ii= i$, 
while the terms  involved in Lemma \ref{zone2lemma} contribute zero for any $\ii$. 
Similarly, for $\ttf^2= \frac12\ttf_{\ii\ii}^>$, 
the terms  involved in Lemma \ref{zone1lemma} contribute zero for any $\ii$, while
the terms involved in Lemma \ref{zone2lemma} contribute zero for $\ii \ne i$ and $1/2$ for $\ii= i$. Therefore, we get contribution $1$ for
$(i,j)=(\ii,\jj)$, as required. In the remaining three cases one has to consider correction terms, similarly to Section \ref{typicase}.

It remains to consider the cases when $i$ or $j$ are equal to $1$ or $n$. For example, let $1<j<i=n$ and assume that the degree of the vertex $(n,j)$ in $Q_{\bfGr,\bfGc}$ equals~6, see Fig.~\ref{fig:njnei}(a). 
It follows from the description of the quiver in Section \ref{thequiver} that $(n,j-1)$ is a mutable vertex. In this case the functions $\tilde\ttf_{i^3j^3}$ and $\tilde\ttf_{i^4j^4}$ satisfy conditions of Case 2 in Section~\ref{case2}, 
and all other functions satisfy conditions of Case 1 in Section~\ref{case1}. Consequently, the leading block of 
$\ttf_{i^1j^1}=\ttf_{n-1,j}$ and $\tilde\ttf_{i^1j^1}=\ttf_{n,j+1}$ is $X_I^J$, the leading block of  
$\ttf_{i^2j^2}=\ttf_{n-1,j-1}$ and $\tilde\ttf_{i^2j^2}=\ttf_{i^3j^3}=\ttf_{ij}$ is $X_{I'}^{J'}$, the leading block of
$\ttf_{i^4j^4}=\ttf_{n,j-1}$ is $X_{I''}^{J''}$, the leading block of $\tilde\ttf_{i^3j^3}=\ttf_{1,k+1}$
with $k=\gammac(j)$ is $Y_{\bar I}^{\bar J}$, and the leading block of $\tilde\ttf_{i^4j^4}=\ttf_{1k}$ is 
$Y_{\bar I'}^{\bar J'}$. 

The treatment of the last three terms in \eqref{xinay} and the last three terms in \eqref{xinax} remains the same as
in Section \ref{typicase}. To proceed further, assume that $X_I^J=X_{I'}^{J'}=X_{I''}^{J''}$ and 
$Y_{\bar I}^{\bar J}=Y_{\bar I'}^{\bar J'}$. In this case it is more convenient to
replace \eqref{dubneisum} with $\ttD_{ij}=\ttd_{ij}^1-\ttd_{ij}^2+\ttd_{ij}^{43}-\tilde\ttd_{ij}^{43}$, where 
$\ttd_{ij}^{43}=\ttf_{n,j-1}-\ttf_{ij}$ and $\tilde\ttd_{ij}^{43}=\ttf_{1k}-\ttf_{1,k+1}$, so that the first 
three terms in $\ttD_{ij}$ are subject to the rules of Case 1, and the last term to the rules of Case 2.

The contributions of the third, ninth and tenth terms in Lemma \ref{zone1lemma} to any one of $\ttd^1_{ij}$, 
$\ttd^2_{ij}$ and $\ttd^{43}_{ij}$ vanish for the same reason as in Section \ref{typicase}. The same holds true 
for the contribution of the third term in Lemma \ref{zone2lemma} to $\tilde\ttd^{43}_{ij}$. 

The first sum in the third term in \eqref{etaleta} contributes the same amount to $\ttd_{ij}^1$ and $\ttd_{ij}^2$, and 
zero to $\ttd_{ij}^{43}$. The same holds true for the first, second and the fifth terms in Lemma \ref{zone1lemma}. The second 
sum in the third term in \eqref{etaleta} vanishes since $\rho(L_t^2)$ for every $X$-block of $\ttf^2$ such that
$\beta_t^2<\beta_p^1$ lies strictly to the left of the column $j-1$. 

Further, 
$\lgrado_{\sigma(\bar K_t^2)}^{\sigma(\bar K_t^2)}$ in the second sum in the fourth term of \eqref{etareta} is an identity
matrix, and hence the contribution of this sum to $\tilde\ttd^{43}_{ij}$ vanishes, since both sides in this difference
depend only on $\ttf^2$. The same reasoning works as well for the first, the fourth
and the fifth terms in Lemma \ref{zone2lemma}, and for the first sum in the fourth term of \eqref{etareta} in the case
$\bar\beta_{t-1}^2>\bar\beta_{p-1}^1$. The contribution of this sum to $\tilde\ttd^{43}_{ij}$ for the case
$\bar\beta_{t-1}^2=\bar\beta_{p-1}^1$ cancels the contribution of the second term in Lemma \ref{zone2lemma} for the 
case $\bar\alpha_{t-1}^2<\bar\alpha_{p-1}^1$.

Let us consider now the contribution of the fourth term in Lemma \ref{zone1lemma}. Assume that
a $t$-th $X$-block of $\ttf^2$ satisfies conditions $\alpha_t^2>\alpha_p^1$ and $\beta_t^2=\beta_p^1$. Consequently, the
$(t-1)$-th $Y$-block of $\ttf^2$ satisfies conditions $\bar\alpha_{t-1}^2\ge\bar\alpha_{p-1}^1$ and 
$\bar\beta_{t-1}^2=\bar\beta_{p-1}^1$. Consider first the case when the inequality above is strict. If the $Y$-block
in question is not the leading block of $\ttf^2$, then the contributions of the $X$-block to $\ttd_{ij}^1[4]$ and 
$\ttd_{ij}^2[4]$ cancel each other, whereas the contribution of the $X$-block
to $\ttd_{ij}^{43}[4]$ cancels the contribution of the $Y$-block to $\tilde\ttd^{43}_{ij}[2]$. The same holds true if 
the $Y$-block is the leading block of $\ttf^2$ and $\hat\jmath<\gammac(j)$. If $\hat\jmath=\gammac(j)$ then the contributions of the $X$-block to $\ttd_{ij}^2[4]$ and $\ttd_{ij}^{43}[4]$ vanish, whereas the contribution of the $X$-block
to $\ttd_{ij}^1[4]$ cancels the contribution of the $Y$-block to $\tilde\ttd^{43}_{ij}[2]$. Finally, if $\hat\jmath>\gammac(j)$
then all the above contributions vanish.

Otherwise, if $\bar\alpha_{t-1}^2=\bar\alpha_{p-1}^1$, the sixth, the seventh and the eights terms
in Lemma \ref{zone2lemma} contribute to both sides of $\tilde\ttd^{43}_{ij}$, since in both cases the exit point
for $\ttf^2$ lies to the left of the exit point for $\ttf^1$. Consequently, the contributions
of the sixth and the eight terms vanish, while the contribution of the $Y$-block to $\tilde\ttd^{43}_{ij}[7]$
equals the total contribution of the $X$-block to $\ttd_{ij}^1[4]$, $\ttd_{ij}^2[4]$ and $\ttd_{ij}^{43}[4]$, similarly
to the previous case.

Assume now that a $t$-th $X$-block of $\ttf^2$ satisfies conditions $\alpha_t^2=\alpha_p^1$ and $\beta_t^2=\beta_p^1$.
We distinguish the following five cases.

A. $\hat\imath-\hat\jmath>n-j+1$; consequently, the sixth, the seventh and the eights terms
 in Lemma \ref{zone1lemma} do not contribute to $\ttD_{ij}$, since in all cases involved the exit point
for $\ttf^2$ lies below the exit point for $\ttf^1$. Besides, $\bar\alpha_{t-1}^2\ge\bar\alpha_{p-1}^1$ and 
$\bar\beta_{t-1}^2=\bar\beta_{p-1}^1$. The treatment of this case is exactly the same as the treatment of the case  
 $\alpha_t^2>\alpha_p^1$ and $\beta_t^2=\beta_p^1$ above.

B. $\hat\imath-\hat\jmath=n-j+1$; consequently, $\bar\alpha_{t-1}^2=\bar\alpha_{p-1}^1$ and 
$\bar\beta_{t-1}^2=\bar\beta_{p-1}^1$. Similarly to the case A, the sixth, the seventh and the eights terms
 in Lemma \ref{zone1lemma} do not contribute to $\ttD_{ij}$, since in all cases involved the exit point
for $\ttf^2$ lies below or coincides with the exit point for $\ttf^1$. On the other hand, the sixth, the seventh and the eights terms in Lemma \ref{zone2lemma} contribute only to the subtrahend of $\tilde\ttd^{43}_{ij}$, but not to the minuend.
If the $Y$-block in question is not the leading block of $\ttf^2$ then the contributions of the $X$-block to
$\ttd_{ij}^1[4]$ and $\ttd_{ij}^2[4]$ cancel each other, the contribution of the $X$-block to
$\ttd_{ij}^{43}[4]$ equals one, while the contributions of the $Y$-block to $\tilde\ttd^{43}_{ij}[6]$, $\tilde\ttd^{43}_{ij}[7]$
and $\tilde\ttd^{43}_{ij}[8]$ are equal to $n+1-\bar\alpha_{t-1}^2-\gammac(j)$, $\gammac(j)-n$ and $\bar\alpha_{t-1}^2$,
respectively. Consequently, the total contribution to $\ttD_{ij}$ vanishes. If the $Y$-block is the leading block of $\ttf^2$
then the contributions of the $X$-block to $\ttd_{ij}^2[4]$ and $\ttd_{ij}^{43}[4]$ vanish. Further, if $\hat\imath>1$ then 
the contribution of the $X$-block to $\ttd_{ij}^1[4]$ vanishes as well, whereas the contributions of 
the $Y$-block to $\tilde\ttd^{43}_{ij}[6]$, $\tilde\ttd^{43}_{ij}[7]$ and $\tilde\ttd^{43}_{ij}[8]$ are equal to 
$n+\hat\imath-\bar\alpha_{t-1}^2-\hat\jmath$, $\hat\jmath-n-1$ and $\bar\alpha_{t-1}^2+1-\hat\imath$, respectively. 
Consequently, the total contribution to $\ttD_{ij}$ vanishes. Finally, if $\hat\imath=1$ then the contribution of the $X$-block to $\ttd_{ij}^1[4]$ equals one, whereas the contributions of 
the $Y$-block to $\tilde\ttd^{43}_{ij}[6]$, $\tilde\ttd^{43}_{ij}[7]$ and $\tilde\ttd^{43}_{ij}[8]$ are equal to 
$n+1-\bar\alpha_{t-1}^2-\gammac(j)$, $\gammac(j)-n$ and $\bar\alpha_{t-1}^2$, respectively, and again the total contribution to $\ttD_{ij}$ vanishes.

C. $\hat\imath-\hat\jmath=n-j$; consequently, $\bar\alpha_{t-1}^2=\bar\alpha_{p-1}^1$ and 
$\bar\beta_{t-1}^2=\bar\beta_{p-1}^1$. Here the sixth, the seventh and the eights terms
 in Lemma \ref{zone2lemma} do not contribute to $\tilde\ttd^{43}_{ij}$, since in both cases involved the exit point
for $\ttf^2$ lies to the right or coincides with the exit point for $\ttf^1$. On the other hand, the sixth, the seventh and the eighth terms in Lemma \ref{zone1lemma} do not contribute to $\ttd_{ij}^1$, $\ttd_{ij}^2$ and to the subtrahend of 
$\ttd^{43}_{ij}$, but contribute to its minuend. If the $X$-block in question is not the leading block of $\ttf^2$ then
its contributions to $\ttd_{ij}^1[4]$ and $\ttd_{ij}^2[4]$ cancel each other, and its contribution to $\ttd_{ij}^{43}[4]$
equals one. The contributions of this block to $\ttd^{43}_{ij}[6]$, $\ttd^{43}_{ij}[7]$ and $\ttd^{43}_{ij}[8]$ are equal to
$\alpha_t^2-j$, $1$ and $j-2-\alpha_t^2$, respectively. Consequently, the total contribution to $\ttD_{ij}$ vanishes. The same 
holds true if this $X$-block is the leading block of $\ttf^2$ and $\hat\imath<n$. If $\hat\imath=n$, and hence
$\hat\jmath=j$, then its contribution to $\ttd_{ij}^2[4]$ and $\ttd_{ij}^{43}[4]$ vanish, and the contribution to
$\ttd_{ij}^1[4]$ equals one. The contributions of this block to $\ttd^{43}_{ij}[6]$, $\ttd^{43}_{ij}[7]$ and 
$\ttd^{43}_{ij}[8]$ are equal to $\alpha_t^2-j$, $1$ and $j-1-\alpha_t^2$, respectively. Consequently, the total contribution to $\ttD_{ij}$ equals one. If the $Y$-block in question is the leading block of $\ttf^2$ then the contributions of the
$X$-block to $\ttd_{ij}^1[4]$, $\ttd_{ij}^2[4]$ and $\ttd_{ij}^{43}[4]$ vanish, as well as the contribution of the $Y$-block
to $\ttd^{43}_{ij}[7]$, and the contributions of $Y$-block to $\ttd^{43}_{ij}[6]$ and $\ttd^{43}_{ij}[8]$ cancel each
other. Consequently, the total contribution to $\ttD_{ij}$ vanishes.

D. $\hat\imath-\hat\jmath=n-j-1$; consequently, $\bar\alpha_{t-1}^2\le\bar\alpha_{p-1}^1$ and 
$\bar\beta_{t-1}^2=\bar\beta_{p-1}^1$. Here the sixth, the seventh and the eighth terms in Lemma \ref{zone1lemma} do not contribute to $\ttd_{ij}^1$, but contribute to $\ttd_{ij}^2$ and $\ttd^{43}_{ij}$. Assume first that 
$\bar\alpha_{t-1}^2=\bar\alpha_{p-1}^1$, then the sixth, the seventh and the eights terms
 in Lemma \ref{zone2lemma} do not contribute to $\tilde\ttd^{43}_{ij}$ similarly to case C.   
If the $X$-block in question is not the leading block of $\ttf^2$ then
its contributions to $\ttd_{ij}^1[4]$ and $\ttd_{ij}^2[4]$ cancel each other, and its contribution to $\ttd_{ij}^{43}[4]$
equals one. Further, its contributions to $\ttd_{ij}^2[6]$ and $\ttd_{ij}^{43}[6]$ vanish, and contributions to
$\ttd_{ij}^2[8]$ and $\ttd_{ij}^{43}[8]$ cancel each other. Finally, its contribution to $\ttd^2_{ij}[7]$ cancels the
contribution to $\ttd_{ij}^{43}[4]$, and hence the total contribution to $\ttD_{ij}$ vanishes. The same holds true if
the $X$-block is the leading block of $\ttf^2$ and $\hat\imath>n-1$. If $\hat\imath=n-1$ the contributions to  $\ttd_{ij}^2[4]$
and $\ttd_{ij}^{43}[4]$ vanish and the contributions to $\ttd_{ij}^1[4]$ and $\ttd^2_{ij}[7]$ cancel each other. If 
$\hat\imath=n$, or if the $Y$-block in question is the leading block of $\ttf^2$ then all the above mentioned contributions vanish. The case $\bar\alpha_{t-1}^2<\bar\alpha_{p-1}^1$ is similar; additionally to the above, the contribution of the 
$Y$-block to $\tilde\ttd_{ij}^{43}$ vanishes.
 
E. $\hat\imath-\hat\jmath<n-j-1$; consequently, $\bar\alpha_{t-1}^2\le\bar\alpha_{p-1}^1$ and 
$\bar\beta_{t-1}^2=\bar\beta_{p-1}^1$. This case is similar to the previous one, with the additional cancellation of the contributions to $\ttd_{ij}^1[7]$ and $\ttd_{ij}^1[8]$. 

Therefore, the total contribution to $\ttD_{ij}$ vanishes in 
all cases except for the case $(\hat\imath,\hat\jmath)=(n,j)$ when it is equal one, 
hence under the assumptions $X_I^J=X_{I'}^{J'}=X_{I''}^{J''}$ and $Y_{\bar I}^{\bar J}=Y_{\bar I'}^{\bar J'}$ relation
\eqref{dubneisum} holds with $\lambda=1$. If these assumptions are violated, 
one has to consider correction terms similarly to Section \ref{typicase}.

\section{Regularity check and the toric action} \label{sec:regtor}
The goal of this section is threefold: 

(i) to check condition (ii) in Proposition \ref{regfun} for the family $F_{\bfGr,\bfGc}$, 

(ii) to prove Theorem \ref{genmainth}(iii), and

(iii) to prove Proposition \ref{frozen}.

\subsection{Regularity check}
We have to prove the following statement.

\begin{theorem}\label{regneighbors}
For any mutable cluster variable $f_{ij}\in F_{\bfGr,\bfGc}$, the adjacent variable $f'_{ij}$ is a regular function on
$\Mat_n$.
\end{theorem}

\begin{proof}
The main technical tool in the proof is the version of the Desnanot--Jacobi identity for minors of a rectangular
matrix  that we have used previously 
for the regularity check in \cite{GSVMem}. Let $A$ be an $(m-1)\times m$ matrix, and $\alpha<\beta<\gamma$ be column indices, then 
\begin{equation}
\label{notjacobi}
\det A^{\hat \alpha} \det A^{\hat \beta \hat \gamma}_{\hat \delta} + 
\det A^{\hat \gamma} \det A^{\hat \alpha\hat \beta}_{\hat \delta} = 
\det A^{\hat \beta} \det A^{\hat \alpha \hat \gamma}_{\hat \delta}, 
\end{equation}
where ``hatted'' subscripts and superscripts indicate deleted rows and columns, respectively. 

Let us assume first that the degree of $(i,j)$ equals six. Following the notation introduced in the previous section, 
denote by $f_{i^1j^1}$ and $\tilde f_{i^1j^1}$ the functions at the vertices to the north and to the east of $(i,j)$, respectively, by $f_{i^2j^2}$ and $\tilde f_{i^3j^3}$ the functions at the vertices to the north-west and to the 
south-east of $(i,j)$, respectively, and by $f_{i^4j^4}$ and $\tilde f_{i^4j^4}$ the functions at the vertices 
to the west and to the south of $(i,j)$, respectively. Let $\L$ be the matrix used to define $f_{i^2j^2}$,
$f_{ij}$ and $\tilde f_{i^3j^3}$, $\L_+$ be the matrix used to define $f_{i^1j^1}$ and $\tilde f_{i^1j^1}$, and $\L_-$
be the matrix used to define $f_{i^4j^4}$ and $\tilde f_{i^4j^4}$.

Assume first that $\deg f_{ij}<\deg f_{i^1j^1}$. Define a 
$\deg f_{i^1j^1}\times(\deg f_{i^1j^1}+1)$ matrix $A$ 
as follows: if $s(i^1,j^1)>1$ then
$A=(\L_+)_{[s(i^1,j^1),N(\L_+)]}^{[s(i^1,j^1)-1,N(\L_+)]}$, 
otherwise $A$ is obtained via adding the first column of $\L_{[s(i,j)-1,N(\L)]}^{[s(i,j)-1,N(\L)]}$
on the left to the matrix $(\L_+)_{[s(i^1,j^1),N(\L_+)]}^{[s(i^1,j^1),N(\L_+)]}$. 
Then it is easy to see that 
$\L_{[s(i,j)-1,N(\L)]}^{[s(i,j)-1,N(\L)]}=A_{[1,\deg f_{ij}+1]}^{[1,\deg f_{ij}+1]}$, and moreover, that
$A_{[1,\deg f_{ij}+1]}^{[1,\deg f_{ij}+1]}$ is a block in the block upper triangular matrix  
$A_{[1,\deg f_{i^1j^1}]}^{[1,\deg f_{i^1j^1}]}$. Consequently, 
\[
f_{i^1j^1}=\det A^{\hat 1}, \quad\tilde f_{i^1j^1}=\det A_{\hat 1}^{\hat 1\hat 2}, \quad f_{i^2j^2}\cdot\det B=\det A^{\hat m}, 
 \quad f_{ij}\cdot\det B=\det A_{\hat 1}^{\hat 1\hat m}
\] 
with $B=A_{[\deg f_{ij}+2,\deg f_{i^1j^1}]}^{[\deg f_{ij}+2,\deg f_{i^1j^1}]}$
and $m=\deg f_{i^1j^1}+1$. Applying \eqref{notjacobi} with $\alpha=1$, $\beta=2$, $\gamma=m$, $\delta=1$, one gets
\[
f_{i^1j^1}\cdot\det A_{\hat 1}^{\hat 2\hat m}+f_{i^2j^2}\cdot\det B\cdot\tilde f_{i^1j^1}=
\det A^{\hat 2}\cdot f_{ij}\cdot\det B.
\]
Note that $\det A_{\hat 1}^{\hat 2\hat m}=\det \bar A_{\hat 1}^{\hat 2}\det B$ with 
$\bar A= A_{[1,\deg f_{ij}+1]}^{[1,\deg f_{ij}+1]}$, and hence 
\begin{equation}\label{firstdj}
f_{i^1j^1}\det \bar A_{\hat 1}^{\hat 2}+f_{i^2j^2}\tilde f_{i^1j^1}=f_{ij}\det A^{\hat 2}.
\end{equation}

Let now $\deg f_{ij}\ge\deg f_{i^1j^1}$. Define a $(\deg f_{ij}+1)\times(\deg f_{ij}+2)$ matrix $A$ via adding the column
$(0,\dots,0,1)^T$ on the right to the matrix $\L_{[s(i,j)-1,N(\L)]}^{[s(i,j)-1,N(\L)]}$. Then it is easy to see that 
$(\L_+)_{[s(i^1,j^1),N(\L_+)]}^{[s(i^1,j^1),N(\L_+)]}=A_{[1,\deg f_{i^1j^1}]}^{[2,\deg f_{i^1j^1}+1]}$, and moreover, that
$A_{[1,\deg f_{i^1j^1}]}^{[2,\deg f_{i^1j^1}+1]}$ is a block in the block lower triangular matrix  
$A_{[1,\deg f_{ij}+1]}^{[2,\deg f_{ij}+2]}$. Consequently, 
\[
f_{i^1j^1}\cdot\det B=\det A^{\hat 1}, \quad\tilde f_{i^1j^1}\cdot\det B=\det A_{\hat 1}^{\hat 1\hat 2}, 
\quad f_{i^2j^2}=\det A^{\hat m}, \quad f_{ij}=\det A_{\hat 1}^{\hat 1\hat m}
\] 
with $B=A_{[\deg f_{i^1j^1}+1,\deg f_{ij}+1]}^{[\deg f_{i^1j^1}+2,\deg f_{ij}+2]}$
and $m=\deg f_{ij}+2$. Applying \eqref{notjacobi} with $\alpha=1$, $\beta=2$, $\gamma=m$, $\delta=1$, one gets
\[
f_{i^1j^1}\cdot\det B\det \bar A_{\hat 1}^{\hat 2}+f_{i^2j^2}\cdot\tilde f_{i^1j^1}\cdot\det B=
\det A^{\hat 2}\cdot f_{ij},
\]
where $\bar A=A_{[1,\deg f_{ij}+1]}^{[1,\deg f_{ij}+1]}$ is the same as in the previous case. Note that
$\det A^{\hat 2}=\det \tilde A^{\hat 2}\det B$, where $\tilde A=A_{[1,\deg f_{i^1j^1}]}^{[1,\deg f_{i^1j^1}+1]}$ is given by the same expression as the whole matrix $A$ in the previous case. Consequently, relation \eqref{firstdj} remains valid in this case as well.

To proceed further, we compare $\deg f_{ij}$ with $\deg f_{i^4,j^4}$ and consider two cases similar to the two cases above. 
Reasoning along the same lines, we arrive to the relation
\begin{equation}\label{secdj}
f_{ij}\det C_{\hat 1}^{\hat 2}+\tilde f_{i^3j^3}f_{i^4j^4}=\tilde f_{i^4j^4}\det \bar A_{\hat 1}^{\hat 2}
\end{equation}
with $C=(\L_-)_{[s(i^4,j^4), N(\L_-)]}^{[s(i^4,j^4), N(\L_-)]}$ and $\bar A$ the same as in \eqref{firstdj}. The linear combination of \eqref{firstdj} and \eqref{secdj} with coefficients $\tilde f_{i^4j^4}$ and $f_{i^1j^1}$, respectively, yields
\begin{equation}\label{onestep}
f_{ij}(\tilde f_{i^4j^4}\det A^{\hat 2}-f_{i^1j^1}\det C_{\hat 1}^{\hat 2})=f_{i^2j^2}\tilde f_{i^1j^1}\tilde f_{i^4j^4}+
f_{i^1j^1}\tilde f_{i^3j^3}f_{i^4j^4}.
\end{equation}
Combining this with Theorem \ref{quiver} we see that 
$f_{ij}'=\tilde f_{i^4j^4}\det A^{\hat 2}-f_{i^1j^1}\det C_{\hat 1}^{\hat 2}$ is a regular function on $\Mat_n$.

For vertices of degree less than six, the claim follows from the corresponding degenerate version of \eqref{onestep}.
For example, for vertices of degree five there are three possible degenerations: 

(i) $\deg f_{i^1j^1}=1$, and hence $\tilde f_{i^1j^1}=1$, which corresponds to the cases shown in Fig.~\ref{fig:1jnei}(b),
Fig.~\ref{fig:innei}(c) and Fig.~\ref{fig:1nnei}(a);

(ii) $\deg f_{i^4j^4}=1$, and hence $\tilde f_{i^4j^4}=1$, which corresponds to the cases shown in Fig.~\ref{fig:i1nei}(b),
Fig.~\ref{fig:njnei}(c) and Fig.~\ref{fig:n1nei}(a);

(iii) $\deg f_{ij}=1$, and hence $\tilde f_{i^3j^3}=1$, which corresponds to the cases shown in Fig.~\ref{fig:njnei}(b),
Fig.~\ref{fig:innei}(b) and Fig.~\ref{fig:nnnei}(a).

Vertices of degrees four and three are handled via combining the above degenerations.
\end{proof}

\subsection{Toric action}\label{sec:toric}
To prove Theorem \ref{genmainth}(iii) we show first that  the  action of $\H_{\bfGr}\times\H_{\bfGc}$ on $SL_n$ given by the 
formula $(H_1,H_2)X=H_1 X H_2$ defines a global toric action of $(\C^*)^{k_{\bfGr}+k_{\bfGc}}$ on $\CC_{\bfGr,\bfGc}$.
In order to show this we first check that the right hand sides of all exchange relations in one cluster are semi-invariants 
of this action. This statement can be expressed as follows.

\begin{lemma}\label{rlsemi} 
Let $f_{ij}(X)f_{ij}'(X)=M(X)$ be an exchange relation in the initial cluster, then 
$M(H_1XH_2)=\chi_L^M(H_1)M(X)\chi_R^M(H_2)$, where $\chi_L^M$ and $\chi_R^M$ are left and right multiplicative 
characters of $\H_{\bfGr}\times\H_{\bfGc}$ depending on $M$. 
\end{lemma}

\begin{proof} Notice first that all cluster variables in the initial cluster are semi-invariants of the action of
$\H_{\bfGr}\times\H_{\bfGc}$. Indeed, recall that by \eqref{f_ij_gen},~\eqref{twof_ii}
any cluster variable $f_{ij}$ in the initial cluster is a minor of a matrix $\L$ of size $N=N(\L)$.  
Clearly, minors are semi-invariant of the left-right action of the torus $\Diag_{N}\times \Diag_{N}$ on $\Mat_N$,
where $\Diag_N$ is the group of invertible diagonal $N\times N$ matrices.
We construct now two injective homomorphisms $r:\H_{\bfGr}\to \Diag_{N}\times \Diag_N$ and 
$c:\H_{\bfGc}\to \Diag_{N}\times \Diag_N$ such that 
the homomorphism $(r,c):\H_{\bfGr}\times\H_{\bfGc}\to \Diag_{N}\times \Diag_{N}$ given by $(r,c)(H_1,H_2)=r(H_1)\cdot c(H_2)$
extends the left-right action of $\H_{\bfGr}\times\H_{\bfGc}$ on $SL_n$  to an action on $\Mat_N$.
Note that $\Diag_{N}\times \Diag_{N}$ is a commutative group, so $(r,c)$ is well-defined.

We describe first the construction of the homomorphism $r$.
Let $\Delta$ be a nontrivial row $X$-run, and $\bar\Delta=\gammar(\Delta)$ be the corresponding row $Y$-run. 
Recall that $\H_{\bfGr}=\exp\h_{\bfGr}$. Consequently, it follows from \eqref{smalltorus} that
for any fixed $T\in\H_{\bfGr}$ there exists a constant $g^\er_{\Delta}(T)\in \C^*$ such that for any pair of 
corresponding indices $i\in\Delta$ and $j\in\bar\Delta$ one has $T_{jj}=g^\er_\Delta(T)\cdot T_{ii}$. 
Clearly, $g^\er_\Delta$  is a multiplicative character of $\H_{\bfGr}$.

Fix a pair of blocks $X_{I_t}^{J_t}$ and $Y_{\bar I_t}^{\bar J_t}$ in $\L$. Let $\Delta_t$ be the row $X$-run 
corresponding to $\Phi_t$, then we put $g_t^\er=g_{\Delta_t}^\er$ and define a  matrix
$A_t^{\er}(T)\in \Diag_N$ such that its entry $(j,j)$ equals $g_t^\er(T)$ for 
$j\in \cup_{i=1}^{t-1}(K_i\cup\bar K_i)\cup(K_t\setminus\Phi_t)$ and~$1$ otherwise, and a matrix
$B_t^{\er}(T)\in \Diag_N$ such that its entry $(j,j)$ equals $\left(g_t^\er(T)\right)^{-1}$ for 
$j\in \cup_{i=1}^{t-1}(L_i\cup\bar L_i)\cup L_t$ and~$1$ otherwise, see Fig.~\ref{fig:ladder}.

Put $A^\er(T)=\prod_{t=1}^s A_t^\er(T)$ and $B^\er(T)=\prod_{t=1}^s B_t^\er(T)$. 
Finally, for any $j\in [1,N]$ define
$\zeta^\er(j)$ as the image of $j$ under the identification of $\bar K_t$ and $\bar I_t$ if $j\in \bar K_t$ and 
as the image of $j$ under the identification of $K_t$ and $I_t$ if $j\in K_t\setminus\Phi_t$, and put 
$C^\er(T)=\diag(T_{\zeta^\er(j),\zeta^\er(j)})_{j=1}^N$. Then, similarly to the proof of Lemma \ref{partrace}, one obtains
$\L(T X,T Y)=A^\er(T)C^\er(T)\L(X,Y)B^\er(T)$, and hence 
$r:T\mapsto (A^\er(T)C^\er(T),B^\er(T))$ is the desired homomorphism.

The construction of the homomorphism $c$ is similar, with $g_t^\ec$ defined by the column $X$-run corresponding to $\Psi_t$, 
$A_t^\ec(T)$ having $g_t^\ec(T)$ as the entry $(j,j)$ for $j\in \cup_{i=1}^{t-1}(L_i\cup\bar L_i)\setminus\Psi_t$ and 1 otherwise,
$B_t^\ec(T)$ having $(g_t^\ec(T))^{-1}$ as the entry $(j,j)$ for $j\in \cup_{i=1}^{t}(K_i\cup\bar K_i)$ and 1 otherwise,
$A^\ec(T)=\prod_{t=1}^s A_t^\ec(T)$, $B^\ec(T)=\prod_{t=1}^s B_t^\ec(T)$, and 
$C^\ec(T)=\diag(T_{\zeta^\ec(j),\zeta^\ec(j)})_{j=1}^N$, where $\zeta^\ec(j)$ is the image of $j$ under the identification 
of $L_t$ and $J_t$ if $j\in L_t$, and the image of $j$ under the identification of $\bar L_t$ and $\bar J_t$ if 
$j\in \bar L_t\setminus\Psi_{t+1}$. Consequently, the desired homomorphism is given by
$C:T\mapsto (A^\ec(T),B^\ec(T)C^\ec(T))$.

We thus see that any minor $P$ of $\L$ is a semi-invariant of the left-right action of $\H_{\bfGr}\times\H_{\bfGc}$ on $SL_n$,
and we can define multiplicative characters $\chi_L^P$ and $\chi_R^P$ as the products of the corresponding minors of 
$A^\er$, $A^\ec$ and $C^\er$, or $B^\er$, $B^\ec$ and $C^\ec$, respectively. 

To prove the lemma,  we consider first the most general case when the degree of the vertex 
$(i,j)$ is 6. Then, borrowing notation from the proof of Theorem \ref{regneighbors}, 
\[
 M(X)=\tilde f_{i^1 j^1}(X) f_{i^2 j^2}(X)\tilde f_{i^4 j^4}(X)+f_{i^1 j^1}(X)\tilde f_{i^3 j^3}(X) f_{i^4 j^4}(X).
\]
 It follows from \eqref{firstdj} that  
$\chi^{\tilde f_{i^1 j^1}}+\chi^{\tilde f_{i^2 j^2}}=\chi^{f_{i^1 j^1}}+\chi^{\det(\bar A_{\hat 1}^{\hat 2})}$, where $\chi$
means $\chi_L$ or $\chi_R$. Similarly, it follows from \eqref{secdj} that 
 $\chi^{\tilde f_{i^4 j^4}}+\chi^{\det(\bar A_{\hat 1}^{\hat 2})}=\chi^{f_{i^4 j^4}}+\chi^{\tilde f_{i^3 j^3}}$.
 Adding to both sides of the first equality $\chi^{\tilde f_{i^4 j^4}}$, to the both sides of the second equality 
 $\chi^{f_{i^1 j^1}}$ and adding these two equations together we obtain 
 \[
\chi^{\tilde f_{i^1 j^1}}+\chi^{\tilde f_{i^2 j^2}}+\chi^{\tilde f_{i^4 j^4}}=\chi^{f_{i^1 j^1}}+\chi^{\tilde f_{i^3 j^3}}+
\chi^{ f_{i^4 j^4}}=\chi^M,
\] 
which proves the assertion of the lemma.
 
 Other cases are obtained from the general case by the same specializations (setting one or more functions above to be $1$) 
that were used in the proof of Theorem \ref{regneighbors} above. This concludes the proof of the lemma.
 \end{proof}

To complete the proof we have to show that any toric action on $\CC_{\bfGr,\bfGc}$ can be obtained in this way. 
To prove this claim, we first note that the dimension of $\H_{\bfGr}$ equals $k_{\bfGr}$, and the dimension of $\H_{\bfGc}$ 
equals $k_{\bfGc}$. Consequently, the construction of Lemma \ref{rlsemi} produces $k_{\bfGr}+k_{\bfGc}$ weight vectors 
that lie in the kernel of the exchange matrix corresponding to $Q_{\bfGr,\bfGc}$, see \cite[Lemma 5.3]{GSVb}. Assume that
there exists a  vanishing nontrivial linear combination of these weight vectors; this would mean that all cluster variables
remain invariant under the toric action induced by a nontrivial right-left action of  $\H_{\bfGr}\times\H_{\bfGc}$ on $SL_n$.
However, by Theorem \ref{laurent} below, every matrix entry of the initial matrix in $SL_n$ can be written as a Laurent polynomial in the cluster variables of the initial cluster. Hence, a generic matrix remains invariant under this nontrivial right-left action on $SL_n$, a contradiction. Note that the proof of Theorem \ref{laurent} does not use the results of
Section \ref{sec:toric}.

\subsection{Proof of Proposition \ref{frozen}}
(i) We will focus on the behavior of $\det\L(X,Y)$ under the right action of $\DD_-=\DD_-^{\ec}$. 
The left action of $\DD_-^{\er}$ can be treated in a similar way. In fact, we will show that $\det\L(X,Y)$ is a semi-invariant 
of the right action of a larger subgroup of $D(GL_n)$. Let $\P_\pm$ be the parabolic subgroups in $SL_n$ that correspond to parabolic subalgebras \eqref{parabolics}, and let $\hat \P_\pm$ be the corresponding  parabolic subgroups in $GL_n$. Elements of $\hat\P_+$ (respectively, $\hat\P_-$) are block upper (respectively, lower) invertible triangular matrices whose 
square diagonal blocks correspond to column $X$-runs (respectively, column $Y$-runs).

It follows from \eqref{d_-} that $\DD_-$ is contained in a subgroup $\tilde\DD_-$ of $\hat\P_+\times\hat\P_-$ defined by the
property that every square diagonal block in the first component determined by a nontrivial column $X$-run $\Delta$ 
coincides with the square diagonal block in the second component determined by the corresponding nontrivial column $Y$-run. 
For $g=(g_1,g_2)\in \tilde\DD_-$, consider the transformation of $\L(X,Y)$ under the action $(X,Y)\mapsto (X,Y)\cdot g$, 
in particular the transformation of the block column $L_t \cup \bar L_{t-1}$ as depicted in Fig.~\ref{fig:ladder}. 
In dealing with the block column we only need to remember that $(g_1,g_2)$ can be written as
\[
(g_1,g_2) =\left (  \begin{bmatrix} A_{11} &A_{12} & A_{13}\\  0& C & A_{23}\\ 0 & 0 & A_{33} \end{bmatrix} ,
 \begin{bmatrix} B_{11} & 0 & 0\\  B_{21} & C & 0\\ B_{31}&  B_{32}& B_{33} \end{bmatrix}
\right ),
\]
where $A_{11}, A_{33}, B_{11} , B_{33}$ and $C$ are invertible and $C$ occupies rows and columns labeled by 
$\Delta(\beta_t)$ in $g_1$ and rows and columns labeled by $\bar\Delta(\bar\beta_{t-1})$ in $g_2$ 
(recall that both these runs correspond to $\Psi_t$). Then the effect of the transformation  $(X,Y)\mapsto (X,Y)\cdot g$ on the 
block column is that it is multiplied on the right by an invertible matrix
\[
\begin{bmatrix} A_{11} & A_{12} & 0\\  0 & C & 0\\  0 &  B_{32}& B_{33} \end{bmatrix}.
\]
The cumulative effect on $\L(X,Y)$ is that it is transformed via a multiplication on the right by an invertible block diagonal matrix with blocks as above, and therefore $\det\L(X,Y)$ is transformed via a multiplication  by the determinant of this matrix. 
The latter, being a product of powers of determinants of diagonal blocks of $g_1$ and $g_2$, is a character of $\tilde\DD_-$, 
which proves the statement.

(ii) The claim follows from a more general statement: $\det\L(X,Y)$ is log-canonical with all matrix entries $x_{ij}$, 
$y_{ij}$ with respect to the Poisson bracket \eqref{sklyadoublegen} which, in our situation, takes the form \eqref{bracket}. 
Semi-invariance of $\det\L(X,Y)$ described in part (i) above, together with the fact that subalgebras $\D_-=\D_-^r$ and $\D_-'=\D_-^c$ are isotropic with respect to  the bilinear form $\langle\langle \ , \ \rangle\rangle$ implies 
\[
\dnabla^L f\in \D_- \dot{+} \left (\D_+ \cap \h\oplus \h\right ),\qquad   
\dnabla^R f\in \D'_- \dot{+} \left (\D_+ \cap \h\oplus \h\right )
\]
for $f=\log\det\L(X,Y)$. This means that in  \eqref{sklyadoublegen} 
\[
R_D ( \dnabla^L f) = -  \dnabla^L f + \pi_{\D_+}\left (\dnabla^L f \right )_0,\qquad 
R'_D ( \dnabla^R f) = -  \dnabla^R f + \pi'_{\D_+}\left (\dnabla^R f \right )_0, 
\]
where $\left (\  \right)_0$ denotes the natural projection  to $D(\h)=\h  \oplus \h$ and $\pi_{\D_+}, \pi'_{\D_+}$ are projections to $\D_+$ along $\D_-,\D'_-$ respectively. Due to the invariance of $\langle\langle \ , \ \rangle\rangle$, 
\eqref{sklyadoublegen} then reduces to 
\[
\{f,\varphi\}^D_{r,r'} = \frac{1}{2}\left (\langle\langle \pi_{\D_+}(\dnabla^L f)_0, \left (\dnabla{^L} \varphi \right )_0\rangle\rangle 
- \langle\langle \pi'_{\D_+}(\dnabla^R f)_0, \left (\dnabla^R \varphi \right )_0\rangle\rangle \right)
\]
for any $\varphi=\varphi(X,Y)$. 

Let now $\varphi(X,Y) = \log x_{ij}$. Then $\left (\dnabla{^L} \varphi \right )_0 = \left (e_{jj}, 0\right )$, 
$\left (\dnabla{^R} \varphi \right )_0 = \left (e_{ii}, 0\right )$. Thus, to prove the desired claim we need to show that 
$\pi_{\D_+}(\dnabla^L f)_0$ and $\pi'_{\D_+}(\dnabla^R f)_0$ do not depend on $X,Y$. To this end, we first recall an explicit formula for $\pi_{\D_+}$: 
\[
\pi_{\D_+}(\xi,\eta) = \left (\xi - R_+(\xi -\eta) ,\xi - R_+(\xi -\eta) \right ),
\]
which can be easily derived using the property $R_+ - R_- = \Id$ satisfied by R-matrices \eqref{r-matrix}. Since in our situation the left gradient $\dnabla^L f$ computed with respect to $\langle\langle \ , \ \rangle\rangle$ is equal to 
$  \left (\nabla_X f \cdot X, - \nabla_Y f\cdot Y \right )$, we conclude that  components of $\pi_{\D_+}(\dnabla^L f)_0$ are equal to
$\left(\nabla_X f \cdot X - R_+\left (E_L f \right ) \right)_0$, where $\left (\  \right)_0$ now means the projection to the diagonal in $\gl_n$. By \eqref{term1left}, \eqref{for_frozen}, \eqref{R0}, 
\begin{multline*}
\left(\nabla_X f \cdot X - R_+\left (E_L f \right ) \right)_0 
= \frac 1 2 \left (-\frac {1}{1-\gamma} \left (\xi_L f\right )_0 + \frac{1} {1 - \gamma^*}\left (\eta_L f\right )_0\right ) \\ 
+ \frac 1 n \left(\Tr(E_L f)\S - \Tr \left((E_L f) \S\right )\one\right).
\end{multline*}
By \eqref{infinv2}, Corollary \ref{traces} and \eqref{traceELS}, the right hand side above is constant. The constancy of 
$\pi'_{\D_+}(\dnabla^R f)_0$ and the case of $\varphi(X,Y) = \log y_{ij}$ can be treated similarly. This completes the proof.

\section{Proof of Theorem \ref{genmainth}(ii)} \label{sec:induction}

As it was explained above in Section \ref{outline}, we have to prove the following statement.

\begin{theorem} \label{laurent}
Every matrix entry can be written as a Laurent polynomial in the initial cluster $F_{\bfGr,\bfGc}$ and 
in any cluster adjacent to it.
\end{theorem}

Below we implement the strategy of the proof outlined in Section \ref{outline}.

\subsection{Proof of Theorem \ref{prototype} and its analogs}\label{matrixmaps}
Given an aperiodic pair $(\bfGr, \bfGc)$  and a non-trivial row $X$-run $\Delta^{\er}$, 
we want to explore the relation between   cluster structures 
$\CC=\CC_{\bfGr,\bfGc}$ and $\tilde\CC=\CC_{\tbfGr,\bfGc}$, where  $\tbfGr=\tbfGr(\overrightarrow{\Delta}^\er)$
is obtained by deletion of the rightmost root in $\Delta^\er$ and its image in $\gamma(\Delta^\er)$.
Note that the pair $(\tbfGr(\overrightarrow{\Delta}^\er),\bfGc)$ remains aperiodic. 

Assume that $\Delta^\er$ is $[p+1, p+k]$, and the corresponding row  $Y$-run  $\gamma(\Delta^\er)$ is
$[q+1, q+k]$. Then, in considering $(\tbfGr(\overrightarrow{\Delta}^\er),\bfGc)$, we replace 
the former one with $[p+1, p+k-1]$, and the latter one  with $[q+1, q+k-1]$. Besides, a trivial row $X$-run
$[p+k,p+k]$ and a trivial row $Y$-run $[q+k,q+k]$ are added. 
The rest of row $X$- and $Y$-runs as well as all column $X$- and $Y$-runs remain unchanged.
In what follows, parameters $p$, $q$ and $k$ are assumed to be fixed.

We say that a matrix $\L\in \bL$ is $r$-{\em piercing\/} for an $r\in [2,k]$ if $\J(p+r,1)=(\L,s_r)$ for some 
$s_r\in [1,N(\L)]$. Note that two distinct matrices cannot be simultaneously $r$-piercing. On the other hand,  
a matrix can be $r$-piercing simultaneously for several distinct values of $r$; the set of all such values is called the {\it piercing set\/} of $\L$. If a piercing set consists of $r_1,\dots,r_l$, we will assume that 
$s_{r_1}>\dots > s_{r_l}$. The subset of all matrices in $\bL$ that are not $r$-piercing for any $r\in [2,k]$ is 
denoted $\bL_\varnothing$. 

Let $\tilde\bL=\bL_{\tbfGr (\overrightarrow{\Delta}^\er),\bfGc}$, 
$\tilde\J=\J_{\tbfGr (\overrightarrow{\Delta}^\er),\bfGc}$, and let the functions
${\tilde\ttf}_{ij}(X,Y)$ and $\tilde f_{ij}(X)$ be defined via the same expressions as ${\ttf}_{ij}(X,Y)$ and $f_{ij}(X)$
with $\bL$ and $\J$ replaced by $\tilde\bL$ and $\tilde\J$. It is convenient to restate Theorem \ref{prototype} in more detail
as follows.

\begin{theorem}
\label{matrixmap1}
Let $Z=(z_{ij})$ be an $n\times n$ matrix. Then there exists a unipotent upper triangular $n\times n$ matrix $\nuu(Z)$ 
whose entries are rational functions in $z_{ij}$ with denominators equal to powers of $\tilde f_{p+k,1}(Z)$
such that for $X=\nuu(Z) Z$ and for any $i,j\in [1,n]$,
$$
f_{ij} (X) =\begin{cases}  
\tilde f_{ij} (Z)\tilde f_{p+k,1} (Z)\quad &\text{if $\J (i,j) = (\L^*, s)$ and $s< s_k$},\\
\tilde f_{ij} (Z)\quad &\text{otherwise},
\end{cases}
$$
where $\L^*$ is the $k$-piercing matrix in $\bL$.
\end{theorem}

\begin{proof} 
In what follows we assume that $i\ne j$, since for $i=j$ the claim of the theorem is trivial.

For any $\L(X,Y) \in \bL$ define $\tilde\L(X,Y)$ obtained from   $\L(X,Y)$ by removing the last row from 
every building block of the form $Y_{[1,q+k]}^{\bar J}$. In particular, 
if $\L(X,Y)$ does not have building blocks like that  then $\tilde\L(X,Y)=\L(X,Y)$. 

Note that all matrices $\tilde \L$ defined above are irreducible except for the one obtained from  the $k$-piercing
matrix $\L^*$. The corresponding matrix $\tilde\L^*$ has two irreducible diagonal blocks 
$\tilde\L_{1}^*$, $\tilde\L_{2}^*$ of sizes $s_k-1$ and $N(\L^*) - s_k+1$, respectively. 
As was already noted in Section \ref{outline}, all maximal alternating paths in $\BD_{\Gammar,\Gammac}$ are preserved in 
$\BD_{\tbfGr (\overrightarrow{\Delta}^\er),\bfGc}$ except for the path that goes through the directed inclined edge
$(p+k-1)\to (q+k-1)$. The latter one is split into two: the initial segment up to the vertex $p+k-1$ and the closing segment
starting with the vertex $q+k-1$. Consequently, $\tilde{\bL}=\{\tilde \L\: \L\in \bL, \L\ne \L^*\}\cup
\{\tilde\L_{1}^*,\tilde\L_{2}^*\}$. 

Further, if $\J (i,j) = (\L, s)$ and $\L\ne \L^*$ 
then $\tilde{\J}(i,j) =(\tilde\L, s)$. Furthermore, if $\L \in \bL_\varnothing$ then 
additionally ${\ttf}_{ij}(X,Y)$ and  ${\tilde {\ttf}}_{ij}(X,Y)$ coincide. 
However, if $\J (i,j) = (\L^*, s)$ then 
\begin{equation*}
\tilde{\J}(i,j) =\begin{cases} (\tilde\L_{1}^*, s) \quad& \text{for $s=s(i,j)  < s_k$}, \\
                               (\tilde\L_{2}^*, s - s_k +1) \quad&  \text{for $s=s(i,j)  \geq s_k$}.
\end{cases}
\end{equation*}															

It follows from the above discussion that the claim of the theorem is an immediate corollary of the equalities
\begin{equation}\label{formeri}
 \det \L(X,X)_{[s, N(\L)]}^{[s, N(\L)]} =  \det \tilde \L(Z,Z)_{[s, N(\L)]}^{[s, N(\L)]}
\end{equation}
for any $\L \in \mathbf L$ and $s\in [1, N(\L)]$. 

To prove \eqref{formeri}, we select a particular "shape" for $\nuu(Z)$. Let
\begin{equation}
\label{mm0}
\nuu_0=\nuu_0(Z) = \one_n + \sum_{\varkappa=1}^{k-1} \alpha_\varkappa(Z) e_{q+\varkappa,q+k},
\end{equation}
where $\alpha_\varkappa(Z)$ are coefficients to be determined, and
\begin{equation}
\label{mm1}
\nuu=\nuu(Z) ={\stackrel {\leftarrow}  {\prod}_{i\geq 0}}\exp(i\gammar) (\nuu_0(Z)).
\end{equation}
Due to the nilpotency of $\gammar$ on $\n_+$, the product above is finite. Clearly, if  $\alpha_\varkappa(Z)$ are 
polynomials in $z_{ij}$ divided by a power of $\tilde f_{p+k,1}$
then the same is true for the entries of $\nuu(Z)$.

The invariance property \eqref{2.1} implies that for every
$(i,j)$,
$$
{\ttf}_{ij}(\nuu Z,\nuu Z) = {\ttf}_{ij}(Z,\exp(\gammar)(\nuu^{-1})\nuu Z) = {\ttf}_{ij} (Z,\nuu_0 Z);
$$
here the second equality follows from \eqref{mm1}. 
Thus, to prove \eqref{formeri}  for $X= \nuu Z$ it is sufficient to select parameters $\alpha_\varkappa(Z)$ in \eqref{mm0}
in such a way that
\begin{equation}
\label{mm2}
 \det \L(Z, \nuu_0 Z)_{[s, N(\L)]}^{[s, N(\L)]} =  \det \tilde \L(Z,Z)_{[s, N(\L)]}^{[s, N(\L)]}\ 
\end{equation}
 for all $\L \in \mathbf L$ and $s\in [1, N(\L)]$. 

Observe, that the equation above is satisfied for any choice of $\alpha_\varkappa$ if 
 $\L \in \bL_\varnothing$, that is, if $\L(X,Y)=\tilde \L(X,Y)$.  
Indeed, in this case any $Y$-block in $\L$ either does not contain any of the rows
$q+1, \ldots, q+k$, or contains all of them but without an overlap with the $X$-block to the right. 
If the former is true, the block rows corresponding to this $Y$-block in $\L(Z, \nuu_0 Z)$ and $\L(Z, Z)$ coincide, 
while if the latter is true, then  the block of $k$ rows under consideration in $\L(Z, \nuu_0 Z)$ is obtained from the corresponding block row of $\L(Z, Z)$ via left multiplication by a $k\times k$ unipotent upper triangular matrix $ \one_k + \sum_{\varkappa=1}^{k-1} \alpha_\varkappa(Z) e_{\varkappa k}$, which does not affect trailing principal minors.

Let us now turn to matrices $\L \in \bL\setminus\bL_\varnothing$. In fact, the same reasoning 
as above shows that for any such matrix, the functions in the left hand side of \eqref{mm2} 
 do not change if $\L(Z, \nuu_0 Z)$ is replaced by $\hat\L(Z,\nuu_0 Z)$  
obtained from $\L(Z,Z)$ via replacing every $Y$-block $Z_{[1,q+k]}^{\bar J}$ by 
$\left (\nuu_0 Z\right )_{[1,q+k]}^{\bar J}$ and retaining all other $Y$-blocks $Z_{\bar I}^{\bar J}$. 
Therefore, in what follows we aim at proving 
\begin{equation}
\label{mm2r}
 \det \hat\L(Z, \nuu_0 Z)_{[s, N(\L)]}^{[s, N(\L)]} =  \det \tilde \L(Z,Z)_{[s, N(\L)]}^{[s, N(\L)]}
\end{equation}
 for all $\L \in \bL\setminus\bL_{\varnothing}$ and $s\in [1, N(\L)]$. 

Assume that $\L=\L(X,Y)$ is $r$-piercing, and so 
there exists $s_r \in [1, N(\L)]$ such that $\L(X,Y)_{s_r s_r} = x_{p+r,1}$; 
the $X$-block of $\L(X,Y)$ that contains the diagonal entry $(s_r,s_r)$ is denoted 
$X^{J^r}_{[p+1,n]}$. 
We can decompose $\hat\L=\hat\L(Z,\nuu_0 Z)$ into blocks as follows:
\begin{equation}
\label{uglymatrix}
\hat\L(Z,\nuu_0 Z)=\begin{bmatrix} \hat A^r_1 & 0\\ \hat A^r_2 & \hat B^r_1\\ 0 & \hat B^r_2 
\end{bmatrix},
\end{equation}
where the sizes of block rows are $s_r-r$, $k$ and $N(\L)-s_r-k+r$, and the sizes of block columns are $s_r-1$ and 
$N(\L)-s_r+1$. Note that the blocks are given by
\begin{equation*}
\hat A^r_1=\begin{bmatrix} \ast & \ast\\ 0 & (\nuu_0 Z)^{\bar J^r}_{[1,q]}
\end{bmatrix}, \qquad
\hat A^r_2= \begin{bmatrix} 0 & (\nuu_0 Z)^{\bar J^r}_{[q+1,q+k]} \end{bmatrix}
\end{equation*}
and
\begin{equation*}
\hat B^r_1= \begin{bmatrix}  Z^{J^r}_{[p+1,p+k]} & 0 \end{bmatrix}, \qquad
\hat B^r_2=\begin{bmatrix} Z^{J^r}_{[p+k+1,n]} & 0\\ \ast & \ast
\end{bmatrix}.
\end{equation*}

It will be convenient to combine $\hat A^r_1$ and $\hat A^r_2$ into one $(s_r+k-r)\times(s_r-1)$ block
$\hat A^r$, and $\hat B^r_1$ and $\hat B^r_2$ into one $\theta_r\times(\theta_r-r+1)$ block 
$\hat B^r$  with $\theta_r=N(\L)-s_r+r$. 
A similar decomposition into blocks of the same size for $\tilde\L=\tilde\L(Z,Z)$ 
contains blocks $\tilde A^r_1$, 
$\tilde A^r_2$, $\tilde B^r_1$ and $\tilde B^r_2$ that may be combined into $\tilde A^r$ and 
$\tilde B^r$, respectively; consequently, the last row of $\tilde A^r_2$ (and hence of $\tilde A^r$) is zero. 
Note that since exactly one matrix in $\bL\setminus\bL_\varnothing$ is $r$-piercing for any fixed $r$, 
notation $\hat A^r$, $\hat B^r$, and $\tilde A^r$, $\tilde B^r$ is unambiguous.

Denote the column set of the second block column in \eqref{uglymatrix} by $M_r$. 
Let
\begin{equation}
\label{alpha}
\alpha_\varkappa (Z) =  \frac{\det (\tilde \L^*)^{M_k}_{(M_k\setminus \{s_k\})\cup 
\{s_k + \varkappa-k \} }}
{\det (\tilde\L^*)_{M_k}^{M_k}}, \quad  \varkappa =1,\ldots, k;
\end{equation}
note that $\alpha_k =1$.
We claim that $\nuu_0(Z)$ given by \eqref{mm0} and \eqref{alpha} satisfies conditions \eqref{mm2r}.
Note that the denominator in \eqref{alpha} equals $\tilde f_{p+k,1}(Z)$, and hence the denominators 
of the entries of $\L$ defined by \eqref{mm1} are powers of  $\tilde f_{p+k,1}(Z)$.

 Assume that the piercing set of $\L$ is $\{r_1,\dots,r_l\}$;  additionally, set $s_{r_{l+1}}=1$.
Recall that $Y$-blocks of the form $Z_{[1,q+k]}^{\bar J}$ do not appear in the columns $M_{r_1}$ in 
$\hat\L$, and hence  
\eqref{mm2r} is trivially satisfied for $ s \geq s_{r_1}$.

For $s_{r_2}\le s \le s_{r_1}-1$, we are in the situation covered by Lemma \ref{blockmatrix} (see Section~\ref{auxi} below) with 
$\M=\hat\L^{M_{r_2}}_{M_{r_2}}$, $\tilde\M=\tilde\L^{M_{r_2}}_{M_{r_2}}$,  $N=\theta_{r_2}-r_2+1$, 
$N_2=\theta_{r_1}-r_1+1$, and $k_1=r_1-1$. Condition (iii) 
in the lemma is satisfied trivially, since in this case $B=\tilde B$. Consequently, \eqref{mm2r} is satisfied if the parameters
$\alpha_\varkappa=\alpha_{\varkappa}(Z)$ satisfy equations  
\begin{equation}
\label{eqs_r}
\sum_{\varkappa \in S}  (-1)^{\varepsilon_{\varkappa S}} \alpha_\varkappa 
\det (\tilde B^{r_1})_{(S\setminus\{\varkappa \})\cup [k+1, \theta_{r_1}]}=0     
\end {equation}
for any  $(k-r_1+2)$-element subset $S$ in $[1, k]$ such that $k\in S$, where
\[
\varepsilon_{\varkappa S} = \# \{ i\in S : i > \varkappa\}.
\]

If $l=1$, there are no other conditions on the parameters $\alpha_\varkappa$, since $s_{r_2}=1$.
Otherwise,  let  $s_{r_3}\le s \le s_{r_2}-1$ and 
consider the block decomposition \eqref{uglymatrix} for $r=r_2$. We claim that the situation is now 
covered by Lemma \ref{blockmatrix} with 
$\M=\hat\L^{M_{r_3}}_{M_{r_3}}$, $\tilde\M=\tilde\L^{M_{r_3}}_{M_{r_3}}$, $N=\theta_{r_3}-r_3+1$, 
$N_2=\theta_{r_2}-r_2+1$, and $k_1=r_2-1$. To check condition (iii) 
in the lemma, we pick an arbitrary subset  $T\subset [s_{r_2}-r_2+1,s_{r_2}-r_2+k]$ of size $k-r_2+1$ and
apply Lemma \ref{blockmatrix} to matrices $\M=\hat\L^{T\cup M_{r_2}\setminus[s_{r_2},s_{r_2}-r_2+k]}_{M_{r_2}}$ and 
$\tilde\M=\tilde\L^{T\cup M_{r_2}\setminus[s_{r_2},s_{r_2}-r_2+k]}_{M_{r_2}}$ with parameters 
$N=\theta_{r_2}-r_2+1$, $N_2=\theta_{r_1}-r_1+1$,  and $k_1=r_1-1$. It follows that the condition in question is guaranteed by the same equations \eqref{eqs_r}. Consequently, by Lemma~\ref{blockmatrix}, equations \eqref{mm2r} for $s_{r_3}\le s \le s_{r_2}-1$ are guaranteed by equations \eqref{eqs_r} with $r_1$ replaced by $r_2$. 

Continuing in the same fashion, we conclude that if conditions
\begin{equation}
\label{eqs_rr}
\sum_{\varkappa \in S}  (-1)^{\varepsilon_{\varkappa S}} \alpha_\varkappa 
\det (\tilde B^{r})_{(S\setminus\{\varkappa \})\cup [k+1, \theta_{r}]}=0     
\end {equation}
are satisfied for any $r\in \{r_1,\dots,r_l\}$ and  any  $(k-r+2)$-element subset $S$ in $[1, k]$ containing $k$, 
then \eqref{mm2r} holds for any $s\in [1, N(\L)]$. It remains to show that \eqref{eqs_rr} are valid  with $\alpha_\varkappa$ defined 
in \eqref{alpha}.  

Rewrite \eqref{alpha} as
\begin{equation}
\label{alphaB}
\alpha_\varkappa (Z) = 
\frac{ \det (\tilde B^k)_{\{\varkappa\}\cup [k+1,\theta_k]}}  
{ \det (\tilde B^k)_{ [k, \theta_k]} }, \quad  \varkappa =1,\ldots, k.
\end{equation}
If $r=k$, and hence $\L=\L^*$, then every $S$ in \eqref{eqs_rr} is a two element set $\{ \varkappa, k\}$ 
with $\varkappa \in [1,k-1]$, 
$\varepsilon_{\varkappa S}=1$, $\varepsilon_{k S}=0$. Plugging \eqref{alphaB} into the left hand side of 
\eqref{eqs_rr} and clearing denominators we obtain two terms that differ only by sign and thus the claim follows.

For $r<k$, we need to evaluate
\begin{equation}
\label{uzhas}
\sum_{\varkappa \in S}  (-1)^{\varepsilon_{\varkappa S}}  \det (\tilde B^k)_{\{\varkappa\}\cup [k+1, \theta_k]}
\det (\tilde B^r)_{(S\setminus\{\varkappa \})\cup [k+1, \theta_r]}.
\end {equation}
Note that the blocks $Z_{[p+1, n]}^{J^k}$ and $Z_{[p+1, n]}^{J^r}$ have the same row set, and the exit point 
of the former lies below the exit point of the latter. Consequently, $J^k\subseteq J^r$, and the first of the blocks is a 
submatrix of the second one. Therefore, we find ourselves in a situation similar to the one discussed in Section \ref{case1}
above while analyzing sequences \eqref{blocks} of blocks. Reasoning along the same lines, we either arrive at the cases (ii) 
and (iii) in Section \ref{case1}, and then
\begin{equation}
\label{compareB1}
\tilde B^k = \left [
\begin{array}{ccc }
U_1 & U_2 & 0\\
0 & V_1 & V_2
\end{array}
\right ], \qquad 
\tilde B^r = \left [
\begin{array}{@{}ccccc@{} }
U_1& U_2 & U_3 & U_4 & 0\\
0 &0 & 0 & W_1 & W_3
\end{array}
\right ],
\end{equation}
where odd block columns and the second block row of $\tilde B^k$ and $\tilde B^r$ might be empty, or
at the cases (i) and (iv) in Section \ref{case1}, and then 
\begin{equation}
\label{compareB2}
\tilde B^k = \left [
\begin{array}{cc }
U_1 &  0\\
U_2 & 0\\
U_3 & 0\\
U_4 & V_1\\
0 & V_2
\end{array}
\right ],\qquad 
\tilde B^r =\left [
\begin{array}{cc }
U_1 &  0\\
U_2 & W_1\\
0 & W_2
\end{array}
\right ],
\end{equation} 
where odd block rows and the second block column of $\tilde B^k$ and $\tilde B^r$ might be empty. 
In particular, if $\tilde B^k$ is a submatrix of $\tilde B^r$ (cf.~case (iv) in Section \ref{case1}) 
then \eqref{compareB1} applies with an empty second block row and third block column in the expression for $\tilde B^k$. Similarly, if $\tilde B^r$ is a submatrix of  $\tilde B^k$ (cf.~case (iii) in Section \ref{case1}) then \eqref{compareB2} 
applies with an empty second block column and third block row in the expression for $\tilde B^r$.

Suppose \eqref{compareB1} is the case. Define $\tau_4>\tau_3\geq \tau_2>\tau_1\geq\tau_0=0$ and $\sigma>0$ so that the size 
of the block $U_i$ equals $\sigma \times ( \tau_i - \tau_{i-1})$ for $1\leq i\leq 4$. 
Note that $\sigma \geq n-p \geq k$ and $\sigma > \tau_3$.
We will use the Laplace expansion of the minors in \eqref{uzhas} with respect to the first block row:
\begin{equation}\label{decomp}
\begin{aligned}
 \det (\tilde B^k)_{\{\varkappa\}\cup [k+1, \theta_k]} &  
 =\sum\limits_{\Theta} (-1)^{\varepsilon_\Theta} \det (\tilde B^k)_{\{\varkappa\}\cup [k+1, \sigma]}^{[1, \tau_1]\cup \Theta} 
\det (\tilde B^k)_{[\sigma+1, \theta_k]}^{\bar\Theta\cup [\tau_2 +1, \theta_k-k+1]},  \\ 
 \det (\tilde B^r)_{(S\setminus\{\varkappa \})\cup [k+1, \theta_r]}&
 =  \sum\limits_{\Xi} (-1)^{\varepsilon_\Xi} \det (\tilde B^r)_{(S\setminus\{\varkappa \})\cup [k+1, \sigma]}^{[1, \tau_3]\cup 
\Xi} 
\det (\tilde B^r)_{[\sigma+1, \theta_r]}^{\bar\Xi\cup [\tau_4 +1, \theta_r-r+1]}.
\end{aligned}
\end{equation}
Here  the first sum runs over all $\Theta \subset [\tau_1+1, \tau_2]$ such that $|\Theta|=\sigma - \tau_1-k+1$, 
and $\bar\Theta$ is the complement of $\Theta$ in $[\tau_1+1, \tau_2]$; the second sum runs over all
$\Xi \subset [\tau_3+1, \tau_4]$ such that $|\Xi|=\sigma - \tau_3-r+1$, and $\bar\Xi$ is the complement of $\Xi$ 
in $[\tau_3+1, \tau_4]$; $\varepsilon_\Theta$ and $\varepsilon_\Xi$ depend only on $\Theta$ and $\Xi$, respectively, and 
$[k+1,\sigma]$ is empty if $\sigma = k$. 
Plug \eqref{decomp} into \eqref{uzhas} and note that for any fixed pair $\Theta$, $\Xi$, the coefficient at 
$$
\det (\tilde B^k)_{[\sigma+1, \theta_k]}^{\bar\Theta\cup [\tau_2 +1, \theta_k-k+1]} 
\det (\tilde B^r)_{[\sigma+1, \theta_r]}^{\bar\Xi\cup [\tau_4 +1, \theta_r-r+1]}
$$ 
is equal to
\begin{equation}
\label{uzhas?}
(-1)^{\varepsilon_\Theta+\varepsilon_\Xi}\sum_{\varkappa \in S}  (-1)^{\varepsilon_{\varkappa S}}  
\det (\tilde B^r)_{\{\varkappa\}\cup [k+1, \sigma]}^{[1, \tau_1]\cup \Theta} 
 \det (\tilde B^r)_{(S\setminus\{\varkappa \})\cup [k+1, \sigma]}^{[1, \tau_3]\cup \Xi},
\end {equation}
since the upper left $\sigma\times \tau_2$ blocks of $\tilde B^r$ and $\tilde B^k$ coincide. 
Observe that $[1, \tau_1]\cup\Theta \subset [1,\tau_3]$, and hence \eqref{uzhas?} is equal to the left-hand 
side of the Pl\"ucker relation \eqref{plu2} with $A=\tilde B^r$, $I=S$, $J=[k+1,\sigma]$, $L= [1, \tau_1]\cup\Theta$ and 
$M = \left ( [1, \tau_3]\cup \Xi\right ) \setminus \left ([1, \tau_1]\cup\Theta\right )$. Thus 
\eqref{uzhas?} vanishes for any $\Theta$, $\Xi$, and so \eqref{uzhas} is zero in the case \eqref{compareB1}. 
The case  \eqref{compareB2} can be treated similarly: using the Laplace expansion with respect to the first block column, 
one concludes that  \eqref{uzhas} is zero.
This proves that with $\alpha_\varkappa$ defined by \eqref{alpha}, all conditions \eqref{eqs_rr} are satisfied, 
and therefore \eqref{mm2r} is valid, which completes the proof of the theorem. 

\end{proof}

As it was explained in Section \ref{outline}, we also need a version of Theorem \ref{prototype} 
relating $\CC=\CC_{\bfGr,\bfGc}$ and $\tilde\CC=\CC_{\tbfGr,\bfGc}$, where  $\tbfGr=\tbfGr(\overleftarrow{\Delta}^\er)$ 
is obtained by the deletion of 
the leftmost root in $\Delta^\er$. The treatment of this case follows the same strategy as above.
Once again, we assume that the non-trivial row  $X$-run that corresponds to $\Delta^{\er}\subset\Gamma^{\er}_1$ is 
$[p+1, p+k]$, and the corresponding row  $Y$-run  is  $[q+1, q+k]$. This time, in considering $(\tbfGr,\bfGc)$, we replace 
the former one with $[p+2, p+k]$, and the latter one  with $[q+2, q+k]$, and add a trivial row $X$-run $[p+1,p+1]$ and
a trivial row $Y$-run $[q+1,q+1]$. The rest of nontrivial row $X$- and $Y$-runs as well as all column $X$- and $Y$-runs remain unchanged. In what follows, parameters $p$, $q$ and $k$ are assumed to be fixed.

Let $\tilde\bL=\bL_{\tbfGr(\overleftarrow{\Delta}^\er),\bfGc}$, 
$\tilde\J=\J_{\tbfGr(\overleftarrow{\Delta}^\er),\bfGc}$, and let the functions
${\tt \tilde f}_{ij}(X,Y)$ and $\tilde f_{ij}(X)$ be defined via the same expressions as ${\tt f}_{ij}(X,Y)$ and $f_{ij}(X)$
with $\bL$ and $\J$ replaced by $\tilde\bL$ and $\tilde\J$. A suitable version of Theorem \ref{prototype} can be stated as follows.

\begin{theorem}
\label{matrixmap2}
Let $Z=(z_{ij})$ be an $n\times n$ matrix. Then there exists a unipotent upper triangular $n\times n$ matrix $\nuu(Z)$ 
whose entries are rational functions in $z_{ij}$ with denominators equal to powers of $\tilde f_{p+2,1}(Z)$
such that for $X=\nuu(Z) Z$ and for any $i,j\in [1,n]$,
$$
f_{ij} (X) =\begin{cases}  
{\tilde f_{ij} (Z)}{\tilde f_{p+2,1} (Z)}\quad&\text{if $\J (i,j) = (\L^*, s)$ and $s< s_2$},\\
\tilde f_{ij} (Z)\quad &\text{otherwise},
\end{cases}
$$
where $\L^*\in\bL$ is the $2$-piercing matrix in $\bL$.
\end{theorem}

\begin{proof} Our approach is similar to that in the proof of Theorem \ref{matrixmap1}. 

For any $\L(X,Y) \in \bL$ define $\tilde\L(X,Y)$ obtained from   $\L(X,Y)$ by removing the first row from 
every building block of the form $X_{[p+1,N]}^{J}$. In particular, 
if $\L(X,Y)$ does not have building blocks like that  then $\tilde\L(X,Y)=\L(X,Y)$. 

Similarly to the previous case, all matrices $\tilde \L$ defined above are irreducible except for the one obtained from  
the $2$-piercing matrix $\L^*$. The corresponding matrix $\tilde\L^*$ has two irreducible diagonal blocks 
$\tilde\L_{1}^*$, $\tilde\L_{2}^*$ of sizes $s_2-1$ and $N(\L^*) - s_2+1$, respectively. As was already noted in 
Section \ref{outline}, all maximal alternating paths in $\BD_{\Gammar,\Gammac}$ are preserved in 
$\BD_{\tbfGr (\overleftarrow{\Delta}^\er),\bfGc}$ except for the path that goes through the directed inclined edge
$(p+1)\to (q+1)$. The latter one is split into two: the initial segment up to the vertex $p+1$ and the closing segment
starting with the vertex $q+1$. Consequently, $\tilde{\bL}=\{\tilde \L\: \L\in \bL, \L\ne \L^*\}\cup
\{\tilde\L_{1}^*,\tilde\L_{2}^*\}$. 

As before, if $\J (i,j) = (\L, s)$ and $\L\ne \L^*$ 
then $\tilde{\J}(i,j) =(\tilde\L, s)$. Furthermore, if $\L \in \bL_\varnothing$ then 
additionally ${\tt f}_{ij}(X,Y)$ and  ${\tilde {\tt f}}_{ij}(X,Y)$ coincide. 
However, if $\J (i,j) = (\L^*, s)$ then 
\[
\tilde{\J}(i,j) =
\begin{cases} (\tilde\L_{1}^*, s)\quad& \text{for $s=s(i,j)  < s_2$},\\  
         (\tilde\L_{2}^*, s - s_2 +1)\quad &\text {for $s=s(i,j)  \geq s_2$}.
\end{cases}				
\]

It follows from the above discussion that the claim of the theorem is an immediate corollary of the equalities
\eqref{formeri} for any $\L \in \bL$ and $s\in [1, N(\L)]$.

Let
\begin{equation}\label{mm02}
\nuu_0(Z) = \one_n + \sum_{\varkappa=2}^{k} \alpha_\varkappa e_{q+1,q+\varkappa}
\end{equation}
and
\begin{equation*}
\nuu(Z) ={\stackrel {\leftarrow}  {\prod}_{t\geq 0}}\gamma^t (\nuu_0(Z)).
\end{equation*}
As before,  the invariance property \eqref{2.1} 
allows to reduce the problem to selecting  parameters $\alpha_\varkappa=\alpha_\varkappa(Z)$ such  that
the analog of \eqref{mm2} with $\nuu_0(Z)$ given by \eqref{mm02} is satisfied for all $\L \in \bL$ and 
$s\in [1, N(\L)]$. 

Once again, this relation is satisfied for any choice of $\alpha_\varkappa$ if 
 $\L \in \bL_\varnothing$, that is, if $\L(X,Y)=\tilde \L(X,Y)$, while for   
matrices $\L \in \bL\setminus\bL_\varnothing$ one has to replace 
$\L(Z, \nuu_0 Z)$ by the matrix $\hat\L(Z,\nuu_0 Z)$ similar to the one defined in the proof of 
Theorem~\ref{matrixmap1}.  
Therefore, in what follows we aim at proving the analog of \eqref{mm2r} 
 for all $\L \in \bL\setminus\bL_{\varnothing}$ and $s\in [1, N(\L)]$. 

We can again use decomposition \eqref{uglymatrix} for $\hat\L$ and $\tilde \L$, except that now $\tilde B^r_1$ is obtained from 
$\hat B^r_1$ by replacing the first row with zeros,  whereas the last row of $\tilde A_2^r$ remains as is, 
unlike the previous case. Consequently, for $s\geq s_{r_1}$ the analog of \eqref{mm2r} is satisfied trivially.

For $s_{r_2}\le s \le s_{r_1}-1$, we are in the situation covered by Lemma \ref{blockmatrix2} with 
$\M=\hat\L^{M_{r_2}}_{M_{r_2}}$, $\tilde\M=\tilde\L^{M_{r_2}}_{M_{r_2}}$, $N=\theta_{r_2}-r_2+1$, 
$N_2=\theta_{r_1} -{r_1}+1$, and $k_1=r_1-1$. Condition (iv) in the lemma is satisfied trivially, since in this case 
$B_{[N_1-k_1+2,N]}=\tilde B_{[N_1-k_1+2,N]}$. Consequently, the analog of \eqref{mm2r} 
holds true if the parameters $\alpha_\varkappa=\alpha_{\varkappa}(Z)$ satisfy equations  
\begin{equation}
\label{eqs_r2}
\sum_{\varkappa \in [1,k]\setminus S}  (-1)^{\varepsilon_{\varkappa S}} \alpha_\varkappa 
\det (\hat B^{r_1})_{S\cup\{\varkappa \}\cup [k+1, \theta_{r_1}]}=0     
\end {equation}
for any  $(k-r_1)$-element subset $S$ in $[2, k]$. 

Continuing in the same way as in the proof of Theorem~\ref{matrixmap1} and using Lemma~\ref{blockmatrix2} instead of 
Lemma~\ref{blockmatrix}, we conclude that if conditions 
\begin{equation}
\label{condition_lemma22}
\sum_{\varkappa \in [1,k]\setminus S}  (-1)^{\varepsilon_{\varkappa S}} \alpha_\varkappa 
\det (\hat B^r)_{S\cup\{\varkappa \}\cup [k+1, \theta_r]} = 0
\end {equation}
are satisfied for any $r\in \{r_1,\dots,r_l\}$ and  any  $(k-r)$-element subset $S$ in $[2, k]$, 
then the analog of \eqref{mm2} holds for any $s\in [1, N(\L)]$. 

In particular, when $r=2$, and hence $\L=\L^*$, every  $S$ in
\eqref{condition_lemma22} is obtained by removing a single index $\varkappa$ from $[2,k]$. Therefore, 
the sum in the left hand side of \eqref{condition_lemma22} is taken over a two-element set $\{1,\varkappa\}$
with $\varkappa\in [2,k]$. Since $\varepsilon_{1S}=k-2$ and $\varepsilon_{\varkappa S}=k-\varkappa$,  
$\alpha_\varkappa$ is determined uniquely as
\begin{equation}
\label{alphaB2}
\alpha_\varkappa (Z) = (-1)^{\varkappa-1} \frac{ \det (\hat B^2)_{[1,\theta_2]\setminus\{\varkappa\}}}
{ \det (\hat B^2)_{ [2,\theta_2]}},  \quad \varkappa =1,\ldots, k.
\end{equation}
Therefore \eqref{condition_lemma22} is equivalent to vanishing of
\begin{equation}
\label{snovauzhas}
\sum_{\varkappa \in [1,k]\setminus S}  (-1)^{\varepsilon_{\varkappa S} +\varkappa} 
\det (\hat B^2)_{[1,\theta_2]\setminus\{\varkappa\}} 
\det (\hat B^r)_{S\cup\{\varkappa \}\cup [k+1,  \theta_r]} = 0.
\end {equation}

Denote $\bar S=[1,k]\setminus S$, then $\varepsilon_{\varkappa S}+\varepsilon_{\varkappa\bar S}=k-\varkappa$, and hence
\eqref{snovauzhas} can be re-written as 
\begin{equation*}
(-1)^k\sum_{\varkappa \in \bar S}  (-1)^{\varepsilon_{\varkappa\bar S}} \det (\hat B^2)_{(\bar S\setminus \{\varkappa\}) 
\cup S \cup [k+1,\theta_2]} \det (\hat B^r)_{\{\varkappa \}\cup S\cup [k+1,  \theta_r]} = 0.
\end {equation*}
The latter equation is similar to \eqref{uzhas} in the proof of Theorem \ref{matrixmap1}, and the current proof can be completed in exactly  the same way taking into account that the denominator in \eqref{alphaB2} equals $\tilde f_{p+2,1}(Z)$.
\end{proof}

There are two more versions of Theorem \ref{prototype} relating the cluster structures $\CC_{\bfGr,\bfGc}$ and
$\CC_{\bfGr,\tbfGc}$, where $\tbfGc=\tbfGc(\overrightarrow{\Delta}^\ec)$ or $\tbfGc=\tbfGc(\overleftarrow{\Delta}^\ec)$
for a nontrivial column $X$-run $\Delta^\ec$. They are obtained easily from Theorems \ref{matrixmap1} and~\ref{matrixmap2}
via the involution
\[
\bL_{\bfGr,\bfGc} \ni \L(X,Y) \mapsto \L(Y^T,X^T)^T\in \bL_{\bfG^\ec_{\rm opp},\bfG^\er_{\rm opp}},
\]
where $\bfG_{\rm opp}=(\Gamma_2,\Gamma_1, \gamma^{-1}: \Gamma_2\to\Gamma_1)$ is the {\em opposite\/} BD triple to
$\bfG=(\Gamma_1,\Gamma_2,\gamma:\Gamma_1\to\Gamma_2)$. Consequently, $X$ is obtained from $Z$ via multiplication by a lower triangular matrix, and the distinguished function $\tilde f_v(Z)$ equals $\tilde f_{1,q+k}(Z)$ for 
$\tbfGc=\tbfGc(\overrightarrow{\Delta}^\ec)$ and equals $\tilde f_{1,q+2}(Z)$ for $\tbfGc=\tbfGc(\overleftarrow{\Delta}^\ec)$.

\subsection{Handling adjacent clusters}\label{adjcl}

Let us continue the comparison of cluster structures $\CC=\CC_{\bfGr,\bfGc}$ and $\tilde\CC=\CC_{\tbfGr,\bfGc}$, where  
$\tbfGr=\tbfGr(\overrightarrow{\Delta}^\er)$. Recall that the corresponding initial quivers $Q$ and $\tilde Q$ differ as follows. The vertex $v=(p+k,1)$ is frozen in $\tilde Q$, but not in $Q$. Three of the edges incident to the vertex  $(p+k,1)$ in 
$Q$---the one connecting it to the vertex $(p+k-1,1)$ and the two connecting it to the vertices $(\gammar(p+k-1),n)$ and
$(\gammar(p+k-1)+1,n)$---are absent in $\tilde Q$ (in more detail, the 
neighborhood of $v$ in $Q$ looks as shown in Fig.~\ref{fig:i1nei}(b), Fig.~\ref{fig:n1nei}(a), or Fig.~\ref{fig:n1nei}(b),
while the neighborhood of $v$ in $\tilde Q$ looks as shown in Fig.~\ref{fig:i1nei}(d), Fig.~\ref{fig:n1nei}(c), or 
Fig.~\ref{fig:n1nei}(d), respectively). 

As it was explained in Section \ref{outline}, we have to establish an analog of Theorem~\ref{prototype} for the fields 
$\FF'=\C(\fy_{11},\dots,\fy'_u,\dots,\fy_{nn})$ and $\tilde\FF'=\C(\tfy_{11},\dots,\tfy'_u,\dots\tfy_{nn})$ and the map
$T': \FF'\to\tilde\FF'$ given by 
\begin{equation}\label{Tprime}
T'(\fy_{ij})=\begin{cases} T(\fy_{ij}) \quad &\text{for $(i,j)\ne u$,}\\
                           \tfy'_u\tfy^{\lambda_u}_v \quad & \text{for $(i,j)=u$}
						\end{cases}							
\end{equation}
for some integer $\lambda_u$, where $T: \FF\to \tilde \FF$ is the map constructed in Theorem \ref{matrixmap1}.  
The map $U:\X\to\ZZ$ is also borrowed from Theorem \ref{matrixmap1}, so condition b) in Theorem~\ref{prototype} 
holds true. Condition c) follows 
immediately from \eqref{Tprime}. Condition a) reads $\tilde f'\circ T'=U\circ f'$.

Recall that cluster mutation formulas provide isomorphisms $\mu: \FF'\to\FF$ and $\tilde\mu:\tilde \FF'\to\tilde\FF$ such that 
$f'=f\circ\mu$ and $\tilde f'=\tilde f\circ\tilde\mu$. Consequently, condition a) above would follow from 
$\tilde\mu\circ T'=T\circ\mu$. The latter statement can be reformulated as follows.  

\begin{proposition} \label{phi-tilde-phi} 
Let $\tilde\psi$ be the cluster variable in $\CC(\tilde Q, \tfy)$ obtained via a sequence of mutations at vertices $(i_1,j_1), \ldots, (i_N,j_N)$ in $\tilde Q$ avoiding $v$, and let $\psi$ be a cluster variable in 
$\CC(Q, \fy)$ obtained via the same sequence of mutations in $Q$. Then $\psi = \tilde \psi \tilde \fy^{\lambda_u}_{v}$ for some integer $\lambda_u$.
\end{proposition}

\begin{proof}
Define a quiver $Q_v$ by freezing the vertex $v$ in $Q$ and retaining all the edges from $v$ to non-frozen vertices. Then any sequence of mutations in $Q$ avoiding $v$ translates into the sequence of mutations in $Q_v$, and all the resulting cluster 
variables in $\CC(Q,\fy)$ and $\CC(Q_v,\fy)$ coincide. We will use 
the statement that describes the relation between cluster variables in two cluster structures 
whose initial quivers are ``almost the same''. That is, there is a bijection between vertices of these quivers that restricts to the bijection of subsets of frozen vertices and under this bijection the two quivers differ only in terms of edges incident to one specified frozen vertex. 

\begin{lemma}\cite[Lemma 8.4]{GSVMem} \label{MishaSha}
Let $\wB$ and $B$ be integer $n\times (n+m)$ matrices that differ in the last column only.
Assume that there exist $\tilde w, w\in\R^{n+m}$ such that $\wB\tilde w=Bw=0$ and
$\tilde w_{n+m}=w_{n+m}=1$. Then for any cluster $(x_1',\dots,x_{n+m}')$ in $\CC(\wB)$ there exists a collection of numbers
$\lambda_i'$, $i\in [1,n+m]$,  such that $x_i' x_{n+m}^{\lambda_i'}$ satisfy exchange relations of 
the cluster structure $\CC(B)$. In particular, for the initial cluster $\lambda_i=w_i-\tilde w_i$,
$i\in [1,n+m]$.
\end{lemma}

 In our current situation, $\wB$ and $B$ are adjacency matrices of quivers $\tilde Q$ and $Q_v$, respectively. 
The last columns of $\wB$  and $B$  correspond to the frozen vertex $(p+k,1)$. To establish the claim of 
Proposition \ref{phi-tilde-phi}, we just need to define appropriate weights $\tilde w$ and $w$ and to
show that for any non-frozen vertex $(i,j)$, $\lambda_{ij}= w_{ij}-\tilde w_{ij}$ coincides with the exponent
of $\tilde f_{p+k,1}(Z)$ in the right hand side of the expression for $f_{ij}(X)$ in Theorem \ref{matrixmap1}.

Put $\tilde d_{ij}=\deg\tilde f_{ij}(Z)$ and $d_{ij}=\deg f_{ij}(X)$. A direct check proves that the vectors
$\tilde d=(\tilde d_{ij})$ and $d=(d_{ij})$ satisfy relations $\wB\tilde d=Bd=0$. Besides, 
$\tilde d_v=d_v=\delta$, and hence vectors $\tilde w=\frac1\delta\tilde d$ and $w=\frac1\delta d$ 
satisfy the conditions of
Lemma \ref{MishaSha}. Moreover, $\tilde d_{ij}$ and $d_{ij}$ coincide for any $f_{ij}$ that is a minor of $\L\ne\L^*$,
or a minor of $\L^*$ with $s(i,j)\ge s_k$. If $f_{ij}$ is  a minor of $\L^*$ with $s(i,j)> s_k$ then 
$d_{ij}-\tilde d_{ij}=\delta$. Consequently $\lambda_{ij}$ satisfies the required condition. 
\end{proof}

\subsection{Base of induction: the case $|\Gamma^\er_1|+|\Gamma^\ec_1|=1$}

It suffices to consider the case $|\Gamma^\er_1|=1$, $|\Gamma^\ec_1|=0$, the other case can then be treated via taking the
opposite BD triple. In this case all the reasoning exhibited in Sections \ref{matrixmaps} and \ref{adjcl} is still valid, so 
to complete the proof we only need to check that every matrix element $x_{\alpha\beta}$ can be expressed as a Laurent polynomial in terms of cluster variables in the cluster $\mu_v(F)$. We will do this directly. 

Let $\bfGr =( \{p\}, \{q\}, p\mapsto q)$ with $q\ne p$ and $\bfGc =\varnothing$.  The functions
forming the initial cluster $F_{\bfGr, \varnothing}$ are $f_{ij}(X)= \det X_{[i,n]}^{[j, n-i +j]}$ for $i\geq j$, 
$f_{ij}(X)= \det X^{[j,n]}_{[i, n-j +i]}$ for $i < j$, $j-i \ne n-q$, and 
$f_{i,n-q+i}(X)= \det \L_{[i,N]}^{[i,N]}$ for $i\in [1, q]$, where $N=n-p+q$ and the $N\times N$ matrix $\L$ is given by
\begin{equation}
\label{easyL}
\L=\begin{bmatrix}  X_{[1,q-1]}^{[n-q+1,n]} & 0\\  X_{[q,q+1]}^{[n-q+1,n]}  &  X_{[p,p+1]}^{[1,n-p]} \\ 0 & X_{[p+2,n]}^{[1,n-p]}
\end{bmatrix}.
\end{equation}
These last $q$ functions distinguish  $F_{\bfGr, \varnothing}$ from $F_{\varnothing, \varnothing}$ that forms an initial cluster for the standard cluster structure on $GL_n$. Also, the function
$f_{p+1,1}(X)= \det X_{[p+1,n]}^{[1, n-p]}$ is a frozen variable in $\CC_{\varnothing, \varnothing}$, but is mutable in 
$\CC_{\bfGr, \varnothing}$. The mutation at $v=(p+1,1)$ transforms $f_{p+1,1}(X)$ into
\begin{equation}\label{newfv}
\begin{aligned}
f'_{p+1,1}(X)&= \frac { f_{p1}(X) f_{p+2,2 }(X) f_{q+1,n}(X) + f_{p+1,2 }(X) f_{qn}(X)}{f_{p+1,1}(X)}\\ &= 
\det \begin{bmatrix}    X_{[q,q+1]}^{[n]}  &  X_{[p,p+1]}^{[2,n-p+1]} \\ 0 & X_{[p+2,n]}^{[2,n-p+1]}
\end{bmatrix}
\end{aligned}
\end{equation}
with $f_{p+2,2}(X)=1$ in case $p=n-1$, see Fig.~\ref{fig:i1nei}(b) and~\ref{fig:n1nei}(b). 
The last equality follows from the short Pl\"ucker relation based on columns $1,2,3, n-p+3$ applied to the 
$(n-p+1)\times (n-p+3)$ matrix
\[
\begin{bmatrix}    \begin{array}{c} 1\\ 0 \end{array} &X_{[q,q+1]}^{[n]}  &  X_{[p,p+1]}^{[1,n-p+1]} \\ 
0 &0 & X_{[p+2,n]}^{[1,n-p+1]}
\end{bmatrix}.
\]

Observe that $\{f_{ij}(X) = f_{ij}\left(X_{[q+1,n]}^{[1,n]}\right): i\in [q+1,n], j\in [1,n]\}$ together with the 
restriction of  $Q_{\varnothing,\varnothing}$ to its lower $n-q$ rows and freezing row $q+1$ form an initial cluster for the standard cluster 
structure $\CC_q$ on $(n-q)\times n$ matrices. It follows immediately from \cite[Prop.~4.15]{GSVb} that every minor of 
$X$ with the row set in $[q+1,n]$ is a cluster variable in $\CC_q$, and hence can be written as a Laurent polynomial in any cluster of $\CC_q$. Note that for $p>q-2$ the variable $f_{p+1,1}(X)$ is frozen in $\CC_q$, therefore, 
by~\cite[Prop.~3.20]{GSVb}, it does not enter the denominator of this Laurent polynomial; for $p\le q-2$ this variable 
does not exist in $\CC_q$. Consequently, all such minors remain Laurent polynomials in the cluster adjacent to the initial one in $\CC_{\bfGr,\varnothing}$ after the mutation at $(p+1,1)$.
In particular, for any $i\in [q+1,n]$, $j\in [1,n]$,  $x_{ij}$   can be written as a Laurent polynomial in this cluster. 

For $s\le q-1$, consider the sequence of consecutive mutations at $(s+1,n), \ldots, (s+1,s), (s+1,s+1), \ldots, (s+1,2)$ starting with the initial cluster in $\CC_{\bfGr,\varnothing}$ and denote the obtained cluster variables $ f'_{s+1,n-t+1}(X)$, 
$t\in [1,n-1]$.
The same sequence of mutations in $\CC_{\varnothing,\varnothing}$ produces cluster variables 
\begin{equation}\label{tfasminors}
\begin{aligned}
\tilde f'_{s+1,n-t+1}(Z)&=\det Z_{\{s\}\cup [s+2,s+t+1]}^{[n-t,n]}, \quad t\in [1,n-s -1],\\ 
\tilde f'_{s+1,n -t+1}(Z)&=\det Z_{\{s\}\cup [s+2,n]}^{[n-t,2n-t -s-1]}, \quad t\in [n-s,n -1]. 
\end{aligned}
\end{equation}
Indeed, every mutation in the sequence is applied to a four-valent vertex, and we obtain consecutively 
\[
\tilde f'_{s+1,n}(Z)=
\frac{\tilde f_{s,n-1}(Z) \tilde f_{s+2,n}(Z) + \tilde f_{s+1,n-1}(Z) \tilde f_{sn}(Z)}{\tilde f_{s+1,n}(Z)}
\]
and
\[ 
\tilde f'_{s+1,n-t}(Z)= 
\frac{\tilde f_{s,n-t-1}(Z) \tilde f_{s+2,n-t}(Z) + \tilde f_{s+1,n-t-1}(Z) \tilde f'_{s+1,n-t+1}(Z)}{\tilde f_{s+1,n-t}(Z)}
\]
for $t\in [1,n-2]$. Explicit formulas \eqref{tfasminors} now follow by applying an appropriate version of the short Pl\"ucker relation.

Recall that by Theorem \ref{matrixmap1},  $X$ and $Z$ differ only in the $q$-th row. Moreover, every minor of $X$ whose row set either does not contain $q$ or contains both  $q$ and $q+1$ is equal to the corresponding minor of $Z$. Let $\tilde\psi(Z)$ be such a minor; invoking  once again \cite[Prop.~4.15]{GSVb}, one can obtain it by a sequence of mutations in 
$\CC_{\varnothing,\varnothing}$. Let $\psi(X)$ be the cluster variable obtained by applying the same sequence of mutations
to the initial seed of   $\CC_{\bfGr,\varnothing}$.  By Proposition~\ref{phi-tilde-phi}, 
$\psi(X) = \tilde\psi(Z) \left (f_{p+1,1}(Z)\right )^\lambda=\tilde\psi(X) \left (f_{p+1,1}(X)\right )^\lambda$ 
for some integer $\lambda$. Clearly, minors in \eqref{tfasminors} satisfy the above condition unless $s+t+1=q$, and hence
\[
f'_{s+1,n-t+1}(X) = \tilde f'_{s+1,n-t+1}(X) \left (f_{p+1,1}(X)\right )^{\lambda_{s+1,n-t+1}}
\]
for $t\ne q-s-1$.
However, the exponents $\lambda_{s+1,n-t+1}$ are easily computed to be all zero. Thus, we conclude that 
\begin{equation}\label{firsttype}
\det X_{\{s\}\cup [s+2,s+t+1]}^{[n-t,n]}= f'_{s+1,n-t+1}(X), \quad t\in [1,n-s -1]\setminus \{q-s-1\}, 
\end{equation}
and
\begin{equation}\label{sectype}
 \det X_{\{s\}\cup [s+2,n]}^{[n-t,2n-t -s-1]}=f'_{s+1,n-t+1}(X), \quad t\in [n-s,n -1],
 \end{equation}
 are cluster variables in  $\CC_{\bfGr,\varnothing}$.

Now we are ready to deal with the entries in the $q$-th row $X$. First, expand $f'_{p+1,1}(X)$ in \eqref{newfv} by the 
first column as
\[
f'_{p+1,1}(X)= x_{qn} f_{p+1,2}(X) + x_{q+1,n} \det X_{\{p\}\cup [p+2,n]}^{[2, n-p+1]}.
\] 
For $p>q$, the row set of $\det X_{\{p\}\cup [p+2,n]}^{[2, n-p+1]}$ lies completely within the last $n-q$ rows of $X$, and hence,
as explained above, it is a Laurent polynomial in the cluster we are interested in. For $p<q$, this determinant is a 
cluster variable in $\CC_{\bfGr,\varnothing}$ by \eqref{sectype} with $t=n-2$, and hence it is a Laurent polynomial in any
cluster in $\CC_{\bfGr,\varnothing}$. Consequently, in both cases $x_{qn}$ is a Laurent polynomial 
in the cluster we are interested in. Further, this claim can be established inductively for $x_{q,n-1}, x_{q,n-2},\ldots, x_{q 1}$ 
by expanding first the minors $f_{q,n-t}(X)=\det X_{[q,q+t]}^{[n-t,n]}$, $t\in [1,n-q]$, and then the minors 
$f_{q,n-t}(X)=\det X_{[q,n]}^{[n-t,2n-t-q]}$, $t\in [n-q+1, n-1]$, by the first row as 
$f_{q,n-t}(X)  = x_{q,n-t} f_{q+1,n-t+1}(X) + P(x_{q, n-t+1},\ldots, x_{qn}, x_{ij}\ : i > q)$,
where $P$ is a polynomial. 
 
Finally, for $s < q$, $x_{sn}$ is a cluster variable in $\CC_{\bfGr,\varnothing}$, and hence is a Laurent polynomial in any cluster.   
For $t=1,  \ldots, q-s-1$, Laurent polynomial expressions for $x_{s,n-t}$ can obtained recursively using expansions of the cluster variable
$f_{s,n-t}(X)=\det X_{[s,s+t]}^{[n-t,n]}$ by the first row exactly as above. For $t=q-s,\dots,n-s-1$, such expressions are obtained 
recursively by expanding the cluster variable $f'_{s+1,n-t+1}(X)$ given by \eqref{firsttype} by the first row as  
$f'_{s+1,n-t+1}(X) = x_{s,n-t} f_{s+2,n-t+1}(X) + P'(x_{s, n-t+1},\ldots, x_{sn}, x_{ij}\ : i > s)$, where $P'$ is a polynomial.
For $t=n-s,\dots,n-1$ we use the same expansion for $f'_{s+1,n-t+1}(X)$ given by \eqref{sectype}. This completes the proof.

\begin{remark}
In fact, one can show that every minor of $X$ whose row set either does not contain $q$ or contains both $q$ and $q+1$ is a 
cluster variable in $\CC_{\bfGr,\varnothing}$.
\end{remark}

\subsection{Auxiliary statements}\label{auxi}
In this section we collected several technical statements that were used before.

\begin{lemma}
\label{blockmatrix}
Let $N=N_1 + N_2$, $k = k_1 + k_2$, and let $\M$, $\tilde \M$ be two $N\times N$ matrices 
\begin{equation}\label{twomatrices}
\M = \left [
\begin{array}{@{}cc@{} }
A_1 & 0\\
A_2 & B_1\\
0 & B_2
\end{array}
\right ],\qquad
\tilde\M = \left [
\begin{array}{@{}cc@{} }
\tilde A_1 & 0\\
\tilde A_2 & \tilde B_1\\
0 & \tilde B_2
\end{array}
\right ],
\end{equation}
with block rows of sizes $N_1-k_1$, $k$ and $N_2-k_2$ and block columns of sizes $N_1$ and $N_2$.
Assume that

{\rm (i)} $A_1=\tilde A_1$;

{\rm (ii)} there exists $A'_2$ such that $A_2=\left (\one_k + \sum_{i=1}^{k-1} \alpha_i e_{ik} \right )A'_2$ 
and $\tilde A_2$ is obtained from $A'_2$ by replacing the last row with zeros;

{\rm (iii)} every maximal minor of $B=\begin{bmatrix} B_1\\ B_2\end{bmatrix}$ that contains the last $N_2 - k_2$ rows 
coincides with the corresponding minor of $\tilde B=\begin{bmatrix} \tilde B_1\\ \tilde B_2\end{bmatrix}$.

Then conditions
\begin{equation}
\label{condition_lemma}
\sum_{\varkappa \in S}  (-1)^{\varepsilon_{\varkappa S}} \alpha_\varkappa 
\det B_{S\setminus\{\varkappa \}\cup [k+1, N_2+k_1]}=0
\end {equation}
for any $S\subset [1,k]$ such that $|S|=k_2+1$ and $k\in S$ guarantee that
\begin{equation}
\label{principal}
\det \M_{[s,N]}^{[s,N]} = \det \tilde \M_{[s,N]}^{[s,N]}
\end{equation}
for all $s\in [1,N]$; here $\varepsilon_{\varkappa S} = \# \{ i\in S : i > \varkappa\}$
and $\alpha_k=1$.
\end{lemma}

\begin{proof}
Denote
\[ 
\xi_s=\det \M_{[s, N]}^{[s, N]},\qquad \tilde\xi_s=\det \tilde\M_{[s, N]}^{[s, N]}.
\]
By condition (iii), we only need to consider $s \leq N_1$. 
First, fix  $s \in [N_1 - k_1 +1, N_1]$, which means that $\M_{ss}$ is in the block $A_2$. 
We use the Laplace expansion of $\xi_s$ and $\tilde\xi_s$ with respect to the second block column. Define $t=s-N_1+k_1$, then
\begin{equation}
\label{laplace1}
\begin{aligned}
\xi_s &= \sum_{T} (-1)^{\varepsilon_T} \det( A_2)_{T}^{\Theta} \det B_{\bar T\cup[k+1, N_2+k_1]},\\
\tilde\xi_s &=  \sum_{T} (-1)^{\varepsilon_T} \det( \tilde A_2)_{T}^{\Theta}  \det \tilde B_{\bar T\cup [k+1, N_2+k_1]},
\end{aligned}
\end{equation}
where  the sum is taken over all  $(N_1 -s +1)$-element subsets $T$ in $[t,  k]$, 
$\bar T=[t, k]\setminus T$, $\Theta=[s, N_1]$  
and $\varepsilon_T=\sum_{i\in T}i+\varepsilon_s$ with $\varepsilon_s$ depending only on $s$.

By condition (ii), 
\begin{equation}
\label{Ztheta1}
\det (A_2)_{T}^{\Theta} =
\begin{cases}
\det ( A_2')_{T}^{\Theta} &\  \mbox{if}\ k \in T,\\
\det ( A_2')_{T}^{\Theta} + \sum\limits_{\varkappa \in T} 
(-1)^{\varepsilon_{\varkappa T}}\alpha_\varkappa\det (A_2')_{\left( T\setminus \{\varkappa\}\right ) \cup \{k\}}^{\Theta} &\  \mbox{if}\ k \notin T,
\end{cases}
\end{equation}
and
\begin{equation}
\label{Ztheta2}
\det (\tilde A_2)_{T}^{\Theta} =\begin{cases}
0 &\ \mbox{if}\ k \in T,\\
\det ( A_2')_{T}^{\Theta} &\  \mbox{if}\ k \notin T.
\end{cases}
\end{equation}
Besides, $\det B_{\bar T\cup[k+1, N_2+k_1]}=\det \tilde B_{\bar T\cup [k+1, N_2+k_1]}$ by condition (iii). 
Therefore, the difference $\xi_s - \tilde\xi_s$ can be written as a linear combination of $\det(A_2')_{T}^{\Theta}$ 
such that $k \in T$. Let $T = T'\cup \{k\}$;  define $S=\bar T'=\bar T\cup\{k\}$, then $|S|=k_2+1$ and $k\in S$. The
coefficient at $\det(A_2')_{T}^{\Theta}$ equals, up to a sign,
\begin{multline}
  \sum_{\varkappa\in [t,k]\setminus T'}  (-1)^{\varepsilon_{\varkappa, T'\cup\{k\}}+\varkappa} \alpha_\varkappa \det B_{(S\setminus\{\varkappa\})\cup [k+1, N_2+k_1]} \\
=(-1)^{k} \sum_{\varkappa \in S}  (-1)^{\varepsilon_{\varkappa S}} \alpha_\varkappa 
\det B_{(S\setminus\{\varkappa \})\cup [k+1, N_2+k_1]},
\end {multline}
since $\varepsilon_{\varkappa, T'\cup\{k\}}+\varepsilon_{\varkappa S}=k-\varkappa$.
Thus  for \eqref{principal} to be valid for $s \in [N_1 - k_1 +1, N_1]$
it is sufficient that \eqref{condition_lemma} be satisfied for any $S\subset [t,k]$, 
$|S|=k_2+1$, $k\in S$.
In fact, since \eqref{Ztheta1} and \eqref{Ztheta2} remain valid for any set $\Theta \subset [1,N_1]$ of size  $|\Theta|=N_1-s+1$,  
similar considerations show that \eqref{condition_lemma} implies
\begin{equation}
\label{Thetas}
\det \M_{[s,N]}^{\Theta\cup [N_1+1,N]} = \det \tilde \M_{[s,N]}^{\Theta\cup [N_1+1,N]}
\end{equation}
for any such $\Theta$ and $s \in [N_1 - k_1 +1, N_1]$. This, in turn, results in \eqref{principal} being valid for all 
$s\in [1, N_1 - k_1]$. To see this, one has to use  the Laplace expansion of $\xi_s$ and $\tilde\xi_s$ with respect to the block row $[s,N_1-k_1]$: 
\begin{align*}
\xi_s &= \sum_{\Theta} (-1)^{\varepsilon_{\bar\Theta}} \det( A_1)_{[s,N_1-k_1]}^{\bar\Theta} 
\det \M_{[N_1-k_1+1, N]}^{\Theta\cup [N_1+1,N]},\\
\tilde\xi_s &= \sum_{\Theta} (-1)^{\varepsilon_{\bar\Theta}} \det(\tilde A_1)_{[s,N_1-k_1]}^{\bar\Theta} 
\det \tilde\M_{[N_1-k_1+1, N]}^{\Theta\cup [N_1+1,N]},
\end{align*}
where $\bar\Theta=[s,N_1]\setminus\Theta$, and the sums are taken over all subsets $\Theta$ in $[s,N_1]$ of size 
$|\Theta|=k_1$. It remains to note that $\det( A_1)_{[s,N_1-k_1]}^{\bar\Theta} =\det(\tilde A_1)_{[s,N_1-k_1]}^{\bar\Theta}$
by condition (i), and $ \det \M_{[N_1-k_1+1, N]}^{\Theta\cup [N_1+1,N]}=\det \tilde\M_{[N_1-k_1+1, N]}^{\Theta\cup [N_1+1,N]}$
is a particular case of \eqref{Thetas} for $s=N_1-k_1 +1$.
\end{proof}

\begin{lemma}
\label{blockmatrix2}
Let $\M$ and $\tilde\M$ be two $N\times N$ matrices given by \eqref{twomatrices} with the same sizes of block rows and block 
columns. Assume that

{\rm (i)} $A_1=\tilde A_1$;

{\rm (ii)} $A_2=\left (\one_k + \sum_{i=2}^{k} \alpha_i e_{1i} \right )\tilde A_2$;

{\rm (iii)}  $\tilde B_1$ is obtained from $B_1$ by replacing the first row with zeros;

{\rm (iv)} every maximal minor of $B=\begin{bmatrix} B_1\\ B_2\end{bmatrix}$ that contains the last $N_2 - k_2$ rows and does not
contain the first row 
coincides with the corresponding minor of $\tilde B=\begin{bmatrix} \tilde B_1\\ \tilde B_2\end{bmatrix}$.

Then conditions 
\begin{equation}
\label{condition_lemma2}
\sum_{\varkappa \in [1,k]\setminus S}  (-1)^{\varepsilon_{\varkappa S}} 
\alpha_\varkappa \det B_{S\cup\{\varkappa \}\cup [k+1, N_2+k_1]}=0
\end {equation}
for any $S\subset [2,k]$ such that $|S|=k_2-1$ guarantee that
\begin{equation}
\label{principal2}
\det \M_{[s,N]}^{[s,N]} = \det \tilde \M_{[s,N]}^{[s,N]}
\end{equation}
for all $s\in [1,N]$; here $\alpha_1=1$.
\end{lemma}

\begin{proof} 
The proof is a straightforward modification of the proof of Lemma \ref{blockmatrix}.
For  $s \in [N_1-k_1+2, N_1]$, Laplace expansions of $\xi_s$ and $\tilde\xi_s$
with respect to the second block column are given by \eqref{laplace1}. By condition (ii),
$\det (A_2)_T^\Theta=\det(\tilde A_2)_T^\Theta$, while by condition (iv), 
$\det B_{\bar T\cup[k+1, N_2+k_1]}=\det \tilde B_{\bar T\cup [k+1, N_2+k_1]}$. Consequently,
$\xi_s-\tilde\xi_s$ vanishes, and hence \eqref{principal2} holds true.

For $s\in [1,N_1-k_1+1]$, the corresponding Laplace expansions are given by
\begin{equation*}
\begin{aligned}
\xi_s &= \sum_{T} (-1)^{\varepsilon_T} \det A_{[s,N_1-k_1]\cup T}^{[s,N_1]} \det B_{\overleftarrow{T}\cup[k+1, N_2+k_1]},\\
\tilde\xi_s &=  \sum_{T} (-1)^{\varepsilon_T}\det \tilde A_{[s,N_1-k_1]\cup T}^{[s,N_1]}  
\det \tilde B_{\overleftarrow{T}\cup [k+1, N_2+k_1]},
\end{aligned}
\end{equation*}
where $T$ runs over all  $k_1$-element subsets in $[N_1-k_1+1, N_1 + k_2]$
and $\overleftarrow{T}=\{ i-N_1+k_1 \: i\in\bar T\}$ for 
$\bar T=[N_1-k_1+1, N_1 + k_2]\setminus T$.

Next, by conditions (i) and (ii), 
\begin{equation*}
\det A_{\Xi\cup T}^{[s,N_1]}  =
\begin{cases}
\det \tilde A_{\Xi\cup T}^{[s,N_1]}   &\  \mbox{if}\ t \notin T,  \\
\det  \tilde A_{\Xi\cup T}^{[s,N_1]} + \sum\limits_{\chi \notin T}   
(-1)^{k_1-1-\varepsilon_{\chi T}}\alpha_\varkappa\det \tilde A_{\Xi\cup\left( T\setminus \{t\}\right ) \cup 
\{\chi\}}^{[s,N_1]} 
&\  \mbox{if}\ t \in T,
\end{cases}
\end{equation*}
where $\Xi=[s,N_1-k_1]$, $t=N_1-k_1+1$ and $\varkappa=\chi-N_1+k_1\in[1, k]$. Further, by conditions (iii) and (iv),
\begin{equation*}
\det \tilde B_{\overleftarrow{T}\cup [k+1, N_2+k_1]}=\begin{cases}
0 &\ \mbox{if}\ t \notin T,  \\
\det B_{\overleftarrow{T}\cup[k+1, N_2+k_1]} &\ \mbox{if}\ t \in T.
\end{cases}
\end{equation*}
Therefore, the difference $\xi_s - \tilde\xi_s$ can be written as a linear combination of $\det \tilde A_{\Xi\cup T}^{[s,N_1]}$
such that $t \notin T$. Let $\bar T=\{t\}\cup\bar T'$; define $S=\overleftarrow{T}'=\overleftarrow{T}\setminus\{1\}$, then 
$S\subset[2, k]$ and $|S|=k_2-1$. Consequently, the coefficient at $\det \tilde A_{\Xi\cup T}^{[s,N_1]}$ equals, up to a sign, 
\begin{equation*}
\sum_{\varkappa \in [1,k]\setminus S}  (-1)^{\varepsilon_{\varkappa S} } \alpha_\varkappa \det B_{S\cup\{\varkappa \}\cup [k+1, N_2+k_1]}, 
\end{equation*}
and the claim follows.
\end{proof}

\begin{lemma}
Let $A$ be a rectangular matrix, $I=(i_1,\ldots i_k)$ and $J$ be disjoint row sets, $L$ and $M$ be disjoint column sets, 
and $|L|=|J| +1$, $|M|=|I|-2$. Then 
\begin{equation}\label{plu2}
 \sum_{\lambda =1}^k (-1)^\lambda\det A_{\{i_\lambda\}\cup J}^L \det A_{(I\setminus \{i_\lambda\})\cup J}^{L\cup M}=0.
\end{equation}
\end{lemma}
 
\begin{proof} The formula can be obtained from standard Pl\"ucker relations via a natural interpretation of minors 
of $A$ as Pl\"ucker coordinates for $\left [ \one \ A\right]$.
\end{proof}


\begin{thebibliography}{GSV7}
\bibitem{BD} A.~Belavin and V.~Drinfeld,
\textit{Solutions of the classical Yang-Baxter equation for simple Lie algebras}.
Funktsional. Anal. i Prilozhen. {\bf16} (1982), 1--29.

\bibitem {BFZ}  A.~Berenstein, S.~Fomin, and A.~Zelevinsky,
\textit{Cluster algebras. III. Upper bounds and double Bruhat cells}. 
Duke Math. J. \textbf{126} (2005), 1--52.


\bibitem{CP} V.~Chari and A.~Pressley, \textit{A guide to quantum groups}.
Cambridge University Press, 1994.

\bibitem {Eis} I.~Eisner,
\textit{Exotic cluster structures on $SL_5$}.
J. Phys. A: Math. Theor. {\bf 47} (2014), 474002--474024.

\bibitem{Eis1} I.~Eisner,
\textit{Exotic cluster structures on $SL_n$ with Belavin-Drinfeld data of minimal size, I. The structure}.
Israel Math. J. {\bf 218} (2017), 391--443.

\bibitem{Eis2} I.~Eisner,
\textit{Exotic cluster structures on $SL_n$ with Belavin-Drinfeld data of minimal size, II. Correspondence between cluster structures and Belavin-Drinfeld triples}.
Israel Math. J. {\bf 218} (2017), 445--487.

\bibitem{ESS} P. Etingof, T. Schedler, O. Schiffmann, 
\textit{Explicit quantization of dynamical $R$-matrices for finite dimensional semisimple Lie algebras}. 
Journal of the AMS {\bf 13} (2000), 595--609.

\bibitem {FZ1}  S.~Fomin and A.~Zelevinsky, 
\textit{Cluster algebras.I. Foundations}. 
J. Amer. Math. Soc. \textbf{15} (2002), 497--529.

\bibitem {FZ2} S.~Fomin and A.~Zelevinsky, 
\textit{The Laurent phenomenon}.
Adv. in Appl. Math. {\bf 28} (2002), 119--144.

\bibitem {GSV1}  M.~Gekhtman, M.~Shapiro, and A.~Vainshtein,
\textit{Cluster algebras and Poisson geometry}.  
Mosc. Math. J. \textbf{3} (2003), 899--934.




\bibitem{GSV2}  M.~Gekhtman, M.~Shapiro, and A.~Vainshtein, 
\textit{Generalized B\"acklund-Darboux transformations of Coxeter-Toda flows from a cluster algebra perspective}.
Acta Math. {\bf 206} (2011), 245--310.


\bibitem{GSVb}  M.~Gekhtman, M.~Shapiro, and A.~Vainshtein,
\textit{Cluster algebras and Poisson geometry}.
Mathematical Surveys and Monographs, 167. American Mathematical Society, Providence, RI, 2010.

\bibitem{GSVM}  M.~Gekhtman, M.~Shapiro, and A.~Vainshtein,  
\textit{Cluster structures on simple complex Lie groups and Belavin--Drinfeld classification},
Mosc. Math. J. \textbf{12} (2012), 899--934.

\bibitem{GSVPNAS}  M.~Gekhtman, M.~Shapiro, and A.~Vainshtein,  
\textit{Cremmer--Gervais cluster structure on $SL_n$}.  
Proc.~Natl.~Acad.~Sci. \textbf{111} (2014), 9688--9695.

\bibitem{GSVMem}  M.~Gekhtman, M.~Shapiro, and A.~Vainshtein,  
\textit{Exotic cluster structures on $SL_n$: the Cremmer--Gervais case}.  
Memoirs of the AMS \textbf{246} (2017), no.~1165, 94pp.

\bibitem{GSVD}  M.~Gekhtman, M.~Shapiro, and A.~Vainshtein,  
\textit{Drinfeld double of $GL_n$ and generalized cluster structures},
Proc. Lond. Math. Soc. \textbf{116} (2018), 429--484.

\bibitem{GSVpest}  M.~Gekhtman, M.~Shapiro, and A.~Vainshtein, 
\textit{Periodic staircase matrices and generalized cluster structures},
Int. Math. Res. Notes (2022), 4181--4221.


\bibitem{GY} K.~Goodearl and M.~Yakimov,
\textit{Cluster algebra structures on Poisson nilpotent algebras},
preprint, arXiv:1801.01963.

\bibitem{r-sts}  A.~Reyman and M.~Semenov-Tian-Shansky,
\textit{Group-theoretical methods in the theory of
finite-dimensional integrable systems}. Encyclopaedia of
Mathematical Sciences, vol.16, Springer--Verlag, Berlin, 1994 pp.116--225.

\bibitem{Ya} M.~Yakimov, 
{\it Symplectic leaves of complex reductive Poisson-Lie groups}. Duke Math. J. \textbf{112} (2002),
453--509.
\end{thebibliography}
\end{document}